\documentclass[10pt,a4paper,reqno]{amsart}
\title{$\mathbb{P}^n$-functors}
\author{Rina Anno}
\email{ranno@math.ksu.edu}
\address{Department of Mathematics \\
Kansas State University \\
138 Cardwell Hall \\
Manhattan, KS 66506\\
USA}
\author{Timothy Logvinenko} 
\email{LogvinenkoT@cardiff.ac.uk} 
\address{School of Mathematics\\ 
Cardiff University\\
Senghennydd Road,\\
Cardiff, CF24 4AG\\
UK}
\usepackage{amsmath,amsfonts,amssymb,amsthm,epsfig,amscd,latexsym,comment}
\usepackage{caption}
\usepackage{tikz-cd}
\usepackage{graphicx}
\usepackage{array}
\usepackage{subfigure}
\usepackage{leftidx}
\usepackage{xparse}
\usepackage[all]{xy}
\newdir^{ (}{{}*!/-5pt/\dir^{(}}

\usepackage[colorlinks=true, pdfpagemode=none, pdfmenubar=false, linkcolor=blue, citecolor=blue, urlcolor=blue]{hyperref}

\let\amsamp=&

\begingroup
\catcode`\&=13
\gdef\smallampmatrix{%
  \begingroup
  \let&=\amsamp
  \begin{smallmatrix}%
}
\gdef\endsmallampmatrix{\end{smallmatrix}\endgroup}
\endgroup

\DeclareMathOperator{\codim}{codim}
\DeclareMathOperator{\krn}{Ker}

\DeclareMathOperator{\img}{Im}

\DeclareMathOperator{\homm}{Hom}
\DeclareMathOperator{\shhomm}{{\it\mathcal{H}om\rm}}
\DeclareMathOperator{\eend}{End}

\DeclareMathOperator{\autm}{Aut}

\DeclareMathOperator{\picr}{Pic}
\DeclareMathOperator{\tot}{Tot}

\DeclareMathOperator{\cl}{Cl}

\DeclareMathOperator{\spec}{Spec}
\DeclareMathOperator{\relspec}{{\mathit{Spec}}}
\DeclareMathOperator{\proj}{Proj\;}

\DeclareMathOperator{\ext}{Ext}

\DeclareMathOperator{\tor}{Tor}

\DeclareMathOperator{\ev}{ev}
\DeclareMathOperator{\trace}{tr}
\DeclareMathOperator{\composition}{cmps}
\DeclareMathOperator{\action}{act}

\DeclareMathOperator{\qcohcat}{QCoh}
\DeclareMathOperator{\cohcat}{Coh}
\DeclareMathOperator{\modd}{\bf Mod}
\DeclareMathOperator{\lder}{\bf L}
\DeclareMathOperator{\rder}{\bf R}
\DeclareMathOperator{\ldertimes}{\overset{\lder}{\otimes}}

\DeclareMathOperator{\id}{Id}

\DeclareMathOperator{\cone}{Cone}
\DeclareMathOperator{\opp}{{opp}}
\DeclareMathOperator{\fg}{{\it fg}}
\DeclareMathOperator{\qrep}{\it \mathcal{Q}r}
\DeclareMathOperator{\hproj}{\mathcal{P}}

\DeclareMathOperator{\acyc}{\it \mathcal{A}c}

\DeclareMathOperator{\semifree}{\mathcal{S}\mathcal{F}}
\DeclareMathOperator{\sffg}{\mathcal{S}\mathcal{F}_{\fg}}
\DeclareMathOperator{\perf}{{\it \mathcal{P}erf}}
\DeclareMathOperator{\hmtpy}{{Ho}}

\DeclareMathOperator{\tria}{{Tria}}
\DeclareMathOperator{\pretriag}{{Pre\text{-}Tr}}

\DeclareMathOperator{\TPair}{{\bf TPair}}
\DeclareMathOperator{\alg}{{\bf Alg}}

\DeclareMathOperator{\fmcatweak}{{\it \mathcal{F}\mathcal{M}uk}}

\begin{document}

\def\bv{\mathbf{v}}
\def\kgc_{K^*_G(\mathbb{C}^n)}
\def\kgchi_{K^*_\chi(\mathbb{C}^n)}
\def\kgcf_{K_G(\mathbb{C}^n)}
\def\kgchif_{K_\chi(\mathbb{C}^n)}
\def\gpic_{G\text{-}\picr}
\def\gcl_{G\text{-}\cl}
\def\trch_{{\chi_{0}}}
\def\regring{{R}}
\def\regrep{{V_{\text{reg}}}}
\def\givrep{{V_{\text{giv}}}}
\def\lbar{{(\mathbb{Z}^n)^\vee}}
\def\genpx_{{p_X}}
\def\genpy_{{p_Y}}
\def\genpcn_{p_{\mathbb{C}^n}}
\def\gnat{gnat}
\def\twalg{{\regring \rtimes G}}
\def\L{{\mathcal{L}}}
\def\O{{\mathcal{O}}}
\def\gcd{\mbox{gcd}}
\def\lcm{\mbox{lcm}}
\def\tf{{\tilde{f}}}
\def\tD{{\tilde{D}}}
\def\A{{\mathcal{A}}}
\def\B{{\mathcal{B}}}
\def\C{{\mathcal{C}}}
\def\D{{\mathcal{D}}}
\def\F{{\mathcal{F}}}
\def\H{{\mathcal{H}}}
\def\L{{\mathcal{L}}}
\def\N{{\mathcal{N}}}
\def\R{{\mathcal{R}}}
\def\barA{{\bar{\mathcal{A}}}}
\def\barAi{{\bar{\mathcal{A}}_1}}
\def\barAj{{\bar{\mathcal{A}}_2}}
\def\barB{{\bar{\mathcal{B}}}}
\def\barC{{\bar{\mathcal{C}}}}
\def\barD{{\bar{\mathcal{D}}}}
\def\M{{\mathcal{M}}}
\def\Aopp{{\A^{\opp}}}
\def\Bopp{{\B^{\opp}}}
\def\Copp{{\C^{\opp}}}
\def\aA{\leftidx{_{a}}{\A}}
\def\bA{\leftidx{_{b}}{\A}}
\def\Aa{{\A_a}}
\def\Ea{E_a}
\def\aE{\leftidx{_{a}}{E}{}}
\def\Eb{E_b}
\def\bE{\leftidx{_{b}}{E}{}}
\def\Fa{F_a}
\def\aF{\leftidx{_{a}}{F}{}}
\def\Fb{F_b}
\def\bF{\leftidx{_{b}}{F}{}}
\def\aM{\leftidx{_{a}}{M}{}}
\def\aMb{\leftidx{_{a}}{M}{_{b}}}
\def\biAMA{\leftidx{_{\A}}{M}{_{\A}}}
\def\biAMC{\leftidx{_{\A}}{M}{_{\C}}}
\def\biCMA{\leftidx{_{\C}}{M}{_{\A}}}
\def\biCMC{\leftidx{_{\C}}{M}{_{\C}}}
\def\biALA{\leftidx{_{\A}}{L}{_{\A}}}
\def\biALC{\leftidx{_{\A}}{L}{_{\C}}}
\def\biCLA{\leftidx{_{\C}}{L}{_{\A}}}
\def\biCLC{\leftidx{_{\C}}{L}{_{\C}}}
\def\Na{{N_a}}
\def\modk{{\modd\text{-}k}}
\def\modA{{\modd\text{-}\A}}
\def\modbar{{\overline{\modd}}}
\def\modbarA{{\overline{\modd}\text{-}\A}}
\def\modbarAopp{{\overline{\modd}\text{-}\Aopp}}
\def\modB{{\modd\text{-}\B}}
\def\modC{{\modd\text{-}\C}}
\def\modD{{\modd\text{-}\D}}
\def\modbarB{{\overline{\modd}\text{-}\B}}
\def\modbarC{{\overline{\modd}\text{-}\C}}
\def\modbarD{{\overline{\modd}\text{-}\D}}
\def\modbarBopp{{\overline{\modd}\text{-}\Bopp}}
\def\AmodA{{\A\text{-}\modd\text{-}\A}}
\def\AmodB{{\A\text{-}\modd\text{-}\B}}
\def\BmodB{{\B\text{-}\modd\text{-}\B}}
\def\BmodA{{\B\text{-}\modd\text{-}\A}}
\def\DmodD{{\D\text{-}\modd\text{-}\D}}
\def\AmodbarA{\A\text{-}{\overline{\modd}\text{-}\A}}
\def\AmodbarB{\A\text{-}{\overline{\modd}\text{-}\B}}
\def\AmodbarC{\A\text{-}{\overline{\modd}\text{-}\C}}
\def\AmodbarD{\A\text{-}{\overline{\modd}\text{-}\D}}
\def\BmodbarA{\B\text{-}{\overline{\modd}\text{-}\A}}
\def\BmodbarB{\B\text{-}{\overline{\modd}\text{-}\B}}
\def\BmodbarC{\B\text{-}{\overline{\modd}\text{-}\C}}
\def\BmodbarD{\B\text{-}{\overline{\modd}\text{-}\D}}
\def\CmodbarA{\C\text{-}{\overline{\modd}\text{-}\A}}
\def\CmodbarB{\C\text{-}{\overline{\modd}\text{-}\B}}
\def\CmodbarC{\C\text{-}{\overline{\modd}\text{-}\C}}
\def\CmodbarD{\C\text{-}{\overline{\modd}\text{-}\D}}
\def\DmodbarA{\D\text{-}{\overline{\modd}\text{-}\A}}
\def\DmodbarB{\D\text{-}{\overline{\modd}\text{-}\B}}
\def\DmodbarC{\D\text{-}{\overline{\modd}\text{-}\C}}
\def\DmodbarD{\D\text{-}{\overline{\modd}\text{-}\D}}
\def\sfA{{\semifree(\A)}}
\def\sfB{{\semifree(\B)}}
\def\sffgA{{\sffg(\A)}}
\def\sffgB{{\sffg(\B)}}
\def\hprojA{{\hproj(\A)}}
\def\hprojB{{\hproj(\B)}}
\def\qrepA{{\qrep(\A)}}
\def\qrepB{{\qrep(\B)}}
\def\opp{{\text{opp}}}
\def\perfsf{{\semifree^{\perf}}}
\def\prfhpr{{\hproj^{\scriptscriptstyle\perf}}}
\def\prfhprA{{\prfhpr(\A)}}
\def\prfhprB{{\prfhpr(\B)}}
\def\prfhprAopp{{\prfhpr(\Aopp)}}
\def\prfhprBopp{{\prfhpr(\Bopp)}}
\def\perfsfA{{\perfsf(\A)}}
\def\perfsfB{{\perfsf(\B)}}
\def\qrhpr{{\hproj^{qr}}}
\def\qrhprA{{\qrhpr(\A)}}
\def\qrhprB{{\qrhpr(\B)}}
\def\qrsf{{\semifree^{qr}}}
\def\qrsf{{\semifree^{qr}}}
\def\qrsfA{{\qrsf(\A)}}
\def\qrsfB{{\qrsf(\B)}}
\def\Aperfsf{{\semifree^{\A\text{-}\perf}(\AbimB)}}
\def\Bperfsf{{\semifree^{\B\text{-}\perf}(\AbimB)}}
\def\Aprfhpr{{\hproj^{\A\text{-}\perf}(\AbimB)}}
\def\Bprfhpr{{\hproj^{\B\text{-}\perf}(\AbimB)}}
\def\Aqrhpr{{\hproj^{\A\text{-}qr}(\AbimB)}}
\def\Bqrhpr{{\hproj^{\B\text{-}qr}(\AbimB)}}
\def\Aqrsf{{\semifree^{\A\text{-}qr}(\AbimB)}}
\def\Bqrsf{{\semifree^{\B\text{-}qr}(\AbimB)}}
\def\modAopp{{\modd\text{-}\Aopp}}
\def\modBopp{{\modd\text{-}\Bopp}}
\def\AmodA{{\A\text{-}\modd\text{-}\A}}
\def\AmodB{{\A\text{-}\modd\text{-}\B}}
\def\AmodC{{\A\text{-}\modd\text{-}\C}}
\def\BmodA{{\B\text{-}\modd\text{-}\A}}
\def\BmodB{{\B\text{-}\modd\text{-}\B}}
\def\BmodC{{\B\text{-}\modd\text{-}\C}}
\def\CmodA{{\C\text{-}\modd\text{-}\A}}
\def\CmodB{{\C\text{-}\modd\text{-}\B}}
\def\CmodC{{\C\text{-}\modd\text{-}\C}}
\def\AbimA{{\A\text{-}\A}}
\def\AbimC{{\A\text{-}\C}}
\def\BbimA{{\B\text{-}\A}}
\def\BbimB{{\B\text{-}\B}}
\def\BbimC{{\B\text{-}\C}}
\def\BbimD{{\B\text{-}\D}}
\def\CbimA{{\C\text{-}\A}}
\def\CbimB{{\C\text{-}\B}}
\def\CbimC{{\C\text{-}\C}}
\def\DbimA{{\D\text{-}\A}}
\def\DbimB{{\D\text{-}\B}}
\def\DbimC{{\D\text{-}\C}}
\def\DbimD{{\D\text{-}\D}}
\def\AhprA{{\hproj\left(\AbimA\right)}}
\def\BhprB{{\hproj\left(\BbimB\right)}}
\def\AhprB{{\hproj\left(\AbimB\right)}}
\def\BhprA{{\hproj\left(\BbimA\right)}}
\def\AbarA{{\overline{\A\text{-}\A}}}
\def\AbarB{{\overline{\A\text{-}\B}}}
\def\BbarA{{\overline{\B\text{-}\A}}}
\def\BbarB{{\overline{\B\text{-}\B}}}
\def\QAbimB{{Q\A\text{-}\B}}
\def\AbimB{{\A\text{-}\B}}
\def\AonebimB{{\A_1\text{-}\B}}
\def\AtwobimB{{\A_2\text{-}\B}}
\def\BbimA{{\B\text{-}\A}}
\def\Aperf{{\A\text{-}\perf}}
\def\Bperf{{\B\text{-}\perf}}
\def\MddA{{M^{\tilde{\A}}}}
\def\MddB{{M^{\tilde{\B}}}}
\def\MhdA{{M^{h\A}}}
\def\MhdB{{M^{h\B}}}
\def\NhdB{{N^{h\B}}}
\def\Cat{{\it \mathcal{C}at}}
\def\twoCat{{\it 2\text{-}\;\mathcal{C}at}}
\def\DGCat{{DG\text{-}Cat}}
\def\HoDGCat{{\hmtpy(\DGCat)}}
\def\HoDGCatV{{\hmtpy(\DGCat_\mathbb{V})}}
\def\tr{{tr}}
\def\pretr{{pretr}}
\def\kctr{{kctr}}
\def\PreTrCat{{\DGCat^\pretr}}
\def\KcTrCat{{\DGCat^\kctr}}
\def\HoPretrCat{{\hmtpy(\PreTrCat)}}
\def\HoKcTrCat{{\hmtpy(\KcTrCat)}}
\def\Aquasirep{{\A\text{-}qr}}
\def\QAquasirep{{Q\A\text{-}qr}}
\def\Bquasirep{{\B\text{-}qr}} 
\def\lderA{{\tilde{\A}}} 
\def\lderB{{\tilde{\B}}} 
\def\adjunit{{\text{adj.unit}}}
\def\adjcounit{{\text{adj.counit}}}
\def\degzero{{\text{deg.0}}}
\def\degone{{\text{deg.1}}}
\def\degminusone{{\text{deg.-$1$}}}
\def\bareta{{\overline{\eta}}}
\def\barzeta{{\overline{\zeta}}}
\def\Ract{{R {\action}}}
\def\barRact{{\overline{\Ract}}}
\def\actL{{{\action} L}}
\def\baractL{{\overline{\actL}}}
\def\Ainfty{{A_{\infty}}}
\def\noddinf{{{\bf Nod}_{\infty}}}
\def\noddinfstr{{{\bf Nod}^{\text{strict}}_{\infty}}}
\def\noddinfA{{\noddinf\A}}
\def\noddinfB{{\noddinf\B}}
\def\noddinfAB{{\noddinf\AbimB}}
\def\noddinfBA{{\noddinf\BbimA}}
\def\noddinfu{{({\bf Nod}_{\infty})_u}}
\def\noddinfuA{{(\noddinfA)_u}}
\def\noddinfhu{{({\bf Nod}_{\infty})_{hu}}}
\def\noddinfhuA{{(\noddinfA)_{hu}}}
\def\noddinfdg{{({\bf Nod}_{\infty})_{dg}}}
\def\noddinfdgA{{(\noddinfA)_{dg}}}
\def\noddinfdgAA{{(\noddinf\AbimA)_{dg}}}
\def\noddinfdgAB{{(\noddinf\AbimB)_{dg}}}
\def\noddinfdgB{{(\noddinfB)_{dg}}}
\def\moddinf{{\modd_{\infty}}}
\def\moddinfA{{\modd_{\infty}\A}}
\def\infbar{{B_\infty}}
\def\infbarA{{B^\A_\infty}}
\def\infbarB{{B^\B_\infty}}
\def\infbarC{{B^\C_\infty}}
\def\inftimes{{\overset{\infty}{\otimes}}}
\def\infhom{{\overset{\infty}{\homm}}}
\def\barhom{{\overline{\homm}}}
\def\barend{{\overline{\eend}}}
\def\bartimes{{\;\overline{\otimes}}}
\def\triaA{{\tria \A}}
\def\TPairdg{{\TPair^{dg}}}
\def\algA{{\alg(\A)}}
\def\Ainfty{{A_{\infty}}}
\def\gpmu{{\boldsymbol{\mu}}}
\def\odd{{\text{odd}}}
\def\even{{\text{even}}}
\def\tta{{TT}}
\def\bartta{{\overline{\tta}}}

\theoremstyle{definition}
\newtheorem{defn}{Definition}[section]
\newtheorem*{defn*}{Definition}
\newtheorem{exmpl}[defn]{Example}
\newtheorem*{exmpl*}{Example}
\newtheorem{exrc}[defn]{Exercise}
\newtheorem*{exrc*}{Exercise}
\newtheorem*{chk*}{Check}
\newtheorem*{remarks*}{Remarks}
\theoremstyle{plain}
\newtheorem{theorem}{Theorem}[section]
\newtheorem*{theorem*}{Theorem}
\newtheorem{conj}[defn]{Conjecture}
\newtheorem*{conj*}{Conjecture}
\newtheorem{prps}[defn]{Proposition}
\newtheorem*{prps*}{Proposition}
\newtheorem{cor}[defn]{Corollary}
\newtheorem*{cor*}{Corollary}
\newtheorem{lemma}[defn]{Lemma}
\newtheorem*{claim*}{Claim}
\newtheorem{Specialthm}{Theorem}
\renewcommand\theSpecialthm{\Alph{Specialthm}}
\numberwithin{equation}{section}
\renewcommand{\textfraction}{0.001}
\renewcommand{\topfraction}{0.999}
\renewcommand{\bottomfraction}{0.999}
\renewcommand{\floatpagefraction}{0.9}
\setlength{\textfloatsep}{5pt}
\setlength{\floatsep}{0pt}
\setlength{\abovecaptionskip}{2pt}
\setlength{\belowcaptionskip}{2pt}

\begin{abstract}
We propose a new theory of (non-split) $\mathbb{P}^n$-functors. 
These are $F\colon \A \rightarrow \B$ for which the adjunction monad
$RF$ is a repeated extension of $\id_\A$ by powers of an
autoequivalence $H$ and three conditions are satisfied: the monad
condition, the adjoints condition, and the highest degree term
condition. This unifies and extends the two earlier notions 
of spherical functors and split $\mathbb{P}^n$-functors. We 
construct the $\mathbb{P}$-twist of such $F$ and prove it to 
be an autoequivalence. We then give a criterion for 
$F$ to be a $\mathbb{P}^n$-functor which is stronger than 
the definition but much easier to check in practice. It involves
only two conditions: the strong monad condition and the weak adjoints
condition. For split $\mathbb{P}^n$-functors, we prove Segal's conjecture 
on their relation to spherical functors.  
Finally, we give four examples of non-split $\mathbb{P}^n$-functors: 
spherical functors, extensions by zero, cyclic covers, and 
family $\mathbb{P}$-twists. For the latter, we show the $\mathbb{P}$-twist
to be the derived monodromy of associated Mukai flop, the so-called
``flop-flop = twist'' formula.  
\end{abstract}

\maketitle

\section{Introduction}
\label{section-introduction}

In the literature to date there appeared several distinct, yet
related notions of \em twist autoequivalences\rm. 
In all of them, an autoequivalence of the derived category $D(X)$ 
of an algebraic variety $X$ is cooked 
up from an object of $D(X)$ or a functor $D(Z) \rightarrow D(X)$ 
from another variety $Z$. The result is usually a non-trivial, 
genuinely derived autoequivalence, which can nonetheless retain 
a lot of geometric sense if the defining object was itself geometric 
in nature. This allowed to construct interesting new categorical 
actions on $D(X)$, to categorify existing such 
actions on the cohomology ring $H^*(X)$, and to explain 
some other phenomena e.g. wall-crossing for moduli of sheaves or 
derived monodromy of flops. 
Before long, these constructions became ubiquitous in algebraic geometry, 
representation theory, and theoretical physics.
In this paper we construct the new theory of 
\em (non-split) $\mathbb{P}^n$-functors\rm. It provides
a common framework for all the existing notions of twist
autoequivalences such as \em spherical functors \rm 
\cite{SeidelThomas-BraidGroupActionsOnDerivedCategoriesOfCoherentSheaves}
\cite{AnnoLogvinenko-SphericalDGFunctors}
and \em split $\mathbb{P}^n$-functors \rm 
\cite{HuybrechtsThomas-PnObjectsAndAutoequivalencesOfDerivedCategories}
\cite{Addington-NewDerivedSymmetriesOfSomeHyperkaehlerVarieties}
\cite{Cautis-FlopsAndAboutAGuide}. It also opens up a wealth of new 
and previously impossible non-split examples. 

Specifically, in this paper:
\begin{enumerate}
\item \em Cyclic covers (\S\ref{section-cyclic-covers}): \rm 
We construct a large new family of geometrical 
$\mathbb{P}^n$-functors which have no prior analogues. 
Indeed, we show that any cyclic cover
of any algebraic variety ramified in an effective Cartier divisor
gives rise to a $\mathbb{P}^n$-functor which is always 
non-split.  
\item \em Family
$\mathbb{P}^n$-twists (\S\ref{section-family-p-twists}): \rm For any
$\mathbb{P}^n$-fibration $P$ over a smooth projective variety $Z$, 
we prove that any embedding of $P$ into a smooth projective variety
$X$ with $\mathcal{N}_{P/X} = \Omega^1_{P/Z}$ yields 
an apriori non-split $\mathbb{P}^n$-functor whose 
twist is the derived monodromy of the associated Mukai
flop. This was previously only done under the
assumption $\mathrm{HH}^{odd}(Z) = 0$ in which case the 
arising $\mathbb{P}^n$-functor is split
\cite{AddingtonDonovanMeachan-MukaiFlopsAndPTwists}.
\item \em Abstract machinery (\S\ref{section-ptwists}-\ref{section-strong-monad-condition}): \rm 
The above is enabled by the abstract theory 
of \em (non-split) $\mathbb{P}^n$-functors \rm 
we develop for arbitrary DG enhanced triangulated categories.
We employ involved DG categorical machinery -- so that others don't have to. 
Our results are stated purely in the language of underlying
triangulated categories, and can be used without any knowledge of
DG enhancements. For example, by algebraic geometers like us who 
work with Fourier-Mukai transforms. 
\item \em Segal's conjecture (\S\ref{section-segals-conjecture}): \rm 
Our abstract machinery allows us to prove a conjecture made by 
Segal in \cite[Remark 4.6]{Segal-AllAutoequivalencesAreSphericalTwists}
that a split $\mathbb{P}^n$-functor can always be deformed to a spherical
functor. 
\item \em Other
examples (\S\ref{section-examples-spherical-functors}-\ref{section-extensions-by-zero}): \rm
We show that non-split $\mathbb{P}^n$-functors now include 
all spherical functors. We also  
show that a $\mathbb{P}^n$-functor can be glued over a
semiorthogonal decomposition with a zero functor and the result is always 
a non-split $\mathbb{P}^n$-functor. We used the zero functor as we
wanted a quick example, but our methods were later employed by Barbacovi 
\cite{Barbacovi-OnTheCompositionOfTwoSphericalTwists} to glue
non-trivial spherical functors. 
\end{enumerate}

It is worth explaining the necessity of working with DG
enhancements. Without them, one can not control convolutions
of complexes of objects. A convolution is a repeated cone: we 
repeatedly replace two consequent objects of the complex by the cone 
of the differential between them until a single object is left. 
Each step involves choices, so apriori a convolution is neither unique
nor guaranteed to exist. Luckily, the complex defining 
the $\mathbb{P}$-twist $P$ has the unique convolution
\cite{AnnoLogvinenko-OnUniquenessOfPTwists}, but most complexes do
not. This makes direct computations, 
e.g. of the composition of $P$ with either of its adjoints, impossible. 

This problem isn't solved by working with Fourier-Mukai kernels or,
more generally, working in some plain triangulated category of enhanced
functors. This only solves the usual problem of taking the cone of
morphism of functors, that is -- taking the convolution of a $2$-term
complex. As long as one works in a plain triangulated category, 
the convolutions of longer complexes of objects are still neither 
unique nor are guaranteed to exist. The direct computations of
the composition of $P$ with its adjoints are still impossible. 

A traditional way to tackle this problem within the axiomatics of
plain triangulated categories is by diagram chasing and repeated use
of the octahedral axiom. However, octahedral axiom offers no
control over morphisms between cones, so while these traditional methods
can tell you something about the object defining $P$, e.g.~its
Fourier-Mukai kernel, they tell you very little about the morphisms
which define any natural transformations involved. Similarly, 
while they can tell you something about what $P$ does on objects, 
they tell you very little about its $\homm$-space maps. This
has long been a problem as to prove $P$ to be an equivalence  
one needs to prove these maps to be isomorphisms. 

For example, in the split case proof of the $\mathbb{P}$-twist
being an equivalence in \cite[Theorem
3]{Addington-NewDerivedSymmetriesOfSomeHyperkaehlerVarieties}
it was shown that the source and the target $\homm$-spaces
are isomorphic. But it it not clear to us how to check that 
the map with which $P$ acts on them is, indeed, an isomorphism. 
Strictly speaking, it is something that needs to be checked to prove 
that $P$ is equivalence. We tried to come up with a way to check it
using just the axiomatics of plain triangulated categories, and
failed. While it may theoretically be possible, it would be very
hard. And would be even harder in the general, non-split case we
work with in this paper.

DG enhancements give a natural solution to this problem. A twisted
complex in the enhancement is the data of a 
complex of objects in the triangulated category plus the choice
of a specific convolution of this complex. We can compute
with twisted complexes -- tensor them, dualise, etc -- and keep track 
of the convolution. Thus direct computations become not only possible, 
but relatively straightforward. Since the complex defining $P$ has 
the unique convolution, we can choose any lift of it to a twisted 
complex and in \S\ref{section-ptwists} we show that there is a 
very natural choice. We then prove that $P$ is an equivalence in 
\S\ref{section-Pn-functors} by using this lift to
directly compute the composition of $P$ with its left adjoint. 

We now explain the results of this paper in more detail, starting 
with the background. \em Spherical objects \rm and their \em spherical
twists \rm were introduced in
\cite{SeidelThomas-BraidGroupActionsOnDerivedCategoriesOfCoherentSheaves}
as mirror symmetric analogues of
Lagrangian spheres on a symplectic
manifold and their associated Dehn twists
\cite[\S5a]{Seidel-GradedLagrangianSubmanifolds}. 
These were generalised to \em spherical functors \rm
\cite{Horja-DerivedCategoryAutomorphismsFromMirrorSymmetry}, 
\cite{Rouquier-CategorificationOfTheBraidGroups}, 
\cite{Toda-OnACertainGeneralizationOfSphericalTwists}, 
\cite{AnnoLogvinenko-SphericalDGFunctors} 
in the setting of enhanced triangulated categories 
and enhanced exact functors 
\cite{BondalKapranov-EnhancedTriangulatedCategories}. 
A functor 
$ F\colon \A \rightarrow \B $ 
with left and right adjoints $L,R\colon \B \rightarrow \A$ 
is spherical if:
\begin{enumerate}
\item 
\label{intro-item-twist-is-equivalence}
The adjunction comonad $FR$ is an extension of $\id_\B$ by
an autoequivalence $T \in \autm(\B)$. 
\item
\label{intro-item-cotwist-is-equivalence}
The adjunction monad $RF$ is a coextension of $\id_\A$ by
an autoequivalence $C[1] \in \autm(\A)$. 
\item 
\label{intro-item-twist-identifies-adjoints}
$R$ is canonically isomorphic to $LT[-1]$. 
\item 
\label{intro-item-cotwist-identifies-adjoints}
$R$ is canonically isomorphic to $CL[1]$. 
\end{enumerate}
The autoequivalences $T$ and $C$ are the 
\em spherical twist \rm and 
\em cotwist \rm of $F$.  
Any two of the above conditions imply all 
four and can be taken as the definition 
\cite[Theorem 5.1]{AnnoLogvinenko-SphericalDGFunctors}.
A spherical object $E \in D(X)$ is a spherical functor
$D(\text{pt}) \rightarrow D(X)$ sending $\mathbb{C}$ to $E$. 
The applications of 
spherical functors include categorifications of link homology 
\cite{KhovanovThomas-BraidCobordismsTriangulatedCategoriesAndFlagVarieties}, 
\cite{CautisKamnitzer-KnotHomologyViaDerivedCategoriesOfCoherentSheavesISL2Case}, moduli and wall-crossing 
on K3 surfaces \cite{Mukai-OnTheModuliSpaceOfBundlesOnK3SurfacesI},
\cite{Bridgeland-StabilityConditionsOnK3Surfaces}, 
\cite{BayerMacri-MMPForModuliOfSheavesOnK3sViaWallcrossingNefAndMovableConesLagrangianFibrations},
McKay correspondence 
\cite{Bridgeland-StabilityConditionsAndKleinianSingularities},
\cite{IshiiUehara-AutoequivalencesOfDerivedCategoriesOnTheMinimalResolutionsOfA_nSingularitiesOnSurfaces}, 
derived monodromy of Atiyah flops
\cite{Toda-OnACertainGeneralizationOfSphericalTwists},
\cite{DonovanWemyss-NoncommutativeDeformationsAndFlops}, 
\cite{BodzentaBondal-FlopsAndSphericalFunctors}, 
semiorthogonal decompositions \cite{Bondal-RepresentationOfAsssociativeAlgebrasAndCoherentSheaves},
variation of GIT and window shifts 
\cite{HalpernLeistnerShipman-AutoequivalencesOfDerivedCategoriesViaGeometricInvariantTheory}, 
\cite{DonovanSegal-WindowShiftsFlopEquivalencesAndGrassmannianTwists}, 
and perverse schobers \cite{KapranovSchechtman-PerverseSchobers}. 

At the same time, Huybrechts and Thomas introduced 
\em $\mathbb{P}^n$-objects \rm
\cite{HuybrechtsThomas-PnObjectsAndAutoequivalencesOfDerivedCategories}.
These were inspired by Lagrangian 
$\mathbb{C}\mathbb{P}^n$s on symplectic manifolds for which Seidel 
constructed analogues of Dehn twists
\cite[\S4b]{Seidel-GradedLagrangianSubmanifolds}. 
The definition of a $\mathbb{P}^n$-object $E \in D(X)$ 
asked for $\ext^*_X(E,E)$ to be isomorphic as a graded ring to
$$
H^*(\mathbb{P}^n,\mathbb{C})
\simeq \mathbb{C} \oplus \mathbb{C}[-2] \oplus \dots \oplus
\mathbb{C}[-2n] \simeq \mathbb{C}[h]/(h^{n+1}) 
\quad \quad 
\text{ with } \deg(h) = 2. 
$$ 
and for $E \otimes \omega_X \simeq E$. 
Thinking of $E$ as a functor $F\colon D(\text{pt}) \rightarrow D(X)$, 
we have $\ext^*_X(E,E) \simeq RF(\mathbb{C})$ as graded algebras. 
This led Addington 
\cite{Addington-NewDerivedSymmetriesOfSomeHyperkaehlerVarieties}
and Cautis \cite{Cautis-FlopsAndAboutAGuide} to define
a \em (split) $\mathbb{P}^n$-functor \rm to be a functor 
$F \colon \A \rightarrow \B$ as above which satisfies 
\begin{enumerate}
\item 
\label{intro-split-pn-functor-RF-condition}
$RF \simeq \id_\A \oplus H \oplus \dots \oplus H^n$ for some
autoequivalence $H \in \autm(\A)$. 
\item 
\label{intro-split-pn-functor-monad-condition}
\em The monad condition. \rm The restriction of the monad multiplication 
$RFRF \xrightarrow{m} RF$ to the map
\begin{equation}
\label{eqn-left-multiplication-by-h-in-RF-monad}
H(\id \oplus \dots \oplus H^{n-1}) \rightarrow H \oplus \dots \oplus H^n
\end{equation}
is an upper triangular matrix with $\id$s down the main diagonal. 
\item 
\label{intro-split-pn-functor-adjoints-condition}
\em The adjoints condition. \rm $R \simeq H^n L$. 
\end{enumerate}
The $\mathbb{P}$-twist of $F$ was constructed as a certain 
convolution (double cone) of the two-step complex 
\begin{equation}
\label{intro-the-formula-for-the-p-twist}
FHR \xrightarrow{\psi} FR \xrightarrow{\trace} \id_\B,
\end{equation}
where $\trace$ is the adjunction counit and $\psi$ is the
map
$FHR \hookrightarrow FRFR \xrightarrow{FR \trace - \trace FR} FR$.  
A $\mathbb{P}^n$-object is a $\mathbb{P}^n$-functor
$D(\text{pt}) \rightarrow D(X)$ with $H = [-2]$.  
The applications of split $\mathbb{P}^n$-functors
include Hilbert schemes of points on K3 and abelian surfaces 
\cite{Addington-NewDerivedSymmetriesOfSomeHyperkaehlerVarieties}
\cite{Meachan-DerivedAutoequivalencesOfGeneralisedKummerVarieties}
\cite{KrugMeachan-UniversalFunctorsOnSymmetricQuotientStacksOfAbelianVarieties}, 
moduli of torsion sheaves on K3 surfaces
\cite{AddingtonDonovanMeachan-ModuliSpacesOfTorsionSheavesOnK3SurfacesAndDerivedEquivalences}, 
and derived monodromy of Mukai flops
\cite{AddingtonDonovanMeachan-MukaiFlopsAndPTwists}. 

From a flurry of applications that followed, it was clear that 
split $\mathbb{P}^n$-functors were an important notion and 
a major step forward. It was equally clear that it was not 
the whole story:
\begin{enumerate}
\item  For a split $\mathbb{P}^n$-functor, the monad $RF$ splits as a
direct sum $\id \oplus H \oplus H^2 \oplus \dots \oplus H^n$. 
For a spherical functor, $RF$ is an extension of $\id$ by $H$. 
It would be logical for 
$\mathbb{P}^1$-functors to be the spherical functors, but
split $\mathbb{P}^1$-functors only give the \em split \rm
spherical functors --- those where $RF$ splits up as $\id \oplus H$.  
\item A split $\mathbb{P}^n$-functor $F$ must have $\krn F = 0$.
By contrast, spherical functors can have a non-trivial kernel and
frequently do, e.g. the derived pullback to a divisor. 
This is due to $RF$ being a non-trivial extension of $\id_\A$ in that case.  
\item 
\label{intro-issues-uniqueness-of-p-twists}
There was an issue with uniqueness of $\mathbb{P}$-twists. 
A two-step complex like \eqref{intro-the-formula-for-the-p-twist} 
can apriori have several convolutions. A certain choice was made
in \cite{Addington-NewDerivedSymmetriesOfSomeHyperkaehlerVarieties}, 
but would making different choices yield a different twist autoequivalence? 
This was eventually resolved by in \cite{AnnoLogvinenko-OnUniquenessOfPTwists}. 

\item In \cite{Addington-NewDerivedSymmetriesOfSomeHyperkaehlerVarieties}
and \cite{Cautis-FlopsAndAboutAGuide} the theory of split 
$\mathbb{P}^n$-functors was developed only for 
the derived categories of smooth, projective varieties and
Fourier-Mukai transforms. For the theory to be more universally
applicable, it needed to work with arbitrary enhanced triangulated 
categories. 
\end{enumerate}

In this paper, we propose the new notion of \em $\mathbb{P}^n$-functors \rm 
which deals with all the issues above and incorporates spherical 
functors and split $\mathbb{P}^n$-functors as special cases. 
Being a $\mathbb{P}^n$-functor is an extra structure on $F$ which
has to satisfy certain conditions. This structure consists of:
\begin{enumerate}
\item An enhanced autoequivalence $H \in \autm(\A)$ with $H(\krn F)
= \krn F$. 
\item A degree $n$ cyclic coextension $Q_n$ of $\id$ by $H$, that is
--- the data of a filtration
\begin{footnotesize}
\begin{equation}
\label{eqn-intro-cyclic-extension-of-id-by-H-of-degree-n}
\begin{tikzcd}
\id 
\ar[phantom]{drr}[description, pos=0.45]{\star}
\ar{rr}{\iota_1}
& &
Q_1 
\ar{rr}{\iota_2}
\ar{ld}{\mu_1}
\ar[phantom]{drr}[description, pos=0.45]{\star}
& &
Q_2  
\ar{r}
\ar{ld}{\mu_2}
&
\dots
\ar{r}
&
Q_{n-2}
\ar[phantom]{drr}[description, pos=0.45]{\star}
\ar{rr}{\iota_{n-1}}
\ar{ld}
& 
& 
Q_{n-1}
\ar{ld}{\mu_{n-1}}
\ar{rr}{\iota_{n}}
& & 
Q_n,
\ar[phantom]{dll}[description, pos=0.45]{\star}
\ar{ld}{\mu_n}
\\
&
H
\ar[dashed]{lu}
& ~ &
H^2 
\ar[dashed]{lu}
& ~ & 
\dots 
\ar[dashed]{lu}
& ~ &
H^{n-1}
\ar[dashed]{lu}
& ~ & 
H^n
\ar[dashed]{lu}
& 
\end{tikzcd}
\end{equation}
\end{footnotesize}
where $\star$ denotes exact triangles and dashed arrows ---  morphisms 
of degree $1$. For any $i$ let $\iota$ denote the composition
$Q_i \xrightarrow{\iota_n \circ \dots \circ \iota_i} Q_n$ 
and let $J_n$ be defined by the exact triangle 
$\id \xrightarrow{\iota} Q_n \xrightarrow{\kappa} J_n \rightarrow \id[1]. $
\item An isomorphism $Q_n \xrightarrow{\gamma} RF$
which intertwines the adjunction unit $\id_\A \xrightarrow{\action} RF$ and 
$\id_\A \xrightarrow{\iota} Q_n$. 
\end{enumerate}

One of the major difficulties in generalising $\mathbb{P}^n$-functors
to include the non-split case was that the map $\psi\colon FHR
\rightarrow FR$ in the definition \eqref{intro-the-formula-for-the-p-twist} of
$\mathbb{P}$-twist involved the direct summand inclusion 
$H \hookrightarrow RF$. In the non-split case, this no longer
exists. However, since $F \xrightarrow{F\action} FRF$ is split, so must be 
$F \xrightarrow{F\iota} FQ_n$. It follows that 
$FR \xrightarrow{F\iota_1 R} FQ_1R$ is split, see 
\S\ref{section-pn-functor-data}. Choose any 
splitting $FQ_1R \simeq FR \oplus FHR$ and
define $\psi$ be the composition of $FHR \hookrightarrow FQ_1R$ with
\begin{equation}
\label{eqn-intro-latter-part-of-psi}
FQ_1R \xrightarrow{\iota_n\circ\ldots \circ \iota_2}
FQ_nR \xrightarrow{ F \gamma R} FRFR \xrightarrow{FR\trace -
\trace FR} FR.
\end{equation}
This is independent of the choice of splitting since 
$\eqref{eqn-intro-latter-part-of-psi} \circ F\iota_1 R 
= (FR \trace - \trace FR) \circ F\action R = 0$. 

\begin{defn*}[Definition \ref{defn-Pn-functor}]
Let $\A$ and $\B$ be enhanced triangulated categories. Let
$F$ be an enhanced functor $\A \rightarrow \B$
with enhanced left and right adjoints 
$L,R\colon \B \rightarrow \A$.
The structure of a \em $\mathbb{P}^n$-functor \rm on $F$ 
is a collection $(H, Q_n, \gamma)$ as above 
which satisfies the following three conditions:
\begin{enumerate}
\item \em The monad condition. \rm The following composition 
is an isomorphism:
\begin{equation}
\label{eqn-monad-condition}
\nu\colon FHQ_{n-1} \xrightarrow{FH \iota_{n-1}} FHQ_n 
\xrightarrow{\psi F} FQ_n \xrightarrow{F\kappa} FJ_n. 
\end{equation}
\item \em The adjoints condition. \rm The following composition 
is an isomorphism:
\begin{equation}
\label{eqn-adjoints-condition}
FR \xrightarrow{FR \action} FRFL \xrightarrow{F\mu_n L} FH^n L.
\end{equation}
 \item \em The highest degree term condition. \rm There exists 
$F H^n L \xrightarrow{\sim} F H H^n H' L$ making the diagram
\begin{equation}
\label{eqn-highest-degree-term-condition}
\begin{tikzcd}[column sep={2cm}]
FHQ_{n-1}L
\ar{r}{FH\iota_n L}
\ar[equals]{d}
&
FHRFL
\ar{r}{\psi FL}
&
FRFL
\ar{r}{F\mu_n L}
&
FH^nL
\ar[dashed]{d}{\text{isomorphism}}
\\
FHQ_{n-1}L
\ar{r}{FH\iota_n L}
&
FHRFL
\ar{r}{FHR\psi'}
&
FHRFH'L
\ar{r}{FH\mu_n H'L}
&
FHH^nH'L
\end{tikzcd}
\end{equation}
commute. Here $\psi': FL \to FH'L$ is the left dual of $\psi$.
\end{enumerate}
\end{defn*}

Our monad condition is strictly weaker than the one 
in the definition of a split $\mathbb{P}^n$-functor above. 
In \S\ref{section-strong-monad-condition-split-case} we show that
in the split case treated by Addington the objects $FHQ_{n-1}$
and $FJ_n$ are both isomorphic to 
\begin{equation}
\label{eqn-intro-FHQ_n-1-FJn}
FH \oplus \dots \oplus FH^n 
\end{equation}
in the way that identifies the map $\nu: FHQ_{n-1} \rightarrow FJ_n$
with the image under $F$ of the map 
\eqref{eqn-left-multiplication-by-h-in-RF-monad} 
of left monad multiplication by $H$ 
minus a strictly upper triangular matrix. Thus our monad condition
is strictly weaker than Addington's condition on two counts: 
we only consider the image of \eqref{eqn-left-multiplication-by-h-in-RF-monad} 
under $F$ and we only ask for that to be invertible rather than 
upper triangular with $\id$'s down the main diagonal. 

Unfortunately, this weakening of the monad condition necessitates 
the highest degree term condition. Equally unfortunately, 
while we weaken Addington's adjoints condition by 
applying $F$ to it, we then strengthen it by asking for
a specific map $FR \rightarrow FH^nL$ to be an isomorphism.
Fortunately, if a certain $\ext^{-1}$-vanishing holds,
both issues can be resolved by strengthening our monad condition
to what is essentially a non-split version of the original monad
condition of
\cite{Addington-NewDerivedSymmetriesOfSomeHyperkaehlerVarieties}. 
To formulate it in the non-split case, 
we use the filtration on $Q_n$ provided by its
structure of a cyclic coextension:
\begin{itemize}
\item \em Strong monad condition: \rm 
For any $0 < j < n$ the map 
\begin{align}
\label{eqn-intro-the-map-Q_1-Q_j-to-Q_n}
Q_1 Q_j \xrightarrow{\iota \iota} Q_nQ_n \xrightarrow{m}  Q_n, 
\end{align}
filters through 
$Q_{j+1} \xrightarrow{\iota} Q_n$ via some map 
$ m_{1j}\colon Q_1Q_j \rightarrow  Q_{j+1}$. 
Here $m$ is the monad multiplication. 

Moreover, there is an isomorphism 
$\rho_{1j}\colon HH^j \xrightarrow{\; \sim \;} H^{j+1}$
making the following diagram commute:
\begin{equation}
\label{eqn-intro diagonal-isomorphisms-in-strong-monad-condition-square}
\begin{tikzcd}[column sep={2cm}]
Q_1 Q_j
\ar{r}{m_{1j}}
\ar{d}{\mu_1\mu_j}
&
Q_{j+1}
\ar{d}{\mu_{1+j}}
\\
HH^j
\ar{r}{\rho_{1j}}
&
H^{j+1}.
\end{tikzcd}
\end{equation}
\item \em Weak adjoints condition: \rm 
There exists \em some \rm isomorphism $FR \simeq FH^nL$.
\end{itemize}

The two main theorems of this paper are:
\begin{theorem*}[Theorem \ref{theorem-strong-monad-and-weak-adjoints-imply-pn-functor}]
Let $(H, Q_n, \gamma)$ be as above and suppose that 
the strong monad condition and the weak adjoints condition hold. 
If additionally
\begin{equation*}
\label{eqn-the-minus-one-ext-assumption}
\homm^{-1}(\id, H^i) = 0, \quad \quad 1 \leq i \leq n
\end{equation*}
then $(H, Q_n, \gamma)$ is a structure of a $\mathbb{P}^n$-functor on $F$. 
\end{theorem*}
\begin{theorem*}[Theorem \ref{theorem-Pn-functor-gives-autoequivalence}]
Let $\A$ and $\B$ be enhanced triangulated categories. Let 
$F$ be an enhanced functor $\A \rightarrow \B$ with 
enhanced left and right adjoints $L,R\colon \B \rightarrow \A$
and a structure $(H, Q_n, \gamma)$ of a $\mathbb{P}^n$-functor. 

Define the $\mathbb{P}$-twist $P_F$ to be the unique convolution of 
the two-step complex
\begin{equation}
\label{eqn-intro-ptwist-definition-nonsplit-theorem}
FHR \xrightarrow{\psi} FR \xrightarrow{\trace} \id_\B. 
\end{equation}
Then $P_F$ is an autoequivalence of $\B$. 
\end{theorem*}

We also prove the conjecture by Segal in 
\cite[Remark 4.6]{Segal-AllAutoequivalencesAreSphericalTwists}. 
Let $F\colon \A \rightarrow \B$ be a split
$\mathbb{P}^n$-functor as above. 
Segal constructed a category $\A_H$ which can be viewed as as 
the derived category of objects supported on the 
zero section of the noncommutative line-bundle over $\A$ defined by $H$. 
It comes with the functor $j_*\colon \A_H \rightarrow \A$ and a 
lift $\tilde{H}$ of $H$. 
The map $h\colon FH \hookrightarrow FRF \rightarrow F$
defines a deformation $\tilde{F}$ of the functor $FHj_*\colon 
\A_H \rightarrow \B$. 
\begin{theorem*}[Theorem \ref{theorem-segals-conjecture}]
If $F\colon \A \rightarrow \B$ is a split 
$\mathbb{P}^n$-functor satisfying the strong monad condition, 
then  
$$\tilde{F}\colon \A_H \rightarrow \B$$ 
is a spherical functor whose twist is the $\mathbb{P}$-twist 
of $F$, and whose cotwist is $\tilde{H}^{n+1}$. 
\end{theorem*}

Next, we give four families of examples of non-split 
$\mathbb{P}^n$-functors. The first two examples are more formal. One 
is the \em spherical functors\rm. We prove that a functor 
$F\colon \A \rightarrow \B$ is spherical if and only if it is a
$\mathbb{P}^1$-functor and in such case $P_F \simeq T_F^2$, 
the $\mathbb{P}$-twist is the square of its spherical twist
(Prop.~\ref{prps-every-spherical-functor-is-a-P1-functor}
and \ref{eqn-P1-twist-is-the-square-of-a-spherical-twist}). 

The other example are \em extensions by zero \rm of existing
$\mathbb{P}^n$-functors. Let $\D \simeq \left< \A, \C \right>$ be 
an enhanced semi-orthogonal decomposition with 
the gluing functor $N\colon \C \rightarrow \A$. Let 
$F\colon \A \rightarrow \B$ have a 
$\mathbb{P}^n$-functor structure $(H,Q_n,\gamma)$. If
$Q_i N = 0$ for odd $i$ and $Q_i N = N[i]$ for even $i$, 
we can extend $F$ to a $\mathbb{P}^n$-functor 
$\tilde{F}\colon \D \rightarrow \B$ with $\tilde{F}|_\A = F$ 
and $\tilde{F}|_\C = 0$ and $P_{\tilde{F}} \simeq P_F$ 
(Prop.~\ref{prps-extension-by-zero-of-a-pn-functor}). 
Such $\mathbb{P}^n$-functor is necessarily non-split since 
it has by construction a non-trivial kernel. Moreover, 
the strong monad condition doesn't hold for such functors. 

The remaining two examples are more geometrical in nature. First 
are \em cyclic covers\rm. Let $Z$ and $X$ be algebraic varieties
and $f\colon Z \rightarrow X$ be a degree $n + 1$ cyclic cover
ramified in an effective Cartier divisor $D \subset X$, as per 
\S\ref{section-cyclic-covers-the-setup}. Let $E$ be the divisor 
$\frac{1}{n+1}f^{-1}(D)$ on $Z$ and 
$\sigma \in \gpmu_{n+1}$ be a primitive generator of the cyclic group 
$\gpmu_{n+1}$ whose action on $Z$ permutes the branches of the cover. 
The fiber product $Z \times_X Z$ has $n+1$ irreducible
components: the graphs $\Gamma_\lambda$ of $\lambda \in \gpmu_{n+1}$
which can be viewed as the 
$\lambda$-twisted diagonals $(z,\lambda z) \subset Z \times_X Z$. 
Hence $\mathcal{O}_{Z \times_X Z}$ has a filtration whose factors
are the sheaves $\mathcal{O}_{\Delta}, \mathcal{O}_{\Gamma_\sigma}, 
\dots, \mathcal{O}_{\Gamma_{\sigma^n}}$. 

\begin{theorem*}[Theorem \ref{theorem-cyclic-cover-as-a-non-split-Pn-functor}]
With $f: Z \rightarrow X$ as above let 
$F_*$ and $F^!$ be the standard Fourier-Mukai enhancements of 
the direct image functor $f_*$ and its right adjoint $f^!$, as per 
\S\ref{section-standard-fourier-mukai-kernels-and-the-key-lemma}. 
Let $H \simeq \mathcal{O}_{\Gamma_\sigma}(0,E) \in D(Z
\times Z)$ be the standard enhancement of the autoequivalence 
$h = \sigma_* (-) \otimes \mathcal{O}_Z(E)$ of $D(Z)$. 

The structure of the degree $n$ cyclic coextension of $\id_Z$ by $H$ 
given on $F^! F_* \simeq \mathcal{O}_{Z \times_X Z}(0,nE)$ 
by the above decomposition of the fiber product $Z \times_X Z$ 
into its irreducible components defines a $\mathbb{P}^n$-functor
structure on $F_*$ whose twist is 
the autoequivalence $(-) \otimes \mathcal{O}_X(D)$ of $D(X)$. 
\end{theorem*}

Our last example are \em family $\mathbb{P}^n$-twists \rm where
$F\colon D(Z) \rightarrow D(X)$ is defined 
by a flat $\mathbb{P}^n$-fibration over $Z$. The $\mathbb{P}$-twist
is the monodromy of the associated Mukai flop, the
``flop-flop = twist'' formula. This was done in
\cite{AddingtonDonovanMeachan-MukaiFlopsAndPTwists} under the
assumption of $HH^{odd}(Z) = 0$ which forced $RF$ to split up as a sum
of its cohomologies. Our non-split theory works just as well with
the standard filtration of $RF$ by cohomology sheaves:

\begin{theorem*}[Theorem \ref{theorem-family-p-twists}]
Let $Z$ be a smooth projective variety and 
let $\mathcal{V}$ be a vector bundle of rank $n+1$ over $Z$. 
Let $P = \mathbb{P}\mathcal{V}$ and let 
$\pi\colon P \twoheadrightarrow Z$ be the resulting
flat $\mathbb{P}^n$-fibration. Let $X$ be another smooth projective variety 
and let $ \iota \colon P \hookrightarrow X $ be a codimension $n$
closed immersion with $\mathcal{N}_{P/X} \simeq  \Omega^1_{P/Z}.$
Let 
$$ f_k \quad\overset{\text{def}}{=}\quad
\iota_*\left(\mathcal{O}_{P}(k) \otimes \pi^*(-)\right) \colon 
D(Z) \rightarrow D(X), $$ 
let $r_k$ be its right adjoint, and let $h = [-2]$. 
Let $F_k$, $R_k$, and $H$ be their standard enhancements. 

The structure of a cyclic coextension of $\id$ by $H$ of degree $n$
on the adjunction monad $R_k F_k$ provided by the standard filtration
by cohomology sheaves makes $F_k$ into a $\mathbb{P}^n$-functor. 
Let $P_{F_k}$ be its $\mathbb{P}$-twist. If the Mukai flop 
$X \xleftarrow{\beta} \tilde{X} \xrightarrow{\gamma} X'$ exists 
we have an isomorphism in $D(X \times X)$:
\begin{equation}
\label{eqn-flop-flop-equals-twist}	
KN'_{-k} \circ KN_{n + k + 1} \simeq P_{F_k},
\quad \quad \quad \text{`` flop-flop $=$ twist ''} 
\end{equation}
where $KN_{n + k + 1}$ and $KN'_{-k}$ are Kawamata-Namikawa 
derived flop equivalences $D(X) \rightleftarrows D(X')$, 
cf.~\S\ref{section-family-the-setup}. 
\end{theorem*}

We finish with some speculation:

\em Symmetricity: \rm If $F\colon \A \rightarrow \B$ is spherical, 
so are its adjoints $L$ and $R$.  
The definition of a spherical functor is symmetrical with 
respect to $\A$ and $\B$. 
Could we have made our definition of a $\mathbb{P}^n$-functor equally
symmetric? Asking also for $FR$ to be a cyclic extension of $\id_B$ by 
some autoequivalence $J$ satisfying analogues of the monad, the
adjoints, and the highest degree term conditions? Thus $L$ and
$R$ would also be $\mathbb{P}^n$-functors.  
This question wasn't asked before, when only the split notion was
considered, since in almost all the known examples $FR$ is not split. 
Indeed, it is probably too strong to expect all known $\mathbb{P}^n$-functors 
to be symmetrical, 
but it would be interesting to study the ones that are. 
Our first example are the cyclic covers $f\colon Z \rightarrow X$ considered 
above. In \S\ref{section-pn-twists-for-f_*-f^*-f^!} we prove that
while $f_*$ is a non-split $\mathbb{P}^n$-functor, $f^*$ and $f^!$
are split $\mathbb{P}^n$-functors. 

\em $\mathbb{P}$-twist data: \rm Let $(H,Q_n,\gamma)$ 
be a $\mathbb{P}^n$-functor struncture on a functor 
$F\colon \A \rightarrow \B$. By its definition 
\eqref{eqn-intro-ptwist-definition-nonsplit-theorem}
the $\mathbb{P}$-twist $P_F$ is completely determined by 
$F$ and the map $\psi$ defined using 
a fraction of the data of $(H,Q_n,\gamma)$. 
In \S\ref{section-ptwists} we axiomatise this and 
define \em $\mathbb{P}$-twist data $(H,Q_1, \gamma_1)$ \rm
for an arbitrary functor $F$, though the resulting $\mathbb{P}$-twist
is not in general an autoequivalence. Each $\mathbb{P}^n$-functor
structure $(H,Q_n,\gamma)$ contains $\mathbb{P}$-twist 
data $(H, Q_1, \gamma_1)$. Is $(H,Q_n,\gamma)$ 
determined by its $(H, Q_1, \gamma_1)$ in some sense? In 
\S\ref{section-truncated-twisted-tensor-algebras}
we give a natural construction which takes 
any $\mathbb{P}$-twist data $(H, Q_1, \gamma_1)$ and produces
a degree $n$ cyclic coextension $Q_n$ 
of $\id_\A$ by $H$ and a map $\gamma\colon Q_n
\rightarrow RF$. If we apply this to the $\mathbb{P}$-twist data of 
a $\mathbb{P}^n$-functor structure is the result 
another $\mathbb{P}^n$-functor structure on $F$? Having the same 
$\mathbb{P}$-twist data as the old structure 
it would therefore have the same $\mathbb{P}$-twist. 
This question is non-trivial even in the split case. 
For a split $\mathbb{P}^n$-functor
satisfying the strong monad condition, i.e.~the definition of 
\cite{Addington-NewDerivedSymmetriesOfSomeHyperkaehlerVarieties}, 
 the answer is yes. We prove this in Corollary
\ref{cor-the-strongest-monad-condition} and we also prove that the
newly obtained $\mathbb{P}^n$-functor structure satisfies the strongest 
possible monad condition: the matrix of the map 
\eqref{eqn-left-multiplication-by-h-in-RF-monad} is not just upper
triangular with $\id$s down the main diagonal, but the identity matrix
itself. Thus, asking whether the same is true of a general split 
$\mathbb{P}^n$-functor, is asking whether the monad
multiplication which satisfies the monad condition can be made
to satisfy the strong monad condition. In the Appendix we provide
some evidence towards this by studying the case when 
$\A = D(\text{Vect})$. 

\em Segal's conjecture: \rm Our proof of Segal's conjecture raises two
questions. The first is whether its converse is true. Let 
$F\colon \A \rightarrow \B$ be an enhanced functor, 
$H$ be an autoequivalence of $\A$ and $\gamma_1\colon H \hookrightarrow RF$ 
be a direct summand inclusion. As before, the map  
$h \colon FH \hookrightarrow FRF \rightarrow F$ defines a functor 
$\tilde{F}$ from the non-commutative line bundle $\A_H$ to $\B$. 
If $\tilde{F}$ is spherical, is there $n$ such that 
$F$ a split $\mathbb{P}^n$-functor 
satisfying the strong monad condition? If 
it is true, our Corollary \ref{cor-the-strongest-monad-condition}
gives a way to try and prove it. It constructs  
$\gamma\colon \id_\A \oplus \dots \oplus H^n \rightarrow RF$ which 
is an isomorphism if $F$ is a $\mathbb{P}^n$-functor. 
To prove the converse to Segal's conjecture one needs to find a way 
to exploit the fact that $\tilde{F}$ is spherical to 
show that $\gamma$ is an isomorphism. However, it doesn't 
seem straightforward. The second question is whether Segal's 
conjecture can be generalised to non-split $\mathbb{P}^n$-functors.
Here, we do not see what should be the analogue of the map $h$. 
In the non-split case, it is perfectly possible to have
$\homm_{D(\AbimB)}^0(FH, F) = 0$. For example, in the cyclic cover
example in \S\ref{section-cyclic-covers-the-setup} 
$F \in D(Z \times X)$ is $\mathcal{O}_{\Gamma_f}$ where $\Gamma_f$
is the graph of the cyclic cover $f\colon Z \rightarrow X$, while 
$FH \simeq \mathcal{O}_{\Gamma_f}(E,0)$ where
$E$ is $\frac{1}{n}$ of the ramification divisor. 
For smooth projective $Z$ and non-trivial $E$
there are no non-zero sheaf morphisms 
$\mathcal{O}_{\Gamma_f}(E,0) \rightarrow \mathcal{O}_{\Gamma_f}$. 
More conceptually, in the split case Segal's construction works  
because the map $\psi\colon FHR \rightarrow FR$ defining 
the $\mathbb{P}$-twist is a difference of two natural maps. 
One is $hR$, where $h$ is the map defining $\tilde{F}$, and 
the other is $Fh'$, where $h'$ is the map defining 
the right adjoint $\tilde{R}$ of $\tilde{F}$. 
Hence $\tilde{F}\tilde{R} \simeq \cone(\psi)$, 
and the twist of $\tilde{F}$ is the $\mathbb{P}$-twist of $F$. 
In the non-split case the map $\psi$ does exist, but it is no
longer a difference of $hR$ and its right adjoint for some map $h$. Thus 
it is not clear how to proceed. 

On the structure of the paper. In \S\ref{section-preliminaries}
we give preliminaries on DG- and $\Ainfty$-categories, enhancements, 
repeated and cyclic extensions, truncated twisted tensor algebras, 
and Fourier-Mukai kernels. In \S\ref{section-ptwists} we give 
the formalism of $\mathbb{P}$-twist data for arbitrary enhanced
functors. In \S\ref{section-Pn-functors} we give our new definition 
of a $\mathbb{P}^n$-functor and prove its
$\mathbb{P}$-twist to be an autoequivalence. In 
\S\ref{section-strong-monad-condition} we first, 
in the split case, compare our definition of a $\mathbb{P}^n$-functor
with the definition of a split $\mathbb{P}^n$-functor
in \cite{Addington-NewDerivedSymmetriesOfSomeHyperkaehlerVarieties}.
With that insight in mind, we go back to the fully general situation, 
write down a non-split analogue of Addington's definition and prove it 
to imply ours. In \S\ref{section-examples-of-split-Pn-functors}
we give a brief survey of all the examples of split $\mathbb{P}^n$-functors
which appeared in literature to date. Then in
\S\ref{section-examples-of-non-split-Pn-functors} we give four brand
new examples of non-split $\mathbb{P}^n$-functors using our new
theory. 

\em Acknowledgements: \rm We would like to thank 
Nicolas Addington, Alexei Bondal, Will Donovan, 
Alexander Efimov, Andreas Hochenegger, Gabriel Kerr, Andreas Krug, 
Alexander Kuznetsov, Zongzhu Lin, and Richard Thomas for useful
discussions in the course of writing this paper. We would like to
thank Richard Thomas for help with the introduction. 
We would also like to thank Ciaran Meachan for bringing Johnstone's Lemma 
to our attention. The first author would like to thank Kansas State
University for providing a stimulating research environment while
working on this paper.  The second author would like to offer similar
thanks to Cardiff University. Both authors would like to thank Banff
International Research Station in Canada for enabling them to work
together in person, and Castello Montegufoni in Tuscany for doing the
same as well as providing some perfect scenery. 

\section{Preliminaries}
\label{section-preliminaries}

\subsection{DG- and $\Ainfty$-categories}
\label{section-preliminaries-dg-and-ainfty-categories}

For a brief introduction to DG-categories, DG-modules and 
the technical notions involved we direct the reader to 
\cite{AnnoLogvinenko-SphericalDGFunctors}, 
\S2-4. The present paper was written with that survey in mind. 
Below we employ freely any notion or 
piece of notation introduced in \cite{AnnoLogvinenko-SphericalDGFunctors},
\S2-4. We particularly stress 
the importance of the material on \em twisted complexes \rm in 
\cite{AnnoLogvinenko-SphericalDGFunctors}, $\S3$. We also recall that the 
\em derived category $D(\A)$ \rm of a DG-category $\A$ is the
localisation of the homotopy category $H^0(\modA)$ of (right)
DG-modules by the class of quasi-isomorphisms. It is constructed as
Verdier quotient of $H^0(\modA)$ by the full subcategory $H^0(\acyc \A)$ 
consisting of acyclic modules, thus it comes with 
a projection $H^0(\modA) \rightarrow D(\A)$.

For an introduction to $\Ainfty$-categories we recommend
\cite{Keller-AInfinityAlgebrasModulesAndFunctorCategories}, 
for a comprehensive technical text -- \cite{Lefevre-SurLesAInftyCategories}.
For the summary of the technical details relevant to this paper,
see
\cite{AnnoLogvinenko-BarCategoryOfModulesAndHomotopyAdjunctionForTensorFunctors}.
In this paper, we employ freely any notion or a piece of notation
introduced in
\cite{AnnoLogvinenko-BarCategoryOfModulesAndHomotopyAdjunctionForTensorFunctors},
particularly that of the bar category of modules $\modbar(\A)$ as a
DG-enhancement of the derived category $D(\A)$ of a DG-category $\A$. 

The rest of this section contains several new technical results 
for setting up the theoretical machinery of
$\mathbb{P}^n$-functors in the language of
DG-bimodules and their bar-categories. 

Recall 
\cite[\S 3.7]{Drinfeld-DGQuotientsOfDGCategories}, 
\cite{Tabuada-UneStructureDeCategorieDeModelesDeQuillenSurLaCategorieDesDG-Categories}, 
\cite[Appendix A]{AnnoLogvinenko-SphericalDGFunctors} that
in any DG-category any homotopy equivalence $y \xrightarrow{\alpha} x$
can be completed to the following system of morphisms and relations:
\begin{equation}
\label{eqn-tabuadas-universal-homotopy-equivalence-category}
\begin{minipage}[c][1in][c]{1.5in}
$d \theta_x = \alpha \circ \beta - \id_x,$ \\
$d \theta_y = \id_y - \beta \circ \alpha,$ \\
$d \alpha = d \beta = 0,$ \\
$d \phi = - \beta \circ \theta_x - \theta_y \circ \beta.$ 
\end{minipage}
\quad\quad
\begin{tikzcd}[column sep={2cm},row sep={1.5cm}] 
x
\ar[bend left=20]{rr}{\beta}
\ar[bend left=50, dashed]{rr}{\phi}
\ar[out=-150,in=150,loop,distance=6em,dotted]{}{\theta_x}
& &
y
\ar[bend left=20]{ll}{\alpha}
\ar[out=30, in=-30,loop,distance=6em, dotted]{}{\theta_y}
\end{tikzcd}
\end{equation}

\begin{lemma}[Replacement Lemma]
\label{lemma-replacement-lemma} 

Let $E \in \modA$ and let its underlying graded module split
as $Q \oplus R$. Let the differential of $E$ with respect to that splitting be
\begin{equation}
\label{eqn-differential-of-the-original-module}
\left(
\begin{matrix}
d_Q & \eta
\\
\zeta & \delta 
\end{matrix}
\right).
\end{equation}

Assume futher that $d_Q$ is a differential on $Q$. 
Let $P \in \modA$ be homotopy equivalent to $Q \in \modA$. 
More specifically, 
let $Q \xrightarrow{\beta} P$ and $P \xrightarrow{\alpha} Q$ 
be homotopy equivalences such that there exist
$\theta_Q \in \homm^{-1}_{\A}(Q,Q)$, 
$\theta_P \in \homm^{-1}_{\A}(P,P)$ and 
$\phi \in \homm^{-2}_{\A}(Q,P)$ as on 
\eqref{eqn-tabuadas-universal-homotopy-equivalence-category}. 

Then $E$ is homotopy equivalent to $P \oplus R$ 
equipped with the differential 
\begin{equation}
\label{eqn-differential-of-the-replaced-module}
\left(
\begin{matrix}
d_P & \beta \circ \eta  
\\
\zeta \circ \alpha & \delta - \zeta \circ \theta_Q \circ \eta
\end{matrix}
\right).
\end{equation}
\end{lemma}
\begin{proof}
By assumption, $d_Q^2 = 0$. Since
\eqref{eqn-differential-of-the-original-module} is 
a derivation, $\delta$ is also a derivation, 
while $\eta$ and $\zeta$ are maps of graded $\A$-modules. 
Since \eqref{eqn-differential-of-the-original-module}
squares to zero, we have  
$$ \eta \circ \zeta = 0, $$
$$ d_Q \circ \eta + \eta \circ \delta = 0, $$
$$ \zeta \circ d_Q + \delta \circ \zeta = 0. $$
$$ \zeta \circ \eta + \delta^2 = 0, $$

We first claim that $\eqref{eqn-differential-of-the-replaced-module}$ 
is a differential, and thus makes $P \oplus R$ into DG $\A$-module.
Indeed, the
 map $\eqref{eqn-differential-of-the-replaced-module}$ is a derivation 
as it is the sum of the derivation 
$\left( 
\begin{smallmatrix}
d_P & 0 \\
0 & \delta 
\end{smallmatrix}
\right)
$
and the graded $\A$-module map 
$\left( 
\begin{smallmatrix}
0 & 
\beta \circ \eta  
\\
\zeta \circ \alpha & - \zeta \circ \theta_Q \circ \eta
\end{smallmatrix}
\right). 
$
Moreover, we have 
\begin{align*}
d_P^2 + \beta \circ \eta \circ \zeta \circ \alpha & = 0 + 
\beta \circ 0 \circ \alpha = 0, \\
d_P \circ \beta \circ \eta + \beta \circ \eta \circ (\delta
- \zeta \circ \theta_Q \circ \eta) &= 
d_P \circ \beta \circ \eta - \beta \circ d_Q \circ \eta
- \beta \circ 0 \circ \theta_Q \circ \eta = 
\\
& = (d_P \circ \beta - \beta \circ d_Q) \circ \eta = 0 \circ \eta = 0, 
\\
\zeta \circ \alpha \circ d_P + (\delta - \zeta \circ \theta_Q \circ \eta)
\circ \zeta \circ \alpha &= 
\zeta \circ \alpha \circ d_P - \zeta \circ d_Q \circ \alpha 
- \zeta \circ \theta_Q \circ 0 \circ \alpha = 
\\
& = \zeta \circ (\alpha \circ d_P - d_Q \circ \alpha) = \zeta \circ 0 = 0, 
\\
\zeta \circ \alpha \circ \beta \circ \eta + 
(\delta - \zeta \circ \theta_Q \circ \eta)^2
&  = \zeta \circ (\id_Q - d\theta_Q) \circ \eta 
+  \delta^2 - \delta \circ \zeta \circ \theta_Q \circ \eta
- \zeta \circ \theta_Q \circ \eta \circ \delta + 
\zeta \circ \theta_Q \circ \eta \circ \zeta \circ \theta_Q \circ \eta = 
\\
& = \zeta \circ \eta + \delta^2 - \zeta \circ d\theta_Q \circ \eta
+ \zeta \circ d_Q \circ \theta_Q \circ \eta
- \zeta \circ \theta_Q \circ d_Q \circ \eta + 
\zeta \circ \theta_Q \circ 0 \circ \theta_Q \circ \eta
=
\\
& = 0 - \zeta \circ d\theta_Q \circ \eta
+ \zeta \circ (d_Q \circ \theta_Q - \theta_Q \circ d_Q) \circ \eta
= 
- \zeta \circ d\theta_Q \circ \eta
+ \zeta \circ d\theta_Q \circ \eta
= 0. 
\end{align*}
 
Now, consider the following degree $0$ maps of $\mathcal{A}$-modules:
\begin{align}
\label{eqn-extraction-lemma-Q+R-to-P+R-homotopy-equivalence}
Q \oplus R 
\xrightarrow{ 
\left(
\begin{smallmatrix}
\beta & 0 \\
- \zeta \circ \theta_Q & \id 
\end{smallmatrix}
\right)
}
P \oplus R, 
\\
\label{eqn-extraction-lemma-P+R-to-Q+R-homotopy-equivalence} 
P \oplus R
\xrightarrow{ 
\left(
\begin{smallmatrix}
\alpha & -\theta_Q \circ \eta \\
0 & \id 
\end{smallmatrix}
\right)
}
Q \oplus R.
\end{align}
These maps are readily seen to commute with the differentials, 
and thus are closed. We claim that they define mutually 
inverse homotopy equivalences 
$Q \oplus R \xrightarrow{\sim} P \oplus R$
and $P \oplus R \xrightarrow{\sim} Q \oplus R$ in $\modA$. Indeed, 
\begin{align}
\eqref{eqn-extraction-lemma-P+R-to-Q+R-homotopy-equivalence}
\circ 
\eqref{eqn-extraction-lemma-Q+R-to-P+R-homotopy-equivalence}
= 
\left(
\begin{smallmatrix}
\alpha & -\theta_Q \circ \eta \\
0 & \id 
\end{smallmatrix}
\right)
\circ 
\left(
\begin{smallmatrix}
\beta & 0 \\
- \zeta \circ \theta_Q & \id 
\end{smallmatrix}
\right)
=
\left(
\begin{smallmatrix}
\alpha \circ \beta & -\theta_Q \circ \eta \\
- \zeta \circ \theta_Q & \id 
\end{smallmatrix}
\right)
=
\left(
\begin{smallmatrix}
\id & 0 \\
0 & \id
\end{smallmatrix}
\right)
-
\left(
\begin{smallmatrix}
d\theta_Q & \theta_Q \circ \eta \\
\zeta \circ \theta_Q & 0 
\end{smallmatrix}
\right)
= 
\id - d
\left(
\begin{smallmatrix}
\theta_Q & 0 \\
0 & 0 
\end{smallmatrix}
\right)
\end{align}
while
\begin{align*}
\eqref{eqn-extraction-lemma-Q+R-to-P+R-homotopy-equivalence}
\circ 
\eqref{eqn-extraction-lemma-P+R-to-Q+R-homotopy-equivalence}
&= 
\left(
\begin{smallmatrix}
\beta & 0 \\
- \zeta \circ \theta_Q & \id 
\end{smallmatrix}
\right)
\circ
\left(
\begin{smallmatrix}
\alpha & -\theta_Q \circ \eta \\
0 & \id 
\end{smallmatrix}
\right)
= 
\left(
\begin{smallmatrix}
\beta \circ \alpha & -\beta \circ \theta_Q \circ \eta \\
-\zeta \circ \theta_Q \circ \alpha & \zeta \circ \theta_Q^2 \circ \eta + \id 
\end{smallmatrix}
\right)
= 
\left(
\begin{smallmatrix}
\id & 0 \\
0 & \id 
\end{smallmatrix}
\right)
-
\left(
\begin{smallmatrix}
-d\theta_P & \beta \circ \theta_Q \circ \eta \\
\zeta \circ \theta_Q \circ \alpha & - \zeta \circ \theta_Q^2 \circ \eta  
\end{smallmatrix}
\right)
= 
\\
& =
\left(
\begin{smallmatrix}
\id & 0 \\
0 & \id 
\end{smallmatrix}
\right)
+
d\left(
\begin{smallmatrix}
\theta_P & 0 \\
0 & 0 
\end{smallmatrix}
\right)
- 
\left(
\begin{smallmatrix}
0 & (\theta_P \circ \beta + \beta \circ \theta_Q) \circ \eta \\
\zeta \circ (\alpha \circ \theta_P + \theta_Q \circ \alpha) & 
- \zeta \circ \theta_Q^2 \circ \eta
\end{smallmatrix}
\right)
. 
\end{align*}
Now let 
\begin{align*}
\psi = \alpha \circ \phi \circ \alpha + \theta_Q^2 \circ \alpha +
\alpha \circ \theta_P^2 + \theta_Q \circ \alpha \circ \theta_P 
\end{align*} 
then 
$$ d\psi = - \alpha \circ \theta_P - \theta_Q \circ \alpha. $$
We then further have 
\begin{align*}
\left(
\begin{smallmatrix}
0 & (\theta_P \circ \beta + \beta \circ \theta_Q) \circ \eta \\
\zeta \circ (\alpha \circ \theta_P + \theta_Q \circ \alpha) & 
- \zeta \circ \theta_Q^2 \circ \eta 
\end{smallmatrix}
\right)
=
d\left(
\begin{smallmatrix}
0 & - \phi \circ \eta \\
\zeta \circ \psi & 0 
\end{smallmatrix}
\right)
+
\left(
\begin{smallmatrix}
0 & 0 \\
0 & \zeta \circ (\psi \circ \beta - \alpha \circ \phi - \theta^2_Q) \circ \eta 
\end{smallmatrix}
\right). 
\end{align*}
For any map $f: Q \rightarrow Q$ we have $d\left(
\begin{smallmatrix}
0 & 0 \\
0 & \zeta \circ f \circ \eta
\end{smallmatrix}
\right) = \left(
\begin{smallmatrix}
0 & 0 \\
0 & \zeta \circ df \circ \eta
\end{smallmatrix}
\right)$, thus it remains to show that 
$\psi \circ \beta - \alpha \circ \phi + \theta^2_Q$ is a boundary.
Indeed, it can be readily checked that  
$$\psi \circ \beta - \alpha \circ \phi - \theta^2_Q  
= d(\alpha \circ \phi \circ \theta_Q + \theta_Q \circ \alpha \circ
\phi + \alpha \circ \theta_P \circ \phi + \theta_Q^3). $$  
\end{proof}

Next is the lemma which allows us to cancel an
additive functor $F$ from certain natural transformations 
involving its right or left adjoint:

\begin{lemma}[Cancellation Lemma]
\label{lemma-cancelling-F-from-natural-transformations-involving-R-or-L}
Let $\A$ and $\B$ be two additive categories. Let $F\colon \A
\rightarrow \B$ be an additive functor and $R, L \colon \B
\rightarrow \A$ be its right and left adjoints. 

Let $\C$ be an additive category. Let 
$G\colon \A \rightarrow \C$, $H\colon \B \rightarrow \C$, 
$I\colon \C \rightarrow \B$, and $J\colon \C \rightarrow \A$
be additive functors.  

\begin{enumerate}
\item \label{item-GR-alpha-H}
For any natural transformation $GR \xrightarrow{\alpha} H$
we have $\alpha = 0$ if and only if $\alpha F = 0$. 
\item For any natural transformation $J \xrightarrow{\alpha} RI$
we have $\alpha = 0$ if and only if $F \alpha = 0$. 
\item For any natural transformation $H \xrightarrow{\alpha} GL$
we have $\alpha = 0$ if and only if $\alpha F = 0$. 
\item For any natural transformation $LI \xrightarrow{\alpha} J$
we have $\alpha = 0$ if and only if $F \alpha = 0$. 
\end{enumerate}
\end{lemma}
\begin{proof}
We only prove the assertion \eqref{item-GR-alpha-H}, as the rest are proved
similarly.

The ``only if'' implication is trivial. For ``if'' implication, 
assume that 
$$ GRF \xrightarrow{ \alpha F } HF $$
is the zero morphism. Then so is the composition  
$$ GRFR \xrightarrow{ \alpha FR } HFR \xrightarrow{H\trace} H. $$ 
By naturality it equals the composition 
$$ GRFR \xrightarrow{GR \trace} GR \xrightarrow{ \alpha } H $$
which is therefore also zero. But $GRFR \xrightarrow{ GR \trace } GR$
has the left inverse $GR \xrightarrow{G \action R} GRFR$, thus we
conclude that $GR \xrightarrow{\alpha} H$ must be zero. 
\end{proof}

Recall that in the bar category of bimodules 
we have left (resp. right) dualisation functor which sends
any left (resp. right) perfect bimodule to its left (resp. right) 
homotopy adjoint \cite[\S4.2]{AnnoLogvinenko-BarCategoryOfModulesAndHomotopyAdjunctionForTensorFunctors}.  
The next lemma describes the action of the left dualisation 
functor on morphisms between bimodules in terms of the corresponding 
homotopy adjunctions. There is also an identical result for 
the right dualisation functor, which we leave to the reader. 
\begin{lemma}[Dualisation Lemma]
\label{lemma-left-dual-morphism}
Let $\A$ and $\B$ be two DG-categories. 
Let $M, N \in \AmodbarB$ be $\A$-perfect.  
For any $f \in \barhom_{\AbimB}(M,N)$ the composition
\begin{equation*}
N^\barA 
\xrightarrow{\action \bartimes \id}
M^\barA \bartimes M \bartimes N^\barA
\xrightarrow{\id \bartimes f \bartimes \id}
 M^\barA\bartimes N \bartimes N^\barA
\xrightarrow{\id \bartimes \trace}
M^\barA
\end{equation*}
is homotopic to the map $f^\barA \in \barhom_{\BbimA}(N^\barA, M^\barA)$.
\end{lemma}
\begin{proof}
The following diagram commutes up to homotopy:
\begin{small}
\begin{equation*}
\begin{tikzcd}[column sep={1.5cm}]
N^\barA 
\ar{r}{\action \bartimes \id}
\ar[equals]{d}
&
M^\barA \bartimes M \bartimes N^\barA
\ar{r}{\id \bartimes f \bartimes \id}
\ar{d}
&
 M^\barA\bartimes N \bartimes N^\barA
\ar{r}{\id \bartimes \trace}
\ar{d}
&
M^\barA
\ar[equals]{d}
\\
\barhom_\A(N,\A)
\ar{r}{\action\bartimes\id}
&
\barhom_\A(M,M) \bartimes \barhom_\A(N,\A)
\ar{r}{(f\circ(-))\bartimes\id}
&
\barhom_\A(M,N) \bartimes\barhom_\A(N,\A)
\ar{r}{\composition}
&
 \barhom_\A(M,\A).
\end{tikzcd}
\end{equation*}
\end{small}
\end{proof}

\begin{cor}
\label{cor-f-and-f-dual-commuting-square}
Let $\A$ and $\B$ be two DG-categories. 
Let $M, N \in \AmodbarB$ be $\A$-perfect bimodules. 
For any $f \in \barhom_{\AbimB}(M,N)$ the 
following square commutes up to homotopy:
\begin{equation*}
\begin{tikzcd}[column sep={2cm}]
\B
\ar{r}{\action}
\ar[']{d}{\action}
&
M^\barA \bartimes M
\ar{d}{\id \bartimes f}
\\
N^\barA \bartimes N
\ar{r}{f^\barA \bartimes \id}
&
M^\barA \bartimes N.
\end{tikzcd}
\end{equation*}
\end{cor}
\begin{proof}
The top contour is the element in $\homm_{D(\BbimB)}(\B, M^\barA\bartimes N)$
that corresponds to $f$
under the isomorphism 
$$ \homm_{D(\AbimB)}(M,N) \simeq \homm_{D(\BbimB)}(\B, M^\barA\bartimes N) $$
and the bottom contour is the element that corresponds to 
$f^\barA$ under the isomorphism
$$\homm_{D(\BbimB)}(\B, M^\barA\bartimes N) \simeq
\homm_{D(\BbimA)}(N^\barA,M^\barA).$$
By Lemma \ref{lemma-left-dual-morphism}
the element of $\homm_{D(\BbimA)}(N^\barA,M^\barA)$ that corresponds 
to $f$ under the composition of the above adjunction isomorphisms 
is $f^\barA$, whence the top and the bottom contour are equal 
in $D(\BbimB)$ as desired. 
\end{proof}

The next two lemmas give a criterion for a map 
between one-sided twisted complexes to be itself one-sided:

\begin{lemma}
\label{lemma-onesidedness-via-true-filtration}
Let $\C$ be a DG category. Let $X=(X_i,\alpha_{ij})$
and $Y=(Y_i, \beta_{ij})$ be one-sided twisted complexes in $\pretriag \C$. 
Denote by $X_{\geq i}$ (resp. $Y_{\geq i}$) the subcomplex of $X$ (resp. $Y$) 
that sits in degrees $\geq i$.
Let $f: X \to Y$ be a map of twisted complexes. If for every $i$ 
the restriction of $f$ to $X_{\geq i}$ is homotopic to a map which 
filters through $Y_{\geq i}$, then $f$ is homotopic to 
a one-sided map of twisted complexes.
\end{lemma}
\begin{proof}
Without loss of generality let $Y$ sit in degrees $0$ to $n$. 
We are going to construct a chain of homotopic maps $f_m: X\to Y$
with $f_{-1}=f$ and such that the projection of $f_m$ onto $Y_{\leq m}$
is one-sided. Suppose we have constructed the map $f_{m-1}$. 
The restriction of $f_{m-1}$ to $X_{\geq m}$ filters through 
$Y_{ \geq m}$. Thus so does the restriction of $f_{m-1}$ 
to $X_{\geq m + 1}$. 
Since $X$ is one-sided the inclusion of $X_{\geq m+1}$ into $X$ is
closed. Hence the restrictions of $f$ and $f_{m-1}$ to $X_{\geq
m+1}$ are homotopic. Hence $f_{m-1}|_{X_{\geq m + 1}}$ 
is still homotopic to a map that filters through $Y_{\geq m+1}$. 
Let 
$$ h_m \colon X_{\geq m+1} \to Y_{\geq m} $$ be the homotopy. 
Let $\tilde{h}_m$ be the map
$X \rightarrow Y$ that has the same components as $h_m$. 
Since $Y$ is one-sided $d\tilde{h}_m$ coincides with $dh_m$ on all terms of $X$ in degrees
$\geq m+1$. Define $f_m$ to be $f_{m-1}+d\tilde{h}_m$. 
By one-sidedness of $Y$ again, 
the projection of $d\tilde{h}_m$ to $Y_{\leq m-1}$
is $0$, thus the projection of $f_m$ to $Y_{\leq m-1}$ is still
one-sided. On the other hand, by construction $f_m$ has no components 
going from the degrees $ \geq m + 1$ to the degrees $\leq m$. We conclude that
the projection of $f_m$ to $Y_{\leq m}$ is also one-sided. 

We can thus construct $f_m$ for any $m \geq -1$. The map $f_n$ is 
one-sided and homotopic to $f$, as desired.  
\end{proof}

\begin{lemma}
\label{lemma-resolution-of-true-filtration}
Let $\C$ be a DG category. Let $X=(X_{ij}, \alpha_{ij,kl})$ be a 
twisted bicomplex over $\C$ which is one-sided in both directions, i.e. 
$\alpha_{ij;kl} = 0$ for $i > k$ or $j > l$.
Then the complex $\tot(X)_{\geq m}$ is homotopy equivalent
to the cone of the map of twisted complexes
\begin{equation}
\label{eqn-resolution-of-true-filtration}
\bigoplus\limits_{i+j=m+1} \tot(X_{\geq i, \geq j})
\xrightarrow{\sum \iota_{ij}^{i-1,j} - \sum \iota_{ij}^{i,j-1}}
\bigoplus\limits_{i+j=m} \tot(X_{\geq i, \geq j}), 
\end{equation}
where $\iota_{ij}^{kl}$ is the totalization of the natural map $X_{\geq i,\geq j} \to X_{\geq k, \geq l}$ for $i\geq k$, $j\geq l$.
\end{lemma}

\begin{proof}
The natural map from the cone of \eqref{eqn-resolution-of-true-filtration} 
to $\tot(X)_{\geq m}$ given by $\Sigma_{m = i +j} \iota_{ij}^m$ where 
$$ \iota_{ij}^m: \tot(X_{\geq i, \geq j})\to \tot(X)_{\geq m} $$
can be readily checked to be a homotopy equivalence. 

\end{proof}

\begin{lemma}
\label{lemma-true-filtration-filtering-via-double-filtration-filtering}
Let $\C$ be a DG category. Let $X=(X_{ij}, \alpha_{ij,kl})$ be 
a twisted bicomplex over $\C$ which is bounded above and one-sided 
in both directions, i.e.  $\alpha_{ij,kl} = 0$ for $i > k$ or $j > l$. 
Let $Y=(Y_i, \beta_{ij})$ be a one-sided bounded below twisted complex 
over $\C$. Assume also that $\homm_{D(\C)}^{-1}(X_{ij}[-(i+j)], Y_k[-k])=0$ 
for $i+j>k$.

Let $f$ be a map $\tot(X) \rightarrow Y$. If for every $i,j$ 
with $i + j \geq m$ 
the restriction of $f$ to $\tot(X_{\geq i, \geq j})$ is homotopic 
to a map which filters through $Y_{\geq m}$,
then the restriction of $f$ to $\tot(X)_{\geq m}$
filters through $Y_{\geq m}$.
\end{lemma}
\begin{proof}
By Lemma \ref{lemma-resolution-of-true-filtration} the complex
$\tot(X)_{\geq m}$ is homotopy equivalent to the totalisation 
of the two-step twisted complex  
\eqref{eqn-resolution-of-true-filtration}
in $\pretriag (\pretriag \C)$. 
Moreover, this homotopy equivalence identifies the map 
$$ \tot(X)_{\geq m} \to \tot(X) \xrightarrow{f} Y \rightarrow Y_{\leq m - 1} $$
in $\pretriag \C$ with (the totalisation of) some map 
\begin{equation}
\label{equation-resolution-of-a-truncated-bicomplex-map}
\begin{tikzcd}[column sep={3cm}]
\bigoplus\limits_{i+j=m+1} \tot(X_{\geq i, \geq j})
\ar{r}{\sum \iota_{ij}^{i-1,j} - \sum \iota_{ij}^{i,j-1}}
\ar{rd}[']{0}
&
\underset{\degzero}{\bigoplus\limits_{i+j=m} \tot(X_{\geq i, \geq j})}
\ar{d}{\sum \eta_{ij}}
\\ 
&
\underset{\degzero}{Y_{\leq m - 1}}
\end{tikzcd}
\end{equation}
of the twisted complexes in $\pretriag (\pretriag \C)$. Now, 
by assumption all $\eta_{ij}$ are null-homotopic. Thus 
we can find degree $-1$ map $h$ with $dh =\sum \eta_{ij}$. Therefore
\eqref{equation-resolution-of-a-truncated-bicomplex-map}
is homotopic to the map 
\begin{equation}
\label{equation-resolution-of-a-truncated-bicomplex-map}
\begin{tikzcd}[column sep={3cm}]
\bigoplus\limits_{i+j=m+1} \tot(X_{\geq i, \geq j})
\ar{r}{\sum \iota_{ij}^{i-1,j} - \sum \iota_{ij}^{i,j-1}}
\ar{rd}[']{-h \circ (\sum \iota_{ij}^{i-1,j} - \sum \iota_{ij}^{i,j-1})}
&
\underset{\degzero}{\bigoplus\limits_{i+j=m} \tot(X_{\geq i, \geq j})}
\ar{d}{0}
\\ 
&
\underset{\degzero}{Y_{\leq m - 1}}
\end{tikzcd}
\end{equation}
Since $\homm^{-1}_{D(\C)}\left(\tot(X_{\geq i, \geq j}),Y_{\leq m-1}\right) = 0$
for all $i + j \geq m$ 
every such map is null-homotopic, as desired. 
\end{proof}

\subsection{DG-enhancements of categories and functors}
\label{section-dg-enhancements-of-categories-and-functors}

Our main setup in this paper is two enhanced triangulated
categories $\A$ and $\B$ and an enhanceable functor 
$$ F\colon \A \rightarrow \B $$ 
which has enhanceable left and right adjoints 
$$ L,R \colon \B \rightarrow \A. $$

For technical details on DG-enhancements, pretriangulated categories, 
quasi-functors, twisted complexes, etc. see
\cite[\S3]{AnnoLogvinenko-SphericalDGFunctors} or
\cite{LuntsOrlov-UniquenessOfEnhancementForTriangulatedCategories}. 
Below we give a brief summary and fix the notation. 

An \em enhanced triangulated category \rm is a pre-triangulated DG-category. 
For any DG-category $\A$ there is a natural triangulated structure 
on $H^0(\modA)$ induced from $\modk$. The DG-category $\A$ 
is \em pre-triangulated \rm if the Yoneda embedding of $H^0(\A)$ 
into $H^0(\modA)$ is a triangulated subcategory. In our context, 
the truncation $H^0(\A)$ is the underlying triangulated category 
and $\A$ its DG-enhancement. We consider enhanced
triangulated categories up to \em quasi-equivalences\rm, 
the DG-functors $\A \rightarrow \B$ for which 
$H^*(\A) \rightarrow H^*(\B)$ and hence 
$H^0(\A) \rightarrow H^0(\B)$ are equivalences. 
An exact functor $f\colon H^0(\A) \rightarrow H^0(\B)$ is 
\em enhanceable \rm if there exists a DG-functor between some pair 
of DG-categories quasi-equivalent to $\A$ and $\B$ whose
$H^0$-truncation is $f$ up to identifications of its domain and
image with $H^0(\A)$ and $H^0(\B)$. See 
\cite{Toen-TheHomotopyTheoryOfDGCategoriesAndDerivedMoritaTheory}
for further detail. 

For technical reasons, it is best to work with Morita enhancements.
These are DG-categories of DG-modules over a small DG-category. 
Given a small enhanced triangulated category $\A$ 
we can enlarge it to $\prfhprA$, the (still small) 
category of $h$-projective, perfect DG $\A$-modules. 
On the underlying triangulated categories this has 
the effect of Karoubi completion, since $H^0(\prfhprA)$ is 
the Karoubi envelope of $H^0(\A)$. More concretely, 
it is the direct summand completion of $H^0(\A)$ inside $H^0(\modA)$. 
It is also canonically isomorphic to $D_c(\A)$, the derived category 
of perfect $\A$-modules. If the underlying triangulated category 
$H^0(\A)$ was Karoubi complete to start with, 
then the Yoneda embedding $\A \hookrightarrow \prfhprA$ is 
a quasi-equivalence, and $\A$ and $\prfhprA$ are equivalent
enhancements of the triangulated category $H^0(\A) \simeq D_c(\A)$. 
Thus any small Karoubi-complete enhanced triangulated category can 
be identified with the \em standard enhancement \rm of 
the derived category $D_c(\A)$ of some small  $\A$. 
Here the standard enhancement is the quasi-equivalence class 
containing e.g. $\prfhprA$ or $\modbar^{\perf}(A)$, the bar-category 
of perfect $\A$-modules introduced in
\cite{AnnoLogvinenko-BarCategoryOfModulesAndHomotopyAdjunctionForTensorFunctors}.

The Morita point of view has the advantage that all functors 
are naturally enhanced by bimodules: any enhanceable exact 
functor $f: D_c(\A) \rightarrow D_c(B)$ is a \em tensor functor\rm, 
that is --- isomorphic to $(-) \ldertimes_\A M$ 
for some $\B$-perfect $M \in D(\AbimB)$ \cite[Theorem
7.2]{Toen-TheHomotopyTheoryOfDGCategoriesAndDerivedMoritaTheory}. 
We say that the bimodule $M$ \em enhances \rm the functor $f$.
Such $f$ has an enhanceable left adjoint $l\colon D_c(\B) \rightarrow
D_c(\A)$ if and only if $M$ is also $\A$-perfect. In such case
$l$ is enhanced by the bimodule $M^\barA$. On the level of full
derived categories, the functor $f$ always has a right adjoint $r\colon
D(\B) \rightarrow D(\A)$ which is enhanced by the bimodule $M^\barB$.
This right adjoint restricts to a functor $D_c(\B) \rightarrow
D_c(\A)$ if and only if $M^\barB$ is $\A$-perfect. See
\cite[\S5.1]{AnnoLogvinenko-BarCategoryOfModulesAndHomotopyAdjunctionForTensorFunctors}
for the details. Thus, enhanceable functors $f\colon D_c(\A) \rightarrow D_c(B)$ with enhanceable left and right adjoints 
form a triangulated category which is equivalent to the full subcategory of $D(\AbimB)$ consisting of $\A$- and $\B$-perfect bimodules $M$ 
for which $M^\barB$ is also $\A$-perfect. 

In this paper we prove our results working in this subcategory of 
$D(\AbimB)$. They immediately apply to all enhanceable functors 
between small Karoubi-complete enhanced triangulated categories 
and to the tensor functors between big categories. These
are precisely the \em continuous \rm enhanceable functors, 
i.e.  those which commute with arbitrary direct sums, cf.
\cite[\S5.1]{AnnoLogvinenko-BarCategoryOfModulesAndHomotopyAdjunctionForTensorFunctors}

Does this mean that our results do not apply to discontinuous 
exact $f\colon \A \rightarrow \B$ between big $\A$ and $\B$? 
No, because we can enlarge our universe to make $\A$ and $\B$ small.
Then $\A$ and $\B$ can be considered in Morita enhancement framework, 
where every enhanceable functor is a tensor functor. A reader may well 
be confused at this stage: how did we make a discontinuous functor into 
a continuous one by mucking about with set-theoretic issues? 
The authors were also pretty confused by this, until they realised 
the following: 

\underline{On continuity, enhanceability, and the universe:}

Let $\A$ be a small enhanced triangulated category. It doesn't have arbitrary 
(indexed by a set)  direct sums. Otherwise, it would have to contain 
a direct sum of all its objects, which is impossible. However, it may still 
have some infinite direct sums, e.g. indexed by $\mathbb{Z}$. E.g. let $\A$ be the unbounded derived category $D_{qc}(X)$ of
quasi-coherent sheaves on a given scheme $X$ and choose a Grothendieck
universe in which $\A$ is small. 

\begin{center}
\underline{\em The embedding $D_c(\A) \hookrightarrow D(\A)$ doesn't preserve infinite direct sums!}
\end{center}

 Thus a continuous $f \colon D(\A) \rightarrow D(\B)$ doesn't have to
respect the direct sums in $D_c(\A)$ and $D_c(\B)$. Hence, somewhat
counter intuitively, in Morita enhancement framework all functors
$D_c(\A) \rightarrow D_c(\B)$, including the discontinuous ones, are
enhanced by bimodules and thus extend to continuous functors $D(\A)
\rightarrow D(\B)$. 
 
 Another way to say the same thing is as folows. Let $\A$ and $\B$ be
two small enhanced triangulated categories and let $f\colon D(\A)
\rightarrow D(\B)$ be a discontinous exact functor.  Then it can't be
enhanced by an $\AbimB$-bimodule.  However $D(\A) \simeq
D_c(\modbarA)$ and $D(\B) \simeq D_c(\modbarB)$.  The induced functor
$f\colon D(\modbarA) \rightarrow D(\modbarB)$ is, by construction,
continuous.  Thus while $f$ can not be enhanced by $\AbimB$-bimodule,
it can be enhanced by a $(\modbarA)$-$(\modbarB)$-bimodule, provided we
enlarge the universe appropriately. 

\subsection{Repeated extensions}
\label{section-repeated-extensions}

We begin by fixing terminology for some well-known concepts: 

\begin{defn}
Let $\C$ be a triangulated category and let $E,F \in \C$.
We say that $X \in \C$ is an \em extension \rm of $F$ by $E$, and, correspondingly, a \em co-extension \rm of $E$ by $F$, if there
exists an exact triangle 
$$ E \rightarrow X \rightarrow F \rightarrow E[1]. $$
\end{defn}

Up to shifts, the extensions correspond to the cones of morphisms of two objects. The following corresponds to convolutions of complexes of several objects, or, more precisely, to the convolutions of twisted complexes of them in a DG-enhancement:

\begin{defn}
Let $\C$ be a triangulated category and let $E_0, E_1, \dots, E_n \in \C$.
We say that $Q \in \C$ is a \em repeated extension \rm of $E_0$ by
$E_1, \dots, E_n$  if there exist objects $Q_1, \dots, Q_{n-1} \in \C$ and 
exact triangles
\begin{align}
\nonumber
E_1 \rightarrow Q_1 \rightarrow E_0 \rightarrow E_1[1] \\
\label{eqn-repeated-extension-data}
E_2 \rightarrow Q_2 \rightarrow Q_1 \rightarrow E_2[1] \\
\nonumber
\dots \\
\nonumber
E_n \rightarrow Q \rightarrow Q_{n-1} \rightarrow E_n[1].
\end{align}

Similarly, we say that $Q \in \C$ is a \em repeated co-extension \rm of $E_n$ by $E_{n-1}, E_{n-2}, \dots, E_1$ if there exist objects $Q_1, \dots, Q_{n-1} \in \C$ and exact triangles
\begin{align}
\nonumber
E_{n} \rightarrow Q_1 \rightarrow E_{n-1} \rightarrow E_{n}[1] \\
Q_1 \rightarrow Q_2 \rightarrow E_{n-2} \rightarrow Q_1[1] \\
\nonumber
\dots \\
\nonumber
Q_{n-1} \rightarrow Q \rightarrow E_{0} \rightarrow Q_{n-1}[1].
\end{align}
\end{defn}

The data \eqref{eqn-repeated-extension-data} fits into the diagram:
\begin{footnotesize}
\begin{equation}
\label{eqn-repeated-extension-postnikov-tower}
\begin{tikzcd}
E_0
\ar[dotted]{dr}
\ar[phantom]{drr}[description, pos=0.45]{\star}
& &
Q_1 
\ar[phantom]{d}[description]{\circlearrowright}
\ar{ll}
\ar{dr}
\ar[phantom]{drr}[description, pos=0.45]{\star}
& &
Q_2  
\ar[phantom]{d}[description]{\circlearrowright}
\ar{ll}
\ar{dr}
&
\dots
\ar{l}
&
Q_{n-2}
\ar[phantom]{d}[description]{\circlearrowright}
\ar[phantom]{drr}[description, pos=0.45]{\star}
\ar{l}
\ar{dr}
& 
& 
Q_{n-1}
\ar[phantom]{d}[description]{\circlearrowright}
\ar{dr}
\ar{ll}
& & 
Q
\ar[phantom]{dll}[description, pos=0.45]{\star}
\ar{ll}
\\
&
E_1
\ar{ur}
\ar[dotted]{rr}
& ~ &
E_2 
\ar[dotted]{rr}
\ar{ur}
& ~ & 
\dots
\ar[dotted]{rr} 
\ar{ur}
& ~ &
E_{n-1}
\ar[dotted]{rr}
\ar{ur}
&~& 
E_n
\ar{ur}
& 
\end{tikzcd}
\end{equation}
\end{footnotesize}
The dotted arrow morphisms define of a differential complex on the
objects $E_0, E_1[1], \dots, E_{n}[n]$ and the diagram
\eqref{eqn-repeated-extension-postnikov-tower} is an instance of a \em
Postnikov tower \rm associated to this complex. In particular, the
object $Q$ is its convolution. Thus repeated extensions of $E_0$ by
$E_1$, \dots, $E_n$ are convolutions of differential complexes with
objects $E_0, E_1[1], \dots, E_{n}[n]$. Similarly, repeated
co-extensions of $E_n$ by $E_{n-1}, \dots, E_0$ are also convolutions
of differential complexes with these objects. This can be made more
precise in the language of twisted complexes: 

\begin{prps}
\label{prps-extensions-coextensions-and-twisted-complexes}
Let $\C$ be an enhanced triangulated category and let $E_0, \dots E_n \in \C$.  For any objects $Q \in \C$ the following are equivalent:

\begin{enumerate}
\item $Q$ is a repeated extension of $E_0$, by $E_1, \dots, E_n$.
\item $Q$ is a repeated co-extension of $E_n$ by $E_{n-1}, \dots, E_0$. 
\item $Q$ is the convolution of a one-sided twisted complex whose objects are $E_i[i]$ in degrees $0 \leq i \leq n$ and $0$ in all other degrees.
\item $Q$ is the convolution of a one-sided twisted complex whose objects 
are $E_{n-i}[-i]$ in degrees $-n \leq i \leq 0$ and $0$ in all other degrees. 
\end{enumerate}
\end{prps}
\begin{proof}
Straightforward verification . 
\end{proof}

\subsection{Cyclic extensions}
\label{section-cyclic-extensions}

Throughout the paper we apply the constructions outlined 
in this section to an autoequivalence $h$ of $D(\A)$. This requires
us to fix a DG-enhancement of $h$. As per 
\S\ref{section-dg-enhancements-of-categories-and-functors}
this is equivalent to a choice of a bimodule $H \in D(\AbimA)$. 
However, the constructions below do not need the functor enhanced by $H$ 
to be an autoequivalence, so we present them in this greater
generality: 

\begin{defn}
Let $H \in D(\AbimA)$. A \em cyclic
extension of degree $n$ \rm (resp. \em co-extension\rm) of $\id_\A$ 
by $H$ is a repeated extension (resp.  co-extension) of 
$\id_\A$ by $H$, $H^2$, \dots, $H^n$. 
\end{defn}

It is more convenient for us in this paper to use the language of
cyclic co-extensions. To this end, we introduce some standard
notation. By the definition above, the structure of a degree 
$n$ cyclic co-extension of $\id$ by $H$ on an object 
$Q_n \in D(\AbimA)$ is the following diagram in $D(\AbimA)$:
\begin{footnotesize}
\begin{equation}
\label{eqn-cyclic-coextension-of-id-by-H-of-degree-n}
\begin{tikzcd}
\id 
\ar[phantom]{drr}[description, pos=0.45]{\star}
\ar{rr}{\iota_1}
& &
Q_1 
\ar{rr}{\iota_2}
\ar{ld}{\mu_1}
\ar[phantom]{drr}[description, pos=0.45]{\star}
\ar[phantom]{d}[description]{\circlearrowright}
& &
Q_2  
\ar{r}
\ar{ld}{\mu_2}
\ar[phantom]{d}[description]{\circlearrowright}
&
\dots
\ar{r}
&
Q_{n-2}
\ar[phantom]{drr}[description, pos=0.45]{\star}
\ar{rr}{\iota_{n-1}}
\ar{ld}
\ar[phantom]{d}[description]{\circlearrowright}
& 
& 
Q_{n-1}
\ar{ld}{\mu_{n-1}}
\ar{rr}{\iota_{n}}
\ar[phantom]{d}[description]{\circlearrowright}
& & 
Q_n.
\ar[phantom]{dll}[description, pos=0.45]{\star}
\ar{ld}{\mu_n}
\\
&
H
\ar[dashed]{lu}{\sigma_1}
& ~ &
H^2 
\ar[dashed]{lu}
\ar[dashed]{ll}{\sigma_2}
& ~ & 
\dots 
\ar[dashed]{ll}
\ar[dashed]{lu}
& ~ &
H^{n-1}
\ar[dashed]{ll}
\ar[dashed]{lu}
& ~ & 
H^n
\ar[dashed]{lu}
\ar[dashed]{ll}{\sigma_n}
& 
\end{tikzcd}
\end{equation}
\end{footnotesize}
Define $\iota: \id \rightarrow Q_n$ to be the composition $\iota_n
\circ \dots \circ \iota_1$. 

On DG level, the data above can be lifted to a twisted complex
$\bar{Q}_n$ over $\AmodbarA$ and an isomorphism of its convolution to $Q_n$. 
The twisted complex is of form
\begin{equation}
\label{eqn-cyclic-coextension-of-id-by-H-of-degree-n-dg}
\begin{tikzcd}
H^{\bartimes n}[-n]
\ar{r}{\sigma_{n}}
\ar[dashed, bend left=20]{rrrrr}
\ar[dashed, bend left=19]{rrrr}
\ar[dashed, bend left=18]{rrr}
& 
H^{\bartimes (n-1)}[-(n-1)]
\ar{r}
\ar[dashed, bend left=19]{rrrr}
\ar[dashed, bend left=18]{rrr}
\ar[dashed, bend left=17]{rr}
& 
\dots 
\ar{r}
&
H^{\bartimes 2}[-2]
\ar{r}{\sigma_2}
\ar[dashed, bend left=17]{rr}
&
H[-1] 
\ar{r}{\sigma_1}
&
\underset{\degzero}{\A}
\end{tikzcd}
\end{equation}
where, by abuse of notation, $\sigma_i$ denote 
\em arbitrary \rm lifts of the maps $\sigma_i$ in 
\eqref{eqn-cyclic-coextension-of-id-by-H-of-degree-n} to 
$\AmodbarA$.
The isomorphism identifies the objects $Q_i$ with the convolutions of
the subcomplexes
\begin{equation}
\label{eqn-object-Qi-dg}
\begin{tikzcd}
H^{\bartimes i}[-i]
\ar{r}
\ar[dashed, bend left=19]{rrrr}
\ar[dashed, bend left=18]{rrr}
\ar[dashed, bend left=17]{rr}
& 
\dots 
\ar{r}
&
H^{\bartimes 2}[-2]
\ar{r}{\sigma_2}
\ar[dashed, bend left=17]{rr}
&
H[-1] 
\ar{r}{\sigma_1}
&
\underset{\degzero}{\A}
\end{tikzcd}
\end{equation}
of \eqref{eqn-cyclic-coextension-of-id-by-H-of-degree-n-dg},
the maps $\iota_i: Q_{i-1} \rightarrow Q_i$ with 
the twisted complex inclusions 
\begin{equation}
\label{eqn-map-iota_i-dg}
\begin{tikzcd}
&
H^{\bartimes {i-1}}[-(i-1)]
\ar{r}
\ar[equals]{d}
& 
\dots 
\ar{r}
&
H^{\bartimes 2}[-2]
\ar{r}{\sigma_2}
\ar[equals]{d}
&
H[-1] 
\ar{r}{\sigma_1}
\ar[equals]{d}
&
\underset{\degzero}{\A} 
\ar[equals]{d}
\\
H^{\bartimes {i}}[-i]
\ar{r}{\sigma_i}
&
H^{\bartimes {i-1}}[-(i-1)]
\ar{r}
& 
\dots 
\ar{r}
&
H^{\bartimes 2}[-2]
\ar{r}{\sigma_2}
&
H[-1] 
\ar{r}{\sigma_1}
&
\underset{\degzero}{\A},
\end{tikzcd}
\end{equation}
and the maps $\mu_i: Q_i \rightarrow H^i$ with the twisted
complex projections
\begin{equation}
\label{eqn-map-mu_i-dg}
\begin{tikzcd}
H^{\bartimes {i}}[-i]
\ar{r}{\sigma_i}
\ar[equals]{d}
&
H^{\bartimes {i-1}}[-(i-1)]
\ar{r}
& 
\dots 
\ar{r}
&
H^{\bartimes 2}[-2]
\ar{r}{\sigma_2}
&
H[-1] 
\ar{r}{\sigma_1}
&
\underset{\degzero}{\A}. 
\\
\underset{\text{deg.}-i}{H^{\bartimes i}[-i]}
&
&
&
&
&
\end{tikzcd}
\end{equation}
For the computations on DG level, we fix a choice of any such DG 
lift $\bar{Q}_n$. We can then replace the objects $Q_i$ and the maps $\iota_i$
and $\mu_i$ with the convolutions of the twisted complexes 
\eqref{eqn-object-Qi-dg} and of the maps \eqref{eqn-map-iota_i-dg} and 
\eqref{eqn-map-mu_i-dg} and work with the latter instead.

\begin{defn}
Let $Q_n$ be a degree $n$ cyclic co-extension of $\id$ by $H$.
Define $J_n \in D(\AbimA)$ as 
\begin{align}
J_n := \cone \left( \id \xrightarrow{\iota} Q_n \right). 
\end{align}
\end{defn}

Let now $\bar{J}_n$ be the following subcomplex of $\bar{Q}_n$
\begin{equation}
\label{eqn-Jn-twisted-complex}
\begin{tikzcd}
H^{\bartimes n}[-n]
\ar{r}{\sigma_{n}}
\ar[dashed, bend left=19]{rrrr}
\ar[dashed, bend left=18]{rrr}
& 
H^{\bartimes (n-1)}[-(n-1)]
\ar{r}
\ar[dashed, bend left=18]{rrr}
\ar[dashed, bend left=17]{rr}
& 
\dots 
\ar{r}
&
H^{\bartimes 2}[-2]
\ar{r}{\sigma_2}
&
\underset{\degminusone}{H[-1]}. 
\end{tikzcd}
\end{equation}
Define $\lambda\colon \bar{J}_n[-1] \rightarrow \A$ to be 
the closed degree zero map
\begin{equation}
\label{eqn-J_n[-1]-to-A-map}
\begin{tikzcd}
H^{\bartimes n}[-n]
\ar{r}{\sigma_{n}}
\ar[dashed]{drrrr}
& 
H^{\bartimes (n-1)}[-(n-1)]
\ar{r}
\ar[dashed]{drrr}
& 
\dots 
\ar{r}
&
H^{\bartimes 2}[-2]
\ar{r}{\sigma_2}
\ar[dashed]{dr}
&
\underset{\degzero}{H[-1]}
\ar{d}{\sigma_1}
\\
&
&
&
&
\underset{\degzero}{\A}
\end{tikzcd}
\end{equation}
whose components are the differentials in 
\eqref{eqn-cyclic-coextension-of-id-by-H-of-degree-n-dg}
with the image $\A$. Define $\kappa\colon \bar{Q}_n \rightarrow \bar{J}_n$ 
to be the closed degree zero map 
\begin{equation}
\label{eqn-Q_n-to-J_n-map}
\begin{tikzcd}
H^{\bartimes n}[-n]
\ar{r}{\sigma_{n}}
\ar[equals]{d}
& 
H^{\bartimes (n-1)}[-(n-1)]
\ar{r}
\ar[equals]{d}
& 
\dots 
\ar{r}
&
H^{\bartimes 2}[-2]
\ar{r}{\sigma_2}
\ar[equals]{d}
&
H[-1]
\ar{r}{\sigma_1}
\ar[equals]{d}
&
\underset{\degzero}{\A}
\\
H^{\bartimes n}[-n]
\ar{r}{\sigma_{n}}
& 
H^{\bartimes (n-1)}[-(n-1)]
\ar{r}
& 
\dots 
\ar{r}
&
H^{\bartimes 2}[-2]
\ar{r}{\sigma_2}
&
\underset{\degminusone}{H[-1]}. 
&
\end{tikzcd}
\end{equation}

We then have:
\begin{lemma}
The following is an exact triangle in $H^0(\pretriag(\AbarA))$
\begin{align}
\label{eqn-J_n-A-Q_n-exact-triangle}
\bar{J}_n[-1]
\xrightarrow{\lambda}
\A
\xrightarrow{\iota} 
\bar{Q}_n
\xrightarrow{\kappa}
\bar{J}_n
\end{align}
\end{lemma}
\begin{proof}
This is an exercise in understanding the total complex functor
$\pretriag(\pretriag(\AbarA)) \xrightarrow{\tot} \pretriag(\AbarA)$ of 
\cite{BondalKapranov-EnhancedTriangulatedCategories}. It is an
equivalence and the Yoneda embedding $\pretriag(\AbarA) \hookrightarrow 
\modd(\pretriag(\AbarA))$ identifies it with the standard convolution
of twisted complexes. Thus, in particular, $\pretriag(\AbarA)$ is
strongly pretriangulated with $\tot$ as the convolution functior. 

Recall, quite generally, that for any strongly pre-triangulated category
$\C$ and any closed degree $0$ morphism $c_1 \xrightarrow{\gamma} c_2$ 
in $\C$, there is an exact triangle in $H^0(\C)$
\begin{align}
c_1 \xrightarrow{\gamma} c_2 \rightarrow \left\{ c_1
\xrightarrow{\gamma} c_2 \right\} \rightarrow c_1[1] 
\end{align}
whose other two maps are induced by the twisted complex maps
\begin{align}
\label{eqn-connecting-maps-of-the-cone-construction-for-twisted-complexes}
\vcenter{
\xymatrix{
& 
\underset{\degzero}{c_2}
\ar[d]^{\id}
\\
c_1 
\ar[r]^{\gamma}
&
\underset{\degzero}{c_2}
}
}
\quad\quad \text{ and } \quad\quad
\vcenter{
\xymatrix{
c_1
\ar[r]^{\gamma}
\ar[d]^{\id}
&
\underset{\degzero}{c_2}
\\
\underset{\degminusone}{c_1}
&
}
}. 
\end{align}
See \cite[\S3.3]{AnnoLogvinenko-SphericalDGFunctors} for the details
and a vastly more general statement. 

Now observe that the map 
$$ \bar{J}_n[-1] \xrightarrow{\lambda} \A$$
is closed of degree $0$ and its total complex is $\bar{Q}_n$. 
Thus, the general fact given above with $\C = \pretriag(\AbarA)$ and 
$\gamma = \lambda$ 
yields the desired exact triangle in $H^0(\pretriag{\AbarA})$. 
\end{proof}

\begin{cor}
For any DG-lift $\bar{Q}_n$ of $Q_n$, we have 
$$ J_n \simeq \left\{ \bar{J}_n \right\}. $$
\end{cor}
We can therefore simply replace $J_n$ by the convolution of
$\bar{J}_n$, and work with it and the exact triangle
\begin{equation}
\id_\A \xrightarrow{\iota} Q_n \xrightarrow{\kappa} J_n 
\xrightarrow{\lambda} \id_\A[1]
\end{equation}
induced by $\eqref{eqn-J_n-A-Q_n-exact-triangle}$, instead.

\subsection{Truncated twisted tensor algebras}
\label{section-truncated-twisted-tensor-algebras}

The most obvious example of a degree $n$ cyclic coextension 
of $\id$ by $H$ is the direct sum 
\begin{align}
H^{n} \oplus \dots \oplus H \oplus \id. 
\end{align}
It has a natural structure of an algebra in $D(\AbimA)$
defined by truncating the tensor algebra $\bigoplus_{i = 0}^{\infty} H^i$. 

An arbitrary cyclic coextension of $\id$ by $H$ does not carry a natural 
structure of an algebra in $D(\AbimA)$. However, 
the definition of a $\mathbb{P}^n$-functor requires the adjunction monad of
the functor to be isomorphic to a cyclic co-extension, thus equipping 
the latter \em a posteriori \rm with an algebra structure. 

We therefore consider the following family of cyclic coextensions with
a natural algebra structure. It is parametrised by the elements
$\sigma \in \ext^1_{D(\AbimA)}(H,\id)$.  The case $\sigma = 0$ is the
truncated tensor algebra example above. We construct these cyclic
co-extensions on DG level, however we then prove it to be independent
of the choices of DG lifts involved.  It is somewhat suprising that,
though the isomorphism class of the pair $(H,\sigma) \in D(\AbimA)$
uniquely determines the isomorphism class of the resulting cyclic
co-extension, the authors were unable to find a construction that
works purely on the level of the triangulated category $D(\AbimA)$.

\begin{defn}
Let $(H,\sigma)$ be a pair consisting of a bimodule $H \in \AmodbarA$
and a closed morphism of degree minus one $H \xrightarrow{\sigma}\A$.
Define  $\bartta_{H, \sigma, n}$ of $(H,\sigma)$
(denoted by $\bartta_n$ where no confusion is possible) 
to be the following twisted complex over $\AmodbarA$:
\begin{align}
\label{eqn-n-truncated-twisted-tensor-algebra-of-bimodule}
\xymatrix{ 
H^{\bartimes n} [-n]
\ar[r]^{\xi_{n}}
& 
\dots 
\ar[r]
&
H^{\bartimes 2}[-2] 
\ar[r]^{\xi_2}
&
H [-1]
\ar[r]^{\xi_1}
&
\underset{\degzero}{\A}
}
\end{align}
where 
$$\xi_{k+1}=\sum_{i=0}^{k} (-1)^i \id^{\bartimes(i)} \bartimes \sigma \bartimes
\id^{\bartimes(k-i)}$$ 
and all the higher differentials are zero.
\end{defn}

The convolution of $\bartta_n$ in $D(\AbimA)$ has a natural $\A$-algebra
structure defined as follows:

\begin{enumerate}
\item The unit map $\epsilon\colon A \rightarrow \left\{ \bartta_n \right\}$ is defined by  the twisted complex map:
\begin{align}
\label{eqn-A-unit-map-for-R_n}
\vcenter{
\xymatrix{ 
& & & & 
\underset{\degzero}{\A} 
\ar[d]^{\id}
\\
H^{\bartimes n}[-n] 
\ar[r]^{\xi_{n}}
& 
\dots 
\ar[r]
&
H^{\bartimes 2}[-2]
\ar[r]^{\xi_2}
&
H[-1]
\ar[r]^{\xi_1}
&
\underset{\degzero}{\A} 
}
}
\end{align}
\item As per \cite[Lemma 3.42(1)]{AnnoLogvinenko-BarCategoryOfModulesAndHomotopyAdjunctionForTensorFunctors}
the bimodule $\left\{ \bartta_n \right\} \bartimes_\A \left\{ \bartta_n \right\}$ is isomorphic in $D(\AbimA)$ to 
the convolution of a twisted complex whose degree $-i$ part is
$\oplus_{k + l = i} H^{\bartimes k} \bartimes_\A H^{\bartimes l}[-i]$ 
for $-2n \leq -i \leq 0$. 

The multiplication map $\mu \colon \left\{ \bartta_n \right\} \bartimes_\A \left\{ \bartta_n \right\} \rightarrow \left\{ \bartta_n \right\}$ 
is defined by the following map from that twisted 
complex to \eqref{eqn-n-truncated-twisted-tensor-algebra-of-bimodule}:
\begin{align}
\label{eqn-A-mult-map-for-R_n}
(-i,-i)\colon \bigoplus_{k+l = i} H^{\bartimes k} \bartimes_\A H^{\bartimes l} [-i]
\xrightarrow{\sum \id} H^{\bartimes i}[-i] \quad \quad \quad -n \leq -i \leq 0. 
\end{align}
\end{enumerate}
  
\begin{lemma}
\label{lemma-derived-category-class-of-R_n-depends-only-on-that-of-H-sigma}
The isomorphism class of the $\A$-algebra $\left\{ \bartta_n \right\}$ in $D(\AbimA)$
depends only on the isomorphism class of the pair $(H, \sigma)$ in $D(\AbimA)$. 
\end{lemma}

\begin{proof}
Let $(H, \sigma)$ and $(H', \sigma')$ be isomorphic in $D(\AbimA)$. 
Then there exists a homotopy equivalence $H \xrightarrow{f} H'$ in 
$\AmodbarA$ such that $\sigma = \sigma' \circ f + d \beta$ 
for some $\beta \in \barhom^{-1}_{\AbimA}(H,\A)$. 

Let $\iota$ be the twisted complex map $\bartta_{H,\sigma, n} \rightarrow \bartta_{H',\sigma',n}$ 
whose component $H^{\bartimes i+k}[-i-k] \rightarrow H'^{\bartimes i}[-i]$ is:
\begin{align}
\label{eqn-induced-map-iota-from-Rn-to-R'n}
\sum_{0 \leq i_1 < \dots < i_k \leq i + k - 1 } 
(-1)^{ik + \frac{1}{2}k(k-1)}
(-1)^{i_1 + \dots + i_k}\;
\phi_0 \bartimes \dots \bartimes \phi_{i+k-1}
\quad
\text{ where }
\phi_j = 
\begin{cases}
\beta  \text{ if } j \in \left\{ i_1, \dots, i_k \right\} \\
f \text{ otherwise. }
\end{cases}
\end{align}

For the readers convenience we've illustrated $\iota$ for $n = 3$:

\begin{tiny}
\begin{align*}
\xymatrix@C=0.75in@R=0.5in{
H\bartimes H\bartimes H[-3]
\ar[rrr]^-{
\sigma\bartimes\id\bartimes\id-\id\bartimes\sigma\bartimes\id+\id\bartimes\id\bartimes\sigma}
\ar[d]|{f\bartimes f \bartimes f}
\ar@{-->}[drrr]|>>>>>>>>>>>>>>>>>>>>{\beta \bartimes f \bartimes f - f
\bartimes \beta \bartimes f + f \bartimes f \bartimes \beta} 
\ar@{.>}[drrrrr]|<<<<<<<<<<<<<<<<<<<<<<<<<<<<<<<<<<<<<<<<<<<{\beta\bartimes\beta\bartimes f
- \beta \bartimes f \bartimes \beta + f \bartimes \beta \bartimes \beta} 
\ar@/^0.75pc/@{.>}[drrrrrr]|<<<<<<<<<<<{\beta\bartimes\beta\bartimes\beta}
&&& 
H\bartimes H[-2]
\ar[rr]^-{\sigma\bartimes\id - \id\bartimes\sigma} 
\ar[d]|{f \bartimes f}
\ar@{-->}[drr]|>>>>>>>>>>>>>>{-(\beta \bartimes f - f \bartimes\beta) \quad} 
\ar@{.>}[drrr]|<<<<<<<<<<<<<<<<<<<<<<<{\beta \bartimes \beta}
&&
H[-1]
\ar[d]|{f}
\ar[r]^{\sigma}
\ar@{-->}[dr]|{\beta}
&
\A
\ar[d]|{\id}
\\ 
H' \bartimes H' \bartimes H'[-3]
 \ar[rrr]_-{\sigma'\bartimes\id\bartimes\id-\id\bartimes\sigma'\bartimes\id+
\id\bartimes\id\bartimes\sigma'}
&&&
H' \bartimes H'[-2]
\ar[rr]_-{\sigma'\bartimes\id-\id\bartimes\sigma'}
&& 
H'[-1]
\ar[r]_{\sigma'}
&
\A
} 
\end{align*}
\end{tiny}

The map $\iota$ is clearly of degree $0$ and is readily seen to be closed. 
Since all the degree zero components of $\iota$ are
quasi-isomorphisms $f^{\bartimes i}$ we conclude by the Rectangle Lemma
\cite[Lemma 2.12]{AnnoLogvinenko-BarCategoryOfModulesAndHomotopyAdjunctionForTensorFunctors}
that $\iota$ itself is a quasi-isomorphism. It therefore defines an isomorphism 
$\left\{ \bartta_{H,\sigma,n} \right\} \xrightarrow{\iota} \left\{  \bartta_{H', \sigma',n} \right\}$ in 
$D(\AbimA)$. This isomorphism is readily seen to intertwine  
the convolutions of the twisted complex maps \eqref{eqn-A-unit-map-for-R_n} and
\eqref{eqn-A-mult-map-for-R_n}. Thus $\iota$ is an isomorphism of $\A$-algebras, as required. 
\end{proof}

The above allows us to make the following definition:

\begin{defn}
Let $H \in D(\AbimA)$ and let $\sigma \in \homm^1_{D(\AbimA)}(H,\id)$. Define the \em $n$-truncated 
twisted tensor algebra \rm $\tta_{H, \sigma, n} \in D(\AbimA)$
(denoted by  $\tta_n$ where no confusion is possible)
to be the convolution of the twisted complex $\bartta_{n}$ for any lift of  $(H, \sigma)$ to 
the DG-enhancement $\AmodbarA$. 
\end{defn}

\subsection{Algebraic geometry and Fourier-Mukai transforms}
\label{section-Fourier-Mukai}

In this section, we introduce the notation and state the generalities 
regarding the bicategory $\fmcatweak$ 
of Fourier-Mukai kernels on separated schemes
of finite type over a field. We then state and prove several key
results which are needed for the Fourier-Mukai computations in 
\S\ref{section-examples-of-non-split-Pn-functors} where we give geometric
examples of non-split $\mathbb{P}^n$-functors. These results  
are a part of a bigger framework which will be treated systematically
and in depth in further papers. 

\subsubsection{Generalities}

Let us fix some generalities. For any object 
$F \in D_{qc}(X \times Y)$ we denote by 
$$ \Phi_F\colon D_{qc}(X) \rightarrow D_{qc}(Y)$$
the corresponding Fourier-Mukai transform
$$ \pi_{Y *}\left(F \otimes \pi_X^* (-) \right). $$
Here $\pi_Y$ and $\pi_X$ denotes the projection from the fibre product 
$X \times Y$ to the corresponding components. Where it is more convenient 
we also use the notation where
$\pi_{i_1 \dots i_k}$ denotes the projection from a fibre product to the product of  its 
components numbered $i_1$, \dots, $i_k$ from left to right. 

For the details on the twisted inverse image functor $f^!$, the
projection formula, the base change morphisms, the K{\"u}nneth
morphism, etc., see an excellent in-depth exposition in
\cite{Lipman-NotesOnDerivedFunctorsAndGrothendieckDuality} or its 
summary in \cite[\S2]{AnnoLogvinenko-OnTakingTwistsOfFourierMukaiFunctors}. 
Let $f\colon X \rightarrow Y$ be a map of separated schemes of finite type 
over $k$. Let
\begin{equation}
\label{eqn-natural-map-f^*-to-Hom-f^!-f^!}
f^* \rightarrow \shhomm(f^!\mathcal{O}_Y, f^!)
\end{equation}
be the morphism right adjoint to the natural map $f^* \otimes
f^!\mathcal{O}_Y \rightarrow f^! $ via Tensor-Hom adjunction. We define
$$ f_! \overset{\text{def}}{=} 
f_* (f^! \mathcal{O}_Y \otimes (-))
\colon \quad
D_{qc}(X) \rightarrow D_{qc}(Y). $$
Thus $f_!$ is the left adjoint to 
the RHS of \eqref{eqn-natural-map-f^*-to-Hom-f^!-f^!}. 
When $f$ is perfect and proper, \eqref{eqn-natural-map-f^*-to-Hom-f^!-f^!} 
is an isomoprhism \cite[Lemma 2.1.10]{AvramovIyengarLipman-ReflexivityAndRigidityForComplexesIISchemes}. 
It follows that $f_!$ is then the left adjoint of $f^*$. The
adjunction unit is given by
$$ \id \rightarrow \shhomm\Bigl(f^! \mathcal{O}_Y, f^! \mathcal{O}_Y \otimes (-)\Bigr)
\rightarrow \shhomm\Bigl(f^! \mathcal{O}_Y, f^! f_*(f^! \mathcal{O}_Y \otimes -)\Bigr)
\xrightarrow{\sim} f^* f_!, $$
where the first two morphisms  are the units of Tensor-Hom adjunction and of $(f_*,f^!)$ adjunction, 
respectively. The adjunction counit is given by 
$$ f_! f^* = f_*(f^! \mathcal{O}_Y \otimes f^*) \xrightarrow{\sim} f_* f^* \rightarrow \id, $$
where the last morphism is the counit of $(f^*, f_*)$ adjunction. 

Given two objects $F \in D_{qc}(X \times Y)$ and $G \in D_{qc}(Y \times Z)$ 
we have the notion of their composition as Fourier-Mukai kernels:
\begin{equation}
\label{eqn-Fourier-Mukai-kernel-composition}
G \star F : =  \pi_{13 *} \left( \pi_{12}^* F \otimes \pi_{23}^* G \right)
\quad \in D_{qc}(X \times Z).   
\end{equation}
Where no confusion is possible, we suppress the $\star$ and write simply $GF$. 
This notion of composition gives the Fourier-Mukai kernels the
structure of a bicategory $\fmcatweak$. 
For the details on bicategories, which are a certain kind of weak 2-categories, 
see \cite{Benabou-IntroductionToBicategories}. The objects of $\fmcatweak$
are separated schemes of finite type over a field, its categories of
$1$-morphisms are the derived categories $D_{qc}(X \times Y)$ of their
products, the horizontal composition is given by
\eqref{eqn-Fourier-Mukai-kernel-composition}, and the identity
$1$-morphism is the structure sheaf of the diagonal 
$\Delta_* \mathcal{O}_X \in D_{qc}(X \times X)$. The associator and
the unitor natural transformations are defined via the usual
combinations of the projection formula and the base-change
isomorphisms. 
For smooth projective algebraic varieties this bicategory was studied
in \cite{CaldararuWillerton-TheMukaiPairingIACategoricalApproach}. 
 
\subsubsection{Standard kernels and the Key Lemma}
\label{section-standard-fourier-mukai-kernels-and-the-key-lemma}

The Fourier-Mukai kernels for the functors of direct and inverse image, 
tensor product, etc.~are well-known. For example, in the case of the direct 
image functor the (structure sheaf of the) graph of the corresponding map of
schemes is usually taken as its Fourier-Mukai kernel. In the lemma below
we define these standard kernels, but in what may seem at first to be 
a slightly unusual form.

\begin{lemma}[Standard kernels]
\label{lemma-standard-fm-kernels}
Let $X$ be a separated scheme of finite type over $k$ and let 
$Q \in D_{qc}(X)$. The object
$$ T_Q  = \pi_2^* Q \otimes \Delta_* \mathcal{O}_X \in D_{qc}(X \times X) $$
is the Fourier-Mukai kernel for $Q \otimes (-)$.  

Let  $f\colon X \rightarrow Y$  be a map of separated schemes of finite type 
over $k$. We have the maps 
$$(\id_X \times f)\colon X \times X \rightarrow X \times Y, $$
$$(\id_Y \times f)\colon Y \times X \rightarrow Y \times Y.$$ 
The objects
$$ F_* = (\id_X \times f)_* \Delta_* \mathcal{O}_X \in D_{qc}(X \times Y), $$
$$ F^* = (\id_Y \times f)^* \Delta_* \mathcal{O}_Y \in D_{qc}(Y \times X), $$
are the Fourier-Mukai kernels for $f_*$ and $f^*$, respectively. If $f$ is perfect, 
then also the object
$$ F^! = (\id_Y \times f)^! \Delta_* \mathcal{O}_Y \in D_{qc}(Y \times X), $$
$$ F_! = (\id_X \times f)_! \Delta_* \mathcal{O}_X \in D_{qc}(X \times Y), $$
are the Fourier-Mukai kernel for $f^!$ and $f_!$, respectively. 
\end{lemma}

We refer to these as the \em standard kernels \rm for these functors. Furthermore, for any 
composition of these functors its standard kernel is the composition of the standard kernels
of its composants. For example, given a functor $f_*(Q \otimes f^*(-))$, 
its standard kernel is $F_* \star T_Q \star F^*$. 

\begin{proof}
This is standard. We have 
$$ \Phi_{T_Q} = \pi_{2 *} \left(\pi_{2}^* Q \otimes \Delta_* \mathcal{O}_X \otimes \pi_1^* \right)
\simeq \pi_{2 *} \Delta_* \Delta^* (\pi_{2}^* Q \otimes \pi_1^*) \simeq Q \otimes \id, $$
by the projection formula and since $\pi_1 \circ \Delta =  \pi_2 \circ \Delta = \id$. 

Next, we have two commutative diagrams:
\begin{equation}
\label{eqn-fourier-mukai-double-product-commutative-squares}
\begin{tikzcd}
X
\ar{r}{\Delta}
\ar[']{dr}{(\id, f)}
&
X \times X
\ar{d}{\id_X \times f}
\ar{r}{\pi_X}
& 
X 
\ar{d}{f}
\\
& 
X \times Y 
\ar{r}{\pi_Y}
&
Y 
\end{tikzcd}
\quad \quad \quad 
\begin{tikzcd}
X
\ar{d}{f} 
\ar{r}{(f,\id)}
& 
Y \times X 
\ar{d}{\id_Y \times f}
\ar{r}{\pi_X}
& 
X
\ar{d}{f}
\\
Y 
\ar{r}{\Delta}
& 
Y \times Y	
\ar{r}{\pi_Y}
&
Y. 
\end{tikzcd}
\end{equation}
All three squares are $\tor$-independent and thus we have
the base change isomorphisms for them. Thus 
$$ \Phi_{F_*} = 
\pi_{Y *}\Bigl((\id_X \times f)_* \Delta_* \mathcal{O}_X \otimes \pi_X^* \Bigr) \simeq
\pi_{Y *}\Bigl((\id,f)_* \mathcal{O}_X \otimes \pi_X^* \Bigr)
\simeq \pi_{Y *} (\id,f)_* (\id,f)^* \pi_X^* \simeq f_* \id^* \simeq f_*, $$
by the projection formula and since $\Delta \circ (\id_X \times f) = (\id,f)$, $\pi_X \circ (\id,f) = \id$, and $\pi_Y \circ (\id,f) = f$. 
Similarly
$$\Phi_{F^*} = \pi_{X *}\Bigl((\id_Y \times f)^* \Delta_* \mathcal{O}_Y \otimes \pi_Y^*  \Bigr) \simeq
\pi_{X *}\Bigl((f,\id)_* \mathcal{O}_X\otimes \pi_Y^*  \Bigr) \simeq 
\pi_{X *} (f,\id)_* (f,\id)^* \pi_Y^*  \Bigr) \simeq f^*. $$

Now, assume $f$ to be perfect. The projections $\pi_X$ and $\pi_Y$ are flat, 
thus the two squares on 
\eqref{eqn-fourier-mukai-double-product-commutative-squares} 
involving them admit 
twisted base change isomorphisms. 
Were $f$ to be proper,
\cite[Theorem 4.7.4]{Lipman-NotesOnDerivedFunctorsAndGrothendieckDuality}
would have given us the twisted base change for $f^!$ and 
$(\id_Y \times f)^!$ in the middle square on 
\eqref{eqn-fourier-mukai-double-product-commutative-squares}. 
Turns out we still have it because the middle square splits 
the rightmost one:
\begin{align*}
(\id_Y \times f)^! \Delta_* \simeq \; 
& (\id_Y \times f)^* \Delta_* \otimes (\id_Y \times f)^! \mathcal{O}_{Y
\times Y} \simeq 
(f,\id)_* f^* \otimes \pi_X^* f^! \mathcal{O}_Y 
\simeq 
\\
\simeq\;
& (f,\id)_* \left(f^* \otimes (f,\id)^* \pi_X^* f^! \mathcal{O}_Y \right) 
\simeq 
(f,\id)_* \left(f^* \otimes f^! \mathcal{O}_Y \right)
\simeq (f,\id)_* f^!,
\end{align*}
Thus 
\begin{align*}
& \Phi_{F^!} \simeq
\pi_{X *}\Bigl((\id_Y \times f)^! \Delta_* \mathcal{O}_Y \otimes \pi_Y^*  \Bigr)
\simeq
\pi_{X *}\Bigl((f,\id)_* f^! \mathcal{O}_Y \otimes \pi_Y^*  \Bigr) \simeq 
\\
\simeq & \;
\pi_{X *} (f,\id)_* \Bigl(f^! \mathcal{O}_Y \otimes (f,\id)^* \pi_Y^*  \Bigr)
\simeq \id_* \Bigl(f^! \mathcal{O}_Y \otimes f^* \Bigr)
\simeq  f^! \mathcal{O}_Y \otimes f^* \simeq f^!
\end{align*}
with the last isomorphism using again the fact that $f$ is perfect. Finally, 
\begin{align*}
\Phi_{F_!} = \;
& \pi_{Y *}\Bigl( (\id_X \times f)_*\left( \Delta_* \mathcal{O}_X \otimes 
(\id_X \times f)^! \mathcal{O}_{X \times Y}\right)\otimes \pi_X^*  \Bigr)
\simeq
\pi_{Y *}\Bigl( (\id_X \times f)_*\left( \Delta_* \mathcal{O}_X \otimes 
\pi_X^* f^! \mathcal{O}_{Y}\right)\otimes \pi_X^*  \Bigr)
\simeq
\\
\simeq \;
& \pi_{Y *}(\id_X \times f)_* \Delta_* \Bigl(\Delta^* \pi_X^* f^!\mathcal{O}_{Y} 
\otimes \Delta^* (\id_X \times f)^* \pi_X^*  \Bigr) \simeq 
f_* \left(f^! \mathcal{O}_{Y}\otimes \id \right) = f_!. 
\end{align*}

\end{proof}

The following crucial result explains the reason for our choice of the
standard kernels above:

\begin{lemma}[The Key Lemma]
\label{lemma-pullbacks-and-pushforwards-of-Fourier-Mukai-kernels}
Let $f\colon X' \rightarrow X$ and $g\colon Y' \rightarrow Y$ be maps 
of separated schemes of finite type over a field. 
Let $M \in D_{qc}(X)$ and $N \in D_{qc}(Y)$. Let 
$Q \in D_{qc}(X \times Y)$ and $Q' \in D_{qc}(X' \times Y')$.   
\begin{enumerate}
	\item 
	\label{item-pullback-of-FMK-over-a-product-map}
	        There exists the following isomorphism which 
                is functorial in $Q$:
		$$ (f \times g)^* Q \simeq G^* \star Q \star F_*. $$
        \item If $f$ and $g$ are perfect, there exist 
                the following isomorphism which is functorial in $Q$:
		$$ (f \times g)^! Q \simeq G^! \star Q \star F_!. $$
	\item There exists the following isomorphism
              which is functorial in $Q'$:
		$$ (f \times g)_* Q' \simeq G_* \star Q' \star F^*. $$
        \item If $f$ and $g$ are perfect, there exists 
              the following isomorphism which is functorial in $Q$:
		$$ (f \times g)_! Q \simeq G^! \star Q \star F_!. $$
        \item There exists the following isomorphism which is functorial in $Q$:
   $$ \pi_X^* M \otimes \pi_Y^* N \otimes Q \simeq T_N \star Q \star T_M. $$
\end{enumerate} 
\end{lemma}
\begin{proof}
We only prove the assertion 
\eqref{item-pullback-of-FMK-over-a-product-map}, the other assertions
are proved similarly. We have
$$ G^* \star Q \star F_*  \simeq \pi_{X'Y' *} \bigl( \pi_{YY'}^* G^* \otimes \pi_{XY}^* Q \otimes \pi_{X'X}^* F_* \bigr)
\simeq 
\pi_{X'Y' *} \bigl( \pi_{YY'}^* (g,\id)_* \mathcal{O}_{Y'} \otimes \pi_{XY}^* Q \otimes \pi_{X'X}^*(\id,f)_* \mathcal{O}_{X'} \bigr). $$
By base change around the squares
\begin{equation}
\begin{tikzcd}[column sep={2.4cm}, row sep={1cm}]
X' \times Y \times Y'
\ar{r}{i := (\id,f) \times \id_{Y \times Y'}}
\ar{d}{\pi_{X'}}
&
X'  \times X \times Y \times Y'
\ar{d}{\pi_{X' X}} 
\\
X'
\ar{r}{(\id,f)}
&
X' \times X. 
\end{tikzcd}
\;
\begin{tikzcd}[column sep={2.4cm}, row sep={1cm}]
X' \times X \times Y'
\ar{r}{j := \id_{X' \times X} \times (g,\id)}
\ar{d}{\pi_{Y'}}
&
X'  \times X \times Y \times Y'
\ar{d}{\pi_{YY'}}
\\
Y'
\ar{r}{(g,\id)}
&
Y \times Y'. 
\end{tikzcd}	
\end{equation}
the above is isomorphic to 
\begin{equation}
\label{eqn-G^*QF_*-kernel-after-the-base-change}
\pi_{X'Y' *} \bigl( j_* \mathcal{O}_{X' \times X \times Y'} 
\otimes \pi_{XY}^* Q \otimes i_* \mathcal{O}_{X' \times Y \times Y'} \bigr).
\end{equation}
The fiber square 
\begin{equation}
\begin{tikzcd}[column sep={1.5cm}, row sep={1cm}]
X' \times Y'
\ar{r}{\id_{X'} \times (g,\id)}
\ar{d}{(\id,f) \times \id_{Y'}}
&
X'  \times  Y \times Y'
\ar{d}{i} 
\\
X'  \times  X \times Y'
\ar{r}{j}
&
X' \times X \times Y \times Y' 
\end{tikzcd}	
\end{equation}
is $\tor$-independent since any pullback from  $X' \times X$ to $X' \times X \times Y \times Y'$ is flat over $Y \times Y'$.
Hence we have the K{\"u}nneth isomorphism 
$$ j_* \mathcal{O}_{X' \times X \times Y'} \otimes  i_* \mathcal{O}_{X' \times Y \times Y'} \simeq h_* 
\mathcal{O}_{X' \times Y'} $$
where $h := (\id,f) \times (g,\id)$ is the composition of the maps along either contour of the  square. 
Thus \eqref{eqn-G^*QF_*-kernel-after-the-base-change} is further isomorphic to:
$$ \pi_{X'Y' *} \bigl( h_* \mathcal{O}_{X' \times Y'} 
\otimes \pi_{XY}^* Q \bigr) \simeq \pi_{X'Y' *} h_* h^* \pi_{XY}^* Q \simeq \id_* (f \times g)^* Q \simeq (f \times g)^* Q.$$ 
\end{proof}

\subsubsection{2-categorical adjunctions for standard kernels}

Lemma \ref{lemma-pullbacks-and-pushforwards-of-Fourier-Mukai-kernels} 
allows us to systematically obtain 2-categorical adjunctions in $\fmcatweak$ 
for the standard kernels of Lemma \ref{lemma-standard-fm-kernels} 
from the functorial adjunctions of derived functors:

\begin{prps}[2-categorical adjunctions for standard kernels] 
\label{prps-2-categorical-adjunctions-for-standard-kernels}
Let $X$ and $Y$ be separated schemes of finite type over
a field and let $f\colon X \rightarrow Y$ be a scheme map.  
\begin{enumerate}
\item  
Let 
\begin{equation}
\id_Y \xrightarrow{\epsilon} F_* \star F^* 
\end{equation}
\begin{equation}
F^* \star F_* \xrightarrow{\mu} \id_X  
\end{equation}
be the morphisms identified by the isomorphisms of the Key
Lemma (Lemma \ref{lemma-pullbacks-and-pushforwards-of-Fourier-Mukai-kernels})
with the morphisms 
\begin{equation}
\Delta_* \mathcal{O}_Y 
\rightarrow 
(\id_Y \times f)_* (\id_Y \times f)^* \Delta_* \mathcal{O}_Y 
\end{equation}
\begin{equation}
(\id_X \times f)^* (\id_X \times f)_* \Delta_* \mathcal{O}_X 
\rightarrow 
\Delta_* \mathcal{O}_X 
\end{equation}
which are the unit and counit of 
the corresponding functorial adjunctions. Then $(F^*, F_*, \epsilon,
\mu)$ is a 2-categorical adjunction in $\fmcatweak$. 

\item Assume further that $f$ is perfect and proper. 
Let 
\begin{equation}
\id_X \xrightarrow{\epsilon} F^! \star F_* 
\end{equation}
\begin{equation}
F_* \star F^* \xrightarrow{\mu} \id_Y  
\end{equation}
be the morphisms identified by the isomorphisms of 
the Key Lemma with the morphisms 
\begin{equation}
\Delta_* \mathcal{O}_X 
\rightarrow 
(\id_X \times f)^! (\id_X \times f)_* \Delta_* \mathcal{O}_X 
\end{equation}
\begin{equation}
(\id_Y \times f)_* (\id_Y \times f)^! \Delta_* \mathcal{O}_Y 
\rightarrow 
\Delta_* \mathcal{O}_Y 
\end{equation}
which are the unit and the counit of 
the corresponding functorial adjunctions. Then $(F_*, F^!, \epsilon,
\mu)$ is a 2-categorical adjunction in $\fmcatweak$. 

\item Assume further that $f$ is perfect and proper. 
Let 
\begin{equation}
\id_X \xrightarrow{\epsilon} F^* \star F_! 
\end{equation}
\begin{equation}
F_! \star F^! \xrightarrow{\mu} \id_Y  
\end{equation}
be the morphisms identified by the isomorphisms of
the Key Lemma 
with the morphisms 
\begin{equation}
\Delta_* \mathcal{O}_X 
\rightarrow 
(\id_X \times f)^* (\id_X \times f)_! \Delta_* \mathcal{O}_X 
\end{equation}
\begin{equation}
(\id_Y \times f)_! (\id_Y \times f)^* \Delta_* \mathcal{O}_Y 
\rightarrow 
\Delta_* \mathcal{O}_Y 
\end{equation}
which are the unit and the counit of 
the corresponding functorial adjunctions. 
Then $(F_!, F^*, \epsilon, \mu)$ is a 2-categorical adjunction 
in $\fmcatweak$. 

\item Let $M \in D(X)$ be a perfect object. Let 
\begin{equation}
\id_X \xrightarrow{\epsilon} T_{M^\vee} \star T_M 
\end{equation}
\begin{equation}
T_M \star T_{M^\vee} \xrightarrow{\mu} \id_X  
\end{equation}
be the morphisms identified by the isomorphisms of 
the Key Lemma with the morphisms 
\begin{equation}
\Delta_* \mathcal{O}_X 
\rightarrow 
\pi_2^* M^\vee \otimes \pi_2^* M \otimes \Delta_* \mathcal{O}_X
\end{equation}
\begin{equation}
\pi_2^* M \otimes \pi_2^* M^\vee \otimes \Delta_* \mathcal{O}_X
\rightarrow 
\Delta_* \mathcal{O}_X 
\end{equation}
which are the the unit and the counit of 
the functorial adjunction $\left((-) \otimes \pi_2^* M, (-) \otimes \pi_2^*
M^\vee)\right)$. Then $(T_M,T_{M^\vee}, \epsilon, \mu)$ 
is a 2-categorical adjunction in $\fmcatweak$. 
\end{enumerate}
\end{prps}
\begin{proof}
Direct verification.  
\end{proof}

The unit and the counit of a 2-categorical adjunction are unique up to
a unique isomorphism. Throughout the rest of the paper we shall refer 
to the morphisms defined in Proposition 
\ref{prps-2-categorical-adjunctions-for-standard-kernels}
as \em the \rm units and \em the \rm counits of the $2$-categorical 
adjunctions $(F^*, F_*)$, $(F_*, F^!)$, and $(T_M, T_{M^\vee})$.  
Up to isomorphism, they admit the following descriptions which are
frequently more convenient for computational purposes:

\begin{prps}
\label{prps-simplifying-units-and-counit-for-standard-kernels}
Let $f\colon X \rightarrow Y$ be a map of separated schemes 
of finite type over a field.
\begin{enumerate}
\item 
\label{item-fm-kernel-for-f^*f_*-alt1}
There is an isomorphism 
\begin{equation}
\label{eqn-fm-kernel-for-f^*f_*-alt1}
F^* \star F_* \simeq (\id_X \times f)^* (\id,f)_* \mathcal{O}_X
\end{equation}
which identifies the adjunction counit 
$F^* \star F_* \xrightarrow{\mu} \id_X$ with 
the morphism 
\begin{equation}
\label{eqn-fm-f^*f_*-counit-alt1}
(\id_X \times f)^* (\id,f)_* \mathcal{O}_X \rightarrow  \Delta_* \mathcal{O}_X
\end{equation}
which is the base change map for the commutative square:
	\begin{equation}
	\label{eqn-X-and-XxY-fiber-square}
		\begin{tikzcd}[column sep={1cm}, row sep={1cm}]
			X 
			\ar[equals]{r}
			\ar{d}{\Delta}
			&
			X
			\ar{d}{(\id,f)}
			\\
			X \times X
			\ar{r}{\id_X \times f}
			&
			X \times Y. 
		\end{tikzcd}
\end{equation}

\item 
\label{item-fm-kernel-for-f^*f_*-alt2}
There is an isomorphism 
\begin{equation}
\label{eqn-fm-kernel-for-f^*f_*-alt2}
F^* \star F_* \simeq (\id_X \times f)^* (f \times \id_Y)^* 
\Delta_* \mathcal{O}_Y
\end{equation}
which identifies the adjunction counit 
$F^* \star F_* \xrightarrow{\mu} \id_X$ with the composition 
\begin{equation}
\label{eqn-fm-f^*f_*-counit-alt2}
(\id_X \times f)^* (f \times \id_Y)^* \Delta_* \mathcal{O}_Y
\xrightarrow{\sim}
(\id_X \times f)^* (\id,f)_* \mathcal{O}_X
\rightarrow 
i_* \mathcal{O}_{X \times_Y X}
\rightarrow 
\Delta_* \mathcal{O}_X
\end{equation}
where the first map is the base change isomorphism for the
$\tor$-independent fiber square at the bottom of
the following commutative diagram:
\begin{equation}
\label{eqn-Y-and-XxX-fiber-square}
\begin{tikzcd}[column sep={1cm}, row sep={1cm}]
X 
\ar{r}{\Delta}
&
X \times_Y X
\ar{r}{i}
\ar{d}{\pi_1}
&
X \times X
\ar{d}{\id_X \times f}
\\
&
X 
\ar{r}{(\id,f)}
\ar{d}{f}
&
X \times Y 
\ar{d}{f \times \id_Y}
\\
&
Y
\ar{r}{\Delta}
&
Y \times Y,
\end{tikzcd}
\end{equation}
the second map is the base change map for its top fiber square, 
and the third map is the natural restriction of sheaves which 
is the image under $i_*$ of the 
adjunction unit 
$\mathcal{O}_{X \times_Y X} \rightarrow \Delta_* \Delta^* 
\mathcal{O}_{X \times_Y X}$. 
 
\item 
\label{item-fm-kernel-for-f_*f^*-alt1}
We have an isomorphism 
\begin{equation}
\label{eqn-fm-kernel-for-f_*f^*-alt1}
F_* \star F^* \simeq \Delta_* f_* \mathcal{O}_X
\end{equation}
which identifies the adjunction unit $\id_Y \xrightarrow{\epsilon} 
F_* \star F^*$ with the morphism 
\begin{equation}
\label{eqn-fm-f^*f_*-unit-alt1}
\Delta_* \mathcal{O}_Y \rightarrow \Delta_* f_* \mathcal{O}_X
		\end{equation}
which is the image under $\Delta_*$ of the adjunction unit for $(f^*, f_*)$. 

\item 
\label{item-fm-kernel-for-f_*f^!-alt1}
If $f$ is perfect and proper, we have an isomorphism
\begin{equation}
\label{eqn-fm-kernel-for-f_*f^!-alt1}
    F_* \star F^! \simeq \Delta_* f_* f^! \mathcal{O}_Y
\end{equation}
which identifies the adjunction counit $F_* \star F^! \xrightarrow{\mu} \id_Y$ 
with the morphism 
\begin{equation}
\label{eqn-fm-f_*f^!-counit-alt1}
\Delta_* f_* f^! \mathcal{O}_Y \rightarrow \Delta_* \mathcal{O}_Y
\end{equation}
which is the image under $\Delta_*$ of the adjunction counit for $(f_*, f^!)$. 
		
\item 
\label{item-fm-kernel-for-f^!f_*-alt1}
If $f$ is perfect and proper, we have an isomorphism
\begin{equation}
\label{eqn-fm-kernel-for-f^!f_*-alt1}
F^! \star F_* \simeq (\id_X \times f)^!  (\id,f)_* \mathcal{O}_X 
\end{equation}
which identifies the adjunction unit $\id_X \xrightarrow{\epsilon} 
F^! \star F_*$ 
with the morphism
\begin{equation}
\label{eqn-fm-f_*f^!-unit-alt1}
  \Delta_* \mathcal{O}_X \rightarrow (\id_X \times f)^! (\id,f)_* \mathcal{O}_X
\end{equation}
which is the twisted base change map 
$\Delta_* \id^! \rightarrow (\id_X \times f)^! (\id, f)_*$
for the commutative square \eqref{eqn-X-and-XxY-fiber-square}.

\item 
\label{item-fm-kernel-for-f^!f_*-alt2}
If $f$ is perfect and proper, we have an isomorphism
\begin{equation}
\label{eqn-fm-kernel-for-f^!f_*-alt2}
F^! \star F_* \simeq (\id_X \times f)^! (f \times \id_Y)^* \Delta_* 
\mathcal{O}_Y
\end{equation}
which identifies the adjunction unit $\id_X \xrightarrow{\epsilon} 
F^! \star F_*$ with the composition
\begin{equation}
\label{eqn-fm-f^!f_*-unit-alt2}
\Delta_* \mathcal{O}_X 
\rightarrow 
i_* \pi_1^! \mathcal{O}_X
\rightarrow 
(\id_X \times f)^! (\id,f)_* \mathcal{O}_X
\xrightarrow{\sim}
(\id_X \times f)^! (f \times \id_Y)^* \Delta_* \mathcal{O}_Y
\end{equation}
where the first map is the image under $i_*$ of the adjunction counit 
$ \Delta_* \Delta^! \pi_1^! \mathcal{O}_X 
\rightarrow \pi_1^!  \mathcal{O}_X, $
the second map is the twisted base change map for the top fiber square
in \eqref{eqn-Y-and-XxX-fiber-square}, and the third map is the base
change isomorphism for the $\tor$-independent bottom fiber square
in \eqref{eqn-Y-and-XxX-fiber-square}. 
\end{enumerate}
\end{prps}

\begin{proof}
We only prove the assertion \eqref{item-fm-kernel-for-f^*f_*-alt1}, 
as the proofs of the remaining assertions are similar. 

Define \eqref{eqn-fm-kernel-for-f^*f_*-alt1} be the composition of 
the isomorphism of the Key Lemma
(Lemma \ref{lemma-pullbacks-and-pushforwards-of-Fourier-Mukai-kernels})
$$ 
F^* \star F_* 
\simeq  
(\id_X \times f)^* (\id_X \times f)_* \Delta_* \mathcal{O}_X
$$
with the image under $(\id_X \times f)^*$ of the 
pseudofunctoriality isomorphism 
\begin{equation}
\label{eqn-pseudofunctoriality-isomoprhism-id,f-to-id-times-f-circ-Delta}
(\id,f)_* \mathcal{O}_X 
\xrightarrow{\sim}
(\id_X \times f)_* \Delta_* \mathcal{O}_X. 
\end{equation}
due to $(\id,f) = (\id_X \times f)_* \Delta_*$. 
Then, by its definition, the adjunction counit 
$F^* \star F_* \xrightarrow{\mu} \id_X$
gets identified by \eqref{eqn-fm-kernel-for-f^*f_*-alt1}  
with the composition 
$$ (\id_X \times f)^* (\id,f)_* \mathcal{O}_X
\xrightarrow{
\eqref{eqn-pseudofunctoriality-isomoprhism-id,f-to-id-times-f-circ-Delta}
}
(\id_X \times f)^* (\id_X \times f)_* \Delta_* \mathcal{O}_X
\rightarrow 
\Delta_* \mathcal{O}_X, $$
where the second morphism is the adjunction counit for 
$ \left( (\id_X \times f)^*, (\id_X \times f)_* \right)$. Thus, 
under this adjunction, the composition above is the morphism adjoint to the
pseudofunctoriality isomorphism 
\eqref{eqn-pseudofunctoriality-isomoprhism-id,f-to-id-times-f-circ-Delta}. 

On the other hand, by its definition the base change map for 
\eqref{eqn-X-and-XxY-fiber-square} is the morphism adjoint to 
the composition 
$$ (\id,f)_* \mathcal{O}_X  
\xrightarrow{\sim} (\id,f)_* \id_{X *} \id_X^* \mathcal{O}_X
\rightarrow 
(\id_X \times f)_* \Delta_* \id_X^* \mathcal{O}_X. 
$$
Given that $\id_{X *} = \id_X^* = \id_{D(X)}$, this is tautologically
equal to the pseudofunctoriality isomorphism 
\eqref{eqn-pseudofunctoriality-isomoprhism-id,f-to-id-times-f-circ-Delta}
above. We conclude that the morphism identified by 
\eqref{eqn-fm-kernel-for-f^*f_*-alt1} with the  
 adjunction counit $F^* \star F_* \xrightarrow{\mu} \id_X$
equals the base change isomorphism for 
\eqref{eqn-X-and-XxY-fiber-square}, as desired. 
\end{proof}

\subsubsection{Exterior algebra, short exact sequences, and dualities}

Let $X$ be a scheme and $\mathcal{V}$ a locally free sheaf 
of rank $n$ on $X$. Denote  
the product in the exterior algebra $\wedge^* \mathcal{V}$ of
$\mathcal{V}$ by
\begin{equation}
w\colon \wedge^* \mathcal{V} \otimes \wedge^* \mathcal{V}
\rightarrow \wedge^* \mathcal{V}. 
\end{equation}
By abuse of notation, we also denote by $w$ 
each of its surjective components
\begin{align*}
w\colon \wedge^i \mathcal{V} \otimes \wedge^j \mathcal{V} 
& \twoheadrightarrow 
\wedge^{i + j} \mathcal{V} \quad \quad \quad i,j \in \mathbb{Z}, 
\\
\left( v_1 \wedge \dots \wedge v_i \right) \otimes 
\left( w_1 \wedge \dots \wedge w_j \right) 
& \rightarrow 
v_1 \wedge \dots \wedge v_i \wedge w_1 \wedge \dots \wedge w_j. 
\end{align*}
These are sometimes referred to as the ``wedging maps''. 

We have the canonical isomorphism 
$$ (\wedge^i \mathcal{V})^\vee \simeq \wedge^i \mathcal{V}^\vee 
\quad\quad\quad i \in \mathbb{Z}
$$
and hence we have two canonical evaluation maps 
$$ \ev\colon \wedge^i \mathcal{V} \otimes \wedge^j \mathcal{V}^\vee \rightarrow \wedge^{i-j} \mathcal{V}, $$
$$ \ev\colon \wedge^i \mathcal{V} \otimes \wedge^j \mathcal{V}^\vee
\rightarrow \wedge^{j-i} \mathcal{V}^\vee. $$
The former map is an isomorphism for $i = n$ and the latter for $j =
n$, giving rise to dualities.

Suppose now that we have a short exact sequence 
\begin{equation}
0 \rightarrow \mathcal{P} 
\overset{\iota}{\hookrightarrow} \mathcal{V}
\overset{\pi}{\twoheadrightarrow} \mathcal{Q} 
\rightarrow 0
\end{equation}
of locally free sheaves on $X$. Let $p$ and $q$ be the ranks of 
$\mathcal{P}$ and $\mathcal{Q}$. 

\begin{prps}
\label{prps-useful-maps-for-exterior-algebra}
\begin{enumerate}
\item For any $i \in \mathbb{Z}$ the composition 
$$ \wedge^p \mathcal{P} \otimes \wedge^i \mathcal{V} 
\xrightarrow{\wedge^p{\iota} \otimes \id} 
\wedge^p \mathcal{V} \otimes \wedge^i \mathcal{V}
\xrightarrow{w} \wedge^{i+p} \mathcal{V} $$
filters through the natural projection 
$$ \wedge^p \mathcal{P} \otimes \wedge^i \mathcal{V}
\overset{\id \otimes \wedge^i \pi}{\twoheadrightarrow} 
\wedge^p \mathcal{P} \otimes \wedge^i \mathcal{Q} $$
and thus defines a canonical map 
\begin{equation}
\wedge^p \mathcal{P} \otimes \wedge^i \mathcal{Q} 
\rightarrow \wedge^{i+p} \mathcal{V}. 
\end{equation}
For $i = q$ this gives the canonical determinantal isomorphism 
\begin{equation}
\label{eqn-canonical-determinantal-isomorphism}
\wedge^p \mathcal{P} \otimes \wedge^q \mathcal{Q}
\xrightarrow{\sim} \wedge^{n} \mathcal{V}. 
\end{equation}

\item For any $i \in \mathbb{Z}$ the composition 
$$ \wedge^i \mathcal{V} \otimes \wedge^q \mathcal{Q}^\vee 
\xrightarrow{\id \otimes \wedge^q \pi^\vee}
\wedge^i \mathcal{V} \otimes \wedge^q \mathcal{V}^\vee 
\xrightarrow{\ev} \wedge^{i-q} \mathcal{V} $$
filters through the inclusion 
$ \wedge^{i-q} \mathcal{P} \hookrightarrow \wedge^{i-q} \mathcal{V} $
giving rise to the canonical projection 
\begin{equation}
\label{eqn-wedging-with-the-top-power-of-the-quotient}
\wedge^i \mathcal{V} \otimes \wedge^q \mathcal{Q}^\vee 
\twoheadrightarrow 
\wedge^{i-q} \mathcal{P}. 
\end{equation}
For $i = p$ this gives the line bundle isomorphism 
\begin{align*}
\wedge^n \mathcal{V} \otimes \wedge^q \mathcal{Q}^\vee
\xrightarrow{\sim} \wedge^p \mathcal{P}
\end{align*}
which is adjoint to the inverse 
\begin{equation}
\label{eqn-the-inverse-of-canonical-determinantal-isomorphism}
\wedge^n \mathcal{V} \xrightarrow{\sim} 
\wedge^q \mathcal{Q} \otimes \wedge^p \mathcal{P}
\end{equation}
of the canonical isomorphism 
\eqref{eqn-canonical-determinantal-isomorphism}. 

\item 
\label{item-wedging-with-the-top-power-of-the-quotion-dual-to-the-inclusion}
For any $i \in \mathbb{Z}$ the following diagram commutes 
\begin{equation}
\label{eqn-duality-for-wedging-with-the-top-power-of-the-quotient}
\begin{tikzcd}[column sep={1cm}, row sep={1cm}]
    \wedge^i \mathcal{V} \otimes \wedge^n \mathcal{V}^\vee  
    \ar{r}{\ev}[']{\sim}
    \ar{d}{\id \otimes \eqref{eqn-the-inverse-of-canonical-determinantal-isomorphism}^\vee}[']{\sim}
    &
    \wedge^{n-i} \mathcal{V}^\vee 
    \ar[two heads]{dd}{\wedge^{n-i} \iota^\vee}
    \\  
    \wedge^i \mathcal{V} \otimes \wedge^q \mathcal{Q}^\vee \otimes \wedge^p \mathcal{P}^\vee 
    \ar{d}[two heads]{\eqref{eqn-wedging-with-the-top-power-of-the-quotient} \otimes \id }
    & 
    \\
    \wedge^{i-q} \mathcal{P} \otimes \wedge^p \mathcal{P}^\vee 
    \ar{r}{\ev}[']{\sim}
    &
    \wedge^{n-i} \mathcal{P}^{\vee}.
\end{tikzcd}
\end{equation}
\end{enumerate}
\end{prps}
\begin{proof}
All the assertions can be checked locally where they follow
by choosing appropriate bases.  
\end{proof}

\subsubsection{The excess bundle formula} 

Let $Z$ be a smooth algebraic variety and let $X$ and $Y$ be locally complete
intersection subvarieties of $Z$ whose intersection $W = X \cap Y$
is a locally complete subvariety of $X$ and $Y$. Consider the fiber square
\begin{equation}
\label{eqn-excess-bundle-formula-fiber-square}
\begin{tikzcd}[column sep={1cm}, row sep={1cm}]
    W 
    \ar{r}{q}
    \ar{d}{p}
    \ar{dr}{h}
    &
    Y
    \ar{d}{j}
    \\
    X
    \ar{r}{i}
    &
    Z,
\end{tikzcd}
\end{equation}
where $i,j,p,q$, and $h$ are the corresponding regular immersions. 

\begin{defn}
The \em excess bundle $\mathcal{E}_W$ \rm of the intersection of $X$
and $Y$ in $Z$ is the locally free sheaf on $W$ which is the cokernel 
of the following natural injection:
\begin{equation}
\label{eqn-excess-bundle} 
\mathcal{N}_{W/Z} \rightarrow 
p^* \mathcal{N}_{X/Z} \oplus q^* \mathcal{N}_{Y/Z}. 
\end{equation}
\end{defn}

The following is a well-known result which appears e.g. in 
\cite{Scala-NotesOnDiagonalsOfTheProductAndSymmetricVarietyOfASurface}
or
\cite{ArinkinCaldararuHablicsek-FormalityOfDerivedIntersectionsAndTheOrbifoldHKRIsomorphism}:
\begin{theorem}[The Excess Bundle Formula]
The cohomology sheaves of the object $i^* j_* \mathcal{O}_Y \in D(X)$ 
are
\begin{equation}
\label{eqn-excess-bundle-formula}
H^{-m}(i^* j_* \mathcal{O}_Y) \simeq p_* (\wedge^m
\mathcal{E}_W^\vee). 
\end{equation}
\end{theorem}

We need some generalisations of it:

\begin{prps}
\label{prps-excess-bundle-formula-generalised}
Let $\mathcal{F} \in \cohcat(Y)$ be such that $q^* \mathcal{F} \in D(W)$ 
is concentrated in the degree zero. If $q^* \mathcal{F} \simeq 0$, 
then $i^* j_* \mathcal{F} \simeq i^! j_* \mathcal{F} \simeq 0$. 
Otherwise:
\begin{enumerate}
\item  
\label{item-excess-bundle-cohomology-sheaves-of-i^*j_*-F}
We have an isomorphism functorial in $\mathcal{F}$:
\begin{equation}
\label{eqn-excess-bundle-formula-generalised}
H^{-m}(i^* j_* \mathcal{F}) \simeq 
p_* (q^* \mathcal{F} \otimes \wedge^m \mathcal{E}_W^\vee). 
\end{equation}
Thus the non-zero cohomologies 
are in 
the degrees from $\codim_Z(W) - \codim_Z(X) - \codim_Z(Y)$ to $0$. 
\item 
\label{item-excess-bundle-base-change-for-i^*j_*-F}
The base change map 
\begin{equation}
\label{eqn-excess-bundle-base-change-for-i^*j_*-F}
i^* j_* \mathcal{F} \rightarrow p_* q^* \mathcal{F} 
\end{equation}
is the projection onto the rightmost (degree $0$) cohomology. 
\item  
\label{item-excess-bundle-cohomology-sheaves-of-i^!j_*-F}
We have an isomorphism functorial in $\mathcal{F}$:
\begin{equation}
\label{eqn-twisted-excess-bundle-formula-generalised}
H^{-m}(i^! j_* \mathcal{F}) \simeq 
p_* (q^* \mathcal{F} \otimes \wedge^{\codim_Z(X)+m} \mathcal{E}_W^\vee) \otimes
\omega_{X/Z}. 
\end{equation}
Thus the non-zero cohomologies 
are in the degrees from $\codim_W(Y)$ to $\codim_Z(X)$. 
\item 
\label{item-excess-bundle-twisted-base-change-for-i^!j_*-F}
The twisted base change map 
$$ p_* q^! \mathcal{F} \rightarrow i^! j_* \mathcal{F} $$
is the inclusion of the leftmost (degree $\codim_Y W$) cohomology. 
\end{enumerate}
\end{prps}
\begin{proof}
\eqref{item-excess-bundle-cohomology-sheaves-of-i^*j_*-F}:

By Lemma \ref{lemma-pullbacks-and-pushforwards-of-Fourier-Mukai-kernels} 
the standard Fourier-Mukai kernel $I^* J_*$ of $i^* j_*$
is isomorphic to the object
$$ (j \times i)^* \Delta_* \mathcal{O}_Z \in D(Y \times X). $$ 
Consider the fiber square 
\begin{equation}
\label{eqn-excess-bundle-formula-FM-kernel-square}
\begin{tikzcd}[column sep={1cm}, row sep={1cm}]
    W 
    \ar{r}{h}
    \ar{d}{(q,p)}
    &
    Z
    \ar{d}{\Delta}
    \\
    Y \times X 
    \ar{r}{j \times i}
    &
    Z \times Z.
\end{tikzcd}
\end{equation}
The excess bundle $\mathcal{E}'_{W}$ 
of the intersection of $Y \times X$ and $Z$
inside $Z \times Z$ is the cokernel of the natural map 
\begin{equation} 
\label{eqn-excess-bundle-formula-FM-kernel-computation}
\mathcal{N}_{W/Z \times Z} \xrightarrow{q_1 \oplus q_2} 
h^* \mathcal{N}_{Z/Z \times Z}
\oplus (j \times i)^*  \mathcal{N}_{Y \times X / Z \times Z}.
\end{equation}
The map $q_1$ is a surjection whose kernel is $\mathcal{N}_{W/Z}
\subset \mathcal{N}_{W/Z \times Z}$. On the other hand, we can identify  
$\mathcal{N}_{Y \times X/Z \times Z}$ with $\pi_Y^* \mathcal{N}_{Y/Z} \oplus
\pi_X^* \mathcal{N}_{X/Z}$. Hence the cokernel of
\eqref{eqn-excess-bundle-formula-FM-kernel-computation} is isomorphic
to the cokernel of 
\eqref{eqn-excess-bundle}, i.e. $\mathcal{E}'_W \simeq \mathcal{E}_W$. 

Now, we have 
$$ i^* j_* F \simeq \pi_{X *} \left( I^*J_* \otimes \pi_{Y}^*
\mathcal{F} \right), $$
so by the excess bundle formula 
$$ H^{-m}(I^* J_*) \simeq (q,p)_* \wedge^m
\mathcal{E}^\vee_W.$$
By the projection formula we have 
\begin{equation}
\label{eqn-cohomology-sheaves-of-I^*J_*-times-F}
H^{-m}(I^* J_*) \otimes \pi_{Y}^* \mathcal{F} \simeq 
\mathcal{F} \simeq (q,p)_* \left( \wedge^m \mathcal{E}^\vee_W
\otimes (q,p)^* \pi_Y^* \mathcal{F}\right)
\simeq (q,p)_* \left(\wedge^m \mathcal{E}^\vee_W \otimes q^*
\mathcal{F}\right). 
\end{equation}
By assumption $q^* \mathcal{F}$ is concentrated in degree zero, and
$\wedge^m \mathcal{E}^\vee_W \otimes (-)$ and $(q,p)_*$ are 
exact. Thus \eqref{eqn-cohomology-sheaves-of-I^*J_*-times-F}
is also concentrated in degree $0$. It follows by a spectral sequence
argument that 
$$ H^{-m} \left( I^* J_* \otimes \pi_{Y}^* \mathcal{F} \right) 
\simeq (q,p)_* \left(\wedge^m \mathcal{E}^\vee_W \otimes q^*
\mathcal{F}\right).$$ 
Finally, since $p_*$ is exact the object 
$$ \pi_{X *} H^{-m} \left( I^* J_* \otimes \pi_{Y}^* \mathcal{F} \right)
\simeq \pi_{X *} (q,p)_* \left(\wedge^m \mathcal{E}^\vee_W \otimes q^*
\mathcal{F}\right) \simeq p_* \left(\wedge^m \mathcal{E}^\vee_W \otimes q^*
\mathcal{F}\right) $$
is again concentrated in degree zero. The isomorphism
\eqref{eqn-excess-bundle-formula-generalised} now 
follows by a similar spectral sequence argument. Its functoriality
is due to the functoriality of the spectral sequences involved. 

\eqref{item-excess-bundle-base-change-for-i^*j_*-F}:

It suffices to prove that the base change map 
\eqref{eqn-excess-bundle-base-change-for-i^*j_*-F}
induces an isomorphism on degree $0$ cohomology sheaves. 
A quick way to establish this is to observe that  
the induced map is just the base change map in $\cohcat(X)$ 
for the ordinary, non-derived functors. That is readily
seen to be an isomorphism.  

There is a more conceptual way which also works for 
\eqref{item-excess-bundle-twisted-base-change-for-i^!j_*-F}. 
Arguing via Fourier-Mukai kernels as in 
\eqref{item-excess-bundle-cohomology-sheaves-of-i^*j_*-F}
we reduce to the case $\mathcal{F} = \mathcal{O}_Y$. Then 
we observe that the question is local on $Z$. Firstly, being an
isomorphism on degree $0$ cohomologies is a local property of 
a morphism in $D(X)$. Secondly, base change maps commute with 
localisation to open subsets or local rings. Indeed, given an open subset 
$U \subset Z$ we can localise the whole fiber square 
\eqref{eqn-excess-bundle-formula-fiber-square} to it. 
The pullback to $X \cap U$ of the base change map 
\eqref{eqn-excess-bundle-base-change-for-i^*j_*-F}
for the original fiber square is then naturally isomorphic to 
the base change map of the localised square applied to $\mathcal{F}_U$. 

Thus it suffices to assume that $X$ and $Y$ are global 
complete intersections in $Z$ and $\mathcal{F} = \mathcal{O}_Y$. 
It is then possible to compute explicitly using Koszul complexes. 
Computing $i^* j_* \mathcal{O}_Y$ via the Koszul complex of 
$j_* \mathcal{O}_Y$ one readily sees that it splits up as 
the direct sum of its cohomologies and the base change map 
is the projection onto the degree zero summand. 

\eqref{item-excess-bundle-cohomology-sheaves-of-i^!j_*-F}:

This follows immediately from the isomorphism
$$ i^! \simeq i^* \otimes i^! \mathcal{O}_Z \simeq i^* \otimes
\omega_{X/Z}[-\codim_Z X]. $$

\eqref{item-excess-bundle-twisted-base-change-for-i^!j_*-F}:

The proof is analogous to that 
of \eqref{item-excess-bundle-base-change-for-i^*j_*-F}: 
reduction to the global complete intersection case, and then 
a Koszul complex computation. Alternatively, the desired 
statement can be obtained 
from \eqref{item-excess-bundle-base-change-for-i^*j_*-F}
via the relative Verdier duality over $Z$ since we have 
$$ i^! j_* \simeq \mathcal{D}_{X/Z} i^* j_* \mathcal{D}_{Y/Z} $$
and the dualisation exchanges adjunction units and adjunction 
counits. See
\cite[\S2.1]{AnnoLogvinenko-OrthogonallySphericalObjectsAndSphericalFibrations}
for more details.  
\end{proof}

Suppose now we further have a locally complete intersection subvariety
$Y' \subset Y$ which carves out a complete intersection 
subvariety $W' \subset W$. We then have the following pair of fiber
squares, all of whose maps are regular immersions:

\begin{equation}
\label{eqn-excess-bundle-formula-FM-kernel-square-pair}
\begin{tikzcd}[column sep={1cm}, row sep={1cm}]
    W' 
    \ar{dr}{g}
    \ar{r}{r}
    \ar{d}{s}
    &
    Y' 
    \ar{d}{k}
    \\
    W 
    \ar{r}{q}
    \ar{d}{p}
    \ar{dr}{h}
    &
    Y
    \ar{d}{j}
    \\
    X
    \ar{r}{i}
    &
    Z.
\end{tikzcd}
\end{equation}

Let $\mathcal{E}_W$ be as before and let $\mathcal{E}_W'$ be the
excess bundle of the intersection of $Y'$ and $W$ at $W'$. We 
then have the following $3 \times 3$ commutative diagram in
$\cohcat(W')$ whose rows and columns are short exact sequences:

\begin{equation}
\label{eqn-two-excess-bundles-3x3-diagram}
\begin{tikzcd}[column sep={1cm}, row sep={1cm}]
    \mathcal{N}_{W'/W} 
    \ar[hook]{d}
    \ar[hook]{r}
    &
    0 \oplus r^* \mathcal{N}_{Y'/Y} 
    \ar[two heads]{r}
    \ar[hook]{d}
    &
    \mathcal{K}
    \ar[hook]{d}
    \\
    \mathcal{N}_{W'/Z}
    \ar[hook]{r}
    \ar[two heads]{d}
    &
    s^* p^* \mathcal{N}_{X/Z} \oplus r^* \mathcal{N}_{Y'/Z} 
    \ar[two heads]{r}
    \ar[two heads]{d}
    &
    \mathcal{E}_{W'}
    \ar[two heads]{d}
    \\
    s^* \mathcal{N}_{W/Z}
    \ar[hook]{r}
    &
    s^* p^* \mathcal{N}_{X/Z} \oplus g^* \mathcal{N}_{Y/Z} 
    \ar[two heads]{r}
    &
    s^* \mathcal{E}_{W},
\end{tikzcd}
\end{equation}
where the sheaf $\mathcal{K}$ is the cokernel of the natural 
inclusion $\mathcal{N}_{W'/W} \hookrightarrow r^* \mathcal{N}_{Y'/Y}$.

\begin{prps}
\label{prps-excess-bundle-formula-double-intersection}
Let $\mathcal{F} \in \cohcat(Y)$ be such that $q^* \mathcal{F} \in
D(W)$, $k^* \mathcal{F} \in D(Y')$ and $g^* \mathcal{F} \in D(W')$
are all concentrated in the degree zero.
\begin{enumerate}
\item 
\label{item-excess-bundle-double-intersection-straight}
Let 
\begin{equation}
\label{eqn-excess-bundle-double-intersection-straight}
i^* j_* \mathcal{F} \rightarrow i^* j_* k_* k^* \mathcal{F} 
\end{equation} 
be the map given by the adjunction unit $\id_Y \rightarrow k_* k^*$.
In other words, by the sheaf restriction 
$\mathcal{F} \twoheadrightarrow \mathcal{F}_{Y'}$. The induced
map on degree $-m$ cohomology sheaves is the map
\begin{equation}
\label{eqn-excess-bundle-double-intersection-twisted}
p_* \left(\wedge^{m} \mathcal{E}_W^\vee \otimes q^* \mathcal{F} \right)
\rightarrow  
p_* \left(s_* \wedge^{m} \mathcal{E}_{W'}^\vee \otimes q^* \mathcal{F} \right) 
\end{equation}
induced by the natural projection $\mathcal{E}_{W'} \rightarrow s^*
\mathcal{E}_W$ of 
\eqref{eqn-two-excess-bundles-3x3-diagram}. 
\item
\label{item-excess-bundle-double-intersection-twisted}
 Let $d = \codim_Y Y'$ and let 
\begin{equation}
i^* j_* k_* k^! \mathcal{F} \rightarrow i^* j_* \mathcal{F}   
\end{equation} 
be the map given by the adjunction counit $k_* k^! \rightarrow \id_Y$.
The induced map on degree $-m$ cohomology sheaves is the map
\begin{equation}
p_* \left(s_*\left( \wedge^{m+d} \mathcal{E}_{W'}^\vee \otimes 
r^* \wedge^d \mathcal{N}_{Y'/Y}\right) \otimes q^* \mathcal{F} \right)
\rightarrow  
p_* \left(s_* \wedge^{m} \mathcal{E}_{W}^\vee \otimes q^* \mathcal{F} \right) 
\end{equation}
induced by the composition 
$$ \wedge^{m+d} \mathcal{E}_{W'}^\vee \otimes 
r^* \wedge^d \mathcal{N}_{Y'/Y} \rightarrow 
\wedge^{m+d} \mathcal{E}_{W'}^\vee \otimes \wedge^d \mathcal{K}
\rightarrow s^* \wedge^m \mathcal{E}_{W}^\vee $$
where the first map is induced by the natural projection 
$r^* \mathcal{N}_{Y'/Y} \twoheadrightarrow \mathcal{K}$
and the second map is induced by the  
rightmost column of \eqref{eqn-two-excess-bundles-3x3-diagram}. 
\end{enumerate}
\end{prps}
\begin{proof}
Both assertions are proved analogously to the proof of Proposition 
\ref{prps-excess-bundle-formula-generalised}
\eqref{item-excess-bundle-base-change-for-i^*j_*-F}: we use
a Fourier-Mukai argument to reduce to the case $\mathcal{F} = \mathcal{O}_X$
and a locality argument to reduce to the case of global complete
intersections. Then it is a straightforward Koszul complex
computation with both sides of 
\eqref{eqn-excess-bundle-double-intersection-straight}
and 
\eqref{eqn-excess-bundle-double-intersection-twisted}
splitting up as sums of their cohomologies. 
\end{proof}

\subsubsection{Splitting trick} 
\label{section-splitting-trick}

Let $X$ and $Y$ be schemes over a base scheme $S$. 
The following result allows us, in certain circumstances, to deduce
that the cohomology sheaves of an external tensor product 
$M \boxtimes_S N$ split up as the direct sums of the external 
tensor products of the cohomology sheaves of $M$ and $N$: 

\begin{lemma}
\label{lemma-splitting-trick}
Let $X$ and $Y$ be schemes over a base scheme $S$. 
Let $X \times_S Y$ be the fiber product and denote by 
$(-) \boxtimes (-)$ the external tensor product 
$\pi_1^*(-) \otimes \pi_2^*(-)$ on $X \times_S Y$. 

Let $M \in D_{qc}(X)$ be such that all its cohomology sheaves $H^i(M)$
are flat over $S$. Then for all $N \in D_{qc}(Y)$ and all 
$r \in \mathbb{Z}$ we have a natural isomorphism in $\qcohcat(X \times Y)$:
\begin{equation}
H^r(M \boxtimes N) \simeq \bigoplus_{p+q = r} H^p(M) \boxtimes H^q(N).
\end{equation}
\end{lemma}
\begin{proof}
This is a typical instance of the K{\"u}nneth map
\begin{equation}
\label{eqn-kunneth-map-for-Tor-of-sheaves} 
\bigoplus_{p+q = r} H^p(M) \boxtimes H^q(N)
\rightarrow H^r(M \boxtimes N) 
\end{equation}
being an isomorphism. For a very general exposition, 
covering this particular case, 
see \cite{Grothendieck-EGA-III-2}, $\S6$ and specifically \em 
Proposition \rm 6.7.7. 

For the benefit of the reader, we give a brief sketch of the argument. 
Let $M_\bullet$ and $N_\bullet$ be resolutions of $M$ and $N$ by
bounded above complexes of flat sheaves. We then have a bicomplex 
$$ C_{p,q} = M_p \boxtimes N_q $$ 
and there is a standard spectral sequence, associated 
to the filtration of its totalisation $\tot C_{\bullet, \bullet}$ 
by columns, see e.g.  
\cite[\S{III.7.5}]{GelfandManin-MethodsOfHomologicalAlgebra}. 
It has 
$$ E^{p,q}_0 = C_{p,q} = M_p \boxtimes N_q, $$
$$ E^{p,q}_1 = H^p (C_{\bullet,q}) \simeq H^p(M) \boxtimes N_q, $$ 
$$ E^{p,q}_2 = H_*^p (H_\bullet^q (C_{*,\bullet})) \simeq 
H^q(H^p(M) \boxtimes N_\bullet) \simeq H^p(M) \boxtimes H^q(N), $$ 
with the last isomorphism due to $H^p(M)$ being flat over $S$. 
Now since $E^{p,q}_2 \simeq H^p(M) \boxtimes H^q(N)$
any class in it can be represented by a genuine cycle, and not just 
$2$-almost cycle. That is, by an element of $M_p \boxtimes N_q$ which 
is killed by the differential of the bicomplex. Therefore 
all the differentials on the page $E_2$ die and the spectral sequence 
degenerates. Since the bicomplex is bounded above in both directions, 
the spectral sequence converges to its total cohomology. 
We thus have a filtration on $H^{r}(M \boxtimes N)$ 
whose assocated graded object is $\bigoplus_{p+q=r}
H^p(M) \boxtimes H^q(N)$. The K{\"u}nneth maps splits this filtration, 
and is therefore an isomorphism. 
\end{proof}

\section{$\mathbb{P}$-twists}
\label{section-ptwists}

In this section, given an enhanced functor between enhanced
triangulated categories equipped with \em $\mathbb{P}$-twist \rm 
data, we construct an endomorphism of the target category 
called the \em $\mathbb{P}$-twist\rm.
 
\subsection{Generalities}
\label{section-ptwists-generalities}

Below we fix some notation to be used throughout 
the section. Let $\A$ and $\B$ be two DG-categories. 
Let $D(\A) \xrightarrow{f} D(\B)$ be a continuous functor, that is ---
it commutes with infinite direct sums. Assume that $f$ has 
continuous left and right adjoints $D(\B) \xrightarrow{l,r} D(\A)$. 
Fix an enhancement $M \in D(\AbimB)$ of $f$. To be more precise, 
for the computations on DG level we fix a choice of a specific 
bimodule $M \in \AmodbarB$. All our constructions, however, 
depend only on the isomorphism class of $M$ in $D(\AbimB)$. 

As $l$ and $r$ exist and are continuous, by 
\cite[Prop. 5.2]{AnnoLogvinenko-BarCategoryOfModulesAndHomotopyAdjunctionForTensorFunctors} 
$M$ is $\A$- and $\B$-perfect and $M^\barA$ and $M^\barB$ are 
DG-enhancements of $l$ and $r$. We therefore denote $M$, $M^\barA$, 
and $M^\barB$ by $F$, $L$, and $R$, respectively. 

Since the tensor product of bimodules enhances the composition of 
the corresponding functors we adopt from now on 
the ``functorial'' notation detailed in
\cite[\S5]{AnnoLogvinenko-BarCategoryOfModulesAndHomotopyAdjunctionForTensorFunctors}.
To sum it up briefly, we suppress the tensor signs and write 
bimodules in the order of the composed functors. For example, 
we write
\begin{align*}
FR \quad\quad\text{ for }\quad\quad  M^{\barB} \bartimes_\A M \in \BmodbarB,  
\\
LF \quad\quad\text{ for }\quad\quad  M \bartimes_\B M^{\barA} \in \AmodbarA.
\end{align*}
Moreover, we denote
the diagonal bimodules $\A$ and $\B$ by $\id_\A$ and $\id_\B$, 
or by $\id$ where no confusion is possible. For further subtleties
concerning diagonal bimodules in this notation see 
\cite[note after Defn.~4.6]{AnnoLogvinenko-BarCategoryOfModulesAndHomotopyAdjunctionForTensorFunctors}. 

\begin{defn}
\label{defn-spherical-twists-and-spherical-cotwists}
\begin{itemize}
\item The \em (spherical) twist \rm $T$ and the \em (spherical) dual
twist \rm $T'$ of $F$ are the convolutions in $\BmodbarB$
of the following twisted complexes in $\pretriag(\BbarB)$:
\begin{align*}
FR \xrightarrow{\trace} \underset{\degzero}{\id_\B}, \\
\underset{\degzero}{\id_\B} \xrightarrow{\action} FL.
\end{align*}
\item the \em (spherical) cotwist \rm $C$ and
the \em dual cotwist \rm of $F$ are the convolutions in $\AmodbarA$
of the following twisted complexes in $\pretriag(\AbarA)$:
\begin{align*}
\underset{\degzero}{\id_\A} \xrightarrow{\action} RF, \\
LF \xrightarrow{\trace} \underset{\degzero}{\id_\A}. \\
\end{align*}
\end{itemize}
\end{defn}

As explained in \cite[\S3.3]{AnnoLogvinenko-SphericalDGFunctors} we have 
the following natural exact triangles in 
$D(\BbimB)$ and $D(\AbimA)$ 
\begin{equation}
\label{eqn-twist-exact-triangle-DG}
\begin{tikzcd}
FR  
\ar{r}{\trace} 
& \id_\B
\ar{r} 
&
T,
\ar[dotted,bend left=20]{ll}
\end{tikzcd}
\end{equation}
\begin{equation}
\label{eqn-dual-twist-exact-triangle-DG}
\begin{tikzcd}
T'
\ar{r}
& \id_\B
\ar{r}{\action}
&
FL,
\ar[dotted,bend left=20]{ll}
\end{tikzcd}
\end{equation}
\begin{equation}
\label{eqn-cotwist-exact-triangle-DG}
\begin{tikzcd}
C
\ar{r}
& \id_\A
\ar{r}{\action}
&
RF,
\ar[dotted,bend left=20]{ll}
\end{tikzcd}
\end{equation}
\begin{equation}
\label{eqn-dual-cotwist-exact-triangle-DG}
\begin{tikzcd}
LF
\ar{r}{\trace} 
& \id_\A
\ar{r} 
&
C'.
\ar[dotted,bend left=20]{ll}
\end{tikzcd}
\end{equation}

Finally, recall that by 
\cite[Theorem 4.1]{AnnoLogvinenko-BarCategoryOfModulesAndHomotopyAdjunctionForTensorFunctors}
there is a twisted complex in $\pretriag(\BbarB)$:

\begin{equation}
\label{eqn-FR-FRFR-FR-Id-twisted-complex}
\begin{tikzcd}[column sep={3cm},row sep={1.5cm}] 
FR
\ar{r}{F{\action}R}
\ar[bend left=20, dashed]{rr}{\xi'_\B}
& 
FRFR
\ar{r}{FR\trace - \trace FR}
&
FR
\ar{r}{\trace}
&
\underset{\degzero}{\B}.
\end{tikzcd}
\end{equation}

\subsection{$\mathbb{P}$-twist data}

Conceptually, the $\mathbb{P}$-twist data for 
$D(\A) \xrightarrow{f} D(\B)$ is an $f$-split coextension $q_1$ of 
$\id_\A$ by an autoequivalence $h \in \autm D(\A)$ together 
with a natural transformation $q_1 \xrightarrow{\gamma} rf$
which intertwines the natural map $\id_\A \rightarrow q_1$ 
with the adjunction unit $\id_\A \rightarrow rf$. Out of this data
we extract a morphism $fhr \xrightarrow{\psi} fr$ and then 
define the $\mathbb{P}$-twist to be the unique convolution 
\cite{AnnoLogvinenko-OnUniquenessOfPTwists}
of the three step complex 
\begin{align*}
fhr \xrightarrow{\psi} fr \xrightarrow{\trace} \id. 
\end{align*}
Aposteriori, the morphism $\psi$ is the only data needed to 
construct the $\mathbb{P}$-twist of $F$. However, defined 
as above the $\mathbb{P}$-twist data is
a natural subset of the data of a $\mathbb{P}^n$-functor, cf. 
\ref{section-pn-functor-data}. 
Conversely, we conjecture in \S\ref{section-introduction}
that up to an isomorphism each 
$\mathbb{P}^n$-functor data is, in fact, uniquely 
determined by its $\mathbb{P}$-twist data.  

As usual, in practice we need to work with DG-enhancements of 
these functors in $D(\AbimA)$:

\begin{defn}
A \em $\mathbb{P}$-twist data $(H, Q_1, \gamma)$ for $F \in D(\AbimB)$ is:
\begin{enumerate}
\item $H \in D(\AbimA)$ such that $h = (-) \ldertimes_\A H$ is 
an autoequivalence of $D(\A)$. 
\item An $F$-split coextension $Q_1 \in D(\AbimA)$ of $\id$ by $H$. 
\item A morphism $Q_1 \xrightarrow{\gamma} RF$ in $D(\AbimA)$
which intertwines 
$\id_\A \xrightarrow{\action} RF$ and 
$\id_\A \xrightarrow{\iota} Q_1$. 
\end{enumerate}
\end{defn}
Here by $F$-split we mean that $FQ_1$ is a split coextension of 
$F$ by $FH$.

To construct the $\mathbb{P}$-twist of $F$ from the $\mathbb{P}$-twist data, 
we need to fix a choice of splitting of $FQ_1R$. We first show that
the resulting morphism $\psi$ doesn't depend on the choice of the
splitting:

\begin{lemma}
\label{lemma-psi-is-independent-of-the-choice-of-phi}
Let $\phi\colon FHR \rightarrow FQ_1R$ be any splitting of the
exact triangle 
\begin{align}
\label{eqn-FQ1R-exact-triangle}
FR \xrightarrow{F \iota R} FQ_1R \xrightarrow{F \mu_1 R} FHR
\end{align}
and define $\psi\colon FHR \rightarrow FR$ to be the composition
\begin{align}
\label{eqn-defn-of-psi}
FHR \xrightarrow{\phi} FQ_1R \xrightarrow{F \gamma R} FRFR
\xrightarrow{FR{\trace} - {\trace}FR} FR.
\end{align}
The map $\psi$ is independent of the choice of $\phi$. 
\end{lemma}

\begin{proof}
Let $\delta \colon FHR \rightarrow F Q_1 R$ be the difference of 
any two splittings of \eqref{eqn-FQ1R-exact-triangle}. Then 
$\delta$ composes to $0$ with 
$F \mu_1 R \colon FQ_1R \rightarrow FHR, $ 
and therefore filters through 
$ F \iota R\colon FR \rightarrow FQ_1R$.

It now suffices to show that the composition 
$$
FR \xrightarrow{F \iota R} FQ_1R \xrightarrow{F \gamma R} FRFR
\xrightarrow{FR{\trace} - {\trace}FR} FR
$$
is zero. But by the definition of $\mathbb{P}$-twist data
we have $F \gamma R \circ F \iota R = F \action R$, and
$$ FR{\trace} \circ F{\action} R = {\trace}FR \circ F{\action} = \id_\B $$
since $R$ is the right adjoint of $F$.  
\end{proof}

It was shown in \cite[Theorem 3.2]{AnnoLogvinenko-OnUniquenessOfPTwists}
that for any $G \in D(\BbimA)$ and any natural
transformation $FG \rightarrow FR$ all convolutions of the three term 
complex $FG \rightarrow FR \xrightarrow{\action} \id_B$ are
isomorphic. We therefore define:

\begin{defn}
Let $F \in D(\AbimB)$ and let $(H,Q_1,\gamma)$ be a $\mathbb{P}$-twist
data for $F$. Then the \em $\mathbb{P}$-twist \rm $P_F$ of $F$ is 
the unique convolution of the three-term complex
\begin{align}
\label{eqn-defn-P-twist}
FHR \xrightarrow{\psi} FR \xrightarrow{\action} \underset{\degzero}{\B}.
\end{align}
\end{defn}

\subsection{DG construction of $\mathbb{P}$-twists}
\label{section-P-twists-via-DG}

Any coextension $Q_1$ of $\id_\A$ by $H$ is realised on the DG-level as the
convolution of a twisted complex 
\begin{equation}
\label{eqn-DG-lift-of-Q1}
H[-1] \xrightarrow{\sigma_1} \underset{\degzero}{\A}
\end{equation}
over $\AmodbarA$ for some closed morphism 
$\sigma_1\colon H \rightarrow \A$ of degree $1$, 
cf. \S\ref{section-repeated-extensions}. 
The class of $\sigma_1$ in $\ext^1_{D(\AbimA)}(H,\id_\A)$ determines 
$Q_1$ up to isomorphism. The condition that $Q_1$ is $F$-split
is equivalent to $F \sigma_1 R$ being zero in 
$\ext^1_{D(\BbimB)}(FHR,FR)$, i.e. 
to $F \sigma_1 R$ being a boundary in $\BmodbarB$. 

Since the convolution functor is a quasi-equivalence, any morphism 
$\gamma\colon Q_1 \rightarrow RF$ in $D(\AbimA)$ lifts to a closed
degree zero map of twisted complexes 
\begin{equation}
\begin{tikzcd}[column sep={2cm},row sep={1cm}] 
\label{eqn-DG-lift-of-gamma-in-Q_1-case}
H[-1]
\ar{r}{\sigma_1}
\ar{dr}{\gamma_1}
&
\underset{\degzero}{\A}
\ar{d}{\gamma_0}
\\
&
\underset{\degzero}{RF}.
\end{tikzcd}
\end{equation}
On the other hand, 
as per \S\ref{section-repeated-extensions}, 
the natural map $\id_\A \xrightarrow{\iota} Q_1$ is 
the convolution of the twisted complex map 
\begin{equation}
\begin{tikzcd}[column sep={2cm},row sep={1cm}] 
\label{eqn-DG-lift-of-iota1}
&
\underset{\degzero}{\A}
\ar[equals]{d}
\\
H[-1]
\ar{r}{\sigma_1}
&
\underset{\degzero}{\A}. 
\end{tikzcd}
\end{equation}
The condition of $\gamma$ intertwining $\iota$
and the adjunction unit 
$\id_\A \xrightarrow{\action} RF$ is therefore equivalent to there
existing a lift $(\gamma_1, \gamma_0)$ of $\gamma$ with $\gamma_0 =
\action$. The if implication is clear, and for the converse observe 
that the composition of an arbitrary lift of $\gamma$ with
\eqref{eqn-DG-lift-of-iota1} is the map $\A \xrightarrow{\gamma_0} FR$. 
If $\gamma \circ \iota = \action$ in $D(\AbimA)$, then in $\AmodbarA$
we have $\gamma_0 = \action + d (\delta)$ for some degree $-1$ map 
$\delta\colon\A \rightarrow RF$. Subtracting from \eqref{eqn-DG-lift-of-gamma-in-Q_1-case} 
the boundary of the map 
\begin{equation}
\begin{tikzcd}[column sep={2cm},row sep={1cm}] 
H[-1]
\ar{r}{\sigma_1}
&
\underset{\degzero}{\A}
\ar{d}{\delta}
\\
&
\underset{\degzero}{RF}
\end{tikzcd}
\end{equation}
we obtain the desired lift of $\gamma$:
\begin{equation}
\begin{tikzcd}[column sep={2cm},row sep={1cm}] 
\label{eqn-DG-lift-of-gamma-in-Q_1-case-with-gamma0-equals-action}
H[-1]
\ar{r}{\sigma_1}
\ar{dr}{\gamma_1 - \delta \circ \sigma_1}
&
\underset{\degzero}{\A}
\ar{d}{\action}
\\
&
\underset{\degzero}{RF}.
\end{tikzcd}
\end{equation}

We therefore define:
\begin{defn}
A \em DG $\mathbb{P}$-twist data $(H, \sigma_1, \gamma_1)$ for $F$ is:
\begin{enumerate}
\item $H \in \AmodbarA$ such that $h = (-) \ldertimes_\A H$ is 
an autoequivalence of $D(\A)$. 
\item A closed degree $1$ map $\sigma_1\colon H \rightarrow \A$
in $\AmodbarA$ with $F \sigma_1 R$ a boundary. 
\item A degree $1$ map $\gamma_1\colon H \rightarrow RF$ in 
$\AmodbarA$ with $d \gamma_1 = \action \circ \sigma_1$. 
\end{enumerate}

The induced $\mathbb{P}$-twist data $(H,Q_1, \gamma)$ in $D(\AbimA)$
is then given by $H$, the convolution of the twisted complex
\begin{equation}
\begin{tikzcd}[column sep={2cm},row sep={1cm}] 
H[-1]
\ar{r}{\sigma_1}
&
\underset{\degzero}{\A}
\end{tikzcd}
\end{equation}
and the convolution of the twisted complex map 
\begin{equation}
\begin{tikzcd}[column sep={2cm},row sep={1cm}] 
H[-1]
\ar{r}{\sigma_1}
\ar{dr}{\gamma_1}
&
\underset{\rm \degzero}{\A}
\ar{d}{\action}
\\
&
\underset{\rm \degzero}{RF}.
\end{tikzcd}
\end{equation}
\end{defn}

We then have on the DG level the following direct formula for 
the $\mathbb{P}$-twist of $F$:
\begin{prps}
\label{prps-dg-construction-of-p-twists}
Let $(H,\sigma_1, \gamma_1)$ be a DG $\mathbb{P}$-twist data for $F$
in $\AmodbarA$. The $\mathbb{P}$-twist $P_F \in D(\BbimB)$ defined by 
the induced $D(\AbimA)$ data is isomorphic to the convolution 
of 
\begin{equation}
\label{eqn-canonical-P-twist-twisted-complex}
\begin{tikzcd}[column sep={2cm},row sep={1cm}] 
FHR
\ar{rr}{(FR\trace - \trace FR)\circ F\gamma_1R - \xi'_B \circ
F{\sigma_1}R}
& &
FR
\ar{r}{\trace}
&
\underset{\rm \degzero}{\id_\B}.
\end{tikzcd}
\end{equation}
\end{prps}
\begin{proof}
It suffices to show that the map 
\begin{equation}
\label{eqn-canonical-form-of-map-psi}
(FR\trace - \trace FR)\circ F\gamma_1R - \xi'_B \circ F{\sigma_1}R 
\end{equation}
equals in $D(\BbimB)$ the map $\psi: FHR \rightarrow FR$ of Lemma 
\ref{lemma-psi-is-independent-of-the-choice-of-phi}.

By assumption, $F \sigma_1 R$ is a boundary. Let therefore
$\beta\colon FR \rightarrow FRFR$ be such that $d\beta = F \sigma_1 R$. 
The composition \eqref{eqn-defn-of-psi} defining $\psi$ then lifts
to the following composition of twisted complex maps:
\begin{equation}
\begin{tikzcd}[column sep={2cm},row sep={1cm}] 
\underset{\degminusone}{FHR[-1]}
\ar[equals]{d}
\ar{dr}{- \beta}
&
\\
FHR[-1]
\ar{r}{F{\sigma_1}R}
\ar{dr}{F{\gamma_1}R}
&
\underset{\degzero}{FR}
\ar{d}{F{\action}R}
\\
&
\underset{\degzero}{FRFR}
\ar{d}{FR\trace - \trace FR}
\\
&
\underset{\degzero}{FR}.
\end{tikzcd}
\end{equation}
Thus $\psi$ is the image in $D(\BbimB)$ of the map 
\begin{equation}
\label{eqn-noncanonical-DG-lift-of-psi}
(FR\trace - \trace FR)\circ F\gamma_1R - 
(FR\trace - \trace FR) \circ F \action R \circ \beta. 
\end{equation}
It remains to show that the difference 
$$ (FR\trace - \trace FR) \circ F \action R \circ \beta
- \xi'_B \circ F{\sigma_1}R $$
between \eqref{eqn-canonical-form-of-map-psi}
and \eqref{eqn-noncanonical-DG-lift-of-psi} is a boundary, 
and thus vanishes in $D(\BbimB)$. 

Recall now the twisted complex
\eqref{eqn-FR-FRFR-FR-Id-twisted-complex}. 
The fact that it is a twisted complex 
implies, in particular, that 
$d \xi'_\B = (FR\trace - \trace FR) \circ F \action R$. 
We thus have
\begin{align}
d(\xi'_\B \circ \beta) = (d \xi'_\B) \circ \beta 
- \xi'_\B \circ (d \beta) = 
(FR\trace - \trace FR) \circ F \action R \circ \beta 
- \xi'_\B \circ F \sigma_1 R 
\end{align}
as desired. 
\end{proof}

\section{$\mathbb{P}^n$-functors}
\label{section-Pn-functors}

Conceptually, a $\mathbb{P}^n$-functor is a functor $f: D(\A)
\rightarrow D(\B)$ whose adjunction monad $rf$ is isomorphic
to a cyclic coextension $q_n$ of $\id_\A$ by an autoequvalence $h
\in \autm D(\A)$. The isomorphism $\gamma: q_n \xrightarrow{\sim} rf$ 
has to intertwine the adjunction unit 
$\action\colon \id_\A \rightarrow RF$ 
and the natural map $\iota\colon \id_\A \rightarrow q_n$, 
and to satisfy three more conditions detailed below: the \em monad
condition\rm, the \em adjoints condition\rm, and the \em 
highest degree term \rm condition.

The data of $(h,q_n, \gamma)$ defines the structure of a $\mathbb{P}^n$-functor
on $f: D(\A) \rightarrow D(\B)$. The restriction of this data to 
the first coextension $q_1$ via the map 
$\eta: q_1 \xrightarrow{\iota_n \circ ... \circ \iota_2} q_n$ is then
a $\mathbb{P}$-twist data for $f$ as per \S\ref{section-ptwists}. 
Indeed, $q_n \xrightarrow{\gamma} rf$ intertwines $\action$ with
$\iota = \eta \circ \iota_1$ and hence 
$q_1 \xrightarrow{\gamma \circ \eta} rf$ intertwines
$\action$ with $\iota_1$. On the other hand, by adjunction  
$fr \xrightarrow{f(\action)r} frfr$ is a retract. 
Since $\gamma$ is an isomorphism intertwining $\action$ and $\iota$, 
$fr \xrightarrow{f(\iota)r} f{q_n}r$ is also split, and hence $f q_1 r$ a
split co-extension of $fr$. Thus $(h,q_1, \gamma \circ \eta)$ is 
a $\mathbb{P}$-twist data for $f$, and we define its $\mathbb{P}$-twist
$P_f$ constructed in \S\ref{section-ptwists} to be the $\mathbb{P}$-twist of
the functor $f$ with the $\mathbb{P}^n$-functor data $(h, q_n, \gamma)$. 

In this section we make this precise and show 
the $\mathbb{P}$-twist of a $\mathbb{P}^n$-functor to be an autoequivalence. 

\subsection{$\mathbb{P}^n$-functor data}
\label{section-pn-functor-data}

Throughout the section we use the assumptions and the notation
introduced in \S\ref{section-ptwists-generalities} which fixes
a DG-enhancement $F \in D(\AbimB)$ of an exact functor 
$D(\A) \rightarrow D(\B)$, its left and right adjoints $L$ and $R$, 
and the corresponding spherical twists $T$, $T'$ and cotwists $C$ and
$C'$. 
 
Let $Q_n$ be any cyclic coextension of $\id_\A$ by $H$ of degree $n$. 
Recall that we have exact triangles
\begin{align}
Q_{n-1} \xrightarrow{\iota_{n-1}} Q_n \xrightarrow{\mu_n} H^n
\rightarrow Q_{n-1}[1] \\
\id_\A \xrightarrow{\iota} Q_n \xrightarrow{\kappa} J_n
\xrightarrow{\lambda} \id_\A[1]
\end{align}
in $D(\AbimA)$ defined on the level of twisted complexes 
over $\AmodbarA$, cf.~\S\ref{section-repeated-extensions}. 

Assume now that there exists an isomorphism $\gamma\colon Q_n \simeq RF$ 
intertwining $\iota$ and the adjunction unit. Note that $J_n$ is
then isomorphic to $C[1]$. Define 
$\eta: Q_1 \rightarrow Q_n$ to be the composition $\iota_n \circ ... \circ \iota_2$, so that $\iota = \eta \circ \iota_1$. 
Then $\gamma \circ \eta: Q_1 \rightarrow RF$ intertwines $\iota_1\colon \id \rightarrow Q_1$ 
and the adjunction unit.  Moreover,  $F \xrightarrow{F \iota} F Q_n$ is a retract, since $F
\xrightarrow{F \action} FRF$ is. But then so is 
$F \xrightarrow{\iota_1} FQ_1$, and thus $Q_1$ is an $F$-split coextension of $\id$ by $H$ and $(H,Q_1, \gamma \circ \eta)$
is a $\mathbb{P}$-twist data for $F$.  This data determines, as 
per \S\ref{section-ptwists}, the map $\psi\colon FHR \rightarrow FR$.
By abuse of notation, the map $FHQ_n \rightarrow FQ_n$ identified by $\gamma$ with 
$FHRF \xrightarrow{\psi F} FRF$ is also denoted by $\psi F$. 

\begin{defn}
\label{defn-Pn-functor}
Let $\A$ and $\B$ be DG-categories and let $F \in D(\AbimB)$ be
an enhanced functor $D(\A) \rightarrow D(\B)$. 
A \em $\mathbb{P}^n$-functor data for $F$\rm, 
or alternatively \em the structure of a $\mathbb{P}^n$-functor on $F$\rm, 
is the collection $(H, Q_n, \gamma)$ of
\begin{enumerate}
\item $H \in D(\AbimA)$, an enhanced autoequivalence of 
$D(\A)$ with $H(\krn F) = \krn F$. 
\item An object $Q_n \in D(\AbimA)$ with a structure of a cyclic
coextension of $\id$ by $H$ of degree $n$. 
\item An isomorphism $Q_n \xrightarrow{\gamma} RF$ in $D(\AbimA)$
which intertwines  
$\id_\A \xrightarrow{\action} RF$ and 
$\id_\A \xrightarrow{\iota} Q_n$.
\end{enumerate}
which satisfy the following three conditions
\begin{enumerate}
\item \em The monad condition: \rm The following composition 
is an isomorphism:
\begin{equation}
\label{eqn-monad-condition}
\nu\colon FHQ_{n-1} \xrightarrow{FH \iota_{n-1}} FH Q_n 
\xrightarrow{\psi F} FQ_n \xrightarrow{F\kappa} FJ_n. 
\end{equation}
\item \em The adjoints condition: \rm The following composition 
is an isomorphism:
\begin{equation}
\label{eqn-adjoints-condition}
FR \xrightarrow{FR \action} FRFL \xrightarrow{F\mu_n L} FH^n L.
\end{equation}
 \item \em The highest degree term condition: \rm There exists 
$F H^n L \xrightarrow{\sim} F H H^n H' L$ making the diagram
\begin{equation}
\label{eqn-highest-degree-term-condition}
\begin{tikzcd}[column sep={2cm}]
FHQ_{n-1}L
\ar{r}{FH\iota_n L}
\ar[equals]{d}
&
FHRFL
\ar{r}{\psi FL}
&
FRFL
\ar{r}{F\mu_n L}
&
FH^nL
\ar[dashed]{d}{\text{isomorphism}}
\\
FHQ_{n-1}L
\ar{r}{FH\iota_n L}
&
FHRFL
\ar{r}{FHR\psi'}
&
FHRFH'L
\ar{r}{FH\mu_n H'L}
&
FHH^nH'L
\end{tikzcd}
\end{equation}
commute. Here $\psi': FL \to FH'L$ is the left dual of $\psi$.
\end{enumerate}
\end{defn}

\bf Remark: \rm
\begin{enumerate}

\item  
We introduce a separate object $Q_n \in D(\AbimA)$ and 
an isomorphism $\gamma$ from it to $RF$ for technical reasons. 
We want the flexibility, as described in \S\ref{section-cyclic-extensions}, 
to replace $Q_n$ and all the intermediate coextensions $Q_i$ by the
convolutions of some twisted complex of form 
\eqref{eqn-cyclic-coextension-of-id-by-H-of-degree-n-dg} and its
corresponding subcomplexes. In practice, the $\mathbb{P}^n$-functor
structure on $F$ can always be given, up to an obvious notion of
isomorphism of such structures, with $Q_n = RF$ and $\gamma = \id$. In 
other words --- by an autoequivalence $H$ and a Postnikov tower 
\eqref{eqn-cyclic-coextension-of-id-by-H-of-degree-n} with $Q_n = RF$
and $\iota$ equal to the adjunction unit. 

\item The condition $H(\krn f) = \krn f$ holds for all the
$\mathbb{P}^n$-functors known to the authors. It can be weakened to
the condition $H^{n+1}(\krn f) = \krn f$ which we show to be 
necessary in 
Corollary
\ref{cor-if-H^n+1-krnF-is-not-krnF-Ptwist-is-not-autoequivalence}. 
However, we would then need to strengthen the adjoints condition. 
Indeed, under the weakened assumption $H^{n+1}(\krn f) = \krn f$
the proof that the $\mathbb{P}^n$-twist is autoequivalence
(Theorem \ref{theorem-Pn-functor-gives-autoequivalence})
works in exactly the same way provided that, in addition to 
the adjoints condition \eqref{eqn-adjoints-condition}, the
following composition is also an isomorphism
\begin{equation}
\label{eqn-adjoints-condition-twisted}
FHR \xrightarrow{FHR \action} FHRFL \xrightarrow{FH\gamma L} FHQ_nL
\xrightarrow{FH\mu_n L} FH^{n+1} L. 
\end{equation}
Of course, \eqref{eqn-adjoints-condition}
and \eqref{eqn-adjoints-condition-twisted} both follow
from the following composition being an isomorphism 
\begin{equation}
\label{eqn-adjoints-condition-true}
R \xrightarrow{R \action} RFL \xrightarrow{\gamma L} FHQ_nL
\xrightarrow{\mu_n L} H^{n} L, 
\end{equation}
which is closer to the original adjoints condition in
\cite{Addington-NewDerivedSymmetriesOfSomeHyperkaehlerVarieties}. 
\end{enumerate}

Our main result is the following:

\begin{theorem}
\label{theorem-Pn-functor-gives-autoequivalence}
Let $\A$ and $\B$ be DG-categories and $F \in D(\AbimB)$ 
an enhanced functor $D(\A) \rightarrow D(\B)$. Let  
$(H, Q_n, \gamma)$ be the structure of a $\mathbb{P}^n$-functor on
$F$.  Then the $\mathbb{P}$-twist $P_F$ that we get from 
the associated $\mathbb{P}$-twist data $(H, Q_1, \gamma\circ\eta)$
is an autoequivalence of the category $D(\B)$.
\end{theorem}

We prove Theorem \ref{theorem-Pn-functor-gives-autoequivalence}
in sections \ref{subsection-left-adjoint}-\ref{subsection-P-equivalence} below.
Meanwhile, let us explain why for $\mathbb{P}^n$-functor structures 
we should only consider those autoequivalences $H$ for which 
$H^{n+1}(\krn F) = \krn F$. Addington proved for split
$\mathbb{P}^n$-functors in
\cite[Prop.~3.3]{Addington-NewDerivedSymmetriesOfSomeHyperkaehlerVarieties}
that $P_F F \simeq FH^{n+1} [2]$. This also holds in our more general
context:
\begin{prps}
\label{prps-PF-is-isomorphic-to-FH^n+1}
Let $\A$ and $\B$ be DG-categories and $F \in D(\AbimB)$
an enhanced functor $D(\A) \rightarrow D(\B)$. Let $H$
be an enhanced autoequivalence of $D(\A)$, $Q_n$ a 
cyclic extension of $\id$ by $H$, and $Q_n \xrightarrow{\gamma} RF$
an isomorphism intertwining $\id_\A \xrightarrow{\action} RF$ 
and $\id_\A \xrightarrow{\iota} Q_n$. Suppose 
the monad condition from Definition \ref{defn-Pn-functor} holds. 

Let $P_F$ be the $\mathbb{P}$-twist we get from the associated 
$\mathbb{P}$-twist data $(H, Q_1, \gamma\circ\eta)$. Then 
$$ P_F F \simeq FH^{n+1}[2].$$
\end{prps}
\begin{proof}
By the definition of the $\mathbb{P}$-twist, $P_F F$ is
isomorphic to any convolution of the two-step complex:
\begin{equation*}
FHRF \xrightarrow{\psi F} FRF \xrightarrow{\trace F}
\underset{\degzero}{F}. 
\end{equation*}
The morphism $FRF \xrightarrow{\trace F} F$ has a right semi-inverse
$F \xrightarrow{F\action} FRF$. Since $\gamma$ intertwines $\action$
and $\iota$, it identifies the resulting direct sum decomposition 
of $FRF$ with $FQ_n \simeq F \oplus FJ_n$ and the projection onto 
the other direct summand with the map 
$FQ_n \xrightarrow{F\kappa} FJ_n$. It follows that $P_F$
is isomorphic to 
\begin{equation*}
\cone \left(
FHQ_n \xrightarrow{F\kappa \circ \psi F} FJ_n
\right)[1]
\end{equation*}
By the monad condition the composition of 
$F\kappa \circ \psi F$  
with 
\begin{equation}
\label{eqn-PF-as-cone-FHQ_n-1-to-FQ_n}
FHQ_{n-1} \xrightarrow{FH\iota_{n-1}} FHQ_n 
\end{equation}
is an isomorphism, hence
$$ \cone(F\kappa \circ \psi F) \simeq \cone (FH\iota_{n-1})[1].$$ 
Thus $P_F F$ is isomorphic to $\cone (FH\iota_{n-1})[2]$. 
By the exact triangle 
\begin{equation*}
Q_{n-1} \xrightarrow{\iota_{n-1}} Q_n \xrightarrow{\mu_n} H^n
\rightarrow Q_{n-1}[1] 
\end{equation*}
we conclude that $P_F F \simeq FH^{n+1}[2]$, as required. 
\end{proof}

It follows immediately that unless $H^{n+1}(\krn F) = \krn F$ the
$\mathbb{P}$-twist of $F$ can not be an autoequivalence: 

\begin{cor}
\label{cor-if-H^n+1-krnF-is-not-krnF-Ptwist-is-not-autoequivalence}
Let $F$ and $(H, Q_n, \gamma)$ be as in 
Proposition \ref{prps-PF-is-isomorphic-to-FH^n+1}. 
Then $H^{n+1}(\krn F) \subset \krn F$. Moreover, if this inclusion 
is strict, then $P_F$ is not an autoequivalence.   
\end{cor}
\begin{proof}
Let $a \in \krn F$. Then $RFa = 0$, and hence $Q_n a = 0$. Then 
$$ FJ_n a \simeq \cone( Fa \xrightarrow{F\iota a} FQ_na) \simeq Fa[1]
\simeq 0, $$
but 
$$ FQ_{n-1}a \simeq \cone \left(FQ_{n}a \xrightarrow{\mu_n a}
FH^{n+1}a\right)[-1] \simeq FH^{n+1}a[-1].$$
By the monad condition the two are isomorphic, and hence $H^{n+1} a \in
\krn F$. 

On the other hand, suppose there exists $a \in \krn F$ such that
$a \notin H^{n+1} (\krn F)$. Then $FH^{-(n+1)}a$ is not a zero object, 
but on the other hand by Proposition \ref{prps-PF-is-isomorphic-to-FH^n+1}
$$ P_F(FH^{-(n+1)}a) \simeq FH^{n+1}H^{-(n+1)}a[2] \simeq Fa [2]
\simeq 0,$$
whence $P_F$ is not an autoequivalence.  
\end{proof}

If $H^{n+1}(\krn F) = \krn F$, then 
in the proof of Theorem \ref{theorem-Pn-functor-gives-autoequivalence} 
we need the condition $H(\krn F) = \krn F$ only for the following:
\begin{lemma}
Let $H(\krn F) \subset \krn F$ and let 
the map \eqref{eqn-adjoints-condition} be an isomorphism. Then the map
\eqref{eqn-adjoints-condition-twisted}
is an isomorphism as well.
\end{lemma}
\begin{proof}
Let $X$ be the cone of the map
\eqref{eqn-adjoints-condition-true}.
The adjoints condition \eqref{eqn-adjoints-condition} 
is equivalent to $FX \simeq 0$,
i.e. $X\in \krn F$. Then $HX\in \krn F$, so $FHX\simeq 0$, which is 
in turn equivalent to 
\eqref{eqn-adjoints-condition-twisted} being an isomorphism.
\end{proof}
Note that $H(\krn F) \subset \krn F$ and $H^{n+1} (\krn F)=\krn F$
imply that $H(\krn F)=\krn F$.

\subsection{The left adjoint} 
\label{subsection-left-adjoint}
Recall as per \S\ref{section-ptwists-generalities} that in an enhanced setting 
if an exact functor $f$ with left and right adjoints $l$ and $r$
is enhanced by an $\AbimB$ bimodule $M$, then $l$ can be enhanced by its left dual $M^\barA$ and 
$r$ can be enhanced by its right dual $M^\barB$. Moreover, $(-)^\barA$ and
$(-)^\barB$ induce exact functors $D(\AbimB)^\opp \to D(\BbimA)$.

Let $F\in D(\AbimB)$ be equipped with the $\mathbb{P}^n$-functor data $(H, Q_n, \gamma)$.
Denote by $H'\in D(\AbimA)$ the left dual to $H\in D(\AbimA)$. Since the functor underlying $H$ is an autoequivalence, 
the same is true for $H'$. For $i=1,\ldots, n$ let $Q'_i$ be the left dual of $Q_i$. 
Let the exact triangles
\begin{align}
(H')^i \xrightarrow{\mu'_i} Q'_i \xrightarrow{\iota'_i} Q'_{i-1} \to (H')^i[1]
\end{align}
be left dual to the exact triangles
\begin{align}
Q_{i-1} \xrightarrow{\iota_i} Q_i \xrightarrow{\mu_i} H^i \to Q_{i-1}[1].
\end{align}
This gives $Q'_n$ the structure of a cyclic extension of $\id_\A$ by $H'$. 

Denote by $\gamma'\colon LF \to Q_n'$ the left dual of the isomorphism $\gamma$.
It then follows from the functoriality of dualisation and from the fact that
the counit $LF \to \id_\A$ is left dual to the unit $\id_\A \to RF$ that $\gamma'$ is
an isomorphism which intertwines the counit $LF \to \id_\A$ with
the composition $\iota'=\iota'_1 \circ \ldots \circ \iota'_n: Q'_n \to \id_\A$.

Let $\psi'\colon FL \to FH'L$ be the left dual to the map 
$\psi\colon FHR \to FR$.
Similarly to the results of \cite{AnnoLogvinenko-OnUniquenessOfPTwists},
one can prove that the left dual
\begin{align}
\B \xrightarrow{\action} FL \xrightarrow{\psi'} FH'L
\end{align}
to the complex \eqref{eqn-defn-P-twist}
has a unique convolution  $P'_F$ 
which is isomorphic in $D(\BbimB)$ to the left dual of $P_F$.

\subsection{DG $\mathbb{P}^n$-functor data}
As described in \S\ref{section-cyclic-extensions}, 
every cyclic coextension of $\id_\A$ by $H$ is isomorphic in $D(\AbimA)$
to the convolution of a twisted complex of form
\eqref{eqn-cyclic-coextension-of-id-by-H-of-degree-n-dg}. Throughout
the rest of this section we make use of the notations and conventions
introduced in \S\ref{section-cyclic-extensions}
for a complex $\bar{Q}_n$ of this form:
the truncated subcomplexes $\bar{Q}_i$ and $\bar{J}_n$, 
the maps $\iota_i, \iota, \mu_i, \kappa$, etc.

\begin{defn}
A \em DG $\mathbb{P}^n$-functor data \rm $(H, \bar{Q}_n, \gamma)$ for $F$ is the collection of
\begin{enumerate}
\item $H \in \AmodbarA$ such that $h = (-) \ldertimes_\A H$ is 
an autoequivalence of $D(\A)$. 
\item A twisted complex $\bar{Q}_n \in \pretriag(\AbarA)$ of form \eqref{eqn-cyclic-coextension-of-id-by-H-of-degree-n-dg}. 
\item An homotopy equivalence $\bar{Q}_n \xrightarrow{\gamma} RF$ in $\pretriag(\AbarA)$.
\end{enumerate}
such that $(H, Q_n, \gamma)$ is a $\mathbb{P}^n$-functor data in $D(\AbimA)$.
\end{defn}

Let us fix some additional notation. For each $(H, \bar{Q}_n,
\gamma)$ as above choose and fix a homotopy inverse to $\gamma$:
\begin{equation}
\gamma^{-1}\colon RF \rightarrow \bar{Q}_n. 
\end{equation}
We use $\gamma$ and $\gamma^{-1}$ to implicitly identify $Q_n$ with $RF$ in the
following sense: for any map to $Q_n$ or from $RF$, e.g. $\iota_n: Q_{n-1}\to Q_n$, denote by the same
letter its composition with $\gamma$, e.g. $\iota_n: Q_{n-1}\to RF$,
and similarly 
for any map from $Q_n$ or to $RF$ and its composition with $\gamma^{-1}$. 

The monad condition asks for the composition
\begin{equation}
\nu\colon 
FHQ_{n-1} 
\xrightarrow{FH\iota_n} 
FHQ_n 
\xrightarrow{FH\gamma}
FHRF
\xrightarrow{\psi F}
FRF
\xrightarrow{F\gamma^{-1}}
FQ_n
\xrightarrow{F\kappa}
FJ_n
\end{equation}
to be a homotopy equivalence. Choose and fix a homotopy inverse $\nu^{-1}$
of $\nu$.

Any DG $\mathbb{P}^n$-functor data contains DG $\mathbb{P}$-twist data
$(H, \bar{Q}_1, \gamma \circ \iota_1)$. Denote by $\bar{P}_F$ the corresponding twisted complex 
\eqref{eqn-canonical-P-twist-twisted-complex} that we have constructed in \S\ref{section-P-twists-via-DG}.
Its first map is a lift of the map $\psi$ from the derived category to the bar category of modules,
so by abuse of notation, we are going to denote it by $\psi$ as well:
\begin{equation}
\bar{P}_F = \left( 
\begin{tikzcd}[column sep={2cm},row sep={1cm}] 
FHR
\ar{r}{\psi}
&
FR
\ar{r}{\trace}
&
\B
\end{tikzcd}
\right).
\end{equation}
By \cite[Lemma 3.43]{AnnoLogvinenko-BarCategoryOfModulesAndHomotopyAdjunctionForTensorFunctors}
we can construct a twisted complex of form 
\begin{equation}
\label{equation-twisted-complex-P-dual-raw-version}
\begin{tikzcd}[column sep={2cm},row sep={1cm}] 
\B
\ar{r}
\ar[dashed, bend left="10"]{rr}
&
FL
\ar{r}
&
FH'L
\end{tikzcd}
\end{equation}
whose convolution is homotopy left dual to the convolution of
$\bar{P}_F$. By \cite[\S 2.2]{AnnoLogvinenko-SphericalDGFunctors}
the map $\id_\B \xrightarrow{\action} FL$ is left dual to the map $FR \xrightarrow{\trace} \id_\B$ 
in $D(\AbimA)$, hence the map $\B \to FL$ in \eqref{equation-twisted-complex-P-dual-raw-version} is homotopic to $\B \xrightarrow{\action} FL$.
Choose and fix a twisted complex $\bar{P}'_F$ homotopy equivalent to
\eqref{equation-twisted-complex-P-dual-raw-version}
where the first map will be $\action$. Denote its second map by $\psi'$
and its convolution by $P'_F$:
\begin{equation}
\label{equation-twisted-complex-P-dual}
\bar{P}'_F = \left( 
\begin{tikzcd}[column sep={2cm},row sep={1cm}] 
\B
\ar{r}{\action}
\ar[dashed, bend left="10"]{rr}
&
FL
\ar{r}{\psi'}
&
FH'L
\end{tikzcd}
\right).
\end{equation}

%

\subsection{Homotopy splitting} 
%
Let $\mathcal{C}$ be a pre-triangulated category. Consider a diagram
in $\mathcal{C}$ 
\begin{equation}
\label{eqn-abstract-homotopy-split-triple}
X \xrightarrow{f} Y \xrightarrow{g} Z
\end{equation}
that becomes an exact triangle in $H^0(\mathcal{C})$.
We say that \eqref{eqn-abstract-homotopy-split-triple} is \em 
homotopy split \rm if $f$ has a homotopy left inverse $h: Y \to
X$ of $f$. Since retracts are split in triangulated categories 
the map $Y \xrightarrow{h \oplus g} X \oplus Z$ is then
a homotopy equivalence. 
This situation occurs twice in our setting: the map  
$FRF \xrightarrow{\trace F} F$  homotopy splits the diagram
$$F \xrightarrow{F\action} FRF \xrightarrow{F\kappa} FJ_n, $$ 
and the map
$FHRF \xrightarrow{\psi F} FRF \xrightarrow{F\kappa } FJ_n
\xrightarrow{\nu^{-1}} FHQ_{n-1}$ homotopy splits the diagram
$$FHQ_{n-1} \xrightarrow{FH \iota_n} FHRF \xrightarrow{FH \mu_n}
FHH^n. $$
Therefore, in each of the following cases there exists a morphism in $\BmodbarB$ which replaces $(*)$ 
producing mutually inverse isomorphisms in $D(\BbimB)$:
\begin{align}
\label{eqn-map-breaking-FRFL}
FRFL  \xrightarrow{\trace FL \oplus F\kappa L} & FL \oplus FJ_nL 
\xrightarrow{F\action L \oplus *} FRFL\\
\label{eqn-map-breaking-FHRFL}
FHRFL  \xrightarrow{FH \mu_n L \oplus (\nu^{-1}L) \circ (F\kappa L)\circ (\psi FL)} & FH^{n+1}L \oplus FHQ_{n-1}L 
\xrightarrow{* \oplus FH\iota_n L} FHRFL\\
\label{eqn-map-breaking-FRFH'L}
FRFH'L  \xrightarrow{\trace FH'L \oplus F\kappa H'L} & FH'L \oplus FJ_nH'L 
\xrightarrow{F\action H'L \oplus *} FRFH'L\\
\label{eqn-map-breaking-FHRFH'L}
FHRFH'L  \xrightarrow{FH\mu_n H'L \oplus (\nu^{-1}H'L) \circ (F\kappa H' L)\circ (\psi FH'L)} & FH^nL \oplus FHQ_{n-1}H'L
\xrightarrow{* \oplus FH\iota_n H'L} FHRFH'L.
\end{align}

Note that in the case of $FRFL$ the map $*: FJ_nL \to FRFL$ 
can be given explicitly as 
$$FJ_nL \xrightarrow{\nu^{-1}} FHQ_{n-1}L \xrightarrow{FH\iota_n L} FHRFL \xrightarrow{\psi FL} FRFL,$$
 since $(\trace FL) \circ (\psi FL)$ equals zero in $D(\BbimB)$.


\subsection{Breaking down the tensor product}
To prove that
$P_F$ is an equivalence, we first establish
that $P_F P'_F$ is homotopy equivalent to $\id_B$. 
Here and below by the convolution
of a twisted bicomplex we mean the convolution of its totalization.
\begin{prps}
\label{prps-PP'-skeleton}
$P_FP'_F$ is homotopy equivalent to the convolution of the twisted bicomplex
\begin{equation}
\label{eqn-diagram-PP'-skeleton}
\begin{tikzcd}[column sep={2cm}]
&
FR
\ar{r}{\trace}
\ar{d}{-(F\kappa L) \circ (FR\action)}
&
\underset{\degzero}{\B}
\\
FHQ_{n-1}L
\ar{r}{\nu L}
\ar[']{d}{(FH\mu_n H'L) \circ (FHR\psi') \circ (FH\iota_n L)}
&
\underset{\degzero}{FJ_nL}
&
\\
\underset{\degzero}{FHH^nH'L}.
&
&
\end{tikzcd}
\end{equation}
\end{prps}
\begin{proof}
Our first claim is that  $P_FP'_F$ 
is homotopy equivalent to the convolution of the twisted bicomplex
\begin{equation}
\label{eqn-diagram-PP'-initial}
\begin{tikzcd}[column sep={2cm}]
FHR
\ar{r}{\psi}
\ar[']{d}{FHR\action}
&
FR
\ar{r}{\trace}
\ar{d}{-FR\action}
&
\underset{\degzero}{\B}
\ar{d}{\action}
\\
FHRFL
\ar{r}{\psi FL}
\ar[']{d}{FHR\psi'}
&
\underset{\degzero}{FRFL}
\ar{r}{\trace FL}
\ar{d}{-FR\psi'}
&
FL
\ar{d}{\psi'}
\\
\underset{\degzero}{FHRFH'L}
\ar{r}{\psi FH'L}
&
FRFH'L
\ar{r}{\trace FH'L}
&
FH'L
\end{tikzcd}
\end{equation}
with some higher differentials going down and right. 
By \cite[Lemma 3.42]{AnnoLogvinenko-BarCategoryOfModulesAndHomotopyAdjunctionForTensorFunctors}, 
since both twisted complexes $\bar{P}_F$ and $\bar{P}'_F$ are one-sided,
$P_FP'_F$ is homotopy equivalent to the convolution of a twisted bicomplex
whose entries are tensor products of those of $\bar{P}_F$ and $\bar{P}'_F$ in $\BmodbarB$.
Its differentials are tensor products of the differentials in the original complexes with
appropriate signs. Thus all differentials in this twisted bicomplex run right and down.
It differs however from \eqref{eqn-diagram-PP'-initial} in the first row and the third column,
because in $\BmodbarB$ the functors $(-)\bartimes \B$ and $\B\bartimes (-)$ are homotopic to
identity, not isomorphic to it. We now use the 
Replacement Lemma (Lemma \ref{lemma-replacement-lemma}) 
to replace the entries in the first row and the third column one by one.
The mutually inverse homotopy equivalences we use are, 
in the notation of \cite[\S 3.3]{AnnoLogvinenko-BarCategoryOfModulesAndHomotopyAdjunctionForTensorFunctors}:
\begin{align*}
(-) \xrightarrow{\beta_{(-)}} (-)\bartimes \B \xrightarrow{\alpha_{(-)}} (-)
\qquad & \text{for $FL$ and $FH'L$}\\
(-) \xrightarrow{\beta_{(-)}} \B \bartimes (-) \xrightarrow{\alpha_{(-)}} (-)
\qquad & \text{for $FR$ and $FHR$} \\
\B \xrightarrow{\beta_{\B}^l} \B\bartimes\B \xrightarrow{\alpha_\B} \B
\qquad & \text{for $\B$ in the top right corner.} 
\end{align*}
By \cite{AnnoLogvinenko-BarCategoryOfModulesAndHomotopyAdjunctionForTensorFunctors},
Prop. 3.27 (1) the degree zero differentials in the resulting complex coincide with those in
\eqref{eqn-diagram-PP'-initial}. We acquire some higher differentials in the process, but since we
replace one node at a time, these new differentials only go down and right.
The claim follows. 

Now, we apply the Replacement Lemma 
to replace $FHRFL$, $FRFL$, $FHRFH'L$, and $FRFH'L$ 
using the homotopy equivalences 
\eqref{eqn-map-breaking-FHRFL},
\eqref{eqn-map-breaking-FRFL},
\eqref{eqn-map-breaking-FHRFH'L}, and 
\eqref{eqn-map-breaking-FRFH'L} respectively.
We obtain the bicomplex:
\begin{equation}
\label{eqn-diagram-PP'-homotopy-split}
\begin{tikzcd}[column sep={2cm}, row sep={1.5cm}]
FHR
\ar{r}{\psi}
\ar[']{d}{\begin{pmatrix}(FH\mu_n L)\circ (FHR\action) \\ *  \end{pmatrix}}
&
FR
\ar{r}{\trace}
\ar{d}{\begin{pmatrix}
* \\ -(F\kappa R) \circ (FR\action)  \end{pmatrix}}
&
\id
\ar{d}{\action}
\\
FHH^nL \oplus FHQ_{n-1}L
\ar{r}{\begin{pmatrix} * \amsamp * \\ \nu L \amsamp * \end{pmatrix}}
\ar[']{d}{\begin{pmatrix} * \amsamp FH(\mu_n H'L\circ R\psi' \circ \iota_n L) \\ * \amsamp * \end{pmatrix}}
&
FL \oplus FJ_nL
\ar{r}{(\sim \ * )}
\ar{d}{\begin{pmatrix} * \amsamp * \\ * \amsamp * \end{pmatrix}}
&
FL
\ar{d}{\psi'}
\\
FHH^nH'L \oplus FHQ_{n-1}H'L
\ar{r}{\begin{pmatrix} * \amsamp * \\ \nu H'L \amsamp * \end{pmatrix}}
&
FH'L \oplus FJ_nH'L
\ar{r}{( \sim \ * )}
&
FH'L
\end{tikzcd}
\end{equation}
where  $(*)$ denotes unknown maps, and  $\sim$ denotes homotopy equivalences.
Once again, the higher differentials only go down and right.
Finally, the Replacement Lemma allows us to remove 
acyclic subcomplexes without acquiring any new differentials 
if there are no arrows going in or no arrows going out of these subcomplexes.
We apply it to remove parts of \eqref{eqn-diagram-PP'-homotopy-split} in the following order:
\begin{enumerate}
\item The subcomplex $FH'L \to FH'L$ in the bottom row (no arrows out);
\item The subcomplex $FHQ_{n-1}H'L \to FJ_nH'L$ in the bottom row (no arrows out after the previous step);
\item The subcomplex $FL \to FL$ in the second row (no arrows out after the previous step);
\item The subcomplex $FHR \to FHH^nL$ in the first column (no arrows in).
\end{enumerate}
We obtain the bicomplex \eqref{eqn-diagram-PP'-skeleton}.
It has no higher differentials since at every step of the proof we ensured that the higher differentials only
went down and right. 
\end{proof}

\subsection{$P'_F$ is fully faithful}
By abuse of notation, we use $\mu_n$ to 
also denote the map $J_n \to H^n$ that comes from the projection
of the twisted complex $\bar{J}_n$ onto its first term $H^n[-n]$.

\begin{prps}
\label{prps-PP'-headless-skeleton-acyclic}
The following bicomplex has an acyclic convolution:
\begin{equation}
\label{bicomplex-headless-skeleton}
\begin{tikzcd}[column sep={2cm}]
&
FR
\ar{d}{(F\kappa L)\circ (FR\action)}
\\
{FHQ_{n-1}L}
\ar{r}{\nu L}
\ar[']{d}{(FH\mu_n H'L)\circ (FHR\psi')\circ (FH\iota_n L)}
&
\underset{\degzero}{FJ_nL}
\\
\underset{\degzero}{FHH^nH'L}. 
\end{tikzcd}
\end{equation}
\end{prps}
\begin{proof}
Since $\nu$ is a homotopy equivalence, the complex 
$FHQ_{n-1}L \xrightarrow{\nu L} FJ_nL$ is acyclic. 
Let $b_0$ and $b_1$ be any degree $-1$ maps that
lift the boundaries $\id-\nu^{-1} \circ \nu$ and $\id-\nu \circ \nu^{-1}$, respectively.
Then the map
\begin{equation}
\begin{tikzcd}[column sep={1.5cm}, row sep={1cm}]
\big(FHQ_{n-1}L
\ar{r}{\nu L} 
\ar[']{d}{-b_1L}
&
FJ_nL\big)
\ar[']{ld}{\nu^{-1} L}
\ar{d}{b_0L}
\\
\big(FHQ_{n-1}L
\ar{r}{\nu L} 
&
FJ_nL\big)
\end{tikzcd}
\end{equation}
is a contracting homotopy.
Thus by the Replacement Lemma the convolution of \eqref{bicomplex-headless-skeleton}
is homotopy equivalent to the convolution of 
\begin{equation}
\label{eqn-extract-of-skeleton}
FR
\xrightarrow{(FH\mu_n H'L)\circ (FHR\psi')\circ (FH\iota_n L) \circ (\nu^{-1}L) \circ (F\kappa L)\circ (FR\action)}
FHH^nH'L.
\end{equation}
It suffices to show that this map is a homotopy equivalence.
By the highest degree term condition, the map 
$(FH\mu_n H'L)\circ (FHR\psi')\circ (FH\iota_n L)$ is homotopic to the composition of
$$
(F\mu_n L)\circ (\psi FL) \circ (\iota_n L) = (F\mu_n L) \circ (\nu L)
$$
with a homotopy equivalence.
It remains to show that 
$$
 (F\mu_n L) \circ (\nu L) \circ (\nu^{-1}L) \circ (F\kappa L)\circ (FR\action)
$$
is a homotopy equivalence. Since $(\nu L) \circ (\nu^{-1}L)\sim\id$,
this follows from the adjoints condition.
\end{proof}

\begin{cor}
\label{cor-functor-P'F-is-fully-faithful}
The functor $P'_F$ is fully faithful.
\end{cor}
\begin{proof}
By Prop. \ref{prps-PP'-skeleton} and \ref{prps-PP'-headless-skeleton-acyclic}
we have $P_FP'_F\simeq \id_\B$ in $D(\BbimB)$.
By Johnstone's Lemma 
\cite[Lemma 1.1.1]{Johnstone-SketchesOfAnElephantAToposTheoryCompendiumV1}
the adjunction unit $\id_B \xrightarrow{\action} P_FP'_F$ is an
isomorphism on every object of $D(\B)$,
whence $P'_F$ is fully faithful.
\end{proof}

\subsection{$P_F$ is an equivalence}
\label{subsection-P-equivalence}

\begin{proof}[Proof of Theorem \ref{theorem-Pn-functor-gives-autoequivalence}]

We have shown in Corollary \ref{cor-functor-P'F-is-fully-faithful}
that the left adjoint $P'_F$ of $P_F$ is fully faithful. By 
\cite[Lemma 2.1]{BKR01} if $\krn P_F = 0$ then 
$P_F$ is an equivalence. Let $a \in \krn P_F$, then the complex 
$$ FHRa \xrightarrow{\psi a} FRa \xrightarrow{\trace a} a $$
is acyclic, whence $a \in \img F$. It now remains to show that $\krn
P_F F = \krn F$. By Proposition 
\ref{prps-PF-is-isomorphic-to-FH^n+1}
we have 
$$ P_F F \simeq FH^{n+1}[2], $$
whence $\krn P_F F = \krn FH^{n+1}$. 
By our assumptions in Definition \ref{defn-Pn-functor} we have 
$H^{n+1}(\krn F) = \krn F$. Thus 
$$ \krn F = \krn FH^{n+1} = \krn P_F F,$$
as desired. 
\end{proof}

We have established in Theorem \ref{theorem-Pn-functor-gives-autoequivalence}
that the exact functor $D(\B) \rightarrow D(\B)$ underlying the
enhanced functor $P_F \in D(\BbimB)$ is invertible. It follows
trivially that $P_F$ is invertible as an enhanced functor.
Indeed, the adjunction unit $\id_\B \xrightarrow{\action} P_FP'_F$ and 
the adjunction counit $P'_F P_F \xrightarrow{\trace} \id_B$ in
$D(\BbimB)$ applied to any object of $D(B)$ give the adjunction unit
and counit of the exact functors underlying $P_F$ and $P'_F$. Since
these exact functors are equivalences, we conclude that the cones of 
$\id_\B \xrightarrow{\action} P_FP'_F$ and 
$P'_F P_F \xrightarrow{\trace} \id_B$ are zero on every object of
$D(\B)$ and hence are isomorphic to zero in $D(\BbimB)$. Hence 
$P_F P'_F \simeq P'_F P_F \simeq \id_B$ in $D(\BbimB)$. 

\section{The strong monad condition}
\label{section-strong-monad-condition}

In this section we first analyse the case where the filtration 
$Q_i$ on the monad $RF$ of a $\mathbb{P}^n$-functor $F$ is split. 
This was the case treated by Addington in 
\cite[\S3]{Addington-NewDerivedSymmetriesOfSomeHyperkaehlerVarieties}. 
We compare his definition of a (split) $\mathbb{P}^n$-functor with our 
Definition \ref{defn-Pn-functor} and show that the former 
implies the latter. Specifically, the definition in  
\cite[\S3]{Addington-NewDerivedSymmetriesOfSomeHyperkaehlerVarieties}
also asks for the adjunction monad $RF$ to be the direct sum of $\id$, 
$H$, \dots, $H^n$, but then it imposes two different conditions which we call
the \em strong monad condition \rm and the \em weak adjoints condition\rm. 
We show, as our choice of these names implies, that the strong monad condition 
implies the monad condition in our Definition \ref{defn-Pn-functor}, 
while the weak adjoints condition follows from its adjoints
condition. We further show that if, in the split case,
the strong monad condition is satisfied, then 
the weak adjoints condition implies the adjoints condition, 
and the highest degree term condition is automatically satisfied.

Motivated by this, we find the appropriate analogue 
of the strong monad condition in the general (non-split) case. 
It turns out to be the condition that the monad multiplication 
of $RF$ respects the filtration $Q_i$ in an obvious way. 
We then show, that as in the split case, the strong monad condition implies 
the monad and the highest degree term conditions of 
our Definition \ref{defn-Pn-functor}. We further show that if 
the strong monad condition holds and, additionally,
$\homm^{-1}_{D(\AbimA)}(\id, H^{i})$ vanishes for all $i > 0$, 
then again the weak adjoints condition implies the adjoints condition. 
Having to demand this $\ext^{-1}$ vanishing is unfortunate, however
it does hold in most examples of $\mathbb{P}^n$-functors known 
to the authors to date.  

\subsection{The split case}
\label{section-strong-monad-condition-split-case}
Suppose we have $\mathbb{P}^n$-functor data $(H, Q_n, \gamma)$ where 
$Q_n$ is completely split:
\begin{equation}
\label{eqn-Q_n-as-a-direct-sum-in-split-case}
Q_n\simeq \id \oplus H \oplus \ldots \oplus H^n. 
\end{equation}
Then we also have
$$ Q'_n \simeq \id \oplus H' \oplus \ldots \oplus (H')^n $$
$$ Q_n Q'_n \simeq \bigoplus_{i,j = 0}^n H^{i} (H')^j $$
and the adjunction unit $\id \xrightarrow{\action} Q_nQ'_n$ 
is the sum of $\id \xrightarrow{\action} H^i (H')^i$.
Let $\gamma_1$ denote the composition
$$ H \hookrightarrow Q_n \xrightarrow{\gamma} RF. $$ 
It is the image in $D(\AbimA)$ of the map $\gamma_1$ of any DG-lift
of $(H, Q_n, \gamma)$. Let $\gamma_1': LF \to H'$ be its left dual, 
the projection of $LF$ onto the direct summand $H'$.
By Lemma \ref{lemma-psi-is-independent-of-the-choice-of-phi} we have 
$$ \psi = (FR \trace - \trace FR) \circ F \gamma_1 R.$$ 

Let 
$$ c_{ij}^k\colon H^iH^j \rightarrow  H^k $$
be the components of the monad multiplication 
$$ RFRF \xrightarrow{R\trace F} RF $$ 
under the identification
\eqref{eqn-Q_n-as-a-direct-sum-in-split-case}. 
Let $A_l$ be the map of the left multiplication by $H$ in the monad $RF$: 
$$ A_l \; := \; HRF \xrightarrow{\gamma_1RF} RFRF \xrightarrow{R \trace F} RF, $$
and let $(a_{ij})$ be the matrix of its component maps, 
in other words $a_{ij}=c_{1j}^i$. Similarly, let  
$A_l$ be the map of the right multiplication by $H$:
$$ A_r\; := \; RFH \xrightarrow{RF \gamma_1} RFRF \xrightarrow{R
\trace F} RF $$
 
The bimodule $RF$ is an algebra in the category $D(\AbimA)$.
The bimodule $LF$ is a coalgebra, and the map
$$ RF \xrightarrow{R\action F} RFLF $$ makes $RF$ a right $LF$-comodule. 
Let $B$ be the map of the right coaction of $H'$ in this
comodule: 
$$ B \; := \; RF \xrightarrow{R\action F} RFLF \xrightarrow{RF
\gamma'_1}
RFH' $$
and let $(b_{ij})$ be its matrix of components, where  
$b_{ij}$ is the component $H^j \to H^iH'$. 
We can relate $A_r$ and $B$:

\begin{lemma}
\label{lemma-mult-H-comult-H'-relation}
The composition $RF \xrightarrow{RF\action} RFHH' \xrightarrow{A_rH'} RFH'$ equals $B$.
\end{lemma}
\begin{proof}
First, the diagram commutes:
\begin{equation}
\label{diagram-mult-comult-relation}
\begin{tikzcd}[column sep={2cm}]
RF 
\ar{r}{RF\action}
\ar[equal]{rd}
&
RFRF
\ar{r}{RFR\action F}
\ar{d}{R\trace F}
&
RFRFLF
\ar{d}{R\trace FLF}
\\
&
RF
\ar{r}{R\action F}
&
RFLF.
\end{tikzcd}
\end{equation}
Therefore, 
$$ B \;=\; 
RF \xrightarrow{RF\action} RFRFLF \xrightarrow{R\trace FLF}
RFLF \xrightarrow{RF\gamma'_1} RFH' $$
and by functoriality of the tensor product we further have
$$ B \;=\; 
RF \xrightarrow{RF\action} RFRFLF \xrightarrow{RFRF\gamma'_1}
RFRFH' \xrightarrow{R\trace FH'} RFH', $$
whence it remains to show that the following diagram commutes:
\begin{equation}
\begin{tikzcd}[column sep={2cm}]
RF
\ar{r}{RF \action}
\ar{d}{RF \action}
&
RFHH'
\ar{d}{RF\gamma_1 H'}
\\
RFRFLF
\ar{r}{RFRF\gamma'_1}
&
RFRFH'. 
\end{tikzcd}
\end{equation}
This diagram is $RF$ applied to the diagram  
\begin{equation}
\begin{tikzcd}[column sep={2cm}]
\id
\ar{r}{\action}
\ar{d}{\action}
&
HH'
\ar{d}{\gamma_1 H'}
\\
RFLF
\ar{r}{RF\gamma'_1}
&
RFH'
\end{tikzcd}
\end{equation}
which commutes by 
Corollary \ref{cor-f-and-f-dual-commuting-square}
since $\gamma_1$ is left dual to $\gamma'_1$.
\end{proof}

\subsection{Split $\mathbb{P}^n$-functors}
In
\cite{Addington-NewDerivedSymmetriesOfSomeHyperkaehlerVarieties}
Addington defined a (split) $\mathbb{P}^n$-functor as 
an enhanced functor $F$ equipped with an isomorphism
$$ RF \simeq \id \oplus H \ldots H^n $$ 
for which the following two conditions hold:
\begin{itemize}
\item \em Strong monad condition: \rm The matrix $A_l$ of the 
left multiplication by $H$ in $RF$ has the form 
\begin{equation}
\left(
\begin{tikzcd}[column sep={0cm}, row sep={0cm}]
* & * & \dots & * & * 
\\
1 & * & \dots & * & * 
\\
0 & 1 & \dots & * & * 
\\
\vdots & \vdots & \ddots & \vdots & \vdots 
\\
0 & 0 & \dots & 1 & *
\end{tikzcd}
\right), 
\end{equation}
in other words $a_{kj}=0$ for $k>j+1$ and
$a_{j+1, j}$ are the identity maps for $0 \leq j < n$.
\item \em Weak adjoints condition: \rm
There exists \em some \rm isomorphism $R \simeq H^nL$.
\end{itemize}

Recall that $a_{ij}$ are the components $c^{i}_{1j}$ of the whole 
monad multiplication of $RF$. The associativity of the monad $RF$ 
implies that the strong monad condition above
is equivalent to an even stronger condition:
\begin{lemma}
\label{lemma-split-case-strong-monad-iff-stronger-monad}
Strong monad condition holds if and only if the following holds:
\begin{itemize}
\item Stronger monad condition: 
\begin{equation}
\label{eqn-stronger-monad-condition}
\begin{cases}
c_{ij}^k=0   & \quad \quad  k > i+j 
\\
c_{ij}^{i+j} \text{ is the identity map }
& \quad \quad i+j\leq n.
\end{cases}
\end{equation}
\end{itemize}
\end{lemma}
\begin{proof}
The strong monad condition is the restriction of 
\eqref{eqn-stronger-monad-condition}
to the case $i = 1$. This shows the ``if'' implication, and
provides the base for proving the ``only if'' part by the induction on $i$. 
Indeed, suppose that we have established 
\eqref{eqn-stronger-monad-condition} for all $i < m$.
Since for $i + j > n$ the condition
\eqref{eqn-stronger-monad-condition} is vacuous, it suffices
to establish \eqref{eqn-stronger-monad-condition}
for $i = m$ and $0 \leq j \leq n - m$. 
By associativity of the algebra $RF$ the following square commutes:
\begin{equation}
\label{eqn-H-H(m-1)-Hj-associativity-square}
\begin{tikzcd}[column sep={2cm}]
HH^{m-1}H^j
\ar{r}
\ar{d}
&
\left( \bigoplus\limits_{s=0}^n H^s \right) H^j
\ar{d}
\\
H \left(\bigoplus\limits_{t=0}^n H^t \right)
\ar{r}
&
\bigoplus\limits_{k=0}^n H^k.
\end{tikzcd}
\end{equation}

Let 
$$ f^k \colon H H^{m-1} H^j \rightarrow H^k \quad \quad k \in [0,n] $$
be the corresponding component of the map obtained by travelling around
either the top or the bottom contour of  
\eqref{eqn-H-H(m-1)-Hj-associativity-square}. 
By applying the induction assumption to its first
map we can decompose the top contour into the sum of two maps 
\begin{align}
\label{eqn-top-contour-Hm-Hj-part}
HH^{m-1}H^j \rightarrow H^mH^j \rightarrow \bigoplus_{k=0}^n H^k, 
\\
\label{eqn-top-contour-the rest}
HH^{m-1}H^j \rightarrow  (\bigoplus_{s=0}^{m-1} H^s)H^j 
\rightarrow \bigoplus_{k=0}^n H^k.
\end{align}
The composition \eqref{eqn-top-contour-Hm-Hj-part} 
contributes the summand 
$$ HH^{m-1}H^j \xrightarrow{c^m_{1,m-1}} H^m H^j
\xrightarrow{c^k_{mj}} H^k $$ 
to each $f^k$. On the other hand, by the induction assumption 
each summand of the second map in 
\eqref{eqn-top-contour-the rest} whose target is $H^k$ with $k \geq
m+j$ vanishes. Hence \eqref{eqn-top-contour-the rest} contributes
nothing to $f^k$ for $k \geq m+j$, and thus we have 
$$ f^k = c^k_{mj} \circ c^m_{1,m-1} \quad \quad \text{ for } k \geq
m+j. $$

On the other hand, applying the induction assumption to both maps in 
the bottom contour, we see that $f^k$ is $0$ for $k > m + j$ and 
the identity map for $k = m + j$. Hence the same is true of 
the maps $c^{k}_{mj}$, since by the induction assumption 
$c^m_{1,m-1}$ is the identity map. 
\end{proof}

The stronger monad condition makes it possible 
to choose a different isomorphism $RF \simeq \id \oplus \dots \oplus H^n$
which renormalises the monad multiplication maps $c^k_{ij}\colon H^i
H^j \rightarrow H^k$ to satisfy the best possible conditions:

\begin{cor}[Strongest monad condition]
\label{cor-the-strongest-monad-condition}
Given the stronger monad condition, there exists a new 
$$\gamma'\colon \id \oplus \dots \oplus H^n \xrightarrow{\sim} RF $$
which has the same inclusion $H \hookrightarrow RF$ (and thus the same map 
$\psi\colon FHR \rightarrow FR$ and the same $\mathbb{P}$-twist), 
but makes the monad multiplication symmetric
$$ \forall\; i,j,k \in [0,\dots, n] \quad \quad \quad c^k_{ij} = c^k_{ji}, $$
and strictly diagonal on lower degree terms
$$
\forall\; i,j,k \in [0,\dots, n] 
\text{ with } i+j \leq n
\quad \quad \quad 
c^k_{ij} = 
\begin{cases}
\id & \text{ if } k = i+j \\
0 & \text{ otherwise.}
\end{cases}
$$
\end{cor}
In particular, this strongest monad condition makes the matrices of
the left monad multiplication map $A_l\colon HRF \rightarrow RF$
and the right monad multiplication map $A_r\colon RFH \rightarrow RF$
both be
\begin{equation}
\left(
\begin{tikzcd}[column sep={0cm}, row sep={0cm}]
0 & 0 & \dots & 0 & * 
\\
1 & 0 & \dots & 0 & * 
\\
0 & 1 & \dots & 0 & * 
\\
\vdots & \vdots & \ddots & \vdots & \vdots 
\\
0 & 0 & \dots & 1 & *
\end{tikzcd}
\right). 
\end{equation}
\begin{proof}
Take the inclusion $\gamma_1: H \hookrightarrow RF$ and 
define for each $1 \leq i \leq n$ a new map  
$$ 
\gamma'_i: 
H^i \xrightarrow{(\gamma_1)^i} (RF)^i \xrightarrow{R (\trace)^{i-1} F} RF. 
$$
Observe that $\gamma'_1 = \gamma_1$. Now set 
$\gamma'_0 = \gamma_0 = \action$, and define 
$$ \gamma' \colon 
\id \oplus \dots \oplus H^n \xrightarrow{\sum \gamma'_i} RF. 
$$
The stronger monad condition ensures that each $\gamma'_i$ filters
through the inclusion 
$$ \id \oplus \dots \oplus H^i \hookrightarrow
\id \oplus \dots \oplus H^n \xrightarrow{\gamma} RF, $$
and is the identity map when composed with the projection
$\gamma_i^{-1} \colon RF \rightarrow H^i$, i.e. the matrix of the map 
$$ \id \oplus \dots \oplus H^n \xrightarrow{\gamma'} RF
\xrightarrow{\gamma^{-1}} \id \oplus \dots \oplus H^n $$
is upper triangular with $\id$s on the main diagonal. The map 
is thus an isomorphism, and hence so is $\gamma'$. The
symmetricity and the lower term strict diagonality of the new maps
$c^k_{ij}$ now follow trivially from the associativity of the monad 
multiplication $RFRF \xrightarrow{\trace} RF$.  
\end{proof}

The strongest monad condition implies trivially 
the stronger and the strong monad condition. Since the weak adjoints
condition doesn't involve $\gamma$, it also 
still holds after the renormalisation of Corollary
\ref{cor-the-strongest-monad-condition}. Thus, renormalising if necessary, 
we can assume that any $F$ which satisfies  
the strong monad and the weak adjoints conditions, 
i.e.~the \cite{Addington-NewDerivedSymmetriesOfSomeHyperkaehlerVarieties}
definition of a split $\mathbb{P}^n$-functor, also satisfies the 
strongest monad condition. 

Next, we show that our monad condition in \S\ref{section-Pn-functors}
is indeed weaker than the strong monad condition:

\begin{lemma}
For split $\mathbb{P}^n$-functors the strong monad condition
implies the monad condition.  
\end{lemma}
\begin{proof}
The monad condition of 
Definition \ref{defn-Pn-functor} asks for the map 
$$ \nu: FHQ_{n-1} \rightarrow FJ_n $$
defined in \eqref{eqn-monad-condition} to be an isomorphism.  
In the split case, we have
\begin{align*}
Q_{n-1} & \simeq \id \oplus H \oplus \dots \oplus H^{n-1}, \\
J_n & \simeq H \oplus H^2 \oplus \dots \oplus H^n,
\end{align*}
and the map $\nu$ is the map 
\begin{align}
\label{eqn-the-map-nu-in-the-split-case}
FH \oplus FH^2 \oplus \dots \oplus FH^n \rightarrow 
FH \oplus FH^2 \oplus \dots \oplus FH^n 
\end{align}
given by
$$ FH(\id) \oplus FH H \oplus \dots \oplus FH H^{n-1} 
\hookrightarrow 
FRFRF \xrightarrow{FR \trace F - \trace FRF } FRF 
\twoheadrightarrow
FH \oplus FH^2 \oplus \dots \oplus FH^n. $$
The summand of $\nu$ which comes from $-\trace FRF$ 
leaves the second $RF$ in $FRFRF$ completely untouched and 
thus only has components of form $(FH \rightarrow F)H^i$.
It is therefore given by a strictly upper triangular matrix. 
On the other hand, the summand of $\nu$ which comes from 
$FR\trace F$ is given by $F$ applied to the matrix $A_l$ of
the left multiplication by $H$ in $RF$ with the top row
and the leftmost column removed. The strong monad condition 
asserts that this matrix is upper triangular and has $\id$s   
on the main diagonal. We conclude that the map $\nu$ is also given 
by an upper triangular matrix with $\id$s on the main diagonal, 
and thus is an isomorphism. 
\end{proof}

\begin{lemma}
For split $\mathbb{P}^n$-functors the strong monad condition and 
the weak adjoints condition imply the adjoints condition.  
\end{lemma}
\begin{proof}
The adjoints condition of Definition \ref{defn-Pn-functor}
asks for the map \eqref{eqn-adjoints-condition}
$$
FR \xrightarrow{FR\action} FRFL \xrightarrow{F\mu_n L} FH^n L
$$
to be an isomorphism. We are going to prove a stronger result: that the map
\eqref{eqn-adjoints-condition-true}
$$
R \xrightarrow{R\action} RFL \xrightarrow{\mu_n L} H^n L
$$
is an isomorphism.
By commutativity of the diagram 
\eqref{diagram-mult-comult-relation}
the map $\eqref{eqn-adjoints-condition-true}F$ equals the composition
\begin{equation}
\label{eqn-adjoints-condition-map-F}
RF \xrightarrow{RF\action} RFRFLF \xrightarrow{R\trace FLF} 
RFLF \xrightarrow{\mu_n LF} H^n LF.
\end{equation}
The first map here is the sum of the isomorphisms 
$RF(\id \rightarrow H^j (H')^j)$. 
The second map is the monad multiplication for which the strong
monad condition holds. The last map is the projection onto the $H^n$
term of the monad $RF$. It follows that the restriction of 
\eqref{eqn-adjoints-condition-map-F} to the summand $H^i$ of $RF$
has the form  
$$ H^i \rightarrow \bigoplus_{j = 0}^n H^i H^j (H')^j
\rightarrow \bigoplus_{j = 0}^{n} \bigoplus_{k = 0}^{\min(i+j,n)} H^{k} (H')^j 
\rightarrow \bigoplus_{j = n-i}^{n} H^{n} (H')^j, $$
Moreover, the component  $H^i \rightarrow H^{n} (H')^{n-i}$ 
is the map 
$$ H^i \xrightarrow{\sim} H^i H^{n-i} (H')^{n-i} \rightarrow H^n
(H')^{n-i} $$
and thus an isomorphism since by the strong monad condition 
its second composant is the identity map.  

We conclude that the map $\eqref{eqn-adjoints-condition-true}F$ is a map 
$$ \id \oplus H \oplus H^2 \oplus \dots \oplus H^n
\rightarrow H^n 
\Bigl( (H')^n \oplus \dots \oplus H' \oplus \id \Bigr) $$
given by an upper triangular matrix with $\id$s on the main
diagonal. It is therefore an isomorphism. The fact that 
$\eqref{eqn-adjoints-condition-true}$ itself is an isomorphism now 
follows from the weak adjoints condition by the result which 
we prove in the next section in Prps.~\ref{prps-R-HnL-canonical-map}
since it also holds in the non-split case: 
if $RF \xrightarrow{\eqref{eqn-adjoints-condition-true}F} H^nLF$ is an isomorphism 
and there exists some isomorphism $R \simeq H^n L$, then 
$R \xrightarrow{\eqref{eqn-adjoints-condition-true}} H^nL$ is also an
isomorphism. 
\end{proof}

\begin{lemma}
For split $\mathbb{P}^n$-functors the strong monad condition implies 
the highest degree term condition.  
\end{lemma}
\begin{proof}
The highest degree term condition of Definition \ref{defn-Pn-functor}
asks for an existence of an isomorphism, marked by a dashed arrow on 
the diagram below, which makes the rest of the diagram commute:
\begin{equation*}
\begin{tikzcd}[column sep={2cm}]
FHQ_{n-1}L
\ar{r}{FH\iota_n L}
\ar[equals]{d}
&
FHRFL
\ar{r}{\psi FL}
&
FRFL
\ar{r}{F\mu_n L}
&
FH^nL
\ar[dashed]{d}
\\
FHQ_{n-1}L
\ar{r}{FH\iota_n L}
&
FHRFL
\ar{r}{FHR\psi'}
&
FHRFH'L
\ar{r}{FH\mu_n H'L}
&
FHH^nH'L
\end{tikzcd}
\end{equation*}
where $\psi': FL \to FH'L$ is the left dual of $\psi$.
In the split case, we can rewrite the diagram above
as: 
\begin{equation*}
\begin{tikzcd}[column sep={2cm}]
FHQ_{n-1}L
\ar{r}{FH\iota_n L}
\ar[equals]{d}
&
FHRFL
\ar{rr}{((FR \trace - \trace FL)\circ F\gamma_1R)FL}
&&
FRFL
\ar{r}{F\mu_n L}
&
FH^nL
\ar[dashed]{d}
\\
FHQ_{n-1}L
\ar{r}{FH\iota_n L}
&
FHRFL
\ar{rr}{FHR((\action FL - FL \action) \circ F\gamma'_1L)}
&
&
FHRFH'L
\ar{r}{FH\mu_n H'L}
&
FHH^nH'L
\end{tikzcd}
\end{equation*}
The summand of the top contour which corresponds to $ - \trace FR$ in
its second composant equals zero since $\mu_n \circ \iota_n=0$.
Similarly, the summand of the bottom contour which corresponds 
to $ - FL\action$ is also zero. Without those summands the  
corresponding maps become $F(A_l)L$, and $F (HB) L$, 
using the terms introduced in
\S\ref{section-strong-monad-condition-split-case}. 
It thus suffices to prove the existence of an isomorphism that makes 
the following diagram commute:
\begin{equation}
\label{eqn-highest-degree-term-condition-split-case-mult-comult}
\begin{tikzcd}[column sep={2cm}]
HQ_{n-1}
\ar{r}{H\iota_n}
\ar[equals]{d}
&
HRF
\ar{r}{A_l}
&
RF
\ar{r}{\mu_n}
&
H^n
\ar[dashed]{d}
\\
HQ_{n-1}
\ar{r}{H\iota_n}
&
HRF
\ar{r}{HB}
&
HRFH'
\ar{r}{H\mu_n H'}
&
HH^nH'
\end{tikzcd}
\end{equation}
whence it becomes clear that the highest degree term condition
is about the left multiplication by $H$ and the right comultiplication by $H'$
having the same effect when projected onto the highest degree term. 

By Lemma \ref{lemma-mult-H-comult-H'-relation} 
we can rewrite the right comultiplication of $H'$ in terms of the right
multiplication by $H$. It 
then remains to show that the top and the bottom contour commute in:
\begin{equation}
\label{eqn-highest-degree-term-condition-as-Al-and-Ar-commuting-on-Hn}
\begin{tikzcd}[column sep={2cm}]
HQ_{n-1}
\ar{r}{H\iota_n}
&
HRF
\ar{r}{A_l}
\ar{d}{HRF\action}
&
RF
\ar{r}{ \mu_n}
&
H^n
\ar[dashed]{d}
\\
&
HRFHH'
\ar{r}{HA_rH'}
&
HRFH'
\ar{r}{H\mu_n H'}
&
HH^nH'.
\end{tikzcd}
\end{equation}
The strong monad condition implies the stronger monad condition, 
hence both $HRF \xrightarrow{\mu_n \circ A_l} H^n$ and $RFH
\xrightarrow{\mu_n \circ A_r} H^n$ are maps 
$$ H \oplus H^2 \oplus \dots \oplus H^n \oplus H^{n+1} \rightarrow H^n $$
which are $0$ on the components $H, \dots, H^{n-1}$ and $\id$ 
on the component $H^n$. Hence the natural isomorphism 
$$ H^n \xrightarrow{H^n \action} H^n H H' = H H^n H' $$
makes the square in the diagram above commute on all components of
$HRF$ except for $H^{n+1}$. Therefore it makes the whole diagram commute, 
since $H^{n+1}$ is the complement of $HQ_{n-1}$ in $HRF$. 
\end{proof}

\subsection{Segal's conjecture}
\label{section-segals-conjecture}

Corollary \ref{cor-the-strongest-monad-condition} allows us to prove a conjecture 
posed by Segal in \cite{Segal-AllAutoequivalencesAreSphericalTwists}. 
We briefly recall the setup. 
Let
$F\colon D(\A) \rightarrow D(\B)$ be a split $\mathbb{P}^n$-functor with 
\begin{equation}
RF \simeq \id_\A \oplus H \oplus \dots \oplus H^n
\end{equation}
where $H$ is an enhanced autoequivalence of $D(\A)$. Equip the 
bimodule
$$ \A_{H} \overset{\text{def}}{=} \A \oplus H^{r\barA}[-1] \in
\AmodA $$
with the structure of a degree $1$ truncated tensor algebra of $H^{r\barA}[-1]$ 
as per \S\ref{section-truncated-twisted-tensor-algebras}. Explicitly, 
the multiplication 
$$ \A_{H} \otimes \A_{H} \rightarrow \A_{H} $$ 
is given by the natural isomorphisms  
\begin{align}
\label{eqn-natural-isomorphisms-tensoring-with-diagonal-bimodule}
A \otimes_\A (-) \xrightarrow{\sim} (-) 
\quad \text{ and } \quad
(-) \otimes_\A \A \xrightarrow{\sim} (-)
\end{align}
on the components involving $\A$ and by the zero map on 
$H^{r\barA} \otimes_\A H^{r\barA} [-2]$. The
unit map 
$$ \A \rightarrow \A_{H} $$
is given by the inclusion of a direct summand. 

Any $\A$-algebra in $\AmodA$ defines the structure of a
new DG-category with the same set of objects as $\A$ and a functor
from $\A$ to this category which is identity on objects. 
For $\A_H$, we denote this new category also by $\A_H$, and 
the functor $\A \hookrightarrow \A_H$ is faithful.  
The category $D_c(\A_H)$ is the category
constructed by Segal in
\cite[\S2.3]{Segal-AllAutoequivalencesAreSphericalTwists}, it should 
be considered as the derived category of objects supported on 
the zero section of a noncommutative line bundle over $D_c(\A)$
corresponding to $H$. 

Now, let
$$ h = FH \hookrightarrow FRF \xrightarrow{\trace F} F \in \AmodbarA $$
and let $\tilde{F}$ be the convolution in $\AmodbarB$ of the following
twisted complex over $\AmodbarB$
$$ H \bartimes_\A F \xrightarrow{h} \underset{\degzero}{F}. $$
We next give it the structure of a left $\A_H$-module by having
$H^{r\barA}$ act by sending $H \bartimes_\A F$ to $F$, and $F$ to zero. To be 
more precise, consider the map of twisted complexes
\begin{equation}
\label{eqn-AH-module-structure-on-tildeF-via-twisted-complexes}
\left(\A \oplus H^{r\barA}[-1]\right) \bartimes_\A 
\left(H \bartimes_\A F \xrightarrow{h} \underset{\degzero}{F}\right) 
\rightarrow 
\left(H \bartimes_\A F \xrightarrow{h} \underset{\degzero}{F}\right) 
\end{equation}
whose $\A \bartimes_\A \bullet$ components are given 
by the natural isomorphisms 
\eqref{eqn-natural-isomorphisms-tensoring-with-diagonal-bimodule}, 
whose $H^{r\barA}[-1] \bartimes_\A H \bartimes_\A F$ component is
given by the adjunction counit $H^{r\barA} \bartimes_\A H \xrightarrow{\trace} \A$, 
and whose $H^{r\barA}[-1] \bartimes_\A F$ component is zero. 
Note that \eqref{eqn-AH-module-structure-on-tildeF-via-twisted-complexes}
is fibrewise left-$\modA$-map in the language of 
\cite[\S3.4]{AnnoLogvinenko-BarCategoryOfModulesAndHomotopyAdjunctionForTensorFunctors},
i.e. all its components are the $\AmodbarB$ maps whose corresponding
$\AmodB$ maps filter through $\tau \otimes \id$.  
The convolution functor $\pretriag \AmodbarB \rightarrow \AmodbarB$ 
filters through $\AmodB$ and from its definition in
\cite[\S3.6]{AnnoLogvinenko-BarCategoryOfModulesAndHomotopyAdjunctionForTensorFunctors}
we see that applying it to  
\eqref{eqn-AH-module-structure-on-tildeF-via-twisted-complexes}
yields an $\AmodB$ map which is a composition 
$$ \barA \otimes_\A \A_H \otimes_\A \tilde{F} 
\xrightarrow{\tau \otimes \id}
\A_H \otimes_\A \tilde{F} \rightarrow \tilde{F},$$
whose second composant we use to give $\tilde{F}$ the structure of an
$\A_H$-$\B$-bimodule. 

On the other hand, 
let $\tilde{H}$ be the convolution in $\AmodbarA$ of the twisted complex 
\begin{equation}
\begin{tikzcd}[column sep = 2cm]
\underset{\degzero}{H}
\ar[phantom]{r}{\oplus}
& 
H \bartimes_\A H^{r\barA}. 
\end{tikzcd}
\end{equation}
Similarly to the above, we give $\tilde{H}$ the structure of an
$\A_H$-$\A_H$-bimodule by having $H^{r\barA}[-1]$ act on the right 
by sending $H$ to $H \bartimes_\A H^{r\barA}$ and 
sending $H \bartimes_\A H^{r\barA}$ to zero, and act on the left by sending $H$
to $H \bartimes_\A H^{r\barA}$ by $H^{r\barA} \bartimes_\A H
\xrightarrow{\trace} \id_\A \xrightarrow{\action} H \bartimes_\A H^{r\barA}$
and sending $H \bartimes_\A H^{r\barA}$ to zero. 

We thus have an enhanced functor $\tilde{F}: D(\A_H) \rightarrow D(\B)$ 
and an enhanced endofunctor $\tilde{H}$ of $D(\A_H)$. The latter 
is an autoequivalence --- its inverse is defined by the convolution 
of the $\AmodbarA$ twisted complex 
\begin{equation}
\begin{tikzcd}[column sep = 2cm]
\underset{\degzero}{H^{r\barA}}
\ar[phantom]{r}{\oplus}
& 
H^{r\barA} \bartimes_\A H^{r\barA}, 
\end{tikzcd}
\end{equation}
with $H^{r\barA}[-1]$ acting on the left and on the right by sending
$H^{r\barA}$ to $H^{r\barA} \bartimes_\A H^{r\barA}$ and 
$H^{r\barA} \bartimes_\A H^{r\barA}$ to zero. 

The following was conjectured by Segal in 
\cite[Remark 4.6]{Segal-AllAutoequivalencesAreSphericalTwists}:
\begin{theorem} 
\label{theorem-segals-conjecture}
If $F\colon D(\A) \rightarrow D(\B)$ is a split 
$\mathbb{P}^n$-functor satisfying the strong monad condition, 
then  
$$\tilde{F}\colon D(\A_H) \rightarrow D(\B)$$ 
is a spherical functor with the cotwist $\tilde{H}^{n+1}$ and the
twist $P_F$, the $\mathbb{P}^n$-twist of $F$. 
\end{theorem}
\begin{proof}
Let $\tilde{R} \in \AmodbarA$ be the convolution of the twisted complex 
$$ \underset{\degzero}{R} \xrightarrow{h'} H^{r\barA}R, $$
where 
$$ h' = R \xrightarrow{\action R} H^{r\barA} H R 
\hookrightarrow H^{r\barA}RFR 
\xrightarrow{\H^{r\barA}R\trace} H^{r\barA} R. $$
Similarly to the above, we give $\tilde{R}$ the structure of 
a right $\A_H$-module by having $H^{r\barA}$ act by sending $R$ to
$H^{r\barA}R$, and $H^{r\barA}R$ to zero. Thus we have an 
enhanced functor $\tilde{R}\colon D(\B) \rightarrow D(\A_H)$. 

It can readily seen that $\tilde{R}$ is a $2$-categorical right
adjoint of $\tilde{F}$ with the following adjunction unit and counit. 
The composition $\tilde{R}\tilde{F}$ is isomorphic in 
$D(\A_H\text{-}\A_H)$ to the object defined by the $\AbimA$-bimodule  
\begin{align}
\label{eqn-adjunction-monad-tildeR-tildeF-as-twisted-complex}
\left\{
H \bartimes_\A F \xrightarrow{h} \underset{\degzero}{F}
\right\} 
\bartimes_\B
\left\{\underset{\degzero}{R} \xrightarrow{h'} R \bartimes_\A H^{r\barA} \right\} 
\end{align}
with left and right actions of $H^{r\barA}[-1]$ as the in definitions of
$\tilde{F}$ and $\tilde{R}$. By \cite[Lemma
3.42(1)]{AnnoLogvinenko-BarCategoryOfModulesAndHomotopyAdjunctionForTensorFunctors}
it is therefore isomorphic to the convolution of the twisted complex 
\begin{align*}
H \bartimes_{\A} F \bartimes_{\B} R 
\xrightarrow{
\left(\begin{smallmatrix}
h \bartimes \id \\
- \id \bartimes h'
\end{smallmatrix}
\right)
}
{\bigl(F \bartimes_{\B} R \bigr)}
\underset{\degzero}{\oplus}
{\bigl(H \bartimes_{\A} F \bartimes_{\B} R \bartimes_{\A} H^{r\barA}\bigr)}
\xrightarrow{
\left(\begin{smallmatrix}
\id \bartimes 
h' \; & \;
h \bartimes \id
\end{smallmatrix}
\right)
}
F \bartimes_{\B} R \bartimes_{\A} H^{r\barA}
\end{align*}
with the corresponding left and and right actions of $H^{r\barA}[-1]$. 
We therefore define the adjunction unit 
$$ \id_{\A_H} \xrightarrow{\action} \tilde{R}\tilde{F} $$
by the following map of twisted complexes over $\AmodbarA$ 
\begin{equation}
\label{eqn-adjunction-unit-for-tildeF-and-tildeR-via-twisted-complexes}
\begin{tikzcd}
&
\id_\A 
\ar{d}[']{
\left(\begin{smallmatrix}
\action 
\\
(\id \bartimes \action \bartimes \id) \circ \action  
\end{smallmatrix}\right)
}
\ar[phantom]{r}{\oplus}
\ar[dotted]{dr}[description]{\xi}
&
H^{r\barA}
\ar{d}{\action \otimes \id}
\\
H \bartimes_{\A} F \bartimes_{\B} R 
\ar{r}
&
{\bigl(F \bartimes_{\B} R \bigr)}
\underset{\degzero}{\oplus}
{\bigl(H \bartimes_{\A} F \bartimes_{\B} R \bartimes_{\A} H^{r\barA}\bigr)}
\ar{r}
&
F \bartimes_{\B} R \bartimes_{\A} H^{r\barA}. 
\end{tikzcd}
\end{equation}
Here $\xi$ is a degree $-1$ map defined as in 
\cite[Prop.4.10
(2)]{AnnoLogvinenko-BarCategoryOfModulesAndHomotopyAdjunctionForTensorFunctors}, 
and it is irrelevant to this proof. The map 
\eqref{eqn-adjunction-unit-for-tildeF-and-tildeR-via-twisted-complexes}
commutes with the left and right actions of $H^{r\barA}[-1]$
defined above, and thus defines a map of $\A_H$-$\A_H$-bimodules.  

On the other hand, the composition $\tilde{F}\tilde{R}$ is isomorphic in 
$D(\B\text{-}\B)$ to the object 
\begin{align*}
\left\{\underset{\degzero}{R} \xrightarrow{h'} R \bartimes_{\A} 
H^{r\barA} \right\} 
\bartimes_{\A_H}
\left\{
H \bartimes_\A F \xrightarrow{h} \underset{\degzero}{F}
\right\}. 
\end{align*}
The tensor product over $\A_H$ can be computed as the cokernel of 
the natural map 
\begin{align*}
\left\{\underset{\degzero}{R} \xrightarrow{h'} R \bartimes_\A H^{r\barA} \right\} 
\bartimes_{\A} 
H^{r\barA}[-1]
\bartimes_{\A}
\left\{
H \bartimes_\A F \xrightarrow{h} \underset{\degzero}{F}
\right\} 
\rightarrow 
\left\{\underset{\degzero}{R} \xrightarrow{h'} R \bartimes_\A H^{r\barA} \right\} 
\bartimes_{\A} 
\left\{
H \bartimes_\A F \xrightarrow{h} \underset{\degzero}{F}
\right\}
\end{align*}
given by the difference between the actions of
$H^{r\barA}[-1]$ on $R \xrightarrow{h'} R \bartimes_\A H^{r\barA}$ 
and $H \bartimes_\A F \xrightarrow{h} F$, respectively. Thus 
$\tilde{F}\tilde{R}$ is isomorphic in $D(\B\text{-}\B)$ to the
convolution of the cokernel of the map 
\begin{scriptsize}
\begin{equation*}
\begin{tikzcd}[column sep={1cm}]
&
R \bartimes_\A (H^{r\barA}) \bartimes_\A H \bartimes_\A F
\ar{r}
\ar{d}{
\left(\begin{smallmatrix}
\id \bartimes \trace \bartimes \id
\\
\id
\end{smallmatrix}\right)
}
& 
\bigl(R \bartimes_\A (H^{r\barA}) \bartimes_\A F\bigr)
\underset{\degzero}{\oplus}
\bigl(R \bartimes_\A H^{r\barA} \bartimes_\A 
(H^{r\barA}) \bartimes_\A H \bartimes_\A F\bigr)
\ar{d}{\begin{pmatrix} 
\id
\amsamp
\id \bartimes \trace \bartimes \id^{\bartimes 2}
\end{pmatrix}
}
\\
R \bartimes_\A H \bartimes_\A F
\ar{r}{
\left(\begin{smallmatrix}
\id \bartimes h 
\\
h' \bartimes \id
\end{smallmatrix}\right)
}
& 
\bigl( R \bartimes_\A F \bigr)
\oplus
\bigl(R \bartimes_\A H^{r\barA} \bartimes_\A H \bartimes_\A F\bigr)
\ar{r}{\begin{pmatrix} 
h' \bartimes \id
\amsamp
- \id \bartimes h 
\end{pmatrix}
}
& 
R \bartimes_\A H^{r\barA} \bartimes_\A F. 
\end{tikzcd}
\end{equation*}
\end{scriptsize}
Since $\id \bartimes h - (\id \bartimes \trace \bartimes \id) \circ
(h' \bartimes \id)$ is the map 
$$
R \bartimes_\A H \bartimes_\A F
\hookrightarrow 
R \bartimes_\A F \bartimes_\B R \bartimes_\A F
\xrightarrow{\id \bartimes \trace - \trace \bartimes \id}
R \bartimes_\A F, 
$$
which is the map $-\psi$, we conclude that $\tilde{F}\tilde{R}$ is 
isomorphic in $D(\BbimB)$ to the convolution of the twisted complex 
$$ 
R \bartimes_\A H \bartimes_\A F
\xrightarrow{\psi}
R \underset{\degzero}{\bartimes_\A}F.
$$
We therefore define the adjunction counit 
$$ \tilde{F} \tilde{R} \xrightarrow{\trace} \id_\B $$
by the following map of twisted complexes over $\BmodbarB$:
\begin{equation}
\label{eqn-adjunction-counit-for-tildeF-and-tildeR-via-twisted-complexes}
\begin{tikzcd}
R \bartimes_\A H \bartimes_\A F
\ar{r}{\psi}
& 
R \bartimes_\A F
\ar{d}{\trace}
\\
& 
\underset{\degzero}{\id_\B}. 
\end{tikzcd}
\end{equation}

With these definitions in mind, we proceed to prove the assertions of
the theorem. The cone of $\tilde{F}\tilde{R} \xrightarrow{\trace}
\id_\B$ is isomorphic
to the convolution of the total complex of
\eqref{eqn-adjunction-counit-for-tildeF-and-tildeR-via-twisted-complexes} 
which is the twisted complex defining the $\mathbb{P}$-twist
of $F$. Hence the spherical twist of $\tilde{F}$ 
coincides with the $\mathbb{P}$-twist of $F$, and thus is an
autoequivalence. 

On the other hand, computing the bar tensor product of the 
twisted complex of $\tilde{H}$ with itself over $\A_H$
shows that $\tilde{H}^{n+1}$ is 
homotopy equivalent to the convolution of the $\AmodbarA$ twisted 
complex
\begin{equation}
\begin{tikzcd}[column sep = 2cm]
\underset{\degzero}{H^{\bartimes k}}
\ar[phantom]{r}{\oplus}
& 
H^{\bartimes k} \bartimes_\A H^{r\barA}, 
\end{tikzcd}
\end{equation}
with the left and the right actions of $H^{r\barA}[-1]$ as in the 
definition of $\tilde{H}$. We now claim that there exists a degree $-1$ map 
$\chi\colon RFH \rightarrow H^{r\barA} H^{n+1}$ such that the following is a closed
degree $0$ map of twisted complexes:
\begin{equation}
\label{eqn-tildeR-tildeF-to-tilde-Hn+1-as-twisted-complexes}
\begin{tikzcd}[column sep=2cm]
H \bartimes_{\A} F \bartimes_{\B} R 
\ar{r}{
\left(\begin{smallmatrix}
h \bartimes \id \\
- \id \bartimes h'
\end{smallmatrix}
\right)
}
\ar{d}{\id \bartimes \mu_n}
\ar[dotted]{dr}[description]{\chi}
&
{\bigl(F \bartimes_{\B} R \bigr)}
\underset{\degzero}{\oplus}
{\bigl(H \bartimes_{\A} F \bartimes_{\B} R \bartimes_{\A} H^{r\barA}\bigr)}
\ar{r}{
\left(\begin{smallmatrix}
\id \bartimes 
h' \; \amsamp \;
h \bartimes \id
\end{smallmatrix}
\right)
}
\ar{d}{
\begin{pmatrix}
(\id \bartimes \action) \circ \mu_n
\;\amsamp\; \id \bartimes \mu_n \bartimes \id
\end{pmatrix}
}
&
F \bartimes_{\B} R \bartimes_{\A} H^{r\barA}
\\
H^{n+1}
\ar[phantom]{r}{\oplus}
&
H^{n+1} \underset{\degzero}{\bartimes_\A} H^{r\barA}. 
&
\end{tikzcd}
\end{equation}
The claim is equivalent to the following diagram commuting in
$D(\AbimA)$:
\begin{equation}
\label{eqn-right-action-analogue-of-the-highest-degree-term-condition}
\begin{tikzcd}[column sep = 1.5cm]
RFH 
\ar[equals]{d}
\ar{r}{Rh}
&
RF
\ar{r}{\mu_n}
&
H^{n} 
\ar{d}{\action H^n}
\\
RFH
\ar{r}{h'FH}
&
H^{r\barA}RFH
\ar{r}{H^{r\barA} \mu_n H}
&
H^{r\barA}H^{n+1}.
\end{tikzcd}
\end{equation}
The map $Rh\colon RFH \rightarrow RF$ is  
the right monad multiplication $A_r$ defined 
in \S\ref{section-strong-monad-condition-split-case},
while $h'F$ is the composition 
$$ RF 
\xrightarrow{\action RF} 
H^{r\barA}HRF
\xrightarrow{H^{r\barA}A_l} 
H^{r\barA}RF, 
$$
where $HRF \xrightarrow{A_l} RF$ is the left monad multiplication.
We can therefore ensure the commutativity of the diagram 
\eqref{eqn-right-action-analogue-of-the-highest-degree-term-condition}
by renormalising the isomorphism $\gamma$ as per Corrolary 
\ref{cor-the-strongest-monad-condition} to have $A_r = A_l$. 

The map \eqref{eqn-tildeR-tildeF-to-tilde-Hn+1-as-twisted-complexes}
commutes with the left and right actions of $H^{r\barA}[-1]$ on its two
twisted complexes. It therefore defines an $\A_H$-$\modbar$-$\A_H$ map 
\begin{equation}
\label{eqn-tildeR-tildeF-to-tilde-Hn+1}
\alpha: \tilde{R} \tilde{F} \rightarrow \tilde{H}^{n+1}[1]. 
\end{equation}
Observe that the twisted complex maps 
\eqref{eqn-adjunction-unit-for-tildeF-and-tildeR-via-twisted-complexes}
and \eqref{eqn-tildeR-tildeF-to-tilde-Hn+1-as-twisted-complexes}
compose to zero. Therefore, they form a two-step twisted complex of twisted 
complexes. For simplicity, we rewrite in the functorial notation:
\begin{equation}
\label{eqn-id-to-tildeRtildeF-to-tildeH^n+1-twisted-complex}
\begin{tikzcd}[column sep=4.70cm]
& 
\id_\A
\ar[phantom]{r}{\oplus}
\ar{d}[']{
\begin{pmatrix}
\action \\ H^{r\barA}{\action}H \circ \action
\end{pmatrix}
}
\ar[dotted]{dr}[description]{\xi}
&
H^{r\barA}
\ar{d}{{\action}H^{r\barA}}
\\
RFH
\ar{r}{
\left(\begin{smallmatrix}
A_r \\
- H^{r\barA}A_l H \circ {\action}RFH
\end{smallmatrix}
\right)
}
\ar{d}{\mu_n H}
\ar[dotted]{dr}[description]{\chi}
&
RF
\oplus
H^{r\barA} RF H
\ar{r}{
\left(\begin{smallmatrix}
H^{r\barA}A_l \circ {\action}RF
\;\amsamp\;
H^{r\barA}A_r
\end{smallmatrix}
\right)
}
\ar{d}{
\begin{pmatrix}
{\action}H^n \circ \mu_n
\;\amsamp\; H^{r\barA} \mu_n H
\end{pmatrix}
}
&
H^{r\barA} RF
\\
H^{n+1}
\ar[phantom]{r}{\oplus}
&
\underset{\degzero}{H^{r\barA} H^{n+1}}. 
&
\end{tikzcd}
\end{equation}
We claim that the convolution  
of \eqref{eqn-id-to-tildeRtildeF-to-tildeH^n+1-twisted-complex} is
acyclic in $\AmodbarA$ and hence in $\A_H$-$\modbar$-$\A_H$. 
Then the cone of $\id_{\A_H} \xrightarrow{\action}
\tilde{R}\tilde{F}$ is isomorphic to $\tilde{H}^{n+1}[1]$ in
$D(\A_H\text{-}\A_H)$. Thus 
$\tilde{F}$ has the spherical twist $P_F$ and the spherical cotwist 
$\tilde{H}^{n+1}$. Since both of these are autoequivalences, 
$\tilde{F}$ is spherical by 
\cite[Theorem 5.1]{AnnoLogvinenko-SphericalDGFunctors}. 

For the claim, apply 
the Replacement Lemma (Lemma \ref{lemma-replacement-lemma})
to replace the vertical arrows $RFH \rightarrow H^{n+1}$, 
$H^{r\barA}RFH \rightarrow H^{r\barA}H^{n+1}$, 
$H^{r\barA} \rightarrow H^{r\barA}RF$, and $\id_\A \rightarrow RF$, 
in that order, by their cones. Since $RF$ splits up as $\id \oplus H
\oplus \dots \oplus H^n$ identifying these four maps with direct
summand inclusions/projections, their cones are simply the direct
sums of the remaining summands. We thus obtain a twisted complex
\begin{equation*}
\label{eqn-id-to-tildeRtildeF-to-tildeH^n+1-twisted-complex-simplified}
\begin{tikzcd}[column sep=2cm]
H \oplus \dots \oplus H^n
\ar[dotted, bend right=10]{rr}[description]{?}
\ar{r}{
\left(\begin{smallmatrix}
A_r \\
?
\end{smallmatrix}
\right)
}
&
\left(H \oplus \dots \oplus H^n\right)
\underset{\degzero}{\oplus}
H^{r\barA} \left(H \oplus \dots \oplus H^n\right)
\ar{r}{
\left(\begin{smallmatrix}
?
\;\amsamp\;
H^{r\barA}A_r
\end{smallmatrix}
\right)
}
&
H^{r\barA}(H \oplus \dots \oplus H^n), 
\end{tikzcd}
\end{equation*} 
where the question signs denote the maps which are irrelevant to this argument. 
By the stronger monad condition established  
in Lemma \ref{lemma-split-case-strong-monad-iff-stronger-monad} the
following map is a homotopy equivalence: 
$$H \oplus \dots \oplus H^n \xrightarrow{A_r} 
H \oplus \dots \oplus H^n. $$
We can therefore remove the corresponding null-homotopic subcomplex, 
leaving us with the complex
\begin{align*}
H^{r\barA} (H \oplus \dots \oplus H^n) 
& \xrightarrow{H^{r\barA}A_r}  
H^{r\barA} (H \oplus \dots \oplus H^n) 
\end{align*}
which is null-homotopic as well. We conclude that 
\eqref{eqn-id-to-tildeRtildeF-to-tildeH^n+1-twisted-complex} is 
homotopy equivalent to zero, as desired.   
\end{proof}

\subsection{The general case}
\label{section-strong-monad-condition-the-general-case}

Throughout this section we adopt the following shorthand where it
doesn't cause any confusion. Given a cyclic degree 
$n$ coextension $Q_n$ of $\id$ by $H$ as on
\eqref{eqn-cyclic-coextension-of-id-by-H-of-degree-n} we denote by
$$ Q_i \xrightarrow{\iota} Q_{j} \quad \quad j \geq i $$
the compositions of the corresponding coextension maps:
$$ Q_i \xrightarrow{\iota_i} Q_{i+1}
\xrightarrow{\iota_{i+1}} \dots 
\xrightarrow{\iota_j} Q_j. $$

The weak adjoints condition of 
\cite{Addington-NewDerivedSymmetriesOfSomeHyperkaehlerVarieties}
applies equally well in the general (non-split) case. 
The strong monad condition doesn't, but it has the following
natural analogue. Observe that in the split case 
$$ Q_i \simeq \id \oplus \dots \oplus H^i, $$
and the condition that $HRF \xrightarrow{A_l} RF$ has 
no components of form $H H^i \rightarrow H^j$ with $i + 1 < j$ is equivalent
to
$$ HQ_i \xrightarrow{\iota} HQ_n \xrightarrow{A_l} Q_n $$
filtering through $Q_{i+1} \xrightarrow{\iota} Q_n$.
This still has the problem that the map 
$A_l$ doesn't exist in the non-split case, but it is further equivalent to 
$$ Q_1 Q_i \xrightarrow{\iota\;\iota} Q_n Q_n \xrightarrow{
\text{monad mult. }} Q_n $$
filtering through $Q_{i+1} \rightarrow Q_n$ as some map 
$$ m_{1,i}\colon Q_1 Q_i \rightarrow Q_{i+1}.$$ 
In this form, the condition makes sense in the fully general case. 
The other condition of the components $H H^{i} \rightarrow H^{i+1}$ of
$HRF \xrightarrow{A_l} RF$ being isomorphisms is then equivalent
to there existing isomorphisms  $H H^{i+1} \rightarrow H^i$
which intertwine with the maps $m_{1,i}$ via the highest degree term 
projections $\mu_{\bullet}$. The stronger monad condition 
can be reformulated in a similar way. 
We thus arrive at:
\begin{defn}
\label{defn-strong-monad-condition-general-case}
Let $\A$ and $\B$ be DG-categories and let $F \in D(\AbimB)$ be
an enhanced functor $D(\A) \rightarrow D(\B)$. 
Let \rm $(H, Q_n, \gamma)$ be a collection of 
\begin{enumerate}
\item $H \in D(\AbimA)$ such that $h = (-) \ldertimes_\A H$ is 
an autoequivalence of $D(\A)$ and $H(\krn F)=\krn F$.
\item A cyclic coextension $Q_n \in D(\AbimA)$ of $\id$ by $H$. 
\item An isomorphism $Q_n \xrightarrow{\gamma} RF$ in $D(\AbimA)$
which intertwines the adjunction unit 
$\id_\A \xrightarrow{\action} RF$ and the natural map 
$\id_\A \xrightarrow{\iota} Q_n$.
\end{enumerate}
Let
$$ m: Q_n Q_n \rightarrow Q_n $$
be the map induced from the monad multiplication on $RF$. 
We define following three conditions:
\begin{itemize}
\item \em Stronger monad condition: \rm 
For $i, j \geq 0$ with $i + j \leq n$ the map 
\begin{align}
\label{eqn-the-map-Q_i-Q_j-to-Q_n}
Q_i Q_j \xrightarrow{\iota \iota} Q_nQ_n \xrightarrow{m}  Q_n 
\end{align}
filters through $Q_{i+j} \xrightarrow{\iota} Q_n$ via some map 
$$ m_{ij}\colon Q_iQ_j \rightarrow  Q_{i+j}. $$ 
Moreover, there is an isomorphism 
$$\rho_{ij}\colon H^iH^j \xrightarrow{\quad \sim \quad} H^{i+j}$$ 
that makes the following diagram commute:
\begin{equation}
\label{eqn-diagonal-isomorphisms-in-strong-monad-condition-square}
\begin{tikzcd}[column sep={2cm}]
Q_i Q_j
\ar{r}{m_{ij}}
\ar{d}{\mu_i\mu_j}
&
Q_{i+j}
\ar{d}{\mu_{i+j}}
\\
H^iH^j
\ar{r}{\rho_{ij}}
&
H^{i+j}.
\end{tikzcd}
\end{equation}
\item \em Strong monad condition: \rm The stronger monad condition 
with $i = 1$. 
\item \em Weak adjoints condition: \rm 
There exists \em some \rm isomorphism $FR \simeq FH^nL$.
\end{itemize}
\end{defn}

In practice, it suffices to only check strong or stronger 
monad conditions for $i, j \geq 1$:

\begin{lemma}
\label{lemma-maps-m_i0-and-m_0j-are-identity}
The compositions 
$$ Q_i Q_0 \xrightarrow{\iota \iota} Q_n Q_n \xrightarrow{m} Q_n $$
$$ Q_0 Q_i \xrightarrow{\iota \iota} Q_n Q_n \xrightarrow{m} Q_n $$
decompose into canonical isomorphisms $Q_i \id_\A \simeq Q_i$ and 
$\id_\A Q_i \simeq Q_i$ followed by the map 
$Q_i \xrightarrow{\iota} Q_n$.  
\end{lemma}
\begin{proof}
By one of our assumptions in Definition \ref{defn-Pn-functor} the map 
$$ Q_0 = \id_{\A} \xrightarrow{\iota} Q_n $$ 
in $D(\AbimA)$ is intertwined by the isomorphism $\gamma$ with 
the adjunction unit $\id_\A \to RF$. Therefore  
$$ Q_nQ_0 \xrightarrow{\iota \iota} Q_n Q_n \xrightarrow{m} Q_n $$ 
is the canonical isomorphism 
$Q_n \id_\A \xrightarrow{\sim} Q_n$.  
Hence the composition  
$$ Q_i Q_0 \xrightarrow{\iota \iota} Q_nQ_n \xrightarrow{m} Q_n $$
equals 
$$ Q_i Q_0 \xrightarrow{\iota Q_0} Q_n Q_0 \xrightarrow{\sim} Q_n $$
and by functoriality of the canonical isomorphism further equals 
$$ Q_i Q_0 \xrightarrow{\sim} Q_i \xrightarrow{\iota} Q_n. $$
The case $i = 0$ is treated similarly. 
\end{proof}

Thus when $i = 0$ or $j = 0$ the map \eqref{eqn-the-map-Q_i-Q_j-to-Q_n} 
always filters through $Q_{i+j} \xrightarrow{\iota} Q_n$, and $m_{ij}$
and $\rho_{ij}$ can be taken to be the identity map. 
Moreover, as in the split case, we have:

\begin{lemma} The strong monad condition implies the monad condition
in Definition \ref{defn-Pn-functor}. 
\end{lemma}
\begin{proof}
Let $\phi: FH \to FQ_1$ be any splitting such that $(F\mu_1)\circ\phi=\id$. 
By Lemma \ref{lemma-psi-is-independent-of-the-choice-of-phi}
the map $\nu$ in the monad condition \eqref{eqn-monad-condition} 
equals the composition
$$
FHQ_{n-1}
\xrightarrow{FH \iota_n}
FHRF
\xrightarrow{\phi RF}
FQ_1RF
\xrightarrow{F\iota RF}
FRFRF
\xrightarrow{FR \trace F - \trace FRF}
FRF
\xrightarrow{F\kappa}
FJ_n.
$$
$FHQ_{n-1}$ has a filtration by $FHQ_i$ and 
$FJ_n$ has a filtration by $FJ_{i+1}$ for $i=0,\ldots, n-1$. 
Define 
$$ \nu_i\colon FHQ_i \to FJ_{i+1} $$
to be the difference of two maps: the composition of the top row
in the commutative diagram
\begin{equation*}
\begin{tikzcd}[column sep={2cm}]
FHQ_{i}
\ar{r}{\phi Q_i}
\ar{d}{FH\mu_i}
&
FQ_1Q_i
\ar{r}{F m_{1i}}
\ar[']{d}{F\mu_1\mu_i}
&
FQ_{i+1}
\ar{r}{F\kappa_{i+1}}
\ar{d}{F\mu_{i+1}}
&
FJ_{i+1}
\ar{d}{F\mu_{i+1}}
\\
FHH^i
\ar[equals]{r}
&
FHH^i
\ar{r}{\rho_{ij}}[']{\sim}
&
FH^{i+1}
\ar[equals]{r}
&
FH^{i+1}
\end{tikzcd}
\end{equation*}
and the composition
\begin{equation}
FHQ_i
\xrightarrow{\phi Q_i}
FQ_1Q_i
\xrightarrow{}
FRFQ_i
\xrightarrow{\trace FQ_i}
FQ_i
\xrightarrow{F\kappa_i}
FJ_i
\xrightarrow{F\iota_{i+1}}
FJ_{i+1}.
\end{equation}
The maps $\nu_i$ are compatible with each other and induce isomorphisms on the factors of the filtration, 
i.e. the following two squares commute:
\begin{equation*}
\begin{tikzcd}[column sep={2cm}]
FHQ_i
\ar{r}{\nu_i}
\ar{d}[']{FH\iota_{i+1}}
&
FJ_{i+1}
\ar{d}{F\iota_{i+2}}
\\
FHQ_{i+1}
\ar{r}{\nu_{i+1}}
&
FJ_{i+2}
\end{tikzcd}
\qquad\qquad
\begin{tikzcd}[column sep={2cm}]
FHQ_i
\ar{r}{\nu_i}
\ar[']{d}{FH\mu_{i}}
&
FJ_{i+1}
\ar{d}{F\mu_{i+1}}
\\
FHH^i
\ar{r}{F\rho_{1i}}[']{\sim}
&
FH^{i+1}.
\end{tikzcd}
\end{equation*}

The map $\nu_0\colon FH \to FH$ can be computed explicitly: 
since the multiplication map $Q_1Q_0\simeq Q_1 \to Q_1$
is the identity map and the second component is zero since $J_0=0$, we have $\nu_0=\id$.
Then by induction all $\nu_i$, including $\nu=\nu_{n-1}$, are isomorphisms.
\end{proof}

Unfortunately to have the strong monad condition and 
the weak adjoints condition imply, as in the split case, 
the adjoints condition and the highest degree term condition 
we need to make an extra assumption. 
The filtrations do not offer as much control as direct sums 
and try as they did the authors couldn't prove the results below
without assuming:
\begin{defn}
With $\A$ and $H$ as in
Definition \ref{defn-strong-monad-condition-general-case}  
we define the \em $\ext^{-1}$-vanishing condition \rm to be:
\begin{align}
\label{eqn-the-minus-one-ext-assumption}
\homm^{-1}_{D(\AbimA)}(\id, H^i) = 0 \quad \quad 1 \leq i \leq n. 
\end{align}
\end{defn}

We first introduce some notation. Recall that, as explained in
\S\ref{section-cyclic-extensions}, we have fixed 
for our given cyclic coextension $Q_n$
a twisted complex $\bar{Q}_n$ over $\AmodbarA$ whose convolution 
is $Q_n$. The intermediate coextensions $Q_i$ are then the
convolutions of the subcomplexes of $\bar{Q}_n$ which sit 
in the degrees from $-i$ to $0$. 

\begin{defn}
For any $i \in [0,n]$ let $\bar{K}_{i+1}$ be the subcomplex
of $\bar{Q}_n$ which sits in the degrees
from $-n$ to $-(i+1)$.  Denote by $\pi_{i+1}$, or just $\pi$ where it doesn't
cause confusion, the natural projection 
$\bar{Q}_n \twoheadrightarrow K_{i+1}$.
Let $K_{i+1}$ be the convolution of $\bar{K}_{i+1}$, we 
then have the following exact triangle in $D(\AbimA)$: 
\begin{equation}
\label{eqn-Q_i-Q_n-K_i+1-exact-triangle}
Q_i \xrightarrow{\iota} Q_n \xrightarrow{\pi} K_{i+1} \rightarrow Q_i[1]. 
\end{equation}

For any $j > i$ denote further by $\pi_{i,j}$, or just $\pi$ where it
doesn't cause confusion, the natural projection 
from $\bar{K}_i$ to $\bar{K}_j$ with $j > i$. 
\end{defn}

\begin{lemma}
\label{lemma-strong-monad-condition-one-sided}
Assume that $\ext^{-1}$-vanishing condition 
\eqref{eqn-the-minus-one-ext-assumption} holds. 
Then the strong monad condition implies that for all $i,j$ 
the following map filters through $Q_{i+j}$:
$$ Q_i Q_j \xrightarrow{\iota \iota}  Q_nQ_n \xrightarrow{m} Q_n. $$ 
\end{lemma}
\begin{proof}
We prove the assertion of the Lemma by the induction on $i$ and $j$. The base
of the induction are the cases where $i = 0$ or $j = 0$. 

By the exact triangle \eqref{eqn-Q_i-Q_n-K_i+1-exact-triangle}  
to establish that the map 
$$ Q_p Q_q \xrightarrow{ \iota \iota } Q_n Q_n \xrightarrow{m} Q_n $$
filters through $Q_{p+q}$ it suffices to prove that 
the composition 
\begin{equation}
\label{eqn-Q_i-Q_j-to-K-i+j+1}
Q_p Q_q \xrightarrow{\iota \iota} Q_n Q_n \xrightarrow{m} Q_n
\xrightarrow{\pi} K_{p+q+1}
\end{equation}
vanishes. Suppose now that we have established our assertion for
all $p$ and $q$ with $p \leq i$ and $q \leq j$ except $p = i$, $q = j$.   
Let us establish it for for $p = i$ and $q = j$. 

Denote by $\bar{X}_{ij}$ the subcomplex of
$\bar{Q}_i\bar{Q}_j$ that sits in degrees $-i-j+1$ to $0$, and 
by $X_{ij}$ its convolution. We thus have an exact triangle 
$$ 
X_{ij} \rightarrow Q_iQ_j \xrightarrow{\mu_i \mu_j} H^iH^j \rightarrow  X_{ij}[1]. 
$$ 
By Lemma \ref{lemma-true-filtration-filtering-via-double-filtration-filtering}
the induction assumption that \eqref{eqn-Q_i-Q_j-to-K-i+j+1} vanishes 
for all $p$ and $q$ with
$p \leq i$ and $q \leq j$ except $p = i$, $q = j$
implies that the composition of $X_{ij} \rightarrow Q_iQ_j$ 
with \eqref{eqn-Q_i-Q_j-to-K-i+j+1} also vanishes.

Consider now the associativity diagram for the monad multiplication:
\begin{equation*}
\begin{tikzcd}[column sep={2cm}]
Q_1Q_{i-1}Q_j
\ar{r}{Q_1m_{i-1,j}}
\ar[']{d}{m_{1,i-1}Q_j}
&
Q_1Q_{i+j-1}
\ar{d}{m \circ \iota\iota}
&
\\
Q_iQ_j
\ar{r}{m \circ \iota\iota}
&
Q_n
\ar{r}{\pi}
&
K_{i+j+1}.
\end{tikzcd}
\end{equation*}
By the strong monad condition 
the composition along the top of the diagram is zero, 
therefore the composition along the bottom is zero too.
Similar to above, let $X_{1,i-1,j}$ be the convolution of 
the subcomplex of $\bar{Q}_1\bar{Q}_{i-1}\bar{Q}_j$ 
that sits in degrees $-i-j+1$ to $0$.
The following diagram commutes:
\begin{equation*}
\begin{tikzcd}[column sep={2cm}]
X_{1,i-1,j}
\ar{d}
&
X_{ij}
\ar{d}
\ar{rd}{0}
&
\\
Q_1Q_{i-1}Q_j
\ar{r}{m_{1,i-1}Q_j}
\ar[']{d}{\mu_1\mu_{i-1}\mu_j}
&
Q_iQ_j
\ar{r}
\ar{d}{\mu_i\mu_j}
&
K_{i+j+1}
\\
HH^{i-1}H^j
\ar{r}{\rho_{1,i-1}H^j}
&
H^iH^j
\ar[dashed]{ru}{\lambda}
&
\end{tikzcd}
\end{equation*}
Since its composition with $X_{ij} \rightarrow  Q_iQ_j$ vanishes,  
the map $Q_iQ_j \to K_{i+j+1}$ filters 
through $\mu_i \mu_j$ as some map $\lambda\colon 
H^iH^j \rightarrow K_{i+j+1}$. 
By commutativity of the diagram the following composition vanishes:
\begin{equation}
Q_1Q_{i-1}Q_j
\xrightarrow{\mu_1 \mu_{i-1}\mu_j}
HH^{i-1}H^j 
\xrightarrow{\rho_{1,i-1} H^j} 
H^iH^j 
\xrightarrow{\lambda}
K_{i+j+1}
\end{equation}
Hence, by exactness of the first column, 
the map $\lambda \circ \rho_{1,i-1} H^j$
filters through some map $X_{1,i-1,j}[1] \rightarrow K_{i+j+1}$.  
Therefore $\lambda \circ \rho_{1,i-1} H^j = 0$ , since 
by assumption $\homm^{-1}_{D(\AbimA)}(X_{1,i-1,j},K_{i+j+1})=0$. 
Since $\rho_{1,i-1}$ is an isomorphism, the map $\lambda$ is zero as
well, and hence so is $Q_iQ_j \to K_{i+j+1}$ as desired. 
\end{proof}

\begin{cor}
\label{cor-from-strong-monad-to-monad-multiplication-one-sided}
Suppose that $\ext^{-1}$-vanishing condition 
\eqref{eqn-the-minus-one-ext-assumption} 
and the strong monad condition hold. Then:
\begin{enumerate}
\item 
\label{item-monad-multiplication-is-DG-one-sided}
The monad multiplication map $m\colon Q_nQ_n\to Q_n$ in
$D(\AbimA)$ lifts to a one-sided map 
$$\bar{m}\colon \bar{Q}_n \bar{Q}_n \rightarrow \bar{Q}_n$$
of twisted complexes over $\AmodbarA$.  
\item 
\label{item-strong-monad-condition-to-stronger-monad-condition}
The stronger monad condition holds. 
The requisite maps $m_{ij}\colon Q_i Q_j \rightarrow Q_{i+j}$ 
and the isomorphisms $\rho_{ij}: H^iH^j \to H^{i+j}$
can be defined by the appropriate components of the map $\bar{m}$. 
Such $m_{ij}$ are, in particular, intertwined by 
the inclusion maps $\iota$.
\end{enumerate}
\end{cor}
\begin{proof}
The first assertion follows from applying the Lemma
\ref{lemma-true-filtration-filtering-via-double-filtration-filtering}
and then Lemma \ref{lemma-onesidedness-via-true-filtration}
to the statement established in Lemma
\ref{lemma-strong-monad-condition-one-sided}. 
The second assertion then follows in a straightforward way: 
the assertion that the maps $\rho_{ij}: H^iH^j \to H^{i+j}$
defined by the appropriate components of the map $\bar{m}$ are
isomorphisms is established by an induction on $i$ and $j$ analogous 
to that in the proof of Lemma \ref{lemma-strong-monad-condition-one-sided}. 
\end{proof}

\begin{lemma} 
\label{lemma-RF-isomorphic-to-HnLF}
Suppose that $\ext^{-1}$-vanishing condition 
\eqref{eqn-the-minus-one-ext-assumption} holds or
that on DG level the monad multiplication is one-sided as per
Corollary \ref{cor-from-strong-monad-to-monad-multiplication-one-sided}\eqref{item-monad-multiplication-is-DG-one-sided}. 
The strong monad condition implies that the canonical map 
$$ RF \xrightarrow{\eqref{eqn-adjoints-condition}F} H^nLF $$ 
is an isomorphism.
\end{lemma}
\begin{proof}
By the commutativity of \eqref{diagram-mult-comult-relation}
the map $\eqref{eqn-adjoints-condition}F$ is 
identified by the isomorphism $\gamma: Q_n \rightarrow RF$ with 
\begin{equation}
\label{eqn-map-RF-HnLF}
Q_n \xrightarrow{Q_n\action} Q_nQ_nQ'_n \xrightarrow{m Q'_n} 
Q_nQ'_n \xrightarrow{\mu_n Q'_n} H^n Q'_n.
\end{equation}
The first and the third maps in this composition canonically
lift to one-sided maps of the corresponding twisted complexes: 
the former has a single degree-$0$-to-degree-$0$ component equal
to the sum of  $\id \xrightarrow{\action} H^i (H')^i$, and the latter
is the projection of $\bar{Q}_n$ onto its highest degree term. 
The second map lifts to a one-sided map of twisted complexes 
by the assumptions of this lemma in the view
of Corollary \ref{cor-from-strong-monad-to-monad-multiplication-one-sided}. 
Thus the whole of \eqref{eqn-map-RF-HnLF} lifts to a one-sided map 
$$ \bar{Q}_n \rightarrow H^n \bar{Q}'_n. $$
of twisted complexes over $\AmodbarA$. 

By the Rectangle Lemma
\cite[Lemma 2.12]{AnnoLogvinenko-BarCategoryOfModulesAndHomotopyAdjunctionForTensorFunctors}
to establish that \eqref{eqn-map-RF-HnLF} is a homotopy equivalence it
suffices to establish that so are its closed degree zero components 
\begin{equation}
\label{eqn-degree-zero-components-of-RF-HnLF}
H^i \rightarrow H^n (H')^{n-i}. 
\end{equation}
We can write these down explicitly: 
they are the compositions
$$ H^i \xrightarrow{H^i \action} H^i H^{n-i} (H')^{n-i}
\xrightarrow{\rho_{i,n-i}(H')^{n-i}} H^n (H')^{n-i}. $$
Both composants in them are homotopy equivalences:
the first since $H$ is an enhancement of an autoequivalence, 
and the second by 
Corollary \ref{cor-from-strong-monad-to-monad-multiplication-one-sided}
\eqref{item-strong-monad-condition-to-stronger-monad-condition}. 
\end{proof}

\begin{prps} 
\label{prps-R-HnL-canonical-map}
Let $G\in D(\CbimB)$ for some $\C$.
If there is some isomorphism $GR \xrightarrow{\phi} GH^n L$
and the canonical map 
$$RF \xrightarrow{\eqref{eqn-adjoints-condition}F} H^nLF $$ 
is an isomorphism, then the canonical map 
\begin{equation}
\label{eqn-adjoints-condition-G}
GR \xrightarrow{G\eqref{eqn-adjoints-condition-true}} GH^nL
\end{equation}
 is also an isomorphism. 
\end{prps}

\begin{proof}
Complete \eqref{eqn-adjoints-condition-G} to an exact triangle
$$ GR \xrightarrow{\eqref{eqn-adjoints-condition-G}} GH^n L
\xrightarrow{\alpha} Z \xrightarrow{\beta} GR[1].$$
Then 
$$ GRF \xrightarrow{\eqref{eqn-adjoints-condition-G}F} GH^n LF
\xrightarrow{\alpha F} ZF \xrightarrow{\beta F} GRF[1]$$
is also an exact triangle. Since $\eqref{eqn-adjoints-condition-G}F$
is an isomorpshim, we have $\alpha F = 0$ and $\beta F = 0$. 
Thus the composition
$$ GRF \xrightarrow{\phi F} GH^nLF \xrightarrow{\alpha F} ZF $$
is zero. By the Cancellation Lemma
\ref{lemma-cancelling-F-from-natural-transformations-involving-R-or-L}\eqref{item-GR-alpha-H}
so is
$$ GR \xrightarrow{\phi} GH^nL \xrightarrow{\alpha} Z. $$
Since $\phi$ is an isomorphism, we conclude that $\alpha = 0$.
Similarly, we obtain that $\beta = 0$. 

Since $\alpha = 0$ and $\beta = 0$ the morphism 
$GR \xrightarrow{\eqref{eqn-adjoints-condition-G}} GH^n L$ admits
both right and left semi-inverse, and is therefore an isomorphism. 
\end{proof}

\begin{lemma} 
Suppose that $\ext^{-1}$-vanishing condition 
\eqref{eqn-the-minus-one-ext-assumption} holds or
that on DG level the monad multiplication is one-sided as per
Corollary \ref{cor-from-strong-monad-to-monad-multiplication-one-sided}\eqref{item-monad-multiplication-is-DG-one-sided}. 
The strong monad condition and the weak adjoints condition 
imply the adjoints condition.
\end{lemma}
\begin{proof}
By Lemma \ref{lemma-RF-isomorphic-to-HnLF} the canonical map $RF \to H^nLF$ is an isomorphism.
Then by Proposition \ref{prps-R-HnL-canonical-map} the canonical map $FR \to FH^nL$ is an isomorphism as well, which
completes the proof.
\end{proof}

\begin{lemma} 
\label{lemma-from-strong-monad-to-highest-degree-term}
Suppose that $\ext^{-1}$-vanishing condition 
\eqref{eqn-the-minus-one-ext-assumption} holds or
that on DG level the monad multiplication is one-sided as per
Corollary \ref{cor-from-strong-monad-to-monad-multiplication-one-sided}\eqref{item-monad-multiplication-is-DG-one-sided}. 
The strong monad condition implies the highest degree term condition.
\end{lemma}
\begin{proof}
As in the proof of Proposition \ref{prps-dg-construction-of-p-twists}, 
we can find a one-sided map of twisted complexes over $\AmodbarA$
\begin{equation}
\label{eqn-FQ1-FH-splitting-on-DG-level}
\phi\colon 
\begin{tikzcd}[column sep={2.5cm}]
\underset{\degminusone}{FH[-1]}
\ar{dr}{\beta}
\ar[equals]{d}
&
\\
FH[-1]
\ar{r}{F\sigma_1}
&
F
\end{tikzcd} 
\end{equation}
whose convolution $\phi: FH \rightarrow FQ_1$ splits 
the projection $F\mu_1: FQ_1 \rightarrow FH$ in $D(\AbimA)$. 
Similarly, by $\phi'$ we denote both the left dual 
of \eqref{eqn-FQ1-FH-splitting-on-DG-level} and 
its convolution $Q_1'L \rightarrow H'L$. 

The highest degree term condition of Definition \ref{defn-Pn-functor}
asks for an isomorphism $FH^nL \to FHH^nH'L$ that makes the following diagram
commute:
\begin{equation}
\label{eqn-highest-degree-term-diagram-via-RF}
\begin{tikzcd}[column sep={2.25cm}]
FHQ_{n-1}L
\ar{r}{(F\gamma \circ F\iota\circ\phi){\iota}L}
\ar[']{d}{FH{\iota}L}
&
FRFRFL
\ar{r}{(FR\trace - \trace FR)FL}
&
FRFL
\ar{r}{F \mu_nL}
&
FH^nL
\ar[dashed]{d}
\\
FHRFL
\ar{r}{FHR( \action FL - FL\action)}
&
FHRFLFL
\ar{r}{FHRF(\phi'\circ \iota'L \circ \gamma'L)}
&
FHRFH'L
\ar{r}{FH\mu_n H'L}
&
FHH^nH'L.
\end{tikzcd}
\end{equation}
The composition 
$$ (F\mu_n L) \circ (\trace FRFL) \circ
((F\gamma \circ F\iota\circ\phi)\iota L) $$ 
vanishes since $\mu_n \circ \iota: Q_{n-1} \rightarrow H^n$ is the
zero map. Similarly, the composition 
$$ (FH\mu_nH'L)\circ (FHRF(\phi'\circ \iota'L \circ \gamma'L)) \circ
(FHRFL\action) \circ (FH\iota L)$$
vanishes as well. On the other hand by the commutativity of 
the diagram \eqref{diagram-mult-comult-relation} 
the composition 
$$
FHRFL
\xrightarrow{FHR\action FL}
FHRFLFL
\xrightarrow{FHRF(\phi'\circ \iota'L \circ \gamma'L)}
FHRFH'L
\xrightarrow{FH\mu_n H'L}
FHH^nH'L
$$
equals the composition
$$
FHRFL
\xrightarrow{FHR{\action}L}
FHRFRFLFL
\xrightarrow{FHR{\trace}FLFL}
FHRFLFL
\xrightarrow{FHRF(\phi'\circ \iota'L \circ \gamma'L)}
FHRFH'L
\xrightarrow{FH\mu_n H'L}
FHH^nH'L.
$$

We can now use the isomorphisms $\gamma\colon Q_n \simeq RF$ and
$\gamma'\colon LF \simeq Q'_n$ to rewrite 
\eqref{eqn-highest-degree-term-diagram-via-RF} as:
\begin{equation}
\label{diagram-highest-degree-term-diagram-via-Q_n}
\begin{tikzcd}[column sep={2.1cm}]
FHQ_{n-1}L
\ar{r}{(F\iota\circ\phi){\iota}L}
\ar[']{d}{FH{\iota}L}
&
FQ_nQ_nL
\ar{r}{FmL}
&
FQ_nL
\ar{r}{F\mu_nL}
&
FH^nL
\ar[dashed]{d}
\\
FHQ_nL
\ar{r}{FHQ_n{\action}L}
&
FHQ_nQ_nQ'_nL
\ar{r}{FHmQ'_nL}
&
FHQ_nQ'_nL
\ar{r}{FH\mu_n(\phi' \circ \iota'L)}
&
FHH^nH'L.
\end{tikzcd}
\end{equation}

Now observe that all maps in the diagram
\eqref{diagram-highest-degree-term-diagram-via-Q_n} lift to one-sided
maps of corresponding twisted complexes over $\AmodbarA$. Moreover, 
the object
$FHQ_{n-1}L$ lifts to the twisted complex 
$ \left\{\underset{\degminusone}{FH[-1]}\right\}  
\bartimes \; \bar{Q}_{n-1} $
which is concentrated in degrees from $-n$ to $-1$. Similarly,
$FH^nL$ and $FHH^nH'L$ lift to twisted complexes concentrated in the
single degree $-n$. Since the maps $FHQ_{n-1}L \rightarrow FH^nL$ and 
$FHQ_{n-1}L \rightarrow FHH^nH'L$ in
\eqref{diagram-highest-degree-term-diagram-via-Q_n} lift to one-sided
maps of twisted complexes, we conclude that their only non-zero
components are the degree $-n$ to degree $-n$ ones. These we can write
out explicitly:
\begin{align*}
&FHH^{n-1}L \xrightarrow{F \rho_{1,n-1}L} FH^nL, \\
&FHH^{n-1}L \xrightarrow{FHH^{n-1}{\action}L} FHH^{n-1}HH'L 
\xrightarrow{FH\rho_{n-1,1}H'L} FHH^nH'L. 
\end{align*}
Since $H$ enhances an autoequivalence and since by our assumptions
the stronger monad condition holds, these are both isomorphisms in $D(\AbimA)$.
Hence they can be intertwined by an isomorphism as desired. 
\end{proof}

Summarizing the results above , we have proved the following:
\begin{theorem}
\label{theorem-strong-monad-and-weak-adjoints-imply-pn-functor}
Suppose that $\ext^{-1}$-vanishing condition 
\eqref{eqn-the-minus-one-ext-assumption} holds or
that on DG level the monad multiplication is one-sided as per
Corollary \ref{cor-from-strong-monad-to-monad-multiplication-one-sided}\eqref{item-monad-multiplication-is-DG-one-sided}. 
Then the strong monad condition and the weak adjoints condition imply 
the three conditions from
Definition \ref{defn-Pn-functor}.
\end{theorem}


\section{Examples of split $\mathbb{P}^n$-functors}
\label{section-examples-of-split-Pn-functors}

In this section, we give an overview of the examples of split 
$\mathbb{P}^n$-functors which appeared in the algebro-geometric 
literature to date. To see that these are indeed
$\mathbb{P}^n$-functors in the sense of \S\ref{section-Pn-functors} 
there is a standard Morita DG-enhancement framework to
apply our results for enhanced triangulated categories in the context
of algebraic geometry. For full details we refer the reader to 
\cite[\S5.2]{AnnoLogvinenko-SphericalDGFunctors},\cite[\S8]{Toen-TheHomotopyTheoryOfDGCategoriesAndDerivedMoritaTheory},
\cite{LuntsSchnurer-NewEnhancementsOfDerivedCategoriesOfCoherentSheavesAndApplications}.

In brief, let $Z$ and $X$ be separated schemes of finite type over an
algebraically closed field $k$. Choosing strong generators identifies 
$D(Z)$ and $D(X)$ with $D(\A)$ and $D(\B)$ for some smooth DG-algebras 
$\A$ and $\B$. We can then identify $D(\AbimB)$, $D(\BbimA)$, $D(\AbimA)$, 
and $D(\BbimB)$ with $D(X \times Z)$, $D(Z \times X)$, $D(X \times X)$, 
and $D(Z \times Z)$ in such a way that DG-enhancing bimodules
correspond to Fourier-Mukai kernels, the derived tensor product of
bimodules corresponds to the standard composition of Fourier-Mukai 
kernels via the triple fibre product, and the dualisation functors are
given by the relative Verdier duality. E.g. DG-enhancements
in $D(\AbimB)$ of exact functors $D(\A) \rightarrow D(\B)$ correspond
to Fourier-Mukai kernels in $D(X \times Z)$ of exact functors $D(X)
\rightarrow D(Z)$. 

All the key definitions and the results of \S\ref{section-ptwists}
and \S\ref{section-Pn-functors} are stated in the terms of
the triangulated subcategories of DG-enhancements of exact functors 
in the derived categories of DG-bimodules over the respective enhanced
triangulated categories. We can therefore apply them instead 
to the triangulated subcategories of Fourier-Mukai kernels in 
the derived categories of the fibre products of respective varieties. 

\subsection{Spherical objects
\cite{SeidelThomas-BraidGroupActionsOnDerivedCategoriesOfCoherentSheaves}}

These were introduced by Seidel and Thomas in
\cite{SeidelThomas-BraidGroupActionsOnDerivedCategoriesOfCoherentSheaves} 
as the mirror-symmetric analogues of Lagrangian spheres on a 
symplectic manifold and their associated Dehn twists. It was 
the genesis of this whole subject.  

Let $X$ be a smooth projective variety. An object $E \in D(X)$ is
said to be a \em spherical \rm if:
\begin{enumerate}
\item $$\homm^i(E,E) = 
\begin{cases}
k, &\quad  i = 0, \dim(X), \\
0, &\quad \text{ otherwise. } 
\end{cases}$$
\item $E \otimes \omega_X \simeq E$. 
\end{enumerate}
For such $E$ an autoequivalence of $D(X)$ 
called the \em spherical twist \rm $T_E$ was constructed in 
\cite{SeidelThomas-BraidGroupActionsOnDerivedCategoriesOfCoherentSheaves}. 
If mirror symmetry identifies $E$ with a Lagrangian sphere, 
then it identifies $T_E$ with the corresponding Dehn twist.  

Let $Z = \spec k$. Any object $E \in D(X)$ defines the functor 
$$ f \colon D(Z) \xrightarrow{(-) \otimes_k E} D(X).$$
Moreover, viewing $E$ as the Fourier-Mukai kernel in 
$D(Z \times X) \simeq D(X)$ gives a DG-enhancement $F$ of $f$. 
If $E$ is a spherical object, then $F$ is a spherical functor
in the sense of \cite{AnnoLogvinenko-SphericalDGFunctors}. 
Moreover, we have 
\begin{align*}
R  & \simeq E^\vee, \\ 
RF & \simeq \rder\homm_X(E,E) \simeq k \oplus k[- \dim (X)],
\end{align*}
thus the adjunction monad $RF$ is a direct sum of $\id_Z$ and the 
autoequivalence $H = [- \dim(X)]$. This decomposition
gives $F$ the structure of a split $\mathbb{P}^1$-functor 
in the sense of this paper and 
\cite{Addington-NewDerivedSymmetriesOfSomeHyperkaehlerVarieties}.  
Indeed, $E$ is a spherical object if and only if $F$ is a
spherical/$\mathbb{P}^1$-functor, see e.g. 
\cite[Example
3.5]{AnnoLogvinenko-OrthogonallySphericalObjectsAndSphericalFibrations}. 

The spherical twist $T_E$ of 
\cite{SeidelThomas-BraidGroupActionsOnDerivedCategoriesOfCoherentSheaves} 
is the cone of the adjunction counit $FR \rightarrow \id_X$. 
The $\mathbb{P}$-twist $P_E$ is its square $T_E^2$, 
as originally observed in  
\cite{HuybrechtsThomas-PnObjectsAndAutoequivalencesOfDerivedCategories}
for spherical objects and later in 
\cite{Addington-NewDerivedSymmetriesOfSomeHyperkaehlerVarieties}
for split spherical functors. Below, in 
\S\ref{section-examples-spherical-functors} we establish the same for
all spherical functors.

The condition of $E \otimes \omega_X \simeq E$ implies that on
varieties whose canonical bundle is ample or anti-ample the support 
of any spherical object must be zero-dimensional. On such varieties, 
therefore, spherical objects do not exist when $\dim X \geq 2$ and are precisely
the point sheaves when $\dim X = 1$. Most of the geometrical examples
of spherical objects occur on Calabi-Yau varieties where they include, 
among many others, all line bundles, the structure sheaves of 
$(-2)$-curves on $K3$-surfaces, the structure sheaves of 
$(-1,-1)$-curves and $\mathbb{P}^2$s with the normal bundle
$\mathcal{O}(-3)$ on Calabi-Yau 3-folds, etc. 

\subsection{$\mathbb{P}^n$-objects and $\mathbb{P}^n[k]$-objects
\cite{HuybrechtsThomas-PnObjectsAndAutoequivalencesOfDerivedCategories}, 
\cite{Krug-VarietiesWithPUnits} 
}

The $\mathbb{P}^n$-objects were introduced by Huybrechts and Thomas
in \cite{HuybrechtsThomas-PnObjectsAndAutoequivalencesOfDerivedCategories},
which began the $\mathbb{P}^n$-chapter of this story. 
They are the mirror symmetric analogue of Lagrangian complex
projective spaces with their associated Dehn twists in the same way
spherical objects are one for Lagrangian spheres.  

Let $X$ be a smooth projective variety. An object $E \in D(X)$ is
said to be a \em $\mathbb{P}^n$-object \rm if $\homm^*_{D(X)}(E,E)$ is
isomorphic as a graded ring to $H^*(\mathbb{P}^n, k)$, and if 
$E \otimes \omega_X \simeq E$. Note that the Serre duality then 
implies that $\dim X = 2n$. 

Again, let $Z = \spec k$ and consider the functor 
$$ f \colon D(Z) \xrightarrow{(-) \otimes_k E} D(X)$$
with a DG-enhancement $F$ given by 
$E \in D(Z \times X) \simeq D(X)$. 
The adjunction monad 
$$ RF \simeq \rder\homm_X(E,E) \simeq 
H^*(\mathbb{P}^n, k) \simeq k \oplus k[-2] \oplus \dots \oplus k[-2n], $$
decomposes as a direct sum of $\id_Z$ and 
$H$, $H^2$, \dots, $H^n$, where $H$ is the autoequivalence $[-2]$. 
As observed in 
\cite{Addington-NewDerivedSymmetriesOfSomeHyperkaehlerVarieties}
this decomposition gives $F$ the structure of 
a split $\mathbb{P}^n$-functor. Indeed, we have 
$$ L \simeq E^\vee \otimes \omega_X [2n], $$
so the weak adjoints condition is satisfied since 
$E \otimes \omega_X \simeq E$. On the other hand, the strong monad 
condition is satisfied since the monad structure on $RF$ is 
the graded ring structure on 
$\rder\homm_X(E,E) \simeq \homm^*_{D(X)}(E,E)$ which is isomorphic
as a graded ring to $H^*(\mathbb{P}^n, k)$.   
Thus we can identify the adjunction monad $RF$ with 
the graded ring $k[h]/h^{n+1}$ where $\deg h = 2$. The filtration $Q_i$ 
on $RF$ corresponds to the filtration by the degree of a polynomial 
on $k[h]/h^{n+1}$ and the projection $Q_i \twoheadrightarrow H^i$ 
to the projection onto the highest degree term $h^i$. The strong monad 
condition now corresponds to the statement that multiplying polynomials 
of degree $\leq 1$ and $\leq n$ we get a polynomial of degree $\leq n+1$ 
and that $h h^i = h^{i+1}$.

Conversely, it is also clear from the above that if $F$ admits a structure 
of a $\mathbb{P}^n$-functor with $H = [-2]$, then $E$ is a
$\mathbb{P}^n$-object. More generally, $D(Z)$ is the derived 
category of vector spaces and its only autoequivalences are 
the shift functors $[k]$ for $k \in \mathbb{Z}$. Thus if
$F$ admits a structure of a $\mathbb{P}^n$-functor for any 
$H \in \autm D(Z)$, then $E \otimes \omega_X \simeq E$ and 
$\homm^*_{D(X)}(E,E)$ is isomorphic as a graded ring to $k[h]/h^{n+1}$
with $\deg h = k$. This very natural generalisation of $\mathbb{P}^n$-objects 
was first made by Krug in \cite[Defn.~2.4]{Krug-VarietiesWithPUnits}, 
who called these the \em $\mathbb{P}^n[k]$-objects \rm and noted that 
they are precisely the split $\mathbb{P}^n$-functors 
$D(\spec k) \rightarrow D(X)$. 

Geometrical examples of $\mathbb{P}^n$-objects include 
the structure sheaves of $\mathbb{P}^n$s on a holomorphic symplectic manifold 
of $\dim = 2n$ and line bundles on a hyperk{\"a}hler manifold of $\dim = 2n$
\cite[Example
1.3]{HuybrechtsThomas-PnObjectsAndAutoequivalencesOfDerivedCategories}.
Geometrical examples of $\mathbb{P}^n[2k]$-objects include all line
bundles on the canonical cyclic cover of the fiber product of $k$ 
strict Enriques varieties \cite[Theorem 4.5]{Krug-VarietiesWithPUnits}. 

\subsection{Hilbert schemes of points on K3 surfaces
\cite{Addington-NewDerivedSymmetriesOfSomeHyperkaehlerVarieties}, 
}

This is the example which motivated Addington to introduce split
$\mathbb{P}^n$-functors in 
\cite{Addington-NewDerivedSymmetriesOfSomeHyperkaehlerVarieties}. 
Let $Z$ be a projective $K3$-surface and let $X = Z^{[n]}$, the
Hilbert scheme of $n$-points on $Z$.
Let $F$ be the DG-enhanced functor $D(Z) \rightarrow D(X)$ 
defined on the Fourier-Mukai level by the universal ideal sheaf 
$\mathcal{I} \in D(Z \times X)$. 

Addington proved in 
\cite[Theorem 2]{Addington-NewDerivedSymmetriesOfSomeHyperkaehlerVarieties}
that we have 
\begin{align}
RF \simeq \id \oplus H \oplus H^2 \oplus ... \oplus H^{n-1} 
\end{align}
for the autoequivalence $H = \id_Z [-2]$ and that this isomorphism 
gives $F$ the structure of a split $\mathbb{P}^{n-1}$-functor. More
specifically, he had shown that the monad multiplication on $RF$
restricts to the map 
$$ H RF \rightarrow RF $$
as the sum of natural isomorphisms $H H^{i} \rightarrow H^{i+1}$
plus an unknown map $H H^{n-1} \rightarrow RF$. In other words,
its restriction to the map $ H Q_{n-2} \rightarrow J_{n-1} $ is 
the map 
$$ H \oplus \dots \oplus H^{n-1} \rightarrow H \oplus \dots \oplus
H^{n-1} $$ 
given by the identity matrix.

\subsection{Generalised Kummer varieties and the Albanese map
\cite{Meachan-DerivedAutoequivalencesOfGeneralisedKummerVarieties}, 
\cite{KrugMeachan-UniversalFunctorsOnSymmetricQuotientStacksOfAbelianVarieties}
}

Let $Z$ be an abelian surface. Let $Z^{[n+1]}$ be the Hilbert scheme 
of $n+1$-points on $Z$ and let $\alpha \colon Z^{[n+1]} \rightarrow Z$
be the Albanese map. We can think of $\alpha$ as the composition
of the Hilbert-Chow morphism $Z^{[n+1]} \rightarrow S^{n+1} Z$ into
the symmetric product, and then the summation map $S^{n+1} Z
\rightarrow Z$. Let $X$ be the corresponding generalised Kummer
variety, that is --- the fiber of $\alpha$ over $0 \in Z$. 
Let $F$ be the DG-enhanced functor $D(Z) \rightarrow D(X)$ defined
by the Fourier-Mukai kernel of the universal ideal sheaf 
$\mathcal{I} \in D(Z \times X)$. 

Generalised Kummer varieties are the other well-known infinite
family of hyperk{\"a}hler manifolds, the first being the Hilbert
schemes of points on K3 surfaces. 
In \cite{Meachan-DerivedAutoequivalencesOfGeneralisedKummerVarieties}
Meachan had observed that, similar to the case of the Hilbert schemes 
of K3 surfaces \cite{Addington-NewDerivedSymmetriesOfSomeHyperkaehlerVarieties}, we have 
\begin{align}
RF \simeq \id \oplus H \oplus H^2 \oplus ... \oplus H^{n-1} 
\end{align}
for the autoequivalence $H = \id_Z [-2]$ and that this isomorphism 
gives $F$ the structure of a split $\mathbb{P}^{n-1}$-functor.
He then showed that the restriction of the monad multiplication on $RF$ 
to the map $H Q_{n-2} \rightarrow J_{n-1}$ is again the map  
$$ H \oplus \dots \oplus H^{n-1} \rightarrow H \oplus \dots \oplus
H^{n-1} $$ 
given by the identity matrix.

The Albanese map $\alpha\colon Z^{[n+1]} \rightarrow Z$ is an
isotrivial fibration whose fibers are each isomorphic to the
hyperk{\"a}hler manifold $X$. Let $F'$ 
be the standard DG-enhancement of the functor 
$\alpha^*\colon D(Z) \rightarrow D(Z^{[n+1]})$ given by 
the Fourier-Mukai kernel $(\alpha, \id)_* \mathcal{O}_{(Z^{[n+1]})}
\in D(Z^{[n+1]} \times Z)$, the graph of $\alpha$. It follows 
that 
\begin{align*}
R'F' \simeq \Delta_* \alpha_* \mathcal{O}_{(Z^{[n+1]})}
\simeq \id \oplus H \oplus H^2 \oplus ... \oplus H^{n-1}
\end{align*}
for $H = \id_Z [-2]$. This decomposition gives $F'$ the structure of a split
$\mathbb{P}^{n-1}$-functor \cite[\S
5]{Meachan-DerivedAutoequivalencesOfGeneralisedKummerVarieties}.

Finally, Krug and Meachan showed in 
\cite{KrugMeachan-UniversalFunctorsOnSymmetricQuotientStacksOfAbelianVarieties}
that the above-mentioned 
$\mathbb{P}^{n-1}$-functor $F: D(Z) \rightarrow D(X)$ extends 
in a family from the zero fibre $X$ to the whole the Albanese map fibration 
$\alpha\colon Z^{[n+1]} \rightarrow Z$. Let $F''$ be the DG-enhanced 
functor $D(Z \times Z) \rightarrow D(Z^{[n+1]})$ defined by the
Fourier-Mukai kernel $(\id_Z, (m,\id_{Z^{[n+1]}}))_* I$ 
which is the direct image of the universal ideal sheaf $I \in D(Z
\times Z^{[n+1]}$ under the map 
$$ Z \times Z^{[n+1]} \xrightarrow{\id_Z, (m,\id_{Z^{[n+1]}}}  Z
\times Z \times Z^{[n+1]}. $$
In 
\cite[\S 2.2]{KrugMeachan-UniversalFunctorsOnSymmetricQuotientStacksOfAbelianVarieties}
it is shown using equivariant methods and the Bridgeland-King-Reid-Haiman
equivalence \cite{BKR01} \cite{Haiman-HilbertSchemesPolygraphsAndTheMacdonaldPositivityConjecture}
that 
$$ R'' F'' \simeq \id \oplus H \oplus H^2 \oplus ... \oplus H^{n-1} $$
for $H = \id_{Z \times Z} [-2]$  and this decomposition gives $F''$
the structure of a split $\mathbb{P}^{n-1}$-functor and that $F''$ 
restricts to $F$ on $Z \times \left\{0\right\}$. 

\subsection{Symmetric quotient stack of a smooth projective surface
\cite{Krug-NewDerivedAutoequivalencesOfHilbertSchemesAndGeneralisedKummerVarieties}
}

Let $Z$ be a smooth projective surface. 
Let $[Z^n/S_n]$ be the quotient stack with respect to the natural action 
of the symmetric group $S_n$ on the fibre product $Z^n$.
Let 
$$ d\colon Z \rightarrow [Z^n/S_n] $$
be the diagonal embedding and let $F$ the DG-enhancement of the 
functor $d_*\colon D(Z) \rightarrow D([Z^n / S_n])$ given 
by the Fourier-Mukai kernel $(d, \id)_* \mathcal{O}_{[Z^n / S_n]} 
\in D(Z \times [Z^n / S_n])$, the graph of $d$. 

Krug showed in 
\cite[\S3]{Krug-NewDerivedAutoequivalencesOfHilbertSchemesAndGeneralisedKummerVarieties} that we have
\begin{align}
RF \simeq \id \oplus H \oplus H^2 \oplus ... \oplus H^{n-1} 
\end{align}
for $H = \Delta_* \omega^{-1}_Z [-2]$, the inverse of the Serre functor on
$D(Z)$ and that this decomposition gives $F$ the structure of a split
$\mathbb{P}^{n-1}$-functor. 

\subsection{$\mathbb{P}^n$-functor versions of the Nakajima operators 
\cite{Krug-P-functorVersionsOfTheNakajimaOperators}
}

Let $X$ be a smooth projective surface and let $n,l \in \mathbb{Z}$
with $n \geq 2$. Let $[X^l/S_l]$ be the symmetric quotient stack with
respect to the natural action of $S_l$ on $X^l$, and similarly for
$[X^{n+l}/S_{n+l}]$. In \cite{Krug-P-functorVersionsOfTheNakajimaOperators} 
Krug had constructed an enhanced functor
$$ P_{l,n}\colon D(X \times [X^l/S_l]) \rightarrow D([X^{n+l}/S_{n+l}]) $$
whose Fourier-Mukai kernel in $D(X \times [X^l/S_l]
\times [X^{n+l}/S_{n+l}])$ is given by the complex of equivariant sheaves  
$$ \underset{\degzero}{\mathcal{P}_0} \rightarrow 
\mathcal{P}_1 \rightarrow \mathcal{P}_2 \rightarrow \dots \rightarrow
\mathcal{P}_l $$
where each $\mathcal{P}_i$ is a direct sum (with a sign-twisted 
action of $S_{n+l}$) of the structure sheaves of the subvarieties 
$$ \Gamma_{I,J,\mu} 
\overset{\text{def}}{=}
\left\{ (x, x_1, \dots, x_l, y_1, \dots, y_{l+n}) \;|\; 
x = x_a = y_b\; \forall\; a \in I, b \in J \text{ and } x_c = y_{\mu(c)}
\; \forall\; c \in \bar{I} \right\} 
\quad \quad \subset X \times X^l \times X^{n+l}, $$
where $I \subset [1,l]$ with $|I| = i$, 
$J \subset [1,n+l]$ with $|J| = n+i$, 
and $\mu$ is a bijection $\bar{I} \rightarrow \bar{J}$. 
That is, $(x_1, \dots, x_l)$ and $(y_1,
\dots, y_{n+l})$ contain $i$ and $n+i$ copies of $x$ in fixed
positions, and the remaining elements of each are the same $l-i$
points of $X$ permuted in a fixed way. Clearly, adding the same element
of $[1,l]$ to both $I$ and $J$ defines a subvariety of 
$\Gamma_{I,J,\mu}$. The differentials in the complex above 
are given by the sheaf restriction maps for such subvarieties.  

Krug had proved in \cite[Theorem C]{Krug-P-functorVersionsOfTheNakajimaOperators}
that for $n > l$
$$ P_{l,n}^{\text{r.adj.}} P_{l,n} \simeq \id_Z \oplus H \oplus H^2
\oplus \dots \oplus H^{n-1} $$
for the enhanced autoequivalence $H$ of $D(X \times [X^l/S_l])$ defined 
by the Fourier-Mukai kernel
$$\Delta_* (\omega_X \boxtimes \mathcal{O}_{X_l})[2] \in D(X \times
[X^l/S_l] \times X \times [X^l/S_l]),$$
and this decomposition gives
$P_{l,n}$ the structure of a $\mathbb{P}^{n-1}$-functor. Feeding this
through the Bridgeland-King-Reid-Haiman
equivalence \cite{BKR01}
\cite{Haiman-HilbertSchemesPolygraphsAndTheMacdonaldPositivityConjecture}
produces a $\mathbb{P}^{n-1}$-functor
$$ P_{l,n} \colon D(X \times X^{[l]}) \rightarrow D(X^{[n+l]}) $$
whose Fourier-Mukai kernel is supported on the correspondence which defines
the Nakajima operator
$$ q_{l,n}\colon H^*(X \times X^{[l]},\mathbb{Q}) \rightarrow
H^*(X^{[n+l]},\mathbb{Q}). $$
These were used in
\cite{Nakajima-HeisenbergAlgebraAndHilbertSchemesOfPointsOnProjectiveSurfaces}
\cite{Grojnowski-InstantonsAndAffineAlgebrasITheHilbertSchemeAndVertexOperators}
to construct the celebrated action of the Heisenberg algebra 
on $\bigoplus_{n \geq 0} H^*(X^{[n]},\mathbb{Q})$. 

Finally, Krug had shown in 
\cite[\S4.12]{Krug-P-functorVersionsOfTheNakajimaOperators}
that when $X$ is an abelian surface, the $\mathbb{P}^{n-1}$ functor 
$$ P_{l,n} \colon D(X \times X^{[l]}) \rightarrow D(X^{[n+l]}) $$
restricts to the functor 
$$ \hat{P}_{l,n} \colon D(K_{l,n}) \rightarrow D(K_{n+l-1}),$$
where $K_{l,n} \subset X \times X^{[l]}$ and $K_{n+l-1} \subset X^{[n+l]}$ 
are the fibers over $0 \in X$ of the summation maps
$n \cdot x + x_1 + \dots + x_l$ and $y_1 + \dots + y_{n + l}$. Note
that $K_{n+l-1}$ is the generalised Kummer variety. Krug had further
proved in \cite[Theorem C']{Krug-P-functorVersionsOfTheNakajimaOperators}
that for $n > l$
$$ \hat{P}_{l,n}^{\text{r.adj.}}\hat{P}_{l,n} \simeq \id_{K_{l,n}}
\oplus \hat{H} \oplus \hat{H}^2 \oplus \dots \oplus \hat{H}^{n-1},$$
for $\hat{H} = [-2]$ and this decomposition gives $\hat{P}_{l,n}$ the
structure of a $\mathbb{P}^{n-1}$-functor. 

\subsection{Truncated universal ideal functors for even-dimensional Calabi-Yau varieties
\cite{KrugSosna-OnTheDerivedCategoryOfTheHilbertSchemeOfPointsOnAnEnriquesSurface}
}

Let $Z$ be a smooth projective variety.
Let $[Z^n/S_n]$ be the quotient stack with respect to the natural action 
of the symmetric group $S_n$ on the fibre product $Z^n$.
Let $D_i \subset Z \times Z^n$ be reduced subvariety supported on  
the closed subset defined by
$$ \left\{ (z, z'_1, \dots, z'_n) | z = z'_i \right\}. $$
The action of $S_n$ permutes $D_i$ giving natural 
equivariant structure to the sheaf $\oplus_{i = 1}^{n}
\mathcal{O}_{D_i}$. The \em truncated universal ideal functor \rm $F$ 
is the enhanced functor $D(Z) \rightarrow D([Z^n/S_n])$ defined by 
the Fourier Mukai kernel 
\begin{align}
\label{eqn-truncation-of-the-complex-resolving-the-ideal-sheaf-on-Hilb-scheme} 
\underset{\degzero}{\mathcal{O}_{Z \times [Z^n/S_n]}} \quad
\longrightarrow \quad
\oplus_{i = 1}^{n} \mathcal{O}_{D_i} 
\hspace{2cm} 
\in \quad D(Z \times [Z^n/S_n]), 
\end{align}
where the differential is the sum of the natural restriction maps. 

Krug and Sosna proved in 
\cite[Theorem 5.4]{KrugSosna-OnTheDerivedCategoryOfTheHilbertSchemeOfPointsOnAnEnriquesSurface} 
that if $Z$ is an even-dimensional Calabi-Yau variety, then we have
\begin{align}
RF \simeq \id \oplus H \oplus H^2 \oplus ... \oplus H^{n-1} 
\end{align}
for $H = \id_Z [-2]$ and that this decomposition gives $F$ the
structure of a split $\mathbb{P}^{n-1}$-functor.

The choice of the name ``truncated universal ideal functor'' is due to 
the fact that the complex 
\eqref{eqn-truncation-of-the-complex-resolving-the-ideal-sheaf-on-Hilb-scheme}
is a truncation of a larger natural complex of equivariant sheaves on 
$Z \times Z^n$ which was shown by Scala
\cite{Scala-CohomologyOfTheHilbertSchemeOfPointsOnASurfaceWithValuesInRepresentationsOfTautologicalBundles}
to define the object in 
$D(Z \times [Z^n/S_n])$ which corresponds 
to the universal ideal sheaf on $Z^{[n]}$ under the
Bridgeland-King-Reid-Haiman derived equivalence $D(Z^{[n]}) \simeq
D([Z^n/S_n])$ \cite{BKR01}
\cite{Haiman-HilbertSchemesPolygraphsAndTheMacdonaldPositivityConjecture}. 
For surfaces, the adjunction monad of $F$ is isomorphic to that of the
full universal ideal functor \cite[Lemma 5.2]{KrugSosna-OnTheDerivedCategoryOfTheHilbertSchemeOfPointsOnAnEnriquesSurface}, and thus $F$ being a $\mathbb{P}^n$-functor for $K3$s is equivalent to the
original result of 
\cite{Addington-NewDerivedSymmetriesOfSomeHyperkaehlerVarieties}. For 
higher dimensional Calabi-Yaus, the adjunction monads of $F$ and
of the universal ideal functor are different and $F$ is the right one
to take. The coincidence of the two in the case of surfaces was
further explained in \cite[Theorem
3.6]{Krug-RemarksOnTheDerivedMcKayCorrespondenceForHilbertSchemesOfPointsAndTautologicalBundles}
where it was shown that there is another natural version of the 
Bridgeland-King-Reid-Haiman derived equivalence $D(Z^{[n]}) \simeq
D([Z^n/S_n])$, where we use the same Fourier-Mukai kernel but in the
opposite direction, and that this version identifies the universal
ideal sheaf directly with the truncated complex 
\eqref{eqn-truncation-of-the-complex-resolving-the-ideal-sheaf-on-Hilb-scheme}. 

\subsection{Family $\mathbb{P}$-twists over the base with $HH^{\odd} = 0$ \cite{AddingtonDonovanMeachan-MukaiFlopsAndPTwists}}
\label{section-examples-split-pnfunctors-split-Mukai-flops}

Let $Z$ be a smooth projective variety whose odd Hochschild cohomology
vanishes: 
$$ HH^{\odd}(Z) = 0. $$
Let $\mathcal{V}$ be a vector bundle of rank $n+1$ over $Z$, 
let $\mathbb{P}\mathcal{V}$ be its 
projectificaton, and let $\pi\colon \mathbb{P}\mathcal{V} \rightarrow Z$ 
be the corresponding $\mathbb{P}^n$-fibration. Let $\iota$  be a closed embedding of
$\mathbb{P}V$ into a smooth  projective variety $X$ with the normal bundle isomorphic to 
$\Omega^1_{\mathbb{P}V/Z}$:
\begin{equation}
\begin{tikzcd}[column sep={1cm}, row sep={1cm}]
\mathbb{P}\mathcal{V} 
\ar[hookrightarrow]{r}{\iota}
\ar[twoheadrightarrow]{d}[']{\pi}
&
X.
\\
Z
&
\end{tikzcd}
\end{equation}

Let $F_k$ be the DG-enhancement of the exact functor 
$$ \iota_* \circ \left(\mathcal{O}_{\mathbb{P}\mathcal{V}}(k)
\otimes \pi^*(-)\right)\colon D(Z) \rightarrow D(X) $$
given by $(\pi, \iota)_*
\mathcal{O}_{\mathbb{P}\mathcal{V}}(k) \in D(Z \times X)$. 
Addington, Donovan, and Meachan proved in 
\cite{AddingtonDonovanMeachan-MukaiFlopsAndPTwists}
that we have 
$$ R_kF_k \simeq \id \oplus H \oplus \dots \oplus H^n $$
with $H = \id_Z[-2]$ and observed that this decomposition gives 
$F_k$ the structure of a split $\mathbb{P}^n$-functor. 

Now, let $\tilde{X} \rightarrow X$ be the blowup of $X$ in $\mathbb{P}\mathcal{V}$. Its exceptional 
divisor $E \subset \tilde{X}$ is isomorphic to the incidence subvariety of
$\mathbb{P}\mathcal{V} \times \mathbb{P}\mathcal{V^\vee}$.  Suppose  the Mukai flop $X \leftarrow \tilde{X} \rightarrow X^{+}$ exists, where $\tilde{X} \rightarrow X^{+}$ contracts $E$ to $\mathbb{P}\mathcal{V^\vee}$.  Addington, Donovan, and Meachan further proved in \cite{AddingtonDonovanMeachan-MukaiFlopsAndPTwists} that the $\mathbb{P}$-twist $P_{F_k}$ is, in a sense, the derived monodromy of this Mukai flop. 

More precisely, it was shown by Kawamata \cite[\S5]{Kawamata-DEquivalenceAndKEquivalence} and Namikawa \cite {Namikawa-MukaiFlopsAndDerivedCategories} that the derived categories of $D(X)$ and $D(X^+)$ are equivalent and 
that the structure sheaf of the reducible subvariety 
$$ \tilde{X} \cup \left( \mathbb{P}\mathcal{V} \times \mathbb{P}\mathcal{V}^\vee \right) \;  \subset \; X \times  X^+ $$
gives the Fourier-Mukai kernel of such equivalence. Let  
$$ KN_m\colon  D(X) \rightarrow D(X^+) $$
$$ KN'_m\colon  D(X^+) \rightarrow D(X) $$
be the enhanced functors defined the Fourier-Mukai kernel above twisted by $m$-th power of the line bundle obtained by gluing $\mathcal{O}(E)$ on $\tilde{X}$ with $\mathcal{O}(-1,-1)$ on $\mathbb{P}\mathcal{V} \times \mathbb{P}\mathcal{V}^\vee$. It is shown in \cite[Theorem B']{AddingtonDonovanMeachan-MukaiFlopsAndPTwists} that 
$$ P_{F_k} \simeq KN'_{-k} \circ KN_{n+k+1} \quad \quad \quad \text{`` flop-flop = twist''}. $$

\subsection{Moduli of genus $g$ curves on $K3$ surfaces of Picard rank $1$ and degree $2g-2$
\cite{AddingtonDonovanMeachan-ModuliSpacesOfTorsionSheavesOnK3SurfacesAndDerivedEquivalences}}

Let $Z$ be a K3 surface of Picard rank $1$ and degree $2g-2$. 
Let $X$ be the moduli space of the stable sheaves on $Z$ with 
the Mukai vector $(0,1,d+1-g)$. Generic element of $X$ is a smooth
genus $g$ curve $C \subset S$ in the linear system
$|\mathcal{O}_Z(1)| \simeq \mathbb{P}^g$
equipped with a degree $d$ line bundle. Thus $X$ can be identified with 
a compactification $\overline{\picr}^d(\C/\mathbb{P}^g)$
of the relative degree $d$ Picard variety of the universal curve
$\mathcal{C}$ over $|\mathcal{O}_Z(1)|$.  

Let $\alpha$ be the Brauer class
on $X$ which obstructs the existence of the universal sheaf. 
Let $F$ be the DG-enhanced functor $D(Z) \rightarrow D(X,\alpha)$
into the derived category of $\alpha$-twisted sheaves on $X$ defined
on the Fourier-Mukai level by the pseudo-universal sheaf 
$\mathcal{F} \in D(Z \times X, 1 \boxtimes \alpha)$.  Addington, Donovan, and Meachan 
found a twisted variant 
$$
FM\colon D(\overline{\picr}^{-1}(\C/\mathbb{P}^g), \beta) \xrightarrow{\sim} D(\overline{\picr}^d(\C/\mathbb{P}^g), \alpha) 
$$
of the original Fourier-Mukai equivalence between an Abelian variety and its dual, which identifies $F$ with the functor
$$
AJ_* \circ \varpi^*(- \otimes \mathcal{O}_S(l))
$$
where $\varpi\colon \mathcal{C} \twoheadrightarrow S$ is the natural $\mathbb{P}^{n-1}$ fibration, $AJ\colon \mathcal{C}
\hookrightarrow \overline{\picr}^{-1}(\C/\mathbb{P}^g)$  is the Abel-Jacobi embedding, and $l \in \mathbb{Z}$ is an integer
\cite[Prop. 3.1]{AddingtonDonovanMeachan-ModuliSpacesOfTorsionSheavesOnK3SurfacesAndDerivedEquivalences}. The embedding $AJ$ satisfies the normal bundle condition in 
\S\ref{section-examples-split-pnfunctors-split-Mukai-flops} with respect to the fibration $\varpi$ and 
the corresponding Mukai flop exists. It follows immediately that 
$$ RF \simeq \id \oplus H \oplus \dots \oplus H^{g-1} $$
giving $F$ the structure of a $\mathbb{P}^{g-1}$-functor whose $\mathbb{P}$-twist is the conjugation by $FM$ of the derived monodromy of the Mukai flop of $\overline{\picr}^{-1}(\C/\mathbb{P}^g)$ with the center $AJ(\mathcal{C})$. 

\section{Examples of non-split $\mathbb{P}^n$-functors}
\label{section-examples-of-non-split-Pn-functors}

\subsection{Spherical functors}
\label{section-examples-spherical-functors}

Let, as in \S\ref{section-ptwists-generalities}, $\A$ and $\B$ be 
two DG categories, let $F \in D(\AbimB)$ be an enhanced functor 
$D(\A) \rightarrow D(\B)$ which has left and right enhanced 
adjoints $L, R \in D(\BbimA)$. Let $T \in D(\BbimB)$ 
be the spherical twist of $F$ with its natural exact triangle 
in $D(\BbimB)$
\begin{equation}
\label{eqn-spherical-functors-twist-exact-triangle}
\begin{tikzcd}
FR  
\ar{r}{\trace} 
& \id_\B
\ar{r} 
&
T,
\ar[dotted,bend left=20]{ll}
\end{tikzcd}
\end{equation}
and let $C \in D(\AbimA)$ 
be the spherical cotwist of $F$ with its natural exact triangle in 
$D(\AbimA)$
\begin{equation}
\label{eqn-spherical-functors-cotwist-exact-triangle}
\begin{tikzcd}
C
\ar{r}
& \id_\A
\ar{r}{\action}
&
RF.
\ar[dotted,bend left=20]{ll}
\end{tikzcd}
\end{equation}

The functor $F$ is \em spherical \rm if the following four conditions hold: 
\begin{enumerate}
\item The spherical twist $T$ is an autoequivalence of $D(\B)$. 
\item 
\label{item-spherical-functor-cotwist-is-an-auto-equivalence}
The spherical cotwist $C$ is an autoequivalence of $D(\A)$. 
\item \em ``The twist identifies adjoints'': \rm  The following composition is
an isomorphism:
\begin{equation}
LT \xrightarrow{\eqref{eqn-spherical-functors-twist-exact-triangle}} LFR[1] \xrightarrow{{\trace}R} R[1]. 
\end{equation}
\item
\label{item-spherical-functor-cotwist-identifies-adjoints}
 \em ``The cotwist identifies adjoints'': \rm 
The following composition is an isomorphism:
\begin{equation}
R \xrightarrow{R{\action}} RFL
\xrightarrow{\eqref{eqn-spherical-functors-cotwist-exact-triangle}} CL[1]. 
\end{equation}
\end{enumerate}
It is, in fact, enough for any two of these conditions to hold
for $F$ to be spherical. See \cite{AnnoLogvinenko-SphericalDGFunctors}
for full details. 

\begin{prps}
\label{prps-every-spherical-functor-is-a-P1-functor}
The functor $F$ is spherical if and only if it is a $\mathbb{P}^1$-functor 
with $H = C[1]$ and the degree $1$ cyclic coextension of 
$\id_\A$ by $H$ structure on $RF$ defined by the exact triangle
\eqref{eqn-spherical-functors-cotwist-exact-triangle}. 
\end{prps}
\begin{proof}
\em ``Only if''\rm : Suppose
\eqref{eqn-spherical-functors-cotwist-exact-triangle} defines a
$\mathbb{P}^1$-functor structure on $F$ as per Definition 
\ref{defn-Pn-functor}. Then $C$ is an autoequivalence, and the condition 
\eqref{item-spherical-functor-cotwist-is-an-auto-equivalence}
in the definition of a spherical functor holds. 

Let $W \in D(\BbimA)$ be defined by the exact triangle
\begin{equation}
\label{eqn-cone-of-R-to-CL1} R \xrightarrow{\mu_1 L \circ R\action} CL[1] \xrightarrow{} W \xrightarrow{} R[1].
\end{equation}
The adjoints condition implies that when we apply $F$ on the left the
first map in this triangle becomes an isomorphism, so $FW\simeq 0$.
Therefore, $RFW\simeq 0$, so by the exact triangle 
\eqref{eqn-spherical-functors-cotwist-exact-triangle}
we have $CW\simeq W$. Applying 
the autoequivalence $C^{-1}[-1]$ to the triangle
\eqref{eqn-cone-of-R-to-CL1} we get
\begin{equation}
C^{-1}R[-1] \xrightarrow{} L \xrightarrow{} W[-1] \xrightarrow{} C^{-1}R.
\end{equation}
By \cite{AnnoLogvinenko-SphericalDGFunctors} Lemma 5.9 (4) the first map in this triangle is
the composition 
\begin{equation}
\label{eqn-C'R(-1)-to-L-map}
C^{-1}R[-1]
\xrightarrow{\eqref{eqn-dual-cotwist-exact-triangle-DG}R} LFR
\xrightarrow{L\trace} L.
\end{equation}
Since $FW\simeq 0$, it follows that the map  
$$
FC^{-1}R[-1] 
\xrightarrow{F \eqref{eqn-C'R(-1)-to-L-map}}
FL
$$
is an isomorphism as well. By the proof of Theorem 5.1 in
\cite{AnnoLogvinenko-SphericalDGFunctors} the adjoints condition 
and the map $F \eqref{eqn-C'R(-1)-to-L-map}$ being an isomorphism 
imply that $T$ is an equivalence. Since $C$ is also an equivalence, 
$F$ is spherical.

\em ``If''\rm : By the exact triangle
\eqref{eqn-spherical-functors-cotwist-exact-triangle}
the monad $RF$ is a coextension of $\id$ by $C[1]$. 
If $F$ is spherical, then $H = C[1]$ is an autoequivalence. 
Moreover, $H|_{\krn F} \simeq \id[1]$ and thus $H^{n+1}(\krn F) = \krn F$. 


We next observe that the adjoints condition holds for $F$ since it
is the ``cotwist identifies adjoints'' condition in the definition 
of a spherical functor. The strong monad condition
also holds. Indeed, it concerns the maps  
$$ Q_1 Q_j \xrightarrow{\iota \iota} Q_n Q_n \xrightarrow{m} Q_n $$
for $0 \leq j \leq n-1$. We have $n = 1$, thus we only need to
consider the case $j = 0$. There, as noted in 
\S\ref{section-strong-monad-condition-the-general-case}, 
the desired assertion holds automatically 
by Lemma \ref{lemma-maps-m_i0-and-m_0j-are-identity}. 

We would now like to apply Theorem 
\ref{theorem-strong-monad-and-weak-adjoints-imply-pn-functor}, but
for that we need either the 
$\ext^{-1}$-vanishing condition \eqref{eqn-the-minus-one-ext-assumption} 
to hold or the monad multiplication be one-sided on DG level as
per Corollary
\ref{cor-from-strong-monad-to-monad-multiplication-one-sided}\eqref{item-monad-multiplication-is-DG-one-sided}.
There is little chance of the former, since $H$ could be  
an arbitrary autoequivalence. Fortunately, since $n = 1$, 
there are not many ways in which monad multiplication can fail to 
be one-sided on DG-level. 

Indeed, $\bar{Q}_1 \bar{Q}_1$ and $\bar{Q}_1$ are twisted complexes of form 
\begin{equation*}
\begin{tikzcd}
H^2
\ar[dotted,bend left=15]{rr}
\ar{r}
& H \oplus H 
\ar{r}{\action}
&
\underset{\degzero}{\id_\A},
\end{tikzcd}
\end{equation*}
\begin{equation*}
\begin{tikzcd}
& H 
\ar{r}{\action}
&
\underset{\degzero}{\id_\A}.
\end{tikzcd}
\end{equation*}
Thus the map $\bar{Q}_1 \bar{Q}_1 \xrightarrow{m} \bar{Q}_1$ 
being homotopic to a one-sided map is equivalent to the composition 
\begin{equation}
\label{eqn-degree-zero-part-of-monad-multiplication-on-DG-level}
\id_\A \xrightarrow{\iota \iota} 
\bar{Q}_1 \bar{Q}_1 \xrightarrow{m} \bar{Q}_1 
\end{equation}
being homotopic to one-sided map. In $D(\AbimA)$ the composition
\eqref{eqn-degree-zero-part-of-monad-multiplication-on-DG-level}
is identified by the isomorphism $\gamma$ with 
$$ \id_\A \xrightarrow{\action \circ \action} RFRF 
\xrightarrow{R \trace F } RF $$
which is just the adjunction unit $\id_\A \xrightarrow{\action} RF$.
Since $\gamma$ intertwines $\id_\A \xrightarrow{\action} RF$
and $\id_\A \xrightarrow{\iota} Q_n$, 
we conclude that the composition
\eqref{eqn-degree-zero-part-of-monad-multiplication-on-DG-level}
is homotopic to the canonical inclusion of $\id_\A$ 
into $\bar{Q}_1$ which is a one-sided map. 

Thus the strong monad condition and the adjoints condition hold, 
while the monad multiplication is one-sided on DG-level as per 
Corollary
\ref{cor-from-strong-monad-to-monad-multiplication-one-sided}\eqref{item-monad-multiplication-is-DG-one-sided}.
We conclude by Theorem 
\ref{theorem-strong-monad-and-weak-adjoints-imply-pn-functor}
that $F$ is a $\mathbb{P}^1$-functor. 
\end{proof}

The left duals of the adjunction unit $\id_\A \rightarrow RF$ and 
the adjunction counit $FR \rightarrow \id_\B$ are the adjunction
counit $LF \rightarrow \id_\A$ and the adjunction unit $\id_\B
\rightarrow FL$ \cite[Lemma 2.13]{AnnoLogvinenko-SphericalDGFunctors}. 
Thus a functor is spherical if and only if both its adjoints are 
spherical. In particular, spherical functors can easily have 
non-trivial kernels. 

Finally, we have:
\begin{prps}
\label{eqn-P1-twist-is-the-square-of-a-spherical-twist}
We have 
$$ P_F = T^2_F $$
where $P_F$ is the $\mathbb{P}$-twist of $F$ as a
$\mathbb{P}^1$-functor and $T_F$ is the spherical twist of $F$. 
\end{prps}
\begin{proof}
In \cite[Theorem 4.2]{AnnoLogvinenko-BarCategoryOfModulesAndHomotopyAdjunctionForTensorFunctors}
it was shown that $T_F^2$ is isomorphic to a (certain) convolution of 
the three-step twisted complex
$$ FR \xrightarrow{F{\action}R} FRFR \xrightarrow{FR\trace - \trace
FR} FR \xrightarrow{\trace} \id_\B. $$
As $FRFR$ splits as $FR \oplus FHR$ identifying $F{\action}R$ 
with the direct summand inclusion, we can replace the first two terms
in the complex above by $FHR$ and the map into $FR$ by the
composition of the direct summand inclusion with  
$FR \trace - \trace FR$. In other words, by the map  
$\psi\colon FHR \rightarrow FR$ in the definition of a
$\mathbb{P}$-twist. Thus $T^2_F$ is isomorphic to a (certain) convolution 
of the complex 
$$ FHR \xrightarrow{\psi} FR \xrightarrow{\trace} \id_\B.$$
By \cite[Theorem 3.2]{AnnoLogvinenko-OnUniquenessOfPTwists} all
convolutions of this twisted complex are isomorphic and
the $\mathbb{P}$-twist $P_F$ is defined to be this unique convolution.  
Thus $P_F = T_F^2$, as desired.
\end{proof}

\subsection{Extensions by zero}
\label{section-extensions-by-zero}

Let, as in \S\ref{section-ptwists-generalities}, $\A$, $\B$, and $\C$ be
DG categories. Let $F \in D(\AbimB)$ be an enhanced functor 
$D(\A) \rightarrow D(\B)$ which has left and right enhanced 
adjoints $L, R \in D(\BbimA)$. Let $n$ be an
odd integer and let $(H,Q_n, \gamma)$ be a $\mathbb{P}^n$-functor
structure on $F$.

Suppose now that $D(\A)$ fits into a DG-enhanced semi-orthogonal 
decomposition. That is --- suppose we have another DG-category $\C$ and 
a gluing DG-bimodule $N \in D(\CbimA)$. Let $\D$ be the DG-gluing 
of $\A$ and $\C$ along $N$ as per
\cite[\S4-\S5]{Efimov-HomotopyFinitenessOfSomeDGCategoriesFromAlgebraicGeometry}
\cite[\S4]{KuznetsovLunts-CategoricalResolutionsOfIrrationalSingularities}. 
The inclusions of $\A$ and $\C$ into $\D$ then induce inclusions 
of $D(\A)$ and $D(\C)$ into $D(\D)$ which form a semiorthogonal decomposition 
$$ D(\D) = \left<D(\A), D(\C)\right>. $$ 
By construction, for any $X \in D(\C)$ and $Y \in D(\A)$ we have
$$
\homm_{D(\D)}(X, Y)=0, \qquad \homm_{D(\D)}(Y,X) \simeq 
\homm_{D(\A)}(Y,X\ldertimes_\C N).
$$

We now describe the matrix notation for $\D$-modules and 
$\DbimD$-, $\BbimD$-, and $\DbimB$-bimodules which we use throughout
the rest of this section. 
A right $\D$-module $E$ is equivalently described by a matrix of modules 
\begin{equation}
\label{eqn-matrix-description-modD-object-matrix}
\left( 
\begin{matrix}
E_\C & E_\A
\end{matrix} \right)
\quad \quad E_\C \in \modC, \; E_\A \in \modA
\end{equation}
and a closed, degree $0$ structure morphism  
$$ \rho \in \homm_{\A}(E_\C\otimes_\C N, E_\A). $$ 
A morphism $\alpha \in \homm^i_{\D}(E,F)$
is then equivalently described by a pair of degree $i$ morphisms 
$$ 
\begin{pmatrix} 
E_\C & E_\A
\end{pmatrix} 
\xrightarrow{(\alpha_\C \quad \alpha_\A)}
\begin{pmatrix} 
F_\C & F_\A
\end{pmatrix} 
$$
in $\modC$ and $\modA$ which commute with the structure morphisms
of $E$ and $F$ in the following sense:
\begin{equation}
\label{eqn-morphism-in-mod-D-commutative-square-condition}
\begin{tikzcd}
E_\C \otimes_\C N 
\ar{r}{\alpha_\C \otimes \id}
\ar{d}{\rho}
&
F_\C \otimes_\C N 
\ar{d}{\rho}
\\
E_\A
\ar{r}{\alpha_\A}
&
F_\A.
\end{tikzcd}
\end{equation}
The differential and composition are computed componentwise in $\modC$
and $\modA$ respectively, e.g. 
$$ d_{\modD} \left(\alpha_\C \quad \alpha_\A\right) = 
\left(d_{\modC} (\alpha_\C) \quad d_{\modA} (\alpha_\A)\right). $$

Likewise, a $\DbimD$-bimodule $M$ is equivalently described 
by a matrix of bimodules
$$
\begin{pmatrix}
\label{eqn-objects-in-D-mod-D}
\biCMC & \biCMA \\
\biAMC & \biAMA
\end{pmatrix}
$$
and a closed, degree $0$ structure morphism 
\begin{align}
\label{eqn-objects-in-D-mod-D-structure-morphism}
\rho \in \;
&\homm_{\CbimA}\left(\biCMC  \otimes_\C N, \biCMA \right) \oplus
\homm_{\AbimA}\left(\biAMC  \otimes_\C N, \biAMA \right) \oplus
\\
&
\oplus 
\homm_{\CbimC}\left(N \otimes_\A (\biAMC), \biCMC \right) \oplus
\homm_{\CbimA}\left(N\otimes_\A (\biAMA), \biCMA \right)
\end{align}
whose four components make the following diagram commute:
$$
\begin{tikzcd}[column sep={1cm}]
N\otimes_\A (\biAMC) \otimes_\C N
\ar{r}
\ar{d}
&
N\otimes_\A (\biAMA)
\ar{d}
\\
(\biCMC) \otimes_\C N
\ar{r}
&
\biCMA.
\end{tikzcd}
$$
A morphism $\alpha \in \homm^i_{\DbimD}(M,L)$ is equivalently
described by a quadruple of degree $i$ morphisms 
\begin{equation}
\begin{pmatrix}
\biCMC & \biCMA \\
\biAMC & \biAMA
\end{pmatrix}
\xrightarrow{\left(\begin{smallmatrix}\alpha_{\C\C} & \alpha_{\C\A} \\
\alpha_{\A\C} & \alpha_{\A\A}\end{smallmatrix}\right)}
\begin{pmatrix}
\biCLC & \biCLA \\
\biALC & \biALA
\end{pmatrix}
\end{equation}
which commute with the components of the structure morphisms of
$M$ and $L$ in the same sense as in
\eqref{eqn-morphism-in-mod-D-commutative-square-condition}. The
differential and the composition are again computed componentwise. 

The diagonal bimodule $\D$ is given by the matrix 
\begin{equation}
\begin{pmatrix}
\C & N \\
0 & \A
\end{pmatrix},
\end{equation} 
with the obvious structure morphisms. The bar complex
bimodule $\barD$, cf. 
\cite[\S2.2]{AnnoLogvinenko-BarCategoryOfModulesAndHomotopyAdjunctionForTensorFunctors}
is given by the matrix 
\begin{equation}
\begin{pmatrix}
\barC & 
\left\{
\barC \otimes_\C N \otimes_\A \barA 
\xrightarrow{\left(
\begin{smallmatrix}
- \id \otimes \tau \\
\tau \otimes \id 
\end{smallmatrix}
\right)
}
\underset{\degzero}{\left(\barC \otimes_\C N\right) \oplus 
\left(N \otimes_\A \barA\right)}
\right\}
\\
0 & \barA 
\end{pmatrix},
\end{equation} 
whose non-zero structure morphism components are given by the twisted
complex maps
\begin{equation}
\begin{tikzcd}
& 
\barC \otimes_\C N
\ar{d}{
\left(
\begin{smallmatrix}
- \id \\
0
\end{smallmatrix}
\right)
}
\\
\barC \otimes_\C N \otimes_\A \barA 
\ar{r}
&
\underset{\degzero}{\left(\barC \otimes_\C N\right) \oplus 
\left(N \otimes_\A \barA\right)}
\end{tikzcd}
\; \text{ and } \;
\begin{tikzcd}
& 
N \otimes_\A \barA
\ar{d}{
\left(
\begin{smallmatrix}
0
\\
\id 
\end{smallmatrix}
\right)
}
\\
\barC \otimes_\C N \otimes_\A \barA 
\ar{r}
&
\underset{\degzero}{\left(\barC \otimes_\C N\right) \oplus 
\left(N \otimes_\A \barA\right)}. 
\end{tikzcd}
\end{equation}

This matrix description of the bar complex of $\D$ allows us to
give a matrix description of the bar category $\modbarD$
similar to the one of $\modD$ above. The category $\modbarD$ 
has the same objects as $\modD$, thus they are given by the data of 
the matrix \eqref{eqn-matrix-description-modD-object-matrix} and 
a structure morphism $\rho$. We implicitly identify $\rho$
with its image under the inclusion 
$\modD \hookrightarrow \modbarD$, cf. 
\cite[Prop. 3.3]{AnnoLogvinenko-BarCategoryOfModulesAndHomotopyAdjunctionForTensorFunctors}. 
A morphism $\alpha \in \barhom^i_{\D}(E,F)$
is equivalently described by a triple $(\alpha_\C, \alpha_\A,
\alpha_{\C\A})$ of degree $(i,i,i-1)$ morphisms 
\begin{equation*}
\begin{tikzcd}
E_\C 
\ar{d}{\alpha_\C}
& E_\A
\ar{d}{\alpha_\A}
\\
F_\C & F_\A
\end{tikzcd}
\quad \quad
\begin{tikzcd}
E_\C  \bartimes_\C N
\ar[dashed]{dr}{\alpha_{\C\A}}
& 
\\
& F_\A
\end{tikzcd}
\end{equation*}
in $\modbarC$, $\modbarA$, and $\modbarA$, respectively. 
The composition is given by
$$\left(\alpha_\C, \alpha_\A, \alpha_{\C\A}\right) \circ 
\left(\beta_\C, \beta_\A, \beta_{\C\A}\right) = 
\left(\alpha_\C \circ \beta_\C, \alpha_\A \circ \beta_\A, \alpha_\A \circ
\beta_{\C\A} + \alpha_{\C\A} \circ (\beta_\C \bartimes N) \right),$$
and the differential is given by
$$
d\left(\alpha_\C, \alpha_\A, \alpha_{\C\A}\right) 
= \bigl( d\alpha_\C, d\alpha_\A, d\alpha_{\C\A} - (-1)^i\left(\alpha_\A \circ
\rho_E - \rho_F \circ (\alpha_\C \bartimes N)\right) \bigr), $$
with all the differentials, compositions, and tensor products computed 
in the corresponding bar categories. In particular, a triple 
$(\alpha_\C, \alpha_\A, \alpha_{\C\A})$ defines a closed degree $0$ 
morphism in $\modbarD$, and hence a morphism in $D(\D)$, if 
$\alpha_\C$ and $\alpha_\A$ are closed degree $0$ morphisms 
in $\modbarC$ and $\modbarA$, while $\alpha_{\C\A}$ is a degree 
$-1$ morphism in $\modbarA$ whose differential is the commutator 
of the square
\eqref{eqn-morphism-in-mod-D-commutative-square-condition}. In 
other words, we now only ask 
that \eqref{eqn-morphism-in-mod-D-commutative-square-condition}
commutes up to homotopy and fix that homotopy. 

Finally, the category $\DmodbarD$ admits a similar description: 
its objects are $2 \times 2$ bimodule matrices
\eqref{eqn-objects-in-D-mod-D} equipped with a structure morphism 
$\rho$ as in \eqref{eqn-objects-in-D-mod-D-structure-morphism}. 
A degree $i$ morphism $\alpha\colon M \rightarrow L$ 
consists of:
\begin{itemize}
\item $4$ degree $i$ morphisms $\alpha_{\C\C}$, $\alpha_{\C\A}$, 
$\alpha_{\A\C}$, $\alpha_{\A\A}$ between the corresponding matrix
entries of $M$ and $L$, 
\item $4$ degree $i-1$ morphisms $\alpha_{\A\C\C\C}$,
$\alpha_{\A\C\A\A}$, $\alpha_{\A\A\C\A}$, $\alpha_{\C\C\C\A}$
each fitting as the diagonal into the four squares similar to 
\eqref{eqn-morphism-in-mod-D-commutative-square-condition}
formed by the degree $i$ morphisms with the components of
$\rho_M$ and $\rho_L$,
\item $1$ degree $i-2$ morphism 
$ \alpha_{\A\C\C\A}\colon 
N \otimes_\A (\biAMC) \otimes_\C N \rightarrow \biCLA$.
Together with the structure morphisms of $M$ and $L$ this data fits
into a cube
\begin{equation}
\label{eqn-the-data-of-a-morphism-in-DmodbarD}
\begin{tikzcd}[row sep={1.5cm}, column sep={1.5cm}]
&
\biCLC 
\ar{rr}[near start, description]{\rho_{\C\C\C\A}}
& 
& 
\biCLA
\\
\biCMC
\ar{rr}[near start, sloped, description]{\rho_{\C\C\C\A}}
\ar{ur}[near start, sloped, description]{\alpha_{\C\C}}
\ar[dotted]{urrr}[pos=0.1, sloped, description]{\alpha_{\C\C\C\A}}
&& 
\biCMA
\ar{ur}[near start, sloped, description]{\alpha_{\C\A}}
&
\\
&
\biALC
\ar{rr}[near start, sloped, description]{\rho_{\A\C\A\A}}
\ar{uu}[near end, sloped, description]{\rho_{\A\C\C\C}}
& 
&
\biALA
\ar{uu}[near end, sloped, description]{\rho_{\A\A\C\A}}
\\
\biAMC
\ar{rr}[near start, sloped, description]{\rho_{\A\C\A\A}}
\ar{ur}[near start, sloped, description]{\alpha_{\A\C}}
\ar{uu}[near end, sloped, description]{\rho_{\A\C\C\C}}
\ar[dotted]{uuur}[near start, sloped, description]{\alpha_{\A\C\C\C}}
\ar[dotted]{urrr}[near start, sloped, description]{\alpha_{\A\C\A\A}}
\ar[dashed, bend left=15]{uuurrr}[sloped, description]{\alpha_{\A\C\C\A}}
& 
& 
\biAMA
\ar{ur}[near start, sloped, description]{\alpha_{\A\A}}
\ar{uu}[near end, sloped, description]{\rho_{\A\A\C\A}}
\ar[dotted]{uuur}[pos=0.125, sloped, description]{\alpha_{\A\A\C\A}}, 
&
\end{tikzcd}
\end{equation}
\end{itemize}
provided we've suppressed tensoring the bimodules by $N$ on either side, e.g. 
the morphism $\rho_{\A\C\C\C}\colon N \bartimes_\A (\biAMC) \rightarrow \biCMC$
is represented by the arrow from $\biAMC$ to $\biCMC$. 

The composition and the differential are defined similarly to the
above. The closed degree $0$ morphisms are those 
whose degree $0$ components are closed and commute with the structure 
morphisms up to homotopy. The four degree $-1$ components are choices of 
these homotopies, and themselves commute with the structure morphisms
up to homotopy. The degree $-2$ component is a choice of this
homotopy. The $\A\C$ components of all the bimodules we work 
with in this section are always zero, so this degree $-2$ component is
always zero. Thus, in particular, we have the following useful fact:

\begin{lemma}
\label{lemma-matrix-criterion-for-bimodules-to-be-iso-in-D-DbimD}
Let $M, L \in D(\DbimD)$ with $\biAMC = \biALC = 0$. 
If there exists a matrix of isomorphisms 
\begin{equation}
\label{eqn-matrix-of-isomorphisms-ensuring-iso-in-D-DbimD}
\begin{pmatrix}
\iota_{\C\C} & \iota_{\C\A} \\
0 & \iota_{\A\A} 
\end{pmatrix}
\end{equation}
in the corresponding derived categories between the components of 
$M$ and $L$ such that the diagram 
\begin{equation}
\label{eqn-condition-for-upper-triag-bimodules-to-be-iso}
\begin{tikzcd}[row sep={1cm}, column sep={0.25cm}]
&
\left(\biCLC\right) \ldertimes_\C N
\ar{rr}[description]{\rho_{\C\C\C\A}}
& 
& 
\biCLA
\\
\left(\biCMC\right) \ldertimes_\C N
\ar{rr}[sloped, description]{\rho_{\C\C\C\A}}
\ar{ur}[sloped, description]{\iota_{\C\C} \ldertimes \id}
&& 
\biCMA
\ar{ur}[sloped, description]{\iota_{\C\A}}
&
\\
&
& 
&
N \ldertimes_\A \left(\biALA\right)
\ar{uu}[sloped, description]{\rho_{\A\A\C\A}}
\\
& 
& 
N \ldertimes_\A \left(\biAMA\right)
\ar{ur}[sloped, description]{\id \ldertimes \iota_{\A\A}}
\ar{uu}[sloped, description]{\rho_{\A\A\C\A}}
&
\end{tikzcd}
\end{equation}
commutes in $D(\CbimA)$, then $M \simeq L$ in $D(\DbimD)$
\end{lemma}
\begin{proof}
Since $\biAMC = \biALC = 0$, the two squares 
on \eqref{eqn-condition-for-upper-triag-bimodules-to-be-iso} are the
only two remaining faces of the cube defining a morphism 
$L \rightarrow M$ in $\DmodbarD$ as per the diagram 
\eqref{eqn-the-data-of-a-morphism-in-DmodbarD}. For a morphism 
to be closed of degree $0$, its components have to have degrees $0$,
$-1$ and $-2$ and the cube has to homotopy commute. In our case, 
each of the two remaining faces has to homotopy commute, independently
of each other. 

A commutative square in $D(\CbimA)$ can always be lifted to a homotopy
commutative square in $\CmodbarA$. Thus we can lift 
\eqref{eqn-matrix-of-isomorphisms-ensuring-iso-in-D-DbimD}
to the data of a closed degree $0$
morphism $\bar{\iota}\colon M \rightarrow L$ whose degree $0$
components are homotopy equivalences. By Rectangle Lemma 
\cite[Lemma
2.12]{AnnoLogvinenko-BarCategoryOfModulesAndHomotopyAdjunctionForTensorFunctors}
it follows that $\bar{\iota}$ itself is a homotopy equivalence, and thus
$M \simeq L$ in $D(\DbimD)$, as desired. 
\end{proof}

\begin{defn}
The \em extension of $F$ by zero \rm from $D(\A)$ to $D(\D)$ is
the enhanced functor 
$$\widetilde{F}\colon  D(\D) \to D(\B)$$
defined by the bimodule
$
\left(
\begin{smallmatrix}
0
\\
F
\end{smallmatrix}
\right)
 \in D(\DbimB).
$ 
\end{defn}
The compositions of $\widetilde{F}$ with the natural inclusions of
$D(\A)$ and $D(\C)$ into $D(\D)$ are $F$ and $0$, respectively, whence
our choice of the terminology. 
The right adjoint $\widetilde{R}\colon D(\B)\to D(\D)$ of
$\widetilde{F}$ is defined by the bimodule
$$
(\widetilde{F})^\barB
\simeq 
\begin{pmatrix}
0
&
R
\end{pmatrix} \in D(\BbimD),
$$
and its left adjoint $\widetilde{L}\colon D(\B) \rightarrow D(\D)$ is 
defined by the bimodule 
\begin{equation}
(\widetilde{F})^\barD
\simeq
\label{eqn-matrix-form-of-tilde-L}
\begin{pmatrix}
\left(N\bartimes_\A F\right)^\barC[-1] 
\quad 
& 
\quad  
\{\underset{\degzero}{\barhom_\A(F,\A)} \xrightarrow{f \mapsto \alpha
\circ (N \bartimes f) } \barhom_\C(N\bartimes_\A F, N)\}
\end{pmatrix}
\in  D(\BbimD),
\end{equation}
where $\alpha$ is the homotopy equivalence $N \bartimes_\A \barA
\rightarrow N$ as per
\cite[Defn.~3.19]{AnnoLogvinenko-BarCategoryOfModulesAndHomotopyAdjunctionForTensorFunctors}. 
The structure morphism $\rho$ of this bimodule is given by the natural map 
$\eta_\C$ defined in 
\cite[Defn.~3.34]{AnnoLogvinenko-BarCategoryOfModulesAndHomotopyAdjunctionForTensorFunctors}:
$$ 
\left(N\bartimes_\A F\right)^\barC \bartimes_\C N[-1] \xrightarrow{\eta_\C} 
\barhom_\C(N\bartimes_\A F, N)[-1].$$

Finally, we have 
$$
\widetilde{R}\widetilde{F} \simeq 
\left(\begin{matrix}
0 & 0 \\
0 & RF 
\end{matrix}\right) \in D(\DbimD).
$$

Now, suppose that in $D(\CbimA)$ we have 
\begin{equation}
\label{eqn-condition-for-extension-by-zero}
Q_i N \simeq 
\begin{cases}
0, & \quad \text{$i$ odd}, \\
N[i], & \quad \text{$i$ even}. \\
\end{cases}
\end{equation}
From the description of a cyclic coextension on 
the diagram \eqref{eqn-cyclic-coextension-of-id-by-H-of-degree-n}
it follows that in $D(\CbimA)$ 
$$ H^iN \xrightarrow{\sigma_iN}  H^{i-1}N[1] $$
is $0$ for even $i$ and an isomorphism for odd $i$. For each
$0 < 2j+1 \leq n$, we have apriori different isomorphisms 
$$ H^{2j+1}N \xrightarrow{H^{2j}\sigma_1N} H^{2j}N[1], $$ 
$$ H^{2j+1}N \xrightarrow{\sigma_{2j+1}N}  H^{2j}N[1]. $$ 
If $Q_n$ is a truncated twisted tensor algebra, as
described in \S\ref{section-truncated-twisted-tensor-algebras}, we 
have $\sigma_{2j+1} =  H^{2j}\sigma_1 - \sigma_{2j}H$. Since
$\sigma_{2j}N = 0$, the two isomorphisms above are the same:
\begin{equation}
 H^{2j}\sigma_1 N = \sigma_{2j+1}N.
\end{equation}
This is the condition we would ideally like to have. 
It turns out something less would do: in $D(\CbimA)$, the map 
\begin{equation}
\label{eqn-condition-for-extension-by-zero-iso-must-come-from-C}
(H^{2j}\sigma_1 N) \circ (\sigma_{2j+1}N)^{-1} \text{ equals }
H^{2j}N[1] \xrightarrow{\alpha_\C^{-1}} 
H^{2j}N[1]\id_\C
\xrightarrow{H^{2j}N[1]\phi_{2j}}
H^{2j}N[1]\id_\C
\xrightarrow{\alpha_\C}
H^{2j}N[1]
\end{equation}
for some automorphism $\phi_{2j}$ of $\id_\C$. Here $\alpha_\C$
is the natural isomorphism $\C \otimes_\C (-) \rightarrow (-)$.  

\begin{prps} 
\label{prps-extension-by-zero-of-a-pn-functor}
Let $(H, Q_n, \gamma)$ be the structure of a $\mathbb{P}^n$-functor on
$F$ with $n$ odd. If the conditions 
\eqref{eqn-condition-for-extension-by-zero} 
and 
\eqref{eqn-condition-for-extension-by-zero-iso-must-come-from-C}
hold, then there exists a structure of a $\mathbb{P}^n$-functor 
on the extension by zero $\widetilde{F}$ of $F$. 
\end{prps}
\begin{proof}
Let $$
\widetilde{H} = 
\left(\begin{matrix}
\C[1] & N[1] \\
0 & H 
\end{matrix}\right)
\in \DmodbarD, 
$$
with structure morphisms
$$ \C \bartimes_\C N[1] \xrightarrow{\alpha} N[1], $$
$$  N \bartimes_\A H \xrightarrow{\id \bartimes \sigma_1} N[1]. $$
One can readily verify that $\tilde{H}$ is an enhanced autoequivalence of $D(\D)$. 
We further have
$$
\widetilde{H}^i
\simeq
\left(\begin{matrix}
\C[i] & N[i] \\
0 & H^i 
\end{matrix}\right) \in \DmodbarD,
$$
with the structure morphisms being $\alpha$ and
$$  N \bartimes_\A H^{\bartimes i} \xrightarrow{(\id \bartimes \sigma_1)^i} N[i]. $$

Let 
$$
\widetilde{Q}_{2i-1} = 
\begin{pmatrix}
0 &  0\\
0 & \bar{Q}_{2i-1}
\end{pmatrix}
\quad \in \DmodbarD,
$$ 
$$
\widetilde{Q}_{2i} = 
\begin{pmatrix}
\C[2i] & N[2i]\\
0 & \bar{Q}_{2i}
\end{pmatrix}
\quad \in \DmodbarD,
$$
with the structure morphisms for the latter being $\alpha$ and
\begin{equation}
\label{equation-structure-morphism-for-extended-Q_2i}
N \bartimes_\A \bar{Q}_{2i} 
\xrightarrow{\id \bartimes \mu_{2i}}
N \bartimes_\A H^{\bartimes(2i)} 
\xrightarrow{(\id \bartimes \sigma_1)^{2i}}
N[2i]. 
\end{equation}

Define
$$ \widetilde{\gamma} \colon 
\widetilde{Q}_n 
\xrightarrow{\begin{pmatrix} 0 & 0 \\ 0 & \gamma \end{pmatrix}}
\widetilde{R}\widetilde{F}, 
\quad \quad  
\widetilde{\iota_j} \colon 
\widetilde{Q}_j 
\xrightarrow{\begin{pmatrix} 0 & 0 \\ 0 & \iota_j \end{pmatrix}}
\widetilde{Q}_{j+1}.$$
Here and below, whenever we only
specify the degree $i$ components of a degree $i$ map in $\DmodbarD$,  
the higher homotopy components are all taken to be zero. 
By Lemma \ref{lemma-matrix-criterion-for-bimodules-to-be-iso-in-D-DbimD}
the map $\widetilde{\gamma}$ is an isomorphism, and we claim 
that the maps $\widetilde{\iota}_j$ give $\widetilde{Q}_n$ the structure of a degree $n$ 
cyclic coextension of $\id_\D$ by $\widetilde{H}$. 

It suffices to show that the cone of $\widetilde{\iota}_j$ is isomorphic 
$D(\DbimD)$ to $\widetilde{H}^{j}$. The cone of a map in $D(\DbimD)$
can be constructed as its convolution in $\DmodbarD$, thus
$\cone(\widetilde{\iota}_{2i})$ is given by 
\begin{equation}
\label{eqn-convolution-of-extended-iota-odd}
\left\{
\begin{pmatrix}
0 &  0\\
0 & \bar{Q}_{2i-1}
\end{pmatrix}
\xrightarrow{\begin{pmatrix} 0 & 0 \\ 0 & \iota_{2i-1}\end{pmatrix}}
\underset{\degzero}
{\begin{pmatrix}
\C[2i] & N[2i] \\
0 & \bar{Q}_{2i}
\end{pmatrix}}
\right\}
= 
\begin{pmatrix}
\C[2i] &  N[2i]\\
0 & \left\{ \bar{Q}_{2i-1} \xrightarrow{\iota_{2i-1}} 
\underset{\degzero}{\bar{Q}_{2i}} \right\}
\end{pmatrix}
\end{equation}
with structure morphisms $\alpha$ and 
$$
\left\{ N \bartimes_\A \bar{Q}_{2i-1} \xrightarrow{\iota_{2i-1}} 
\underset{\degzero}{N \bartimes_\A \bar{Q}_{2i}} \right\}
\xrightarrow{\begin{pmatrix} 0 & 
\eqref{equation-structure-morphism-for-extended-Q_2i} \end{pmatrix}}
N[2i]. 
$$
By definition of the maps $\iota_\bullet$ and $\mu_\bullet$ in 
\S\ref{section-cyclic-extensions}, the map 
\begin{equation}
\label{eqn-map-from-cone-of-iota-2i-1-to-H^2i}
\left\{ \bar{Q}_{2i-1} \xrightarrow{\iota_{2i-1}} 
\underset{\degzero}{\bar{Q}_{2i}} \right\} 
\xrightarrow{\begin{pmatrix} 0 & \mu_{2i} \end{pmatrix}}
H^{2i}
\end{equation}
is a homotopy equivalence. In the derived categories the matrix of
isomorphisms 
$$
\begin{pmatrix}
\id &  \id\\
0 & \eqref{eqn-map-from-cone-of-iota-2i-1-to-H^2i}
\end{pmatrix}
$$
is readily seen to satisfy the commutation condition 
\eqref{eqn-condition-for-upper-triag-bimodules-to-be-iso}. Thus, by
Lemma \ref{lemma-matrix-criterion-for-bimodules-to-be-iso-in-D-DbimD}
we have $\cone(\widetilde{\iota}_{2i}) \simeq \widetilde{H}^{2i}$, as desired. 
Similarly, $\cone(\widetilde{\iota}_{2i+1})$ is given by 
\begin{equation}
\label{eqn-convolution-of-extended-iota-even}
\left\{
\begin{pmatrix}
\C[2i] & N[2i] \\
0 & \bar{Q}_{2i}
\end{pmatrix}
\xrightarrow{\begin{pmatrix} 0 & 0 \\ 0 & \iota_{2i}\end{pmatrix}}
\underset{\degzero}
{\begin{pmatrix}
0 &  0\\
0 & \bar{Q}_{2i+1}
\end{pmatrix}}
\right\}
= 
\begin{pmatrix}
\C[2i+1] &  N[2i+1]\\
0 & \left\{ \bar{Q}_{2i} \xrightarrow{\iota_{2i}} 
\underset{\degzero}{\bar{Q}_{2i+1}} \right\}
\end{pmatrix}
\end{equation}
with the structure morphisms $\alpha$ and   
$$ 
\left\{ N \bartimes_\A \bar{Q}_{2i} \xrightarrow{\iota_{2i}} 
\underset{\degzero}{N \bartimes_\A \bar{Q}_{2i+1}} \right\}
\xrightarrow{\begin{pmatrix} 
\eqref{equation-structure-morphism-for-extended-Q_2i}
&
0
\end{pmatrix}}
N[2i+1]. 
$$
By construction of the automorphism $\phi_{2j}$ of $\C$, 
the matrix of isomorphisms 
$$
\begin{pmatrix}
\phi_{2j} & \eqref{eqn-condition-for-extension-by-zero-iso-must-come-from-C}
\\
0 &  (0 \quad \mu_{2j+1}) 
\end{pmatrix}
$$
satisfies 
\eqref{eqn-condition-for-upper-triag-bimodules-to-be-iso}, and 
thus $\cone(\widetilde{\iota}_{2i+1}) \simeq \widetilde{H}^{2i + 1}$ 
in $D(\DbimD)$, as desired. 

It remains to check the monad, the adjoints, and 
the highest degree term condition for $(\widetilde{H}, \widetilde{Q}_n,
\widetilde{\gamma})$. Define a $\DbimD$-bimodule $M$ to have property $(\dagger)$ if
$\rho\colon \,_\C M_\C \otimes N \to \,_\C M_\A$ is an isomorphism 
and $_\A M_\C=0$. The bimodules $\widetilde{H}^j$ and $\widetilde{Q}_j$ 
all have this property. This property is preserved by the bar tensor product. 
Finally, the bar tensor product of a
$\DbimD$ bimodule $M$ with property $(\dagger)$ and a $\DbimB$ bimodule
$K$ with $\leftidx{_\C}K_\B = 0$ has the only non-zero component:
$(_\A M_\A) \otimes_\A (\!_\A K_\B)$. 
 
The monad condition for $(\widetilde{H}, \widetilde{Q}_n, \widetilde{\gamma})$ 
is readily seen to be equivalent to the monad condition for $(H,Q_n,\gamma)$:
its every term is the tensor product of $\widetilde{F}$ whose $\CbimB$ component 
is $0$ and a bimodule with property $(\dagger)$. 

Now, by assumptions \eqref{eqn-condition-for-extension-by-zero}
we have $Q_n N \simeq 0$. Since $Q_n\simeq RF$, we have $RFN \simeq 0$.
This implies $FN \simeq 0$. Thus 
the matrix form \eqref{eqn-matrix-form-of-tilde-L} of 
the bimodule $\widetilde{L}$ simplifies to $(0 \ L)$.
With this, all the maps in the adjoints and the highest degree 
term conditions for
$(\widetilde{H}, \widetilde{Q}_n, \widetilde{\gamma})$ 
are simply equal to their counterparts for $(H, Q_n, \gamma)$. 

\end{proof}

\subsection{Cyclic covers}
\label{section-cyclic-covers}

\subsubsection{The setup}
\label{section-cyclic-covers-the-setup}

Let $X$ be a variety over an algebraically closed field $k$ of 
characteristic coprime to $n+1$. Let $f\colon Z \rightarrow X$ be 
a cyclic $(n+1)$-fold cover of $X$ ramified in an effective Cartier 
divisor $D \subset X$. In other words, $f$ is a finite flat map 
of degree $n+1$ and $Z$ carries 
a fibrewise action of the group $\gpmu_{n+1} \subset k$ of 
$(n+1)$-st roots of unity which fixes the divisor 
$E = \frac{1}{n+1} f^{-1}(D) \subset Z$ 
and acts freely outside it. Choose and fix a primitive generator $\sigma
\in \gpmu_{n+1}$ and denote its action by $\sigma\colon Z \rightarrow Z$:

\begin{equation}
\begin{tikzcd}[column sep={2cm},row sep={1cm}] 
E
\ar[hookrightarrow]{r}
\ar{d}{f}[']{1:1\text{ map }} 
&
Z
\ar{d}{f}[']{(n+1):1 \text{ map }}
\ar[loop,out=-20,in=20,distance=20,swap, "\sigma"]
\\
D 
\ar[hookrightarrow]{r}
&
X.
\end{tikzcd}
\end{equation}

Let $\gpmu_{n+1}^\vee$ be the group of characters $\gpmu_{n+1} 
\rightarrow k^\times$. We identify it with  
$\mathbb{Z}/(n+1)$ as follows: for any $i \in \mathbb{Z}/(n+1)$ 
let $\xi_i$ be the character which sends $\sigma$ to $\sigma^i$. 
The direct image $f_* \mathcal{O}_Z$ is a locally free sheaf of 
rank $(n+1)$. It carries a natural action of $\gpmu_{n+1}$ and 
decomposes with respect to it as
\begin{equation}
\label{eqn-the-decomposition-of-f_*-O_Z}
f_* \mathcal{O}_Z \simeq \mathcal{O}_X \oplus \L^{-1} \oplus
\L^{-2} \oplus \dots \oplus \L^{-n}
\end{equation}
for some $\L \in \picr(X)$ with an action of $\gpmu_{n+1}$ by $\xi_{-1}$.
We further have 
\begin{align}
f^*(\L^{-1}) &\simeq \mathcal{O}_Z(-E),
\\
\L^{-(n+1)}  &\simeq \mathcal{O}_X(-D),
\end{align}
and the structure of $\mathcal{O}_X$-algebra on $f_* \mathcal{O}_Z$ 
is given in terms of the decomposition
\eqref{eqn-the-decomposition-of-f_*-O_Z}
by the natural 
isomorphisms 
$\L^{-i} \otimes_{\mathcal{O}_X} \L^{-j} \xrightarrow{\sim} \L^{-(i+j)}$ 
and the inclusion 
$\mathcal{L}^{-(n+1)} \simeq \mathcal{O}_X(-D) \hookrightarrow \mathcal{O}_X$. 

Conversely, this whole setup can be recovered starting from 
$\L \in \picr(X)$ such that $\L^{-(n+1)} \simeq \mathcal{O}_X(-D)$ for
some divisor $D \subset X$. Define a sheaf of
algebras 
\begin{align}
\label{eqn-f_*-O_Z-direct-sum-algebra}
\mathcal{O}_X \oplus \L^{-1} \oplus \dots \oplus \L^{-n}
\end{align}
on $X$ as above and give it a $\gpmu_{n+1}$-equivariant structure where 
$\gpmu_{n+1}$ acts on each $\L^{-i}$ by $\xi_{i}$. Define $Z$ to be its 
relative $\spec$. Then \eqref{eqn-f_*-O_Z-direct-sum-algebra} is 
the sheaf of algebras $f_* \mathcal{O}_Z$ and its 
$\gpmu_{n+1}$-equivariant structure defines an action 
of $\gpmu_{n+1}$ on $Z$. In particular, the map $\sigma\colon Z \rightarrow Z$
corresponds to the endomorphism of $f_* \mathcal{O}_Z$
which is the multiplication by $\sigma^i$ on each $\L^{-i}$. 
This is a helpful point of view, 
as some objects on $Z$ and its fiber products with itself are
easier understood as sheaves on $X$ of modules over 
\eqref{eqn-f_*-O_Z-direct-sum-algebra} and its tensor powers.  

Since $f$ is a finite morphism, by
\cite[\S{III.6}]{Hartshorne-Residues-and-Duality} 
the relative dualizing complex $f^! \mathcal{O}_X $
is the object of $D(Z)$ which corresponds to the following 
object of $D(f_* \mathcal{O}_Z\text{-}\modd)$:
\begin{align*}
\shhomm_X\left( f_* \mathcal{O}_Z, \mathcal{O}_X \right) 
& \simeq 
\shhomm_X
\left( \mathcal{O}_X \oplus \L^{-1} \oplus \dots \oplus \L^{-n},
\mathcal{O}_X \right) 
\simeq 
\\
& \simeq 
\left(\mathcal{O}_X \oplus \L^{-1} \oplus \dots \oplus \L^{-n}\right) \otimes
\L^n
\simeq 
\\
& \simeq f_*(\mathcal{O}_Z) \otimes \L^{n} \simeq 
f_* f^* \mathcal{L}^{n} \simeq 
f_*(\mathcal{O}_Z(nE)). 
\end{align*}
We thus write $\omega_{Z/X}$ for the sheaf $\mathcal{O}_Z(nE)$ and we have
$$ f^!(\mathcal{O}_X) = \omega_{Z/X} = \mathcal{O}_Z(nE).$$

\subsubsection{The geometry of the fiber product $Z \times_X Z$}
\label{section-cyclic-covers-the-geometry-of-Z_times_X_Z}

We next look at the geometry of the fiber product $Z \times_X Z$. 
It is a reducible subvariety of $Z \times Z$ whose $n+1$ irreducible
components are the $\gpmu_{n+1}$-twisted diagonals 
$$ \Gamma_{\id}, \Gamma_{\sigma}, \dots, \Gamma_{\sigma^n}, $$
where for each $\lambda \in \gpmu_{n+1}$ the twisted diagional
$\Gamma_\lambda$ is the graph of the action of $\lambda$ on $Z$:
\begin{align}
\Gamma_\lambda \overset{\text{def}}{=}
\left\{
(z, \lambda z) \; \middle| \; 
z \in Z
\right\}
\subset Z \times_X Z.
\end{align}

Let $\lambda_0, \dots, \lambda_k \in \gpmu_{n+1}$. We have the natural 
restriction map 
\begin{equation}
\label{eqn-Z-times-X-projection-from-k-to-k-1-components}
\mathcal{O}_{\Gamma_{\lambda_0} \cup \dots \cup \Gamma_{\lambda_k}} 
\twoheadrightarrow
\mathcal{O}_{\Gamma_{\lambda_0} \cup \dots \cup \Gamma_{\lambda_{k-1}}} 
\end{equation}
which sends any function on 
$\Gamma_{\lambda_0} \cup \dots \cup \Gamma_{\lambda_k}$ to its
restriction to $\Gamma_{\lambda_0} \cup \dots \cup
\Gamma_{\lambda_{k-1}}$. Its kernel 
are the functions which vanish on 
$\Gamma_{\lambda_0}, \dots, \Gamma_{\lambda_{k-1}}$, 
and, in particular, vanish with the multiplicity of at least $k$ along $E$. 
One can verify this locally on $X$: if $X \simeq \spec R$ for some
$k$-algebra $R$ and the ideal of $D$ is $(w) \subset S$, then 
$Z \simeq \spec R[t]/(t^{n+1} - w)$ 
and $Z \times_X Z \simeq \spec R[t,t'](t^{n+1} - t'^{n+1}, t^{n+1} - w)$ with 
the twisted diagonals having the coprime ideals $(t' - \sigma^k t)$. It follows 
that restricting the kernel of 
\eqref{eqn-Z-times-X-projection-from-k-to-k-1-components}
to $\Gamma_{\lambda_k}$ is injective, and the image is the ideal subsheaf
$\mathcal{O}_{\Gamma_{\lambda_k}}(-kE) \subset 
\mathcal{O}_{\Gamma_{\lambda_k}}$. We thus have
an exact triangle
\begin{equation}
\label{eqn-exact-triag-projection-from-k-to-k-1-components}
\mathcal{O}_{\Gamma_{\lambda_k}}(-kE) 
\hookrightarrow 
\mathcal{O}_{\Gamma_{\lambda_0} \cup \dots \cup \Gamma_{\lambda_k}} 
\twoheadrightarrow
\mathcal{O}_{\Gamma_{\lambda_0} \cup \dots \cup \Gamma_{\lambda_{k-1}}}. 
\end{equation}
Similarly, we have an exact triangle 
\begin{equation}
\label{eqn-exact-triag-projection-from-k-to-1-components-with-twist}
\mathcal{O}_{\Gamma_{\lambda_0} \cup \dots \cup
\Gamma_{\lambda_{k-1}}}(0,(k-1)E)
\hookrightarrow 
\mathcal{O}_{\Gamma_{\lambda_0} \cup \dots \cup \Gamma_{\lambda_k}}(0,kE)
\twoheadrightarrow
\mathcal{O}_{\Gamma_{\lambda_k}}(kE) 
\end{equation}
which is Verdier dual to
\eqref{eqn-exact-triag-projection-from-k-to-k-1-components}
relative to the projection $Z \times_X Z \xrightarrow{\pi_2} Z$. Note
that since we aim to work with Fourier-Mukai kernels as described in 
the next section, we apply the convention where the Verdier duality on 
a fiber product $S_1 \times S_2$ relative to either of the projections 
$\pi_i\colon S_1 \times S_2 \rightarrow S_i$ takes values in 
the derived category $D(S_2 \times S_1)$. Thus the relative Verdier dual of 
$\mathcal{O}_{\Gamma_{\lambda_0} \cup \dots \cup \Gamma_{\lambda_k}}$
with respect to $\pi_2$ is 
$\mathcal{O}_{\Gamma_{\lambda_0} \cup \dots \cup
\Gamma_{\lambda_k}}(0,kE)$. 

By above we have a sequence of projections 
\begin{equation}
\label{eqn-projection-decomposition-of-O-ZxX}
\mathcal{O}_{\Delta \cup \Gamma_{\sigma} \dots \cup \Gamma_{\sigma^n}} 
\twoheadrightarrow
\mathcal{O}_{\Delta \cup \dots \cup \Gamma_{\sigma^{n-1}}}  
\twoheadrightarrow
\dots
\twoheadrightarrow
\mathcal{O}_{\Delta \cup \Gamma_{\sigma}}  
\twoheadrightarrow
\mathcal{O}_{\Delta}  
\twoheadrightarrow
0
\end{equation}
with kernels $\mathcal{O}_{\Gamma_{\sigma^n}}(-nE), \dots, 
\mathcal{O}_{\Gamma_{\sigma}}(-E), \mathcal{O}_\Delta$. Verdier dually
relative to $\pi_2$ we have a filtration 
\begin{equation}
\label{eqn-filtration-of-O-ZxX-(0,E)}
0
\hookrightarrow 
\mathcal{O}_{\Delta}
\hookrightarrow 
\mathcal{O}_{\Delta \cup \Gamma_{\sigma}}(0,E)  
\hookrightarrow 
\dots
\hookrightarrow 
\mathcal{O}_{\Delta \cup \dots \cup \Gamma_{\sigma^{n-1}}}(0,(n-1)E)
\hookrightarrow 
\mathcal{O}_{\Delta \cup \Gamma_{\sigma} \dots \cup
\Gamma_{\sigma^n}}(0, nE)
\end{equation}
with quotients $\mathcal{O}_{\Delta},
\mathcal{O}_{\Gamma_{\sigma}}(E), \dots,
\mathcal{O}_{\Gamma_{\sigma^n}}(nE)$.

\subsubsection{Enhancing the adjunction monads and comonads}
\label{section-cyclic-covers-dg-enhancing-adjunction-monads-and-comonads}

To describe the structure of a $\mathbb{P}^n$-functor on $f_*$
we first need to fix its DG-enhancement, i.e. a Fourier-Mukai kernel. 
Let $F_* \in D(Z \times X)$ and $F^*, F^! \in D(X \times Z)$ be 
the standard kernels of $f_*$, $f^*$, $f^!$
as per \S\ref{section-standard-fourier-mukai-kernels-and-the-key-lemma}. 
We have
\begin{align}
\begin{array}{c c l}
F_*  & \simeq & (\id_Z,f)_* \mathcal{O}_Z, \\
F^*  & \simeq & (\id_Z,f)_* \mathcal{O}_Z, \\ 
F^!  & \simeq & (\id_Z,f)_* \omega_{Z/X}. 
\end{array}
\end{align}

We next compute the adjunction monads and comonads. Recall that by
Proposition \ref{prps-2-categorical-adjunctions-for-standard-kernels}
we have 2-categorical adjunctions $(F^*, F_*)$ and $(F_*, F^!)$. Their 
units and counits are unique up to unique isomorphism, thus we can
speak of \em the \rm unit and \em the \rm counit of each adjunction. 
By Proposition \ref{prps-simplifying-units-and-counit-for-standard-kernels}
we have isomorphisms
\begin{align}
\begin{array}{c c l c l}
\label{eqn-cyclic-cover-direct-sum-decomposition-of-FL}
F^! F_* & \simeq &  \mathcal{O}_{Z \times_X Z}(0, nE) & \quad \quad \quad & 
\in D(Z \times Z), \\
F_* F^! & \simeq & \Delta_* f_* \omega_{Z/X} \simeq \Delta_* \left( \L^n \oplus
\L^{n-1} \oplus \dots \oplus \mathcal{O}_X \right) & \quad \quad \quad & \in D(X \times X), 
\\
F^*F_*  & \simeq &  \mathcal{O}_{Z \times_X Z} & \quad \quad \quad & \in D(Z \times Z), \\
F_* F^*  & \simeq &  \Delta_* f_* \mathcal{O}_Z \simeq  
\Delta_* \left( \mathcal{O}_X \oplus \L^{-1}
\oplus \dots \oplus \L^{-n} \right) & \quad \quad \quad & \in D(X \times X),
\end{array}
\end{align}
and these isomorphisms identify:
\begin{itemize}
\item The adjunction counit $F^*F_* \xrightarrow{\mu} \id_Z$ with the map
\begin{align}
\mathcal{O}_{Z \times_X Z} \rightarrow 
\mathcal{O}_\Delta 
\end{align} 
which is the unit of $(\Delta^*, \Delta_*)$ adjunction, i.e.
the restriction to the diagonal $\Delta \in Z \times Z$. 
\item 
The adjunction unit $\id_X \xrightarrow{\epsilon} F_*F^*$ with the map 
\begin{align}
\Delta_* \mathcal{O}_X  \rightarrow 
\Delta_* \left( \mathcal{O}_X \oplus \L^{-1} \oplus \dots \oplus \L^{-n} \right)
\end{align} 
which is $\Delta_*$ applied to the $(f_*, f^*)$ adjunction unit, 
i.e.~the inclusion of a direct summand. 
\item
The adjunction counit $F_* F^! \xrightarrow{\mu} \id_X$ with the map
\begin{align}
\Delta_* \left( \L^n \oplus
\L^{n-1} \oplus \dots \oplus \mathcal{O}_X \right)
\rightarrow \Delta_* \mathcal{O}_X,
\end{align} 
which is $\Delta_*$ applied to the $(f_*, f^!)$ adjunction counit, 
i.e.~the projection onto a direct summand. 
\item 
The adjunction unit $\id_Z \xrightarrow{\epsilon} F^!F_*$ with the map
\begin{align}
\mathcal{O}_\Delta 
\rightarrow
\mathcal{O}_{Z \times_X Z}(0,nE) 
\end{align}
which is the twist by $(0,nE)$ of the inclusion 
$\mathcal{O}_\Delta(-nE) \hookrightarrow \mathcal{O}_{Z \times_X Z}$
of the ideal sheaf of $\mathcal{O}_{\Gamma_{\id, \sigma} \cup \Gamma_{
\id, \sigma^2} \cup \dots \cup \Gamma_{\id, \sigma^n}}$. In other
words, it is the composition of the inclusion maps in the filtration 
\eqref{eqn-filtration-of-O-ZxX-(0,E)}. 
\end{itemize}

\subsubsection{The structure of a $\mathbb{P}^n$-functor on $f_*$}
\label{section-f_*-as-a-P^n-functor}

Let 
$$ h \overset{\text{def}}{=} \sigma_*(-) \otimes \mathcal{O}_Z(E).$$
It is an autoequivalence of $D(Z)$. Let $H \in D(Z \times Z)$ be 
its standard Fourier-Mukai kernel. We have 
$$ H \simeq \mathcal{O}_{\Gamma_{\sigma}}(E), $$
$$ H^k \simeq \mathcal{O}_{\Gamma_{\sigma^k}}(kE). $$
The filtration 
\eqref{eqn-filtration-of-O-ZxX-(0,E)}
\begin{equation*}
0
\hookrightarrow 
\mathcal{O}_{\Delta}
\hookrightarrow 
\mathcal{O}_{\Delta \cup \Gamma_{\sigma}}(0,E)  
\hookrightarrow 
\dots
\hookrightarrow 
\mathcal{O}_{\Delta \cup \dots \cup \Gamma_{\sigma^{n-1}}}(0,(n-1)E)
\hookrightarrow 
\mathcal{O}_{Z \times_X Z}(0,nE) \simeq F^!F_*
\end{equation*}
has the quotients 
$$\id \simeq \mathcal{O}_{\Delta} ,
H \simeq \mathcal{O}_{\Gamma_{\sigma}}(E), \dots,
H^n \simeq \mathcal{O}_{\Gamma_{\sigma^n}}(nE),$$ 
and thus gives $F^! F_*$ 
the structure of cyclic coextension of $\id$ by $H$ of degree $n$.

\begin{theorem}
\label{theorem-cyclic-cover-as-a-non-split-Pn-functor}
Let $Z$ and $X$ be algebraic varieties. 
Let $f\colon Z \rightarrow X$ be a cyclic cover of degree $n+1$ ramified
in an effective Cartier divisor $D \subset X$. 
Let $E = \frac{1}{n+1}f^{-1}(D)$ on $Z$
and let $\sigma \in \gpmu_{n+1}$ be a primitive
generator of the cyclic group $\gpmu_{n+1}$ whose action on $Z$ permutes
the branches of the cover. 

Let $h: D(Z) \rightarrow D(Z)$ be 
the autoequivalence $\sigma_*(-)\otimes_Z \mathcal{O}_Z(E)$. 
Let $F_*, F^*, F^!,$ and $H$ be the standard enhancements of 
$f_*, f^*, f^!,$ and $h$, as per 
\S\ref{section-standard-fourier-mukai-kernels-and-the-key-lemma}. 
 
Then the structure of a cyclic coextension of $\id$ by $H$ of degree $n$
on the adjunction monad $F^!F_*$ provided by the filtration 
\eqref{eqn-filtration-of-O-ZxX-(0,E)} makes $F_*$ into a $\mathbb{P}^n$-functor. 
\end{theorem}
\begin{proof}
Since $f$ is affine, $\krn F_* = 0$. 
The condition of $H(\krn F_*) = \krn F_*$ is thus trivially fulfilled. 
Next, observe that the $\ext^{-1}$-vanishing condition 
\eqref{eqn-the-minus-one-ext-assumption} holds in our setup 
since for any $1 \leq i \leq n$
\begin{align*}
\homm^{-1}_{D(Z \times Z)}(\id_Z, H^i) = 
\homm^{-1}_{D(Z \times Z)}(\Delta_* \mathcal{O}_Z, \Delta_*
\mathcal{O}_{\Gamma_{\sigma^i}}(iE)) = 0 
\end{align*}
as there are no negative $\ext$s between sheaves. 
Thus, by Theorem 
\ref{theorem-strong-monad-and-weak-adjoints-imply-pn-functor}, 
to show $F_*$ to be a $\mathbb{P}^n$-functor 
it suffices to show that the strong monad condition and 
the weak adjoints condition hold. 

\begin{enumerate}
\item \em Strong monad condition: \rm We need to show that 
for any $0 \leq k \leq n-1$ the map
\begin{equation}
\label{eqn-Q_1Q_k-to-Q_n-map-for-cyclic-cover-direct-image} 
Q_1 Q_k \hookrightarrow  Q_n Q_n \xrightarrow{m} Q_n 
\end{equation}
filters through $Q_{k + 1} \hookrightarrow Q_n$ 
as some map 
\begin{equation}
\label{eqn-Q_1Q_k-to-Q_k+1-map-for-cyclic-cover-direct-image} 
m_k \colon Q_1 Q_k  \rightarrow Q_{k+1} 
\end{equation}
for which there exists an isomorphism 
$H H^{k} \xrightarrow{\sim} H^{k+1}$ which makes 
the following diagram commute:
\begin{equation}
\label{eqn-strong-monad-condition-for-cyclic-cover-direct-image}
\begin{tikzcd}
Q_1 Q_k 
\ar{d}{\mu_1 \mu_k}
\ar{r}{m_k}
&
Q_{k+1}
\ar{d}{\mu_{k+1}}
\\
H H^k 
\ar[dashed]{r}{\sim}
&
H^{k+1}.
\end{tikzcd}
\end{equation}

It is possible to establish this directly, using a local computation.
However, we prefer to give a more geometric proof. For this, we work 
with the left adjoint $F^*$, which is more geometric, and establish 
for it the following equivalent of the strong monad condition above. 

Let $Q'_n = F^* F_* \simeq \mathcal{O}_{Z \times_X Z}$ and 
$Q'_k$ be the terms in the cofiltration 
\eqref{eqn-projection-decomposition-of-O-ZxX}. Let $H' = 
\mathcal{O}_{\Gamma_{\sigma}}(-E)$, thus 
$\mu'_k\colon (H')^k \hookrightarrow Q'_k$ 
are the kernels of the maps in 
\eqref{eqn-projection-decomposition-of-O-ZxX}.
We claim that for $0 \leq k \leq n-1$ the map 
\begin{equation}
\label{eqn-Q'n-to-Q'kQ'1-map}
Q'_n \xrightarrow{m'} Q'_n Q'_n \twoheadrightarrow Q'_k Q'_1, 
\end{equation}
where $m'$ is the comonad comultiplication, filters  
through $Q'_n \twoheadrightarrow Q'_{k+1}$ 
as some map 
\begin{equation}
\label{eqn-Q'_k+1-to-Q'_1Q'_k-map-for-cyclic-cover-direct-image} 
m'_k \colon Q'_{k+1} \rightarrow  Q'_k Q'_1
\end{equation}
for which there exists an isomorphism 
$(H')^{k+1} \xrightarrow{\sim} (H')^{k} H'$ making 
the following diagram commute:
\begin{equation}
\label{eqn-strong-comonad-condition-for-cyclic-cover-direct-image}
\begin{tikzcd}
(H')^{k+1}
\ar[dashed]{r}{\sim}
\ar[hook]{d}{\mu'_{k+1}}
&
(H')^k H'.
\ar[hook]{d}{\mu'_1 \mu'_k}
\\
Q'_{k+1}
\ar{r}{m'_k}
&
Q'_k Q'_1. 
\end{tikzcd}
\end{equation}
The strong monad condition above is the 2-categorical right adjoint
of this statement. In other words, it follows by applying the relative 
Verdier duality functor $\rder\shhomm_{Z \times Z}(-,\pi_2^! \mathcal{O}_Z)$.
The map $m_k$ and the diagram 
\eqref{eqn-strong-monad-condition-for-cyclic-cover-direct-image}
are obtained as the images of the map $m'_k$ and the diagram 
\eqref{eqn-strong-comonad-condition-for-cyclic-cover-direct-image}. 

Let $S_k$ be the subscheme 
$\Delta \cup \Gamma_\sigma \cup \dots \cup \Gamma_{\sigma^k} 
\subset Z \times Z$, thus 
$$ Q'_k  = \mathcal{O}_{S_k} \in D(Z \times Z). $$
The subschemes $S_k$ are flat over both 
projections to $Z$, and therefore 
$$Q'_i Q'_j \simeq \pi_{13 *} \mathcal{O}_{S_j \times_Z S_i} \in D(Z \times Z).$$
We have the commutative diagram of scheme maps
\begin{equation}
\begin{tikzcd}
S_1 \times_Z S_k 
\ar[hook]{r} 
\ar{d}{\pi_{13}}
&
S_n \times_Z S_n 
\ar[hook]{r} 
\ar{d}{\pi_{13}}
& 
Z \times Z \times Z
\ar{d}{\pi_{13}}
\\
S_{k+1} 
\ar[hook]{r} 
&
S_n 
\ar[hook]{r} 
&
Z \times Z.
\end{tikzcd}
\end{equation} 
It induces the commutative square of objects in $D(Z \times Z)$
\begin{equation}
\label{eqn-cyclic-covers-commutative-square-filtering-comonad-comult}
\begin{tikzcd}
Q'_k Q'_1 
&
Q'_n Q'_n 
\ar[two heads]{l} 
\\
Q'_{k+1} 
\ar{u}{\id \rightarrow \pi_{13 *} \pi_{13}^*}
&
Q'_n 
\ar{u}[']{\id \rightarrow \pi_{13 *} \pi_{13}^*}
\ar[two heads]{l}.
\end{tikzcd}
\end{equation}
It follows from the description of the adjunction unit 
$\id_Z \rightarrow F^* F_*$ in 
\S\ref{section-cyclic-covers-dg-enhancing-adjunction-monads-and-comonads}
that the left vertical arrow in  
\eqref{eqn-cyclic-covers-commutative-square-filtering-comonad-comult}
is the comonad comultiplication. We thus have shown that
the map \eqref{eqn-Q'n-to-Q'kQ'1-map} does filter through $Q'_{k+1}$
and we define $m'_k$ to be the right vertical arrow in 
\eqref{eqn-cyclic-covers-commutative-square-filtering-comonad-comult}, 
i.e. the adjunction unit 
\begin{equation}
\label{eqn-cyclic-covers-the-map-m'k-explicitly}
\mathcal{O}_{S_{k+1}} \rightarrow \pi_{13 *} \mathcal{O}_{S_1
\times_Z S_k}. 
\end{equation}

Finally, the subsheaf 
$(H')^{k+1} \hookrightarrow Q'_{k + 1}$ is the ideal sheaf of the subscheme 
$$ \overline{S_{k+1} \setminus \Gamma_{\sigma^{k+1}}} \subset S_{k+1}, $$
which is the union of all irreducible components 
complementing $\Gamma_{\sigma^{k+1}}$ in $S_{k+1}$. 
Similarly, the subsheaf $(H')^k H' \hookrightarrow Q'_k Q'_1$ is
the image under $\pi_{13 *}$ of the ideal sheaf of
$$ \overline{S_{1} \times_Z S_{k} \setminus \Gamma_{\sigma} \times_Z
\Gamma_{\sigma^k}} \subset S_{k+1}. $$ 
Since $\pi_{13}: S_{1} \times_Z S_k \rightarrow S_{k+1}$
maps the component $\Gamma_{\sigma} \times_Z \Gamma_{\sigma^k}$
isomorphically onto the component $\Gamma_{\sigma^{k+1}}$, 
the adjunction unit \eqref{eqn-cyclic-covers-the-map-m'k-explicitly}
restricts to an isomorphism on the ideal sheaves of their
complementing components. That is, the map $m'_k$ restricts to 
an isomorphism $(H')^{k+1} \xrightarrow{\sim} (H')^k H'$, as required. 

\item \em Weak adjoints condition: \rm Straightforward. 
We have $F^* \simeq (\id_Z,f)_* \mathcal{O}_Z$, 
$F^! \simeq (\id_Z,f)_* \mathcal{O}_Z(nE)$ and 
$H^n \simeq \mathcal{O}_{\Gamma_{\sigma^n}}(nE)
\simeq (\id_Z, \sigma^n)_* \mathcal{O}_Z(nE) \simeq (\sigma, \id_Z)_*
\mathcal{O}_Z(nE)$. Thus, by the Key Lemma  
(Lemma \ref{lemma-pullbacks-and-pushforwards-of-Fourier-Mukai-kernels})
$$ H^n F^* \simeq (f \times \id_Z)_* (\sigma,\id_Z)_* \mathcal{O}_Z(nE)
\simeq F^!. $$
\end{enumerate}
\end{proof}

\subsubsection{The structure of a split $\mathbb{P}^n$-functor on $f^*$
and on $f^!$}
\label{section-f^*-and-f^!-as-split-P^n-functor}

Let $J$ be the standard kernel of the autoequivalence 
$(-) \otimes \mathcal{L}^{-1}$ of $D(X)$, as per 
\S\ref{section-standard-fourier-mukai-kernels-and-the-key-lemma}. 
Then we have
$$ J \simeq \Delta_* \mathcal{L}^{-1} \in D(X \times X), $$
$$ J^k \simeq \Delta_* \mathcal{L}^{-k} \in D(X\times X) . $$
The direct sum decomposition 
\eqref{eqn-cyclic-cover-direct-sum-decomposition-of-FL}
$$ F_*F^* \quad \simeq \quad \Delta_* f_* \mathcal{O}_Z \simeq  
\Delta_* \left( \mathcal{O}_X \oplus \L^{-1}
\oplus \dots \oplus \L^{-n} \right) \in D(X \times X) $$
gives the adjunction monad $F_* F^*$ a structure of a (trivial) 
cyclic coextension of $\id$ by $J$ of degree $n$.

On the other hand, the right adjoint $f^!$ of $f_*$ has itself a
right adjoint $f^{!!} = f_*(\mathcal{O}_Z(nE) \otimes - )$.
Let $F^{!!} \in D(Z \times X)$ be the standard kernel of $f^{!!}$
as per \S\ref{section-standard-fourier-mukai-kernels-and-the-key-lemma}.
We have
$$ F^{!!} \simeq (\id_Z,f)_* \mathcal{O}_Z(nE) \in D(Z \times X). $$
The same computation as in 
\S\ref{section-cyclic-covers-dg-enhancing-adjunction-monads-and-comonads}
shows that we have 
\begin{equation}
\label{eqn-cyclic-cover-direct-sum-decomposition-of-RR'}
F^{!!}F^! \simeq \Delta_* f_* \mathcal{O}_Z(nE) \simeq
\Delta_*
\left( \L^{n} \oplus \dots \oplus \L \oplus \mathcal{O}_X \right) 
\quad \in D(X \times X) 
\end{equation}
and the adjunction counit is again the projection onto the direct summand 
$\Delta_* \mathcal{O}_X$. This gives the adjunction monad $F^{!!}F^!$ 
as well the structure of a (trivial) cyclic coextension of $\id$ by
$J^{-1}$ of degree $n$.

\begin{prps}
\label{prps-cyclic-cover-as-a-split-Pn-functor}
Under the same assumptions as those of Theorem 
\ref{theorem-cyclic-cover-as-a-non-split-Pn-functor} the direct sum
decompositions \eqref{eqn-cyclic-cover-direct-sum-decomposition-of-FL}
and \eqref{eqn-cyclic-cover-direct-sum-decomposition-of-RR'} of 
the adjunction monads $F_* F^*$ and $F^{!!}F^!$ 
give $F^*$ and $F^!$ 
the structure of (split) $\mathbb{P}^n$-functors. 
\end{prps}
\begin{proof}
We only treat the case of $F^*$, as the case of $F^{!}$ is similar. 
Since $F_* F^*$ has $\id_X$ as a direct summand, we have $\krn F^* = 0$
and the condition $J(\krn F^*) = \krn F^*$ is trivially fulfilled. 
As before, the $\ext^{-1}$-vanishing condition 
\eqref{eqn-the-minus-one-ext-assumption} holds in our setup 
since there are no negative $\ext$s between sheaves. 
Thus it suffices, again, to demonstrate that the strong monad condition 
and the weak adjoints condition hold.
\begin{enumerate}
\item \em Strong monad condition: \rm 
We have 
$$ F_* F^* \simeq \id \oplus J \oplus J^2 \oplus \dots \oplus J^n $$
and in the split case we need to establish that for any 
$0 \leq k \leq n-1$ the map 
$$ J J^k \hookrightarrow F_* F^* F_* F^* \xrightarrow{m} F_* F^* $$
filters through the inclusion 
$$ \id \oplus J \oplus \dots \oplus J^{k+1}
\hookrightarrow 
\id \oplus J \oplus \dots \oplus J^{n} \simeq F_* F^*
$$
with the component $J J^k \rightarrow J^{k+1}$ being an isomorphism.
We claim that, in fact, $J J^k \hookrightarrow F_* F^*$ filters through 
$J^{k+1} \hookrightarrow F_* F^*$ as the isomorphism $J J^k \simeq J^{k+1}$.  

By the Key Lemma  
(Lemma \ref{lemma-pullbacks-and-pushforwards-of-Fourier-Mukai-kernels})
we have 
$$ F_* F^* F_* F^* \simeq (f,f)_* \mathcal{O}_{Z \times_X Z} \simeq
\Delta_* \phi_* \mathcal{O}_{Z \times_X Z} $$
where $\phi$ is the map $Z \times_X Z \rightarrow X$ of projection 
to $X$ and thus $(f,f) = \Delta \circ \phi$. 
Since 
$$Z \times_X Z \simeq
\relspec_{\mathcal{O}_X} f_* \mathcal{O}_Z \otimes f_* \mathcal{O}_Z $$
we have
$$ \phi_* \mathcal{O}_{Z \times_X Z} \simeq  
f_* \mathcal{O}_Z \otimes f_* \mathcal{O}_Z 
\simeq (\mathcal{O}_X \oplus \dots \oplus \L^{-n}) \otimes 
(\mathcal{O}_X \oplus \dots \oplus \L^{-n}) $$
and the map $J J^k \hookrightarrow F_* F^* F_* F^*$ corresponds
to the image under $\Delta_*$ of the map 
$$ \L \otimes \L^{-k} \rightarrow 
(\mathcal{O}_X \oplus \dots \oplus \L^{-n}) \otimes 
(\mathcal{O}_X \oplus \dots \oplus \L^{-n}) $$
given by the tensor product of the direct summand inclusions. 

On the other hand, the monad multiplication 
$m$
is the image under $(f,f)_*$ of the restriction map
\begin{equation}
\label{eqn-restriction-from-O_Z-times_X-Z-to-O_Delta}
\mathcal{O}_{Z \times_X Z} \twoheadrightarrow \mathcal{O}_\Delta. 
\end{equation}
The image of this restriction under $\phi$ is the map  
$$  f_* \mathcal{O}_Z \otimes f_* \mathcal{O}_Z 
\rightarrow f_* \mathcal{O}_Z $$
given by the $\mathcal{O}_X$-algebra multiplication on 
$ f_* \mathcal{O}_Z = \mathcal{O}_Z \oplus \dots \oplus \L^{-n}$.
We conclude that the monad multiplication map corresponds
to the image under $\Delta_*$ of the map  
$$ (\mathcal{O}_X \oplus \dots \oplus \L^{-n}) \otimes 
(\mathcal{O}_X \oplus \dots \oplus \L^{-n}) 
\rightarrow 
(\mathcal{O}_X \oplus \dots \oplus \L^{-n}) $$
given by the maps 
$$ \L^{-i} \otimes \L^{-j} \xrightarrow{\sim} \L^{-(i+j)}
\hookrightarrow \L^{-(i + j \mod n+1)} $$
where the last inclusion is induced by the inclusion $\L^{-(n+1)}
\simeq \mathcal{O}_X(-D) \hookrightarrow \mathcal{O}_X. $

Since $k \leq n-1$, the above means that the map 
$$ J J^k \rightarrow F_* F^* $$
is given by the image under $\Delta_*$ of the map 
$$ \L^{-1} \otimes \L^{-k} \xrightarrow{\sim} \L^{-(k+1)} 
\hookrightarrow \mathcal{O}_X \oplus \L^{-1} \dots \oplus \L^{-n}, $$
whence the claim. 

\item \em Weak adjoints condition: \rm Let $F_!$ be
the standard kernel of the left adjoint $f_!$ of $f^*$. 
We have
$$ F_! \simeq (\id_Z,f)_* \mathcal{O}_Z(nE) \in D(Z \times X). $$ 
Since $J$ is the standard kernel of $(-) \otimes \mathcal{L}^{-1}$, 
we have by the Key Lemma  
(Lemma \ref{lemma-pullbacks-and-pushforwards-of-Fourier-Mukai-kernels})
\begin{align*}
J^n F_! \simeq 
(\id_Z,f)_* \mathcal{O}_Z(nE) \otimes \pi_2^* \L^{-n}
\simeq 
(\id_Z,f)_* \left( \mathcal{O}_Z(nE) \otimes f^* \L^{-n} \right)
\simeq (\id_Z,f)_* \mathcal{O}_Z = F_*. 
\end{align*}
\end{enumerate}
\end{proof}

\subsubsection{The $\mathbb{P}$-twists for $f_*$, $f^*$ and $f^!$}
\label{section-pn-twists-for-f_*-f^*-f^!}

In this section we compute the $\mathbb{P}$-twists of $f_*$, $f^*$, 
$f^!$ with the $\mathbb{P}^n$-functor structures defined in \S
\ref{section-f_*-as-a-P^n-functor}-\ref{section-f^*-and-f^!-as-split-P^n-functor}:

\begin{prps}
Under the assumptions of Theorem  
\ref{theorem-cyclic-cover-as-a-non-split-Pn-functor} and
Prop.~\ref{prps-cyclic-cover-as-a-split-Pn-functor} 
the $\mathbb{P}$-twists $P_{F_*}$, $P_{F^*}$ and $P_{F^!}$ 
of the $\mathbb{P}^n$-functors $F_*$, $F^*$, and $F^!$ 
are:
\begin{align}
P_{F_*} & \simeq J^{-(n+1)}[2], \\
P_{F^*} & \simeq H^{-(n+1)}[2],  \\
P_{F^!} & \simeq H^{n+1}[2]. 
\end{align}
\end{prps}

\begin{proof}

By definition, the $\mathbb{P}$-twist $P_{F_*}$ is the unique convolution of 
the following two step complex in $D(X \times X)$:
$$ F_*HF^! \xrightarrow{\psi} F_* F^! \xrightarrow{\mu} \id_X. $$
where $\psi$ is the map defined in Lemma 
\ref{lemma-psi-is-independent-of-the-choice-of-phi}. 
We have computed in \S\ref{section-cyclic-covers-dg-enhancing-adjunction-monads-and-comonads}:
\begin{align*}
\id &\simeq \Delta_* \mathcal{O}_X \\ 
 F_* F^! & \simeq \Delta_* \left(\L^{n} \oplus \dots \oplus \L \oplus
\mathcal{O}_X \right)
\end{align*}
with $\mu: F_* F^! \rightarrow \id$ being the projection onto the direct
summand. Computing similarly, we obtain 
\begin{align*}
F_*HF^! & \simeq \Delta_* \left( \L^{n+1} \oplus \dots \oplus \L \right). 
\end{align*} 
The map $\psi$ is then identified with the image under $\Delta_*$ of 
a certain map  
\begin{align}
\label{eqn-cyclic-cover-psi-in-terms-of-line-bundles}
\L^{n+1} \oplus \dots \oplus \L 
\rightarrow 
\L^{n} \oplus \dots \oplus \L \oplus \mathcal{O}_X. 
\end{align}
One could compute this directly by working, since everything is
affine over $X$, with sheaves of algebras on $X$ and sheaves of modules 
over them. One would obtain that $\psi$ is linear combination of canonical
inclusions $\L^i \hookrightarrow \L^j$ for every $i \leq j$ whose  
coefficients are various powers of $\sigma$. However it is an involved
computation, and there is the following simpler argument.  

Taking the cone of $\mu\colon F_* F^! \rightarrow \id$ first, 
$P_{F_*}$ is isomorphic to the image under $\Delta_*[1]$ of the cone of a map 
\begin{equation}
\label{eqn-cyclic-cover-psi-in-terms-of-line-bundles-iso-part}
\L^{n+1} \oplus \L^{n} \oplus \dots \oplus \L 
\rightarrow 
\L^{n} \oplus \dots \oplus \L .
\end{equation}
By the monad condition 
$\eqref{eqn-cyclic-cover-psi-in-terms-of-line-bundles-iso-part}F_*$
has a left semi-inverse. So does the natural projection from the LHS
to the RHS of
\eqref{eqn-cyclic-cover-psi-in-terms-of-line-bundles-iso-part}. 
Arguing as in Prop.~\ref{prps-R-HnL-canonical-map} we then see that \eqref{eqn-cyclic-cover-psi-in-terms-of-line-bundles-iso-part} must also have a left semi-inverse. 

On the other hand, $\homm_{X}(\L^i, \L^j) = 0$ for $i > j$. In other 
words, both  
\eqref{eqn-cyclic-cover-psi-in-terms-of-line-bundles-iso-part}
and its left semi-inverse are one-sided with respect to the grading
by the powers of $\L$. It follows immediately that the diagonal 
components $\L^i \rightarrow \L^i$ of
\eqref{eqn-cyclic-cover-psi-in-terms-of-line-bundles-iso-part}
are isomorphisms, and thus the cone of 
\eqref{eqn-cyclic-cover-psi-in-terms-of-line-bundles-iso-part}
is $\L^{n+1}[1]$. We conclude that
$$ P_{F_*} \simeq \Delta_* \L^{n+1}[2] \simeq J^{-(n+1)}[2]. $$

Similarly, the $\mathbb{P}$-twist $P_{F^*}$ of $F^*$ is the unique convolution 
of the following complex in $D(Z \times_X Z)$:
$$ F^*JF_* \xrightarrow{\psi} F^* F_* \xrightarrow{\mu} \id_Z. $$
As per \S\ref{section-cyclic-covers-dg-enhancing-adjunction-monads-and-comonads}
we have $F^* F_* \simeq \mathcal{O}_{Z \times_X Z}$ and therefore
$$ F^*JF_* \simeq p^* \L^{-1} $$
where $p$ is the projection $Z \times_X Z \rightarrow X$.
The counit $\mu$ is the natural restriction 
$\mathcal{O}_{Z \times_X Z} \twoheadrightarrow \Delta_* \mathcal{O}_Z$ 
of sheaves. 
The map $\psi$ is also easy to describe, due to the monad
$F_* F^*$ being split. By definition, $\psi$ is the composition 
$$ F^*JF_* \rightarrow F^* F_* F^* F_* 
\xrightarrow{F^* F_* \mu - \mu F^* F_* } F^* F_* $$
where the first map comes, as described in \S\ref{section-ptwists}, 
from $F^* \xrightarrow{F^*\epsilon} F^* F_* F^*$ being split and $F^*J$ 
having a canonical map into its cone. As per the proof of Proposition 
\ref{prps-cyclic-cover-as-a-split-Pn-functor},
we have 
$$ 
F^* F_* F^* F_* \simeq 
p^* \left(\mathcal{O}_X \oplus \L^{-1} \oplus \dots \oplus \L^{-n} \right), $$
and $F^*JF_* \rightarrow F^*F_*F^*F_*$ corresponds to 
the direct summand inclusion. 

On the other hand, the maps $F^*F_*\mu$ and $\mu F^*F_*$ correspond
to the maps $\sum_{i=0}^n l_i$ and $\sum_{i=0}^n r_i$, respectively,
\begin{align*}
l_i\colon  p^* \L^{-i} \simeq \mathcal{O}_{Z \times_X Z}(-iE,0)
\hookrightarrow \mathcal{O}_{Z \times_X Z}, \\
r_i\colon  p^* \L^{-i} \simeq \mathcal{O}_{Z \times_X Z}(0, -iE)
\hookrightarrow \mathcal{O}_{Z \times_X Z}. 
\end{align*}
We conclude that $\psi$ can be identified with the map 
$$ p^* \L^{-1} \xrightarrow{l_1 - r_1} \mathcal{O}_{Z \times_X Z}. $$
Its kernel and co-kernel are readily seen to be the maps
\begin{align*}
\mathcal{O}_{\Delta}(-nE) \otimes p^* \mathcal{L}^{-1}
& \hookrightarrow  p^* \mathcal{L}^{-1},  \\
\mathcal{O}_{Z \times_X Z} \twoheadrightarrow \mathcal{O}_{\Delta},
\end{align*}
which are the inclusion of the sections which vanish outside the
diagonal and the restriction to the diagonal, respectively. Thus,
with the identifications as above, the counit $\mu$
coincides with the cokernel of $\psi$. Hence 
$$ P_{F^*} \simeq \mathcal{O}_{\Delta}(-nE) \otimes p^* \mathcal{L}^{-1}
[2] \simeq \mathcal{O}_{\Delta}(-(n+1)E)[2] \simeq H^{-(n+1)}[2].$$
\end{proof}

\subsection{Family $\mathbb{P}$-twists and Mukai flops}
\label{section-family-p-twists}

\subsubsection{The setup}
\label{section-family-the-setup}

Let $Z$ be a smooth projective variety of dimension $m$ and let $\mathcal{V}$ be 
a vector bundle of rank $n+1$ over $Z$. Let $P$ denote the
projectification $\mathbb{P}\mathcal{V}$ of $\mathcal{V}$ and let
$$\pi\colon P \twoheadrightarrow Z $$
be the structure map, it is a flat $\mathbb{P}^n$-fibration over $Z$.
Let $X$ be a smooth projective variety and let  
$$ \iota \colon P \hookrightarrow X $$
be a closed embedding with 
$$ \mathcal{N}_{P/X} \simeq  \Omega^1_{P/Z}. $$
Since $P$ and $X$ are both smooth, $\iota$ is a regular immersion 
\cite[Prop. 1.10]{Berthelot-ImmersionsRegulieresEtCalculDuKDUnSchemaEclate}. 
  
We have 
\begin{equation}
\quad\quad\quad
\begin{tikzcd}[column sep={1cm}, row sep={1cm}]
P 
\ar[hookrightarrow]{r}{\iota}
\ar[twoheadrightarrow]{d}[']{\pi}
&
X.
\\
Z
&
\end{tikzcd}
\end{equation}
and we define $f_k: D(Z) \rightarrow D(X) $ to be the exact functor  
$$ \iota_* \circ \left(\mathcal{O}_{P}(k)
\otimes \pi^*(-)\right). $$

By \cite[\S{III.2}, \S{III.7}]{Hartshorne-Residues-and-Duality} 
the relative dualizing complexes of $P$ over $Z$ and 
$X$ are
\begin{align} 
\pi^! \mathcal{O}_Z 
&\simeq 
\omega_{P/Z}[n]
\simeq 
\mathcal{O}_{P}(-n-1)[n]
\\
\iota^! \mathcal{O}_X 
& \simeq 
\omega_{P / X} [-n]
\simeq 
\wedge^n \mathcal{N}_{P / X}[-n]
\simeq
\mathcal{O}_{P}(-n-1)[-n].
\end{align} 

Suppose now that there exists the Mukai flop 
$$ X \xleftarrow{\beta} \tilde{X} \xrightarrow{\gamma} X'. $$
Here $\beta$ is the blowup of $X$ along $P$ and its exceptional divisor
$E \subset \tilde{X}$ can therefore be identified with the incidence
subvariety $F \subset \mathbb{P}\mathcal{V} \times_Z \mathbb{P}\mathcal{V}^\vee$. 
By saying that the Mukai flop exists, we mean there existing $\gamma$ which 
blows down $E$ to $P' \subset X'$ in a way that can 
be identified with the projection of 
$F$ onto $\mathbb{P}\mathcal{V}^\vee$.  

In general, the naive flop functor 
$$ \gamma_* \circ \beta^*\colon D(X) \rightarrow D(X') $$
is not a derived equivalence. Namikawa 
\cite{Namikawa-MukaiFlopsAndDerivedCategories} and 
Kawamata \cite[\S5]{Kawamata-DEquivalenceAndKEquivalence} described
how to modify it to obtain one. The natural map 
$$ \tilde{X} \xrightarrow{\beta, \gamma} X \times X' $$
is a closed immersion, thus we can consider $\tilde{X}$ to be a
subvariety of $X \times X'$. Define 
$$ 
\hat{X} \overset{\text{def}}{=} \tilde{X} \cup (P \times_Z P') \subset X \times X'
$$
and denote as folows that natural projections
$$ \hat{\beta}\colon \hat{X} \rightarrow X
\quad \quad \text{ and } \quad \quad 
\hat{\gamma}\colon \hat{X} \rightarrow X'. $$
The functors 
$\hat{\gamma}_* \circ \hat{\beta}^*$ and $\hat{\beta}_* \circ
\hat{\gamma}^*$ are derived equivalences $D(X) \rightleftarrows D(X')$. 
More generally, let $\mathcal{L} \in \picr(\hat{X})$ be the 
line bundle which restricts to $\tilde{X}$ as
$\mathcal{O}_{\tilde{X}}(E)$ and to $P \times_Z P'$ as $\mathcal{O}_{P
\times_Z P'}(-1,-1)$. For any $k \in \mathbb{Z}$ the functors 
\begin{equation}
kn_{i} 
\quad \overset{\text{def}}{=} \quad 
\hat{\gamma}_* \left(\mathcal{L}^i \otimes
\hat{\beta}^*(-)\right) \colon D(X) \rightarrow D(X')
\end{equation}
\begin{equation}
kn'_{j} 
\quad \overset{\text{def}}{=} \quad
\hat{\beta}_* \left(\mathcal{L}^j \otimes
\hat{\gamma}^*(-)\right) \colon D(X') \rightarrow D(X)
\end{equation}
are the \em Kawamata-Namikawa derived equivalences \rm associated 
to the Mukai flop. These are not mutually inverse --- the composition
$kn'_j \circ kn_i$ is a non-trivial derived autoequivalence of $D(X)$
which we can think of as derived monodromy of the flop. Below we 
prove that for appropriately chosen $i$ and $j$ this monodromy
is a $\mathbb{P}$-twist for a natural $\mathbb{P}^n$-functor structure
on $f_k$. 

Let $KN_i$ and $KN'_j$ be the standard enhancements of $kn_{i}$ and
$kn'_{j}$ as per 
\S\ref{section-standard-fourier-mukai-kernels-and-the-key-lemma}. 
We then have 
$$ KN_i \simeq \mathcal{O}_{\hat{X}}\otimes \mathcal{L}^i 
\quad \quad \in \quad D(X \times X'), $$
$$ KN_j \simeq \mathcal{O}_{\hat{X}}\otimes \mathcal{L}^j 
\quad \quad \in \quad D(X' \times X). $$

\subsubsection{Enhancing adjunction monads and comonads for $\iota_*$} 
\label{section-family-pn-twist-enhanced-adjunction-monads-and-comonads-for-a-regular-immersion}

To compute enhancements of the adjunction monads and comonads for 
$F_k$, we first compute these for $\iota_*$.   
This is completley general, so for the rest of this sections
the reader may assume $X$ and $P$ to be arbitrary smooth varieties
with a closed immersion $\iota\colon P \hookrightarrow X$ of
codimension $n$. 
Let $I_*$, $I^*$ and $I^!$ be the standard enhancements of
$\iota_*$, $\iota^*$ and $\iota^!$ as per 
\S\ref{section-standard-fourier-mukai-kernels-and-the-key-lemma}:
\begin{equation}
I_* \simeq (\id_P \times \iota)_*
\mathcal{O}_{\Delta(P)} \in D(P \times X), 
\end{equation}
\begin{equation}
I^* \simeq (\id_X \times \iota)^*
\mathcal{O}_{\Delta(X)} \in D(X \times P), 
\end{equation}
\begin{equation}
I^! \simeq (\id_X \times \iota)^*
\mathcal{O}_{\Delta(X)} \in D(X \times P).
\end{equation}
\begin{prps}
\label{prps-adjunction-monads-and-comonads-for-a-closed-immersion}
\begin{enumerate}
\item 
\label{item-JI-object-and-the-adjunction-unit}
$I^! I_*$ is an object in 
$D(P \times P )$ with 
\begin{equation}
\label{eqn-cohomology-sheaves-of-JI}
H^i(I^! I_*) =  
\begin{cases}
\Delta_* \wedge^i \mathcal{N}_{P/X}  
& 
0 \leq i \leq n, \\
0 & \text{ otherwise,}
\end{cases}
\end{equation}
and the adjunction unit
$\id_{P} \xrightarrow{\epsilon} I^! I_*$
is the inclusion of the leftmost (degree $0$) cohomology. 
\item
\label{item-IJ-object-and-the-adjunction-counit}
$I_* I^!$ is an object of $D(X \times X)$ with 
\begin{equation}
I_* I^! \simeq  
\Delta_* \iota_* \iota^! \mathcal{O}_X 
\simeq
\Delta_* \iota_* \omega_{P/X} [-n]
\end{equation}
and the adjunction counit $I_* I^!  \xrightarrow{\mu} \id_{X}$
is given by the morphism 
$$ \Delta_* \iota_* \iota^! \mathcal{O}_X \rightarrow \Delta_* \mathcal{O}_X $$
which is an instance of 
the functorial adjunction counit $\iota_* \iota^!  \rightarrow \id$. 
In other words, by the element of 
$\ext^n_{X}(\iota_* \omega_{P/X}, \mathcal{O}_X)$ 
relative Serre dual to $\id \in \homm_P \left(\omega_{P/X}, 
\omega_{P/X}\right)$.
\item 
\label{item-JIJI-and-the-monad-multiplication}
$I^! I_* I^! I_*$ is an object of $D(P \times P)$ with 
\begin{equation}
\label{eqn-cohomology-sheaves-of-JIJI}
H^i(I^! I_* I^! I_*) =  
\begin{cases}
\bigoplus_{p+q = i} \Delta_* 
\left(
\wedge^p \mathcal{N}_{P/X} \otimes \wedge^q \mathcal{N}_{P/X}
\right)
& 
0 \leq i \leq 2n, \\
0 & \text{ otherwise,}
\end{cases}
\end{equation}
and the monad multiplication 
\begin{equation}
I^! I_* I^! I_* \rightarrow I^! I_* 
\end{equation}
is given on the level of cohomology by 
$\Delta_*$ applied to the sum of the wedging maps 
\begin{equation}
\wedge^p \mathcal{N}_{P/X} \otimes \wedge^q \mathcal{N}_{P/X}
\rightarrow 
\wedge^{p+q} \mathcal{N}_{P/X}. 
\end{equation}
\end{enumerate}
\end{prps}

\begin{proof}
\eqref{item-JI-object-and-the-adjunction-unit}:
\newline

By Lemma \ref{lemma-pullbacks-and-pushforwards-of-Fourier-Mukai-kernels}
we have
\begin{equation*}
I^! I_* \simeq (\id_P \times \iota)^! (\iota \times \id_P)^* 
\Delta_* \mathcal{O}_X
\simeq (\iota \times \iota)^* \Delta_* \mathcal{O}_X \otimes 
\pi_2^* \iota^! \mathcal{O}_X.  
\end{equation*}
The subvarieties $P \times P$ and $\Delta(X)$ intersect inside
$X \times X$ at $\Delta(P)$, forming the fiber square
\begin{equation}
\label{eqn-PxP-intersect-Delta(X)-square}
\begin{tikzcd}[column sep={1cm}, row sep={1cm}]
P 
\ar{r}{\iota}
\ar{d}{\Delta}
&
X
\ar{d}{\Delta}
\\
P \times P
\ar{r}{\iota \times \iota}
&
X \times X. 
\end{tikzcd}
\end{equation}
The excess bundle $\mathcal{E}$ of this map is the the cokernel of the map 
$$ \mathcal{N}_{\Delta(P)/ X \times X} \xrightarrow{p_1 \oplus p_2} 
\left(\mathcal{N}_{P \times P / X \times X}\right)|_P \oplus 
\left(\mathcal{N}_{\Delta(X) / X \times X}\right)|_P $$
where $p_i$ are the natural surjections. Thus $\mathcal{E}$ is
isomorphic to the quotient of 
$\mathcal{N}_{P \times P / X \times X}|_P$ by $p_1 \krn(p_2)$. 
Identifying $\mathcal{N}_{P \times P / X \times X}|_P$ with
$\mathcal{N}_{P/X} \oplus \mathcal{N}_{P/X}$
we obtain a short exact sequence 
\begin{equation}
0 \rightarrow 
\mathcal{N}_{P/X} \xrightarrow{\id \oplus \id}
\mathcal{N}_{P/X} \oplus \mathcal{N}_{P/X}
\rightarrow 
\mathcal{E}
\rightarrow 0, 
\end{equation}
and, in particular, $\mathcal{E} \simeq \mathcal{N}_{P/X}$. 
We have 
\begin{align*}
H^i(I^! I_*)  
\simeq 
H^i \left( (\iota \times \iota)^* \Delta_* \mathcal{O}_X \otimes 
\pi_2^* \iota^! \mathcal{O}_X\right) \simeq
H^{i-n} \left((\iota \times \iota)^* \Delta_* \mathcal{O}_X \right) 
\otimes \pi_2^* \left( \wedge^n \mathcal{N}_{P/X} \right),
\end{align*}
and by the excess bundle formula 
(Proposition \ref{prps-excess-bundle-formula-generalised}
\eqref{item-excess-bundle-cohomology-sheaves-of-i^*j_*-F}) we 
further have 
\begin{align}
\label{eqn-identifying-H^i-I^!I_*-with-wedge^i-E}
 H^i(I^! I_*) 
\simeq
\Delta_* \left( \wedge^{n-i} \mathcal{E}^\vee \right) 
\otimes \pi_2^* \left( \wedge^n \mathcal{N}_{P/X}\right) 
\simeq 
\Delta_*\left( \wedge^{n-i} \mathcal{E}^\vee 
\otimes \wedge^n \mathcal{E} \right) 
\simeq 
\Delta_* \wedge^i \mathcal{E}, 
\end{align}
whence the assertion \eqref{eqn-cohomology-sheaves-of-JI}. 
 
On the other hand, 
by Proposition \ref{prps-simplifying-units-and-counit-for-standard-kernels}
\eqref{item-fm-kernel-for-f^!f_*-alt1}
we can identify the adjunction unit 
$\id_P \xrightarrow{\epsilon} I^!  I_*$ with a twisted base change
map for a certain intersection fiber square. 
It then follows by
Proposition \ref{prps-excess-bundle-formula-generalised}
\eqref{item-excess-bundle-twisted-base-change-for-i^!j_*-F}
that this map is the inclusion of the leftmost 
cohomology sheaf, as desired. 

\eqref{item-IJ-object-and-the-adjunction-counit}:

By Lemma \ref{lemma-pullbacks-and-pushforwards-of-Fourier-Mukai-kernels}
we have 

$$ I_* I^! \simeq 
(\iota \times \iota)_* 
\left(\pi_1^* \iota^! \mathcal{O}_X \otimes \Delta_* \mathcal{O}_X \right) 
\simeq 
(\iota \times \iota)_* \Delta_* \Delta^* \pi_1^* \iota^! \mathcal{O}_X 
\simeq \Delta_* \iota_* \iota^! \mathcal{O}_X $$
and by 
Proposition 
\ref{prps-simplifying-units-and-counit-for-standard-kernels}
\eqref{item-fm-kernel-for-f_*f^!-alt1}
this identifies the counit $I_* I^! \rightarrow \id_X$ with  
the morphism 
$$ \Delta_* \iota_* \iota^! \mathcal{O}_X \rightarrow \Delta_*
\mathcal{O}_X, $$
induced by the functorial adjunction 
counit $\iota_* \iota^! \rightarrow \id_X$, as desired. 

\eqref{item-JIJI-and-the-monad-multiplication}:

We proceed similarly to the proofs of 
\eqref{item-JI-object-and-the-adjunction-unit} and 
\eqref{item-IJ-object-and-the-adjunction-counit}.
By Lemma \ref{lemma-pullbacks-and-pushforwards-of-Fourier-Mukai-kernels}
we have
\begin{equation}
\label{eqn-I^!I_*I^!I_*-explicit-form}
I^! I_*  I^! I_*  \simeq 
\bigl( (\iota \times \iota)^* \Delta_* \iota_* \iota^! \mathcal{O}_X
\bigr) \otimes \pi_2^* \iota^! \mathcal{O}_X
\simeq 
(\iota \times \iota)^* \Delta_* \iota_* \wedge^n \mathcal{N}_{P/X} 
\otimes \pi_2^* \wedge^n  \mathcal{N}_{P/X} [-2n],
\end{equation}
\begin{align*}
H^{i} \left( I^! I_*  I^! I_* \right)
\simeq 
H^{i-2n} \bigl( (\iota \times \iota)^* \Delta_* \iota_* 
\wedge^n \mathcal{N}_{P/X} \bigr) \otimes 
\pi_2^* \wedge^n \mathcal{N}_{P/X}.
\end{align*}
The subvarieties $P \times P$ and $\Delta(P)$ intersect inside 
$X \times X$, forming the fiber square which can be decomposed into
two fiber squares via the embedding $\Delta(P) \subset \Delta(X)$:
\begin{equation}
\label{eqn-PxP-intersect-Delta(P)-and-Delta(X)-double-square}
\begin{tikzcd}[column sep={0.75cm}, row sep={1cm}]
P 
\ar[equals]{d}
\ar[equals]{r}
&
P
\ar{d}{\iota}
\\
P
\ar{r}{\iota}
\ar{d}{\Delta}
&
X 
\ar{d}{\Delta}
\\
P \times P
\ar{r}{\iota \times \iota}
&
X \times X. 
\end{tikzcd}
\end{equation}
Let $\mathcal{E}'$ denote the excess bundle of this intersection. 
It is the cokernel of the natural map 
$$ \mathcal{N}_{\Delta(P)/ X \times X} \xrightarrow{p_1 \oplus p_2} 
\left(\mathcal{N}_{P \times P / X \times X}\right)|_P \oplus 
\left(\mathcal{N}_{\Delta(P) / X \times X}\right), $$ 
and hence we have $\mathcal{E}' \simeq \mathcal{N}_{P/X} \oplus
\mathcal{N}_{P/X}$. 
Thus, by the excess bundle formula
(Proposition \ref{prps-excess-bundle-formula-generalised}
\eqref{item-excess-bundle-cohomology-sheaves-of-i^*j_*-F}) we have
\begin{align}
\label{eqn-identifying-H^i-I^!I_*I^!I_*-with-wedge^i-E'}
H^{i} \left( I^! I_*  I^! I_* \right) 
& \simeq 
\Delta_* \bigl(\wedge^{2n-i} (\mathcal{E}')^\vee \otimes 
\wedge^n \mathcal{N}_{P/X} \bigr) \otimes 
\pi_2^* \wedge^n \mathcal{N}_{P/X} \simeq 
\\
\nonumber
&
\simeq 
\Delta_* 
\bigl(
\wedge^{2n-i} (\mathcal{E}')^\vee \otimes 
\wedge^n \mathcal{N}_{P/X} \otimes 
\wedge^n \mathcal{N}_{P/X} 
\bigr)  
\simeq 
\\
\nonumber
&\simeq 
\Delta_* 
\bigl(
\wedge^{2n-i}
(\mathcal{E}')^\vee
\otimes 
\wedge^{2n}
(\mathcal{E}')
\bigr)
\simeq 
\Delta_* 
\bigl(
\wedge^{i}
(\mathcal{E}')
\bigr), 
\end{align}
whence we conclude that, as desired, 
\begin{align*}
H^{i} \left( I^! I_*  I^! I_* \right)
\simeq 
\Delta_* \left( \wedge^i (\mathcal{N}_{P/X} \oplus
\mathcal{N}_{P/X})\right)
\simeq 
\Delta_* \left( \bigoplus_{i = p + q} \left( \wedge^p \mathcal{N}_{P/X} \oplus
\wedge^q \mathcal{N}_{P/X} \right) \right). 
\end{align*}

By \eqref{item-IJ-object-and-the-adjunction-counit}
the monad multiplication $I^! I_*  I^! I_* \rightarrow I^! I_*$
can be identified with the map 
\begin{equation}
\label{eqn-monad-multiplication-as-double-excess-bundle-intersection}
\bigl( (\iota \times \iota)^* \Delta_* \iota_* \iota^! \mathcal{O}_X
\bigr) \otimes \pi_2^* \iota^! \mathcal{O}_X
\rightarrow 
\bigl( (\iota \times \iota)^* \Delta_* \mathcal{O}_X
\bigr) \otimes \pi_2^* \iota^! \mathcal{O}_X  
\end{equation}
induced by the adjunction counit $\iota_* \iota^! \mathcal{O}_X
\rightarrow \mathcal{O}_X$. 
We are now in position to apply 
Proposition \ref{prps-excess-bundle-formula-double-intersection}\eqref{item-excess-bundle-double-intersection-twisted}. 
From the diagram \eqref{eqn-two-excess-bundles-3x3-diagram}
we have the natural short exact sequence of excess bundles 
$$ 
0 \rightarrow 
\mathcal{N}_{P/X}
\hookrightarrow 
\mathcal{E'} 
\twoheadrightarrow 
\mathcal{E} 
\rightarrow 0, $$ 
which can be identified with the short exact sequence
$$ 
0 \rightarrow 
\mathcal{N}_{P/X}
\xrightarrow{\id \oplus \id}
\mathcal{N}_{P/X} \oplus \mathcal{N}_{P/X}
\xrightarrow{\id \oplus - \id}
\mathcal{E} 
\rightarrow 0. $$ 
It follows that the map induced by 
\eqref{eqn-monad-multiplication-as-double-excess-bundle-intersection}
on the degree $i$ cohomology sheaves is given by 
$$
\Delta_* \bigl(\wedge^{2n-i} (\mathcal{E}')^\vee \otimes 
\wedge^n \mathcal{N}_{P/X} \otimes \wedge^n \mathcal{N}_{P/X} 
\bigr) 
\xrightarrow{\Delta_*
(\eqref{eqn-wedging-with-the-top-power-of-the-quotient} \otimes \id)}
\Delta_*  \bigl(\wedge^{n-i} 
\mathcal{E}^\vee   
\otimes \wedge^n \mathcal{N}_{P/X}\bigr). 
$$
By Proposition 
\ref{prps-useful-maps-for-exterior-algebra}
\eqref{item-wedging-with-the-top-power-of-the-quotion-dual-to-the-inclusion}
under the identifications
\eqref{eqn-identifying-H^i-I^!I_*-with-wedge^i-E} above 
and
\eqref{eqn-identifying-H^i-I^!I_*I^!I_*-with-wedge^i-E'} above 
this becomes the map 
$$ \Delta_* \wedge^i \mathcal{E}' \rightarrow \Delta_* \wedge^i \mathcal{E}, $$
induced by the projection $\mathcal{E}' \twoheadrightarrow
\mathcal{E}$, i.e. the projection 
$ \mathcal{N}_{P/X} \oplus \mathcal{N}_{P_X} \xrightarrow{\id \oplus
- \id} \mathcal{N}_{P/X}$. 
Thus the monad multiplication is indeed given on the degree $i$
cohomology sheaves by the image under $\Delta_*$ of the wedging map 
$$ \bigoplus_{i = p+q} \wedge^p \mathcal{N}_{P/X} \otimes \wedge^q
\mathcal{N}_{P/X} \rightarrow \wedge^{i} \mathcal{N}_{P/X}. $$
\end{proof}
 
\subsubsection{Enhancing adjunction monads and comonads for $F_k$}
\label{section-family-pn-twist-enhanced-adjunction-monads-and-comonads}

We now return to the setup of \S\ref{section-family-the-setup}. Let
$\Pi_*$, $\Pi^*$, $\Pi^!$, and $\mathcal{T}_{\mathcal{O}(k)}$
be the standard enhancements of $\pi_*$, $\pi^*$, $\pi_!$, 
and $(-) \otimes \mathcal{O}_P(k)$
as per
\S\ref{section-standard-fourier-mukai-kernels-and-the-key-lemma}: 
\begin{equation}
\Pi^* \simeq (\id_Z \times \pi)^* \mathcal{O}_{\Delta(Z)} \in D(Z \times P),
\end{equation}
\begin{equation}
\Pi_* \simeq (\id_P \times \pi)_* \mathcal{O}_{\Delta(P)} \in D(P \times Z), 
\end{equation}
\begin{equation}
\Pi_! \simeq (\id_P \times \pi)^! \mathcal{O}_{\Delta(P)} \in D(P \times Z), 
\end{equation}
\begin{equation}
T_{\mathcal{O}(k)}  \simeq \pi^*_2(\mathcal{O}_P(k))
\in D(P \times P). 
\end{equation}
\begin{lemma}
\label{lemma-units-and-counits-of-Pi_*-Pi^*-and-T_O(k)-T_O(-k)-adjunctions}
\begin{enumerate}
\item 
\label{item-adjunction-unit-and-counit-for-Pi_*-Pi^*}
The adjunction unit 
\begin{equation}
\label{eqn-adjunction-unit-for-Pi_*-Pi^*}
\id_Z \xrightarrow{\epsilon} \Pi_* \Pi^* 
\end{equation}
is an isomorphism in $D(Z \times Z)$. On the other hand, 
there is an isomorphism
$$ \Pi^* \Pi_* \simeq \mathcal{O}_{P \times_Z P} $$
in $D(P \times P)$ which identifies the adjunction counit 
\begin{equation}
\label{eqn-adjunction-counit-for-Pi_*-Pi^*}
\Pi^* \Pi_* \xrightarrow{\mu} \id_P 
\end{equation}
with the restriction of sheaves  
$$ \mathcal{O}_{P \times_Z P} \twoheadrightarrow \mathcal{O}_{\Delta(P)}. $$
\item 
\label{item-adjunction-unit-and-counit-for-T_O(k)-T_O(-k)}
The adjunction unit and counit  
$$ \id_P \rightarrow T_{\mathcal{O}(-k)} T_{\mathcal{O}(k)}, $$
$$ T_{\mathcal{O}(k)} T_{\mathcal{O}(-k)} \rightarrow \id_P $$
are isomorphisms in $D(P \times P)$.  
\end{enumerate}
\end{lemma}
\begin{proof}
\eqref{item-adjunction-unit-and-counit-for-Pi_*-Pi^*}:

The natural transformation of exact functors induced by 
\eqref{eqn-adjunction-unit-for-Pi_*-Pi^*}
is the adjunction unit $\id_Z \rightarrow \pi_* \pi^*$ 
for the $(\pi^*, \pi_*)$ adjunction. 
This adjunction unit is an isomorphism, because  
$\pi\colon P \rightarrow Z$ is an instance of a relative $\proj$-construction, 
and therefore $\pi^*$ is fully faithful. More generally, the projection 
formula identifies the unit $\id_{D(Z)} \rightarrow \pi_* \pi^* $ with the 
tensor product of $\id_{D(Z)}$ and the adjunction unit $\mathcal{O}_Z \rightarrow \pi_* \pi^* \mathcal{O}_Z$, which is 
an isomorphism since $\pi$ is a flat Fano fibration. 

We can therefore conclude that \eqref{eqn-adjunction-unit-for-Pi_*-Pi^*} 
itself is an isomorphism by the general fact that a morphism of Fourier-Mukai 
kernels which induces an isomorphism of functors is itself an isomorphism. 
Indeed, consider its cone. It is a kernel whose induced Fourier-Mukai
transform is the zero functor. In particular, every skyscraper sheaf
gets sent to zero. Hence the pullback of the kernel to every closed
point of the product is zero, and hence the kernel itself is zero. 

For the adjunction counit \eqref{eqn-adjunction-counit-for-Pi_*-Pi^*}
the desired statement follows from Propostion 
\ref{prps-simplifying-units-and-counit-for-standard-kernels}
\eqref{item-fm-kernel-for-f^*f_*-alt2}
since $\pi$ is flat and hence the base change around  
the top fiber square in \eqref{eqn-Y-and-XxX-fiber-square} is an
isomorphism. 

\eqref{item-adjunction-unit-and-counit-for-T_O(k)-T_O(-k)}:

Similarly, this follows from the induced natural transformations 
of exact functors being isomorphisms as they are 
the unit and the counit for the adjunction of 
autoequivalences $\left((-) \otimes \mathcal{O}_P(k), (-)
\otimes \mathcal{O}_P(-k)\right)$. 
\end{proof}

We choose the following Fourier-Mukai kernel for $f_k$:
$$ F_k \overset{\text{def}}{=} I_* T_{\mathcal{O}(k)} \Pi^* \in D(Z \times X), $$
and by the Key Lemma  
(Lemma \ref{lemma-pullbacks-and-pushforwards-of-Fourier-Mukai-kernels})
we have
\begin{align*}
F_k\;   
\simeq 
(\pi \times \iota)_* \Delta_* \mathcal{O}_{P}(k)
\simeq 
(\pi,\iota)_* \mathcal{O}_{P}(k). 
\end{align*}

Similarly, we choose the following Fourier-Mukai kernels for the left and 
right adjoints of $f_k$:
\begin{align}
L_k \overset{\text{def}}{=} \Pi_! T_{\mathcal{O}(-k)} I^* \in D(X \times Z), 
\\
R_k \overset{\text{def}}{=}  \Pi_* T_{\mathcal{O}(-k)} I^! \in D(X \times Z), 
\end{align}
and we similarly have
\begin{align}
\label{eqn-L_k-explicitly}
L_k & \simeq (\iota,\pi)_* \mathcal{O}_{P}(-n-k-1)[n],
\\
\label{eqn-R_k-explicitly}
R_k &\simeq 
(\iota,\pi)_* \mathcal{O}_{P}(-n-k-1)[-n].  
\end{align}

\begin{prps}
\label{prps-enhancing-monad-RkFk-and-comonad-FkRk}
\begin{enumerate}
\item 
\label{item-RkFk-and-the-adjunction-unit}
$R_k F_k$ is an object in $D(Z \times Z)$ with 
\begin{equation}
\label{eqn-cohomology-sheaves-of-RkFk}
H^i(R_k F_k) \simeq 
\begin{cases}
\Delta_* \mathcal{O}_Z & \quad 0 \leq i \leq 2n, i = 2j \\
0 & \quad \text{ otherwise},  
\end{cases}
\end{equation}
and the adjunction unit
\begin{equation}
\id_Z = \Delta_* \mathcal{O}_Z \rightarrow R_k F_k  
\end{equation}
is an isomorphism on $H^0(-)$. 

\item 
\label{item-FkRk-and-the-adjunction-counit}
$F_k R_k$ is an object in $D(X \times X)$ with 
\begin{equation}
F_k R_k \simeq 
(\iota \times \iota)_* z_* \mathcal{O}_{P \times_Z P}(-k,k) \otimes
\pi_1^* \iota^! \mathcal{O}_X
\simeq 
(\iota \times \iota)_* z_* \mathcal{O}_{P \times_Z P}(\omega_{P/X}-k,k)[-n]
\end{equation}
where $z\colon P \times_Z P \hookrightarrow P \times P$ is the natural
closed immersion.

The adjunction counit $F_k R_k \xrightarrow{\mu} \id_{X}$
is given by the composition 
$$
(\iota \times \iota)_* z_* \mathcal{O}_{P \times_Z P}(-k,k) \otimes
\pi_1^* \iota^! \mathcal{O}_X
\twoheadrightarrow 
I^! I_* 
\rightarrow 
\id_{X}
$$
where the first morphism is the restriction of a (shifted) line bundle 
from $P \times_Z P$ to $\Delta(P)$:
\begin{equation}
\label{eqn-adjunction-counit-FkRk-first-part}
(\iota \times \iota)_* z_* \mathcal{O}_{P \times_Z P}(\omega_{P/X}-k,k) [-n]
\rightarrow 
\Delta_* \iota_* \omega_{P/X}[-n],
\end{equation}
while the second one is the adjunction counit for $I^!I_*$
$$ \Delta_* \iota_* \omega_{P/X}[-n] \rightarrow \Delta_* \mathcal{O}_X $$
given by the element of 
of $\ext^n_{X}(\omega_{P/X}, \mathcal{O}_X)$ 
relative Serre dual to 
$\id_{\omega_{P/X}} \in \homm_{P}\left(\omega_{P/X}, \omega_{P/X}\right)$. 

\item 
\label{item-RkFkRkFk-and-the-monad-multiplication}

$R_k F_k R_k F_k$ is an object in $D(Z \times Z)$ with
\begin{equation}
\label{eqn-cohomology-sheaves-of-RkFkRkFk}
H^j(R_k F_k R_k F_k) \simeq 
\begin{cases}
\Delta_* \mathcal{O}_Z^{\oplus (n+1-|i-n|)}  & \quad  0 \leq j \leq 4n, j = 2i \\
0 & \quad \text{ otherwise},  
\end{cases}
\end{equation}
and the monad multiplication
\begin{equation}
R_k F_k R_k F_k \rightarrow R_k F_k  
\end{equation}
is given on cohomologies by the maps 
$$ \Delta_* \mathcal{O}_Z^{\oplus (i+1)} \xrightarrow{\sum \id}
\Delta_* \mathcal{O}_Z. $$
\end{enumerate}
\end{prps}

\begin{proof}
\eqref{item-RkFk-and-the-adjunction-unit}:
 
We have 
$$ R_k F_k = \Pi_* T_{\mathcal{O}(-k)} I^! I_* T_{\mathcal{O}(k)} \Pi^*, $$
and by Lemma \ref{lemma-pullbacks-and-pushforwards-of-Fourier-Mukai-kernels} 
the object $R_k F_k$ is the image of $I^! I_*$ under the functor 
\begin{equation}
\label{eqn-RkFk-via-I^!I_*-functor}
W = (\pi \times \pi)_* \left(\pi_2^* \mathcal{O}_P(-k)
\otimes \pi_1^* \mathcal{O}_P(k) \otimes (-)\right). 
\end{equation}
By Proposition \ref{prps-adjunction-monads-and-comonads-for-a-closed-immersion}\eqref{item-JI-object-and-the-adjunction-unit} we have 
$$ H^i (I^! I_*) \simeq \Delta_* \wedge^i \mathcal{N}_{P/X} \simeq \Delta_* \Omega^i_{P/Z} .$$
We now observe that there is a functorial isomorphism 
\begin{align}
\label{eqn-isomorphism-reducing-the-functor-applied-to-I^!I_*I^!I_*}
W \Delta_* = 
& (\pi \times \pi)_* \left(
\pi_2^* \mathcal{O}_P(-k) \otimes \pi_1^* \mathcal{O}_P(k) \otimes 
\Delta_* \right) 
\simeq (\pi \times \pi)_* \Delta_* 
\left(\Delta^* \pi_2^* \mathcal{O}_P(-k)
\otimes \Delta^* \pi_1^* \mathcal{O}_P(k) \otimes (-) \right) 
\simeq
\\
\nonumber
\simeq &
(\pi \times \pi)_* \Delta_* 
\left(\mathcal{O}_P(-k)
\otimes \mathcal{O}_P(k) \otimes (-) \right) 
\simeq
(\pi \times \pi)_* \Delta_* \simeq \Delta_* \pi_*, 
\end{align}
with the second isomorphism due to projection formula, the third
due to the identities $\pi_1 \circ \Delta = \pi_2 \circ \Delta = \id$, 
and the last due to the commutativity of the square
\begin{equation}
\begin{tikzcd}[column sep={1cm},row sep={1cm}]
P 
\ar{d}{\pi}
\ar{r}{\Delta}
& 
P \times P 
\ar{d}{\pi \times \pi}
\\ 
Z 
\ar{r}{\Delta}
&
Z \times Z.
\end{tikzcd}
\end{equation}
Hence 
$$ W( H^i(I^! I_*) ) \simeq \Delta_* \pi_* \Omega^i_{P/Z}, $$
and by relative Bott vanishing we have 
$$ \pi_* \Omega^i_{P/Z} \simeq \mathcal{O}_Z [i]. $$
Since $R_k F_k = W(I^! I_*)$ a standard spectral sequence argument 
now yields the assertion \eqref{eqn-cohomology-sheaves-of-RkFk}, 
and, in particular, for $0 \leq i \leq n$ the natural isomorphisms 
\begin{equation}
\label{eqn-2i-th-cohomology-of-RkFk}
H^{2i}(R_k F_k) \simeq H^i(W(H^i(I^! I_*)) \simeq 
H^i(\Delta_* \pi_* \Omega^i_{P/Z}) \simeq \Delta_* \mathcal{O}_Z. 
\end{equation}

The adjunction unit $\id_Z \rightarrow R_k F_k$ is given 
by the composition of adjunction units
$$ \id_Z \rightarrow \Pi_* \Pi^* \rightarrow \Pi_* T_{\mathcal{O}(-k)} T_{\mathcal{O}(k)} \Pi^* 
\rightarrow \Pi_* T_{\mathcal{O}(-k)} I^! I_* T_{\mathcal{O}(k)} \Pi^*. $$
The first two are isomorphisms by Lemma
\ref{lemma-units-and-counits-of-Pi_*-Pi^*-and-T_O(k)-T_O(-k)-adjunctions}. 
The third is 
an isomorphism on $H^0(-)$ by Proposition \ref{prps-adjunction-monads-and-comonads-for-a-closed-immersion}\eqref{item-JI-object-and-the-adjunction-unit}. 	

\eqref{item-FkRk-and-the-adjunction-counit}:

We have 
$$ F_k R_k  = I_* T_{\mathcal{O}_{(k)}} \Pi^* \Pi_* T_{\mathcal{O}(-k)} I^! , $$
By 
Lemma
\ref{lemma-units-and-counits-of-Pi_*-Pi^*-and-T_O(k)-T_O(-k)-adjunctions}
we have $\Pi^* \Pi_* \simeq \mathcal{O}_{P \times_Z P}$, hence by 
Lemma \ref{lemma-pullbacks-and-pushforwards-of-Fourier-Mukai-kernels}
we have 
\begin{align*}
F_k R_k \simeq (\iota \times \iota)_* z_* \left(\pi_1^* \iota^!
\mathcal{O}_X \otimes \pi_1^* \mathcal{O}_P(-k) \otimes \pi_2^* \mathcal{O}_P(k)  \otimes \mathcal{O}_{P \times_Z P} \right) 
\simeq (\iota \times \iota)_* z_* 
\mathcal{O}_{P \times_Z P}(-k+\omega_{P/X},k)[-n]. 
\end{align*}
The adjunction counit $F_k R_k \rightarrow \id_X$ is given by the 
composition of the adjunction counits 
$$ I_* T_{\mathcal{O}(k)} \Pi^* \Pi_* T_{\mathcal{O}(-k)} I^!
\rightarrow I_* T_{\mathcal{O}(k)} T_{\mathcal{O}(-k)} I^!
\rightarrow I_* I^! \rightarrow \id_X. $$
By Lemma
\ref{lemma-units-and-counits-of-Pi_*-Pi^*-and-T_O(k)-T_O(-k)-adjunctions}
the first two are induced by the sheaf restriction $\mathcal{O}_{P
\times_Z P} \twoheadrightarrow \mathcal{O}_{\Delta(P)}$ and 
by the isomorphism $\mathcal{O}_P(-k) \otimes
\mathcal{O}_P(k) \simeq \mathcal{O}_P$, respectively. 
Hence their composition is the line bundle restriction  
\eqref{eqn-adjunction-counit-FkRk-first-part}. 


\eqref{item-RkFkRkFk-and-the-monad-multiplication}:

As in the proof of \eqref{item-RkFk-and-the-adjunction-unit}
by Lemma \ref{lemma-pullbacks-and-pushforwards-of-Fourier-Mukai-kernels} 
we have 
\begin{equation}
\label{eqn-RkFkRkFk-via-I^!FkRkI*}
R_k F_k R_k F_k \simeq 
W (I^! F_k R_k I_*)
\end{equation}
where $W$ is the functor \eqref{eqn-RkFk-via-I^!I_*-functor}. 
By \eqref{item-FkRk-and-the-adjunction-counit} 
we've 
$$ F_k R_k \simeq (\iota \times_Z \iota)_* \mathcal{O}_{P \times_Z
P}(\omega_{P/X} - k ,k)[-n], $$
and therefore 
by Lemma \ref{lemma-pullbacks-and-pushforwards-of-Fourier-Mukai-kernels} 
again
\begin{equation}
\label{eqn-I^!FkRkI_*-explicit-form}
I^! F_k R_k I_* \simeq 
(\iota \times \iota)^* (\iota \times_Z \iota)_* 
\mathcal{O}_{P \times_Z P}(\omega_{P/X}-k,\omega_{P/X}+k)[-2n].
\end{equation}

The subvarieties $P \times P$ and $P \times_Z P$ intersect within 
$X \times X$ forming the fiber square  
\begin{equation}
\begin{tikzcd}
P \times_Z P
\ar[equals]{r}
\ar{d}{z}
&
P \times_Z P
\ar{d}{(\iota \times \iota) \circ z}
\\
P \times P 
\ar{r}{\iota \times \iota}
& 
X \times X.
\end{tikzcd}	
\end{equation}
Since $P \times_Z P$ lies entirely within in $P \times P$, 
the excess bundle $\mathcal{E}''$ of their intersection is given by 
$$ \mathcal{E}'' \simeq \mathcal{N}_{P \times P/X \times X}|_{P \times_Z P}
\simeq \pi_1^* \mathcal{N}_{P/X} \oplus \pi_2^* \mathcal{N}_{P/X}. $$
It follows by the excess bundle formula 
(Proposition \ref{prps-excess-bundle-formula-generalised}
\eqref{item-excess-bundle-cohomology-sheaves-of-i^*j_*-F})
that 
\begin{align}
\label{eqn-identifying-H^i-I^!FkRkI_*-with-wedge^i-E''}
H^i(I^! F_k R_k I_*) 
& 
\simeq
z_*
\bigl( \wedge^{2n-i} (\mathcal{E}'')^\vee \otimes 
\mathcal{O}_{P \times_Z P}(\omega_{P/X}-k,\omega_{P/X}+k)
\bigr)
\simeq
\\
\nonumber
& 
\simeq 
z_*
\bigl(
\wedge^{2n-i} (\mathcal{E}'')^\vee 
\otimes \wedge^{2n} \mathcal{E}'' \otimes \mathcal{O}_{P \times_Z P}(-k,k)
\bigr)
\simeq 
\\
\nonumber
&
\simeq
z_*
\bigl(
\wedge^{i} \mathcal{E}'' \otimes \mathcal{O}_{P \times_Z P}(-k,k)
\bigr). 
\end{align}
We conclude that 
\begin{equation}
\label{eqn-cohomology-sheaves-of-I^!FkRkI_*}
H^i(I^! F_k R_k I_*)
\simeq 
z_*
\bigl(
\bigoplus_{i = p+q}
\Omega^p_{P/Z}(-k) \boxtimes \Omega^q_{P/Z}(k) 
\bigr).
\end{equation}

We next compute the action of the functor $W$ on the cohomology sheaves
of $I^! F_k R_k I_*$:
\begin{align*}
W H^i (I^! F_k R_k I_*) 
&
\simeq 
(\pi \times \pi)_* \bigl(
\pi_1^* \mathcal{O}_P(k) \otimes 
\pi_2^* \mathcal{O}_P(-k) \otimes 
H^i (I^! F_k R_k I_*) 
\bigr)
\simeq 
\\
&
\simeq 
(\pi \times \pi)_* z_* 
\bigoplus_{i = p+q} \Omega^p_{P/Z} \boxtimes \Omega^q_{P/Z}
\simeq
\\
&
\simeq 
\Delta_* (\pi \times_Z \pi)_* \bigoplus_{i = p+q} 
\Omega^p_{P/Z} \boxtimes \Omega^q_{P/Z}, 
\end{align*}
where $\pi \times_Z \pi$ is the diagonal map 
in the fiber square
\begin{equation}
\label{eqn-P-P-Z-fiber-square}
\begin{tikzcd}
P \times_Z P
\ar{r}{\pi_2}
\ar{d}{\pi_1}
\ar{dr}{\pi \times_Z \pi}
&
P
\ar{d}{\pi}
\\
P 
\ar{r}{\pi}
& 
Z. 
\end{tikzcd}	
\end{equation}
Since the square is $\tor$-independent its K{\"u}nneth map
is an isomorphism, and we have 
$$ 
W H^i (I^! F_k R_k I_*) \simeq 
\Delta_* \bigoplus_{i = p+q}
\pi_*(\Omega^p_{P/Z}) \otimes \pi_*(\Omega^q_{P/Z}).$$
Finally, by relative Bott vanishing we further conclude that  
$$
W H^i (I^! F_k R_k I_*)
\simeq \Delta_* \bigoplus_{i = p+q}
\mathcal{O}_Z[p]\otimes \mathcal{O}_Z[q]
\simeq 
\Delta_* \mathcal{O}_Z^{\oplus (n+1-|i-n|)}[i]. $$
Since $R_k F_k R_k F_k= W(I^! F_k R_k I_*)$ the assertion 
\eqref{eqn-cohomology-sheaves-of-RkFkRkFk}
follows by the usual spectral sequence argument. 

For the monad multiplication $R_k F_k R_k F_k \rightarrow R_k F_k$,
by \eqref{item-FkRk-and-the-adjunction-counit} 
it is the image under $W$ of the composition 
\begin{equation}
\label{eqn-monad-multiplication-for-RkFk-before-applying-W}
I^! F_k R_k I_* \rightarrow I^! I_* I^! I_* \rightarrow I^! I_* 
\end{equation}
whose first map is induced by the map 
$F_k R_k \xrightarrow{\eqref{eqn-adjunction-counit-FkRk-first-part}} I_* I^!$
which is a restriction of (shifted) line bundles from $P \times_Z P$
to $\Delta(P)$ and whose second map is the monad multiplication 
for $I^! I_*$. 

We next compute the maps induced by 
$\eqref{eqn-monad-multiplication-for-RkFk-before-applying-W}$ on 
the cohomology sheaves. Under 
the identifications \eqref{eqn-I^!FkRkI_*-explicit-form} and
\eqref{eqn-I^!I_*I^!I_*-explicit-form} the first map in 
\eqref{eqn-monad-multiplication-for-RkFk-before-applying-W}
becomes the map 
\begin{equation}
\label{eqn-I^!FkRkI*-to-I^!I_*I!I_*-map-explicitly}
(\iota \times \iota)^* (\iota \times_Z \iota)_* 
\mathcal{O}_{P \times_Z P}(\omega_{P/X}-k,\omega_{P/X}+k)[-2n]
\rightarrow 
(\iota \times \iota)^* \Delta_* \iota_* \omega_{P/X} \otimes \omega_{P/X} [-2n]
\end{equation}
induced by the restriction of a shifted line bundle 
$\mathcal{O}_{P \times_Z P}(\omega_{P/X}-k,\omega_{P/X}+k)[-2n]$ from 
$P \times_Z P$ to $\Delta(P)$. We are therefore in a position 
to apply Proposition 
\ref{prps-excess-bundle-formula-double-intersection}\eqref{item-excess-bundle-double-intersection-straight}. 

We have a commutative diagram of fibre squares
\begin{equation}
\begin{tikzcd}
P 
\ar{d}{\Delta}
\ar[equals]{r}
&
P 
\ar{d}{\Delta}
\\
P \times_Z P
\ar[equals]{r}
\ar{d}{z}
&
P \times_Z P
\ar{d}{(\iota \times \iota) \circ z}
\\
P \times P 
\ar{r}{\iota \times \iota}
& 
X \times X.
\end{tikzcd}	
\end{equation}
Above we've computed the excess bundle $\mathcal{E}''$ of the
intersection of $P \times P$ and $P \times_Z P$ to be 
$$ \mathcal{E}'' \simeq 
\pi_1^* \mathcal{N}_{P/X} \oplus \pi_2^* \mathcal{N}_{P/X}. $$
In Proposition 
\ref{prps-adjunction-monads-and-comonads-for-a-closed-immersion}
\eqref{item-JI-object-and-the-adjunction-unit}
we've shown the excess bundle $\mathcal{E}$ of the intersection 
of $P \times P$ and $\Delta(P)$ to be 
$$ \mathcal{E} \simeq \mathcal{N}_{P/X} \oplus \mathcal{N}_{P/X}. $$ 
The natural map $\mathcal{E} \rightarrow \Delta^* \mathcal{E}''$
is readily seen to be an isomorphism. 
Correspondingly, the induced map $(\mathcal{E}'')^\vee \rightarrow 
\Delta_* \mathcal{E}^\vee$ is 
the natural restriction of sheaves from $P \times_Z P$ to $\Delta(P)$. 
It follows 
by Proposition 
\ref{prps-excess-bundle-formula-double-intersection}\eqref{item-excess-bundle-double-intersection-straight}
that the map induced by 
\eqref{eqn-I^!FkRkI*-to-I^!I_*I!I_*-map-explicitly}
on the degree $i$ cohomology is the map 
$$ z_* \bigl(
\wedge^{2n-i} 
(\mathcal{E}'')^\vee
\otimes \mathcal{O}_{P \times_Z P}(\omega_{P/X}-k,\omega_{P/X}+k)\bigl)
\rightarrow 
z_* \Delta_* 
\bigl(
\wedge^{2n-i} \mathcal{E}^\vee \otimes \omega_{P/X} \otimes \omega_{P/X} 
\bigr) 
$$
which is the restriction of a locally free sheaf 
$\wedge^{2n-i} 
(\mathcal{E}'')^\vee
\otimes \mathcal{O}_{P \times_Z P}(\omega_{P/X}-k,\omega_{P/X}+k)
$
from $P \times_Z P$ to $\Delta(P)$. Under the identifications 
\eqref{eqn-identifying-H^i-I^!FkRkI_*-with-wedge^i-E''} and 
\eqref{eqn-identifying-H^i-I^!I_*I^!I_*-with-wedge^i-E'} 
this further becomes the map 
$$ z_* \left( \wedge^i \mathcal{E}'' \otimes \mathcal{O}_{P \times_Z
P}(-k,k) \right) \rightarrow z_* \Delta_* \wedge^i \mathcal{E} $$
given again by a restriction of a locally free sheaf 
$\wedge^i \mathcal{E}'' \otimes \mathcal{O}_{P \times_Z P}(-k,k)$ 
from $P \times_Z P$ to $\Delta(P)$. 

We conclude that the map induced by the first composant 
of \eqref{eqn-monad-multiplication-for-RkFk-before-applying-W}
on the degree $i$ cohomology can be identified with the natural 
sheaf restriction map  
$$ 
z_*
\bigl(
\bigoplus_{i = p+q}
\Omega^p_{P/Z}(-k) \boxtimes \Omega^q_{P/Z}(k) 
\bigr)
\twoheadrightarrow 
\Delta_* \bigl(
\bigoplus_{i = p+q}
\Omega^p_{P/Z} \otimes \Omega^q_{P/Z} 
\bigr). 
$$
On the other hand, the second composant in 
\eqref{eqn-monad-multiplication-for-RkFk-before-applying-W}
is the monad multiplication for $I^! I_*$ and we've computed in 
\ref{prps-adjunction-monads-and-comonads-for-a-closed-immersion}\eqref{item-JIJI-and-the-monad-multiplication}
that the induced map on the degree $i$ cohomology cohomology sheaves 
$$ \Delta_* \bigl(
\bigoplus_{i = p+q}
\Omega^p_{P/Z} \otimes \Omega^q_{P/Z} 
\bigr)
\rightarrow 
\Delta_* \bigl(\wedge^i \Omega^p_{P/Z}\bigr)$$
is the image under $\Delta$ of the wedging map. 

Thus, the map 
$$ H^i(I^! R_k F_k I_*) \rightarrow H^i(I^! I_*) $$
induced by 
\eqref{eqn-monad-multiplication-for-RkFk-before-applying-W}
on the degree $i$ cohomology sheaves is the composition 
\begin{equation}
\label{eqn-monad-multiplication-for-RkFk-before-applying-W-on-deg-i-coh}
z_*
\bigl(
\bigoplus_{i = p+q}
\Omega^p_{P/Z}(-k) \boxtimes \Omega^q_{P/Z}(k) 
\bigr)
\rightarrow 
\Delta_* \bigl(
\bigoplus_{i = p+q}
\Omega^p_{P/Z} \otimes \Omega^q_{P/Z} 
\bigr)
\rightarrow
\Delta_* \bigl(\Omega^i_{P/Z}\bigr)
\end{equation}
of the sheaf restriction from $P \times_Z P$ to $\Delta(P)$
and the wedging map. Note that this is zero unless $0 \leq i \leq n$.  

The image of the first composant of 
\eqref{eqn-monad-multiplication-for-RkFk-before-applying-W-on-deg-i-coh}
under the functor $W$ is therefore the map 
\begin{equation}
\Delta_* (\pi \times_Z \pi)_* 
\bigl(
\bigoplus_{i = p+q} \Omega^p_{P/Z} \boxtimes \Omega^q_{P/Z}
\bigr)
\rightarrow 
\Delta_* \pi_* 
\bigl( \bigoplus_{i = p+q} \Omega^p_{P/Z} \otimes \Omega^q_{P/Z} \bigr)
\end{equation}
which is the the image under $\Delta_* (\pi \times_Z \pi)_*$
of the restriction of sheaves from $P \times_Z P$ to $\Delta(P)$. 
The K{\"u}nneth isomorphism for the square \eqref{eqn-P-P-Z-fiber-square}
identifies it further with the map 
\begin{equation}
\Delta_* \bigl( 
\bigoplus_{i = p+q} \pi_* \Omega^p_{P/Z} \otimes \pi_* \Omega^q_{P/Z}
\bigr)
\rightarrow 
\Delta_* 
\bigl( 
\bigoplus_{i = p+q} 
\pi_* \bigl(\Omega^p_{P/Z} \otimes \Omega^q_{P/Z} 
\bigr)
\bigr)
\end{equation}
which is the image under $\Delta_*$ of the natural map 
$\pi_*(-) \otimes \pi_*(-) \rightarrow \pi_*(- \otimes -)$. 
On the other hand, since $W \Delta_* \simeq \Delta_* \pi_*$
the image of the second composant of 
\eqref{eqn-monad-multiplication-for-RkFk-before-applying-W-on-deg-i-coh}
under the functor $W$
is the map 
\begin{equation}
\Delta_* \pi_* \bigl(
\bigoplus_{i = p+q}
\Omega^p_{P/Z} \otimes \Omega^q_{P/Z} 
\bigr)
\rightarrow
\Delta_* \pi_* \bigl(\wedge^i \Omega^p_{P/Z}\bigr) 
\end{equation}
which is the image under $\Delta_* \pi_*$ of the wedging map. 

We conclude that the image 
$$ W H^i(I^! R_k F_k I_*) \rightarrow WH^i(I^! I_*) $$
under the functor $W$ of the 
\eqref{eqn-monad-multiplication-for-RkFk-before-applying-W-on-deg-i-coh}
is therefore the image under $\Delta_*$ of the map 
\begin{equation}
\label{eqn-W-on-monad-multiplication-for-RkFk-before-applying-W-on-deg-i-coh}
\bigoplus_{i = p+q} \pi_* \Omega^p_{P/Z} \otimes \pi_* \Omega^q_{P/Z}
\rightarrow 
\bigoplus_{i = p+q} 
\pi_* (\Omega^p_{P/Z} \otimes \Omega^q_{P/Z})
\rightarrow
\pi_* \Omega^i_{P/Z}. 
\end{equation}
For $0 \leq i \leq n$ each individual map 
$$
\pi_* \Omega^p_{P/Z} \otimes \pi_* \Omega^q_{P/Z}
\rightarrow 
\pi_* (\Omega^p_{P/Z} \otimes \Omega^q_{P/Z})
\rightarrow
\pi_* \Omega^{i}_{P/Z}
$$
is an isomorphism and the relative Bott vanishing identifications 
$\pi_* \Omega^p_{P/Z} \simeq \mathcal{O}_Z[p]$, 
$\pi_* \Omega^q_{P/Z} \simeq \mathcal{O}_Z[q]$, 
$\pi_* \Omega^i_{P/Z} \simeq \mathcal{O}_Z[i]$ identify this
isomorphism with the identity map 
$$ \mathcal{O}_Z[i] \rightarrow \mathcal{O}_Z[i]. $$
Since we have identified the 
monad multiplication $R_k F_k R_k F_k \rightarrow R_k F_k$
with the image under $W$ of 
the map \eqref{eqn-monad-multiplication-for-RkFk-before-applying-W} 
it now follows by the same spectral sequence argument that 
the induced map on degree $2i$ cohomology for $0 \leq i \leq n$
is the sum of identity maps 
$$ \Delta_* \mathcal{O}_Z^{\oplus (n+1)} \rightarrow \Delta_*
\mathcal{O}_Z, $$
as desired. 
\end{proof}

\subsubsection{The structure of a $\mathbb{P}^n$-functor on $F_k$ and
its $\mathbb{P}$-twist}
\label{section-family-pn-functor-structure-on-Fk}

Let $H$
be the standard enhancements of an autoequivalence $[-2]$ of $D(Z)$
as per
\S\ref{section-standard-fourier-mukai-kernels-and-the-key-lemma}. 
Then:
$$ H \simeq \Delta_* \mathcal{O}_Z[-2] \in D(Z \times Z), $$
$$ H^i \simeq \Delta_* \mathcal{O}_Z[-2i].$$
In \S\ref{section-family-pn-twist-enhanced-adjunction-monads-and-comonads}
we have shown that the cohomology sheaves of the object $R_k F_k \in D(Z
\times Z)$ are  
\begin{equation*}
H^i(R_k F_k) \simeq 
\begin{cases}
\Delta_* \mathcal{O}_Z & \quad 0 \leq i \leq 2n, i = 2j \\
0 & \quad \text{ otherwise}.  
\end{cases}
\end{equation*}

We have the \em standard filtration \rm on $R_k F_k$ built using 
the truncation functors 
$$ \tau_{\leq n}\colon D(Z) \rightarrow D^{\leq n}(Z), $$ 
where $D^{\leq n}$ is the subcategory of objects whose cohomologies
are concentrated in degrees $\leq n$. The functor $\tau_{\leq n}$
is the right adjoint to the natural embedding 
$$ i_{\leq n}\colon D^{\leq n}(Z) \rightarrow D(Z)$$ 
and for any $Q \in D(Z)$ we have
\begin{equation*}
H^i(\tau_{\leq n} Q) \simeq 
\begin{cases}
H^i(Q) & \quad i \leq n \\
0 & \quad \text{ otherwise},  
\end{cases}
\end{equation*}
see 
\cite[\S1.10]{Lipman-NotesOnDerivedFunctorsAndGrothendieckDuality}
or 
\cite[IV.\S4.5]{GelfandManin-MethodsOfHomologicalAlgebra}. Let now
$$ Q_n = R_k F_k \in D(Z \times Z)$$
and define $Q_0, \dots, Q_{n-1}$ inductively by 
$$ Q_i = \tau_{\leq 2i} Q_{i+1} \simeq \tau_{\leq 2i} Q_n. $$
The $(i_{\leq n}, \tau_{\leq n})$ adjunction counits
$Q_i \rightarrow Q_{i+1}$ are termwise injections
of underlying complexes. They form the standard  filtration  on $R_k F_k$ 
\begin{equation}
\label{eqn-the-standard-filtraton-on-RkFk}
0 \rightarrow Q_0 \rightarrow Q_1 \rightarrow \dots \rightarrow Q_{n-1} \rightarrow
Q_n \simeq R_k F_k, 
\end{equation}
whose factors, by the above desription of the cohomology sheaves of $R_k F_k$,
are 
$$ \id_Z = \Delta_* \mathcal{O}_Z, H = \Delta_* \mathcal{O}_Z[-2],
\dots, H^n \simeq \Delta_* \mathcal{O}_Z[-2n]. $$
In particular, the filtration \eqref{eqn-the-standard-filtraton-on-RkFk} gives
$R_k F_k$ the structure of a cyclic coextension of $\id$ by $H$ of
degree $n$. 

We need the following technical result:
\begin{lemma}
\label{lemma-cohomology-sheaves-of-Ql-Qm}
For any $0 \leq l,m \leq n$ we have a natural isomorphism 
\begin{equation}
\label{eqn-cohomology-sheaves-of-Ql-Qm}
H^i(Q_l Q_m) \simeq \bigoplus_{i = p + q} H^p(Q_l) H^q(Q_m),
\end{equation}
which for $l = m = n$ coincides with the isomorphism 
\eqref{eqn-cohomology-sheaves-of-RkFkRkFk}. 
\end{lemma}
\begin{proof} 
We begin by showing that the truncation functors commute with with 
the functor $W$ used in the proof of Proposition 
\ref{prps-enhancing-monad-RkFk-and-comonad-FkRk}. For this we break up 
$W$ into two: define  
$$ W_1 = (\id_Z \times \pi)_* 
\left(\pi_2^* \mathcal{O}_P(-k) \otimes (-)\right) $$
$$ W_2 = (\pi \times \id_P)_* 
\left(\pi_1^* \mathcal{O}_P(k) \otimes (-)\right), $$
and note that $W \simeq W_1 W_2$.
By Lemma \ref{lemma-pullbacks-and-pushforwards-of-Fourier-Mukai-kernels}
we have 
$$ R_k F_k \simeq W_1 (I^! F_k) \simeq W_2 (R_k I_*), $$
and since $W_1$ and $W_2$ applied to degree $k$ cohomology sheaves 
of $I^! F_k$ and  $R_k I_*$ yields sheaves which are shifted by $k$ to the right, 
we have by the usual spectral sequence argument: 
\begin{equation*}
Q_l = \tau_{\leq 2l} (R_k F_k) \simeq  \tau_{\leq 2l} W_1 (I^! F_k) 
\simeq W_1\left( \tau_{\leq l} \left( I^! F_k \right) \right),
\end{equation*}
\begin{equation*}
Q_m = \tau_{\leq 2m} (R_k F_k) \simeq  \tau_{\leq 2m} W_2 (R_k I_*) 
\simeq W_2\left( \tau_{\leq m} \left( R_k I_* \right) \right).
\end{equation*}

It follows that there exists an isomorphism
\begin{equation}
\label{eqn-Ql-Qm-via-truncations-of-I^!Fk-and-RkI_*}
Q_l Q_m \xrightarrow{\sim} 
W\left(\tau_{\leq l}(I^! F_k) \tau_{\leq m}(R_k I_*)\right), 
\end{equation}
which makes the following square commute
\begin{equation}
\begin{tikzcd}
Q_l Q_m  
\ar{d}{\iota \iota}
\ar{r}{\eqref{eqn-Ql-Qm-via-truncations-of-I^!Fk-and-RkI_*}}[']{\sim}
&
W\left(\tau_{\leq l}(I^! F_k) \tau_{\leq m}(R_k I_*)\right)
\ar{d}
\\
Q_n Q_n 
\ar{r}{\eqref{eqn-RkFkRkFk-via-I^!FkRkI*}}[']{\sim}
& 
W\left(I^! F_kR_k I_*\right).
\end{tikzcd}	
\end{equation}

Recall that in the proof 
of Proposition \ref{prps-enhancing-monad-RkFk-and-comonad-FkRk} 
we've obtained the formula 
\eqref{eqn-cohomology-sheaves-of-RkFkRkFk}
for the cohomology sheaves of $R_k F_k R_k F_k$ by establishing 
an isomorphism 
\begin{equation*}
H^i(I^! F_k R_k I_*)
\underset{\sim}{\xrightarrow{\eqref{eqn-cohomology-sheaves-of-I^!FkRkI_*}}}
z_*
\bigl(
\bigoplus_{i = p+q}
\Omega^p_{P/Z}(-k) \boxtimes \Omega^q_{P/Z}(k) 
\bigr),
\end{equation*}
and feeding it through the functor $W$ via a 
K{\"u}nneth isomorphism and spectral sequence argument 
to obtain 
\begin{equation*} 
H^{2i} (R_k F_k R_k F_k) \simeq 
\left(
\bigoplus_{{i = p + q,}} 
\Delta_* \left(\pi_* \Omega^p_{P/Z} \otimes \pi_* \Omega^q_{P/Z}\right)
\right).
\end{equation*}
Finally, applying the isomorphisms 
$H^{2i}(R_k F_k) \xrightarrow{\eqref{eqn-2i-th-cohomology-of-RkFk}}
H^i(\Delta_* \pi_* \Omega^i_{P/Z})$
we obtain a natural isomorphism 
\begin{equation} 
\label{eqn-cohomologies-of-RkFkRkFk-via-those-of-RkFk-and-RkFk}
H^{2i} (R_k F_k R_k F_k) \simeq 
\left(
\bigoplus_{i = p + q} 
H^{2p}(R_k F_k) H^{2q}(R_k F_k) 
\right). 
\end{equation}
We now claim that the isomorphism
\eqref{eqn-cohomology-sheaves-of-I^!FkRkI_*}
restricts to an isomorphism 
\begin{equation}
\label{eqn-cohomology-sheaves-of-truncated-I^!FkRkI_*}
H^i\left(\tau_{\leq l}(I^! F_k)\tau_{\leq m}(R_k I_*)\right)
\xrightarrow{\sim}
z_*
\bigl(
\bigoplus_{{i = p + q,}} 
\Omega^p_{P/Z}(-k) \boxtimes \Omega^q_{P/Z}(k) 
\bigr). 
\end{equation}
The assertion of the lemma then follows immediately by feeding 
\eqref{eqn-cohomology-sheaves-of-truncated-I^!FkRkI_*}
through the identifications above to obtain the truncated version of  
\eqref{eqn-cohomologies-of-RkFkRkFk-via-those-of-RkFk-and-RkFk}, 
as desired. 

For the claim, it suffices to prove that the natural map 
of sheaves on $P \times P$
$$ H^i\left(\tau_{\leq l}(I^! F_k)\tau_{\leq m}(R_k I_*)\right)
\rightarrow 
H^i\left(I^! F_k R_k I_*\right) $$
is injective and has the right image specified in 
terms of the isomorphism \eqref{eqn-cohomology-sheaves-of-I^!FkRkI_*}.
Thus the whole question is local on $P \times P$. 
On the other hand, the isomorphism 
\eqref{eqn-cohomology-sheaves-of-I^!FkRkI_*} was obtained
by identifying $H^i(I^! F_k R_k I_*)$ with 
the twisted exterior powers $\wedge^i \mathcal{E}''(-k,-k)$
of the excess bundle $\mathcal{E}''$ of the intersection 
of $P \times P$ and $P \times_Z P$ in $X \times X$. 
The excess bundle $\mathcal{E}''$ lives on $P \times_Z P$
and splits naturally as 
$$ \mathcal{E}'' \simeq \pi_1^* \mathcal{N}_{P/X} \oplus \pi_2^*
\mathcal{N}_{P/X}. $$  
Locally, these two direct summands are generated by (the duals of) 
the common generators of the ideal sheaves of $\pi_{12}^* I^!$ and 
$\pi_{23}^* F_k$ and of the ideal sheaves of $\pi_{34}^* R_k$ and
$\pi_{45}^* I_*$ on $P \times X \times Z \times X \times P$ 
restricted to $P \times_Z P$ via the embedding
$$ P \times_Z P \xrightarrow{(\id,\iota,\pi) \times_Z (\pi, \iota, \id)}
P \times X \times Z \times X \times P. $$
It follows that, locally, truncating $I^! F_k$ and $R_k I_*$ 
at cohomology degrees $\leq l$ and $\leq m$, respectively,
corresponds to truncating the direct sum decomposition 
$$ \wedge^i \mathcal{E}'' \simeq \bigoplus_{i = p + q}
\wedge^p (\pi_1^*  \mathcal{N}_{P/X}) \otimes \wedge^q 
(\pi_2^*  \mathcal{N}_{P/X}) $$
at the exterior powers $\leq l$ of $(\pi_1^*  \mathcal{N}_{P/X})$
and $\leq m$ of $\pi_2^* \mathcal{N}_{P/X}$. 
Since 
$$ \wedge^p (\pi_1^*  \mathcal{N}_{P/X}) \otimes \wedge^q 
(\pi_2^*  \mathcal{N}_{P/X}) \otimes \mathcal{O}_{P \times_Z P}(-k,k) 
\simeq \Omega^p_{P/X}(-k) \boxtimes \Omega^q_{P/X}(k), $$
our claim follows. 

\end{proof}

We now prove the main theorem of this section: this cyclic coextension 
structure makes $F_k$ into a $\mathbb{P}^n$-functor, and if the Mukai
flop of $X$ exists, the $\mathbb{P}$-twist of $F_k$ is its derived
monodromy: 

\begin{theorem}
\label{theorem-family-p-twists}
Let $Z$ be a smooth projective variety and let $\mathcal{V}$ be a vector bundle
of rank $n + 1$ on $Z$. Let $P$ be the projectification of $\mathcal{V}$ and
let $\pi \colon P \twoheadrightarrow Z$ be the corresponding $\mathbb{P}^n$-fibration. Let $X$ be 
another smooth projective variety and let $\iota \colon P \hookrightarrow X$ be a codimension $n$
closed immersion with $\mathcal{N}_{P/X} \simeq \Omega^1_{P/Z}$. 

Let 
$$ f_k \quad\overset{\text{def}}{=}\quad
\iota_*\left(\mathcal{O}_{P}(k) \otimes \pi^*(-)\right) \colon 
D(Z) \rightarrow D(X), $$ 
let $r_k$ and $l_k$ be its right and left adjoints, and let
$$h \quad\overset{\text{def}}{=}\quad \id_Z[-2] 
\colon D(Z) \rightarrow D(Z).$$ 
Let $F_k$, $R_k$, $L_k$, and $H$ be their standard
enhancements as described in \S \ref{section-standard-fourier-mukai-kernels-and-the-key-lemma}. 

The structure of a cyclic coextension of $\id$ by $H$ of degree $n$
on the adjunction monad $R_k F_k$ provided by the filtration 
\eqref{eqn-the-standard-filtraton-on-RkFk} makes $F_k$ into a $\mathbb{P}^n$-functor. 

Let $P_{F_k}$ be the $\mathbb{P}$-twist of $F_k$. 
If the Mukai flop 
$X \xleftarrow{\beta} \tilde{X} \xrightarrow{\gamma} X'$ exists 
we have an isomorphism in $D(X \times X)$:
\begin{equation}
\label{eqn-flop-flop-equals-twist}	
KN'_{-k} \circ KN_{n + k + 1} \simeq P_{F_k},
\quad \quad \quad \text{`` flop-flop $=$ twist ''} 
\end{equation}
where $KN_{n + k + 1}$ and $KN'_{-k}$ are Kawamata-Namikawa 
derived flop equivalences $D(X) \rightleftarrows D(X')$, 
cf.~\S\ref{section-family-the-setup}. 
\end{theorem}

\begin{proof}
The condition that the autoequivalence $h$ preserves $\krn f_k$ is
trivially fulfilled since $h = [-2]$. We next 
observe that the $\ext^{-1}$-vanishing condition 
\eqref{eqn-the-minus-one-ext-assumption} holds in our setup 
since for any $1 \leq i \leq n$
\begin{align*}
\homm^{-1}_{D(Z \times Z)}(\id_Z, H^i) = 
\homm^{-1}_{D(Z \times Z)}(\Delta_* \mathcal{O}_Z, \Delta_*
\mathcal{O}_Z[-2i]) \simeq 
\ext^{-2i-1}_{Z \times Z}(\Delta_* \mathcal{O}_Z, \Delta_*
\mathcal{O}_Z) = 0,
\end{align*}
as there are no negative $\ext$s between sheaves. 
Thus, we are once again in the situation where 
Theorem 
\ref{theorem-strong-monad-and-weak-adjoints-imply-pn-functor}
applies and it suffices to show 
that the strong monad condition and the weak adjoints condition hold: 

\begin{enumerate}
\item \em Strong monad condition: \rm 
We need to show that for any $0 \leq k \leq n-1$ the composition
\begin{equation}
\label{eqn-Q_1-Q_k-to-Q_n-monad-multiplication-map}
Q_1 Q_k \xrightarrow{\iota \iota}  Q_n Q_n \xrightarrow{m} Q_n 
\end{equation}
filters through $Q_{k + 1} \xrightarrow{\iota} Q_n$ 
and that the resulting map 
$$ Q_1 Q_k \rightarrow Q_{k + 1} $$
descends to an isomorphism 
$$ H H^{k} \xrightarrow{\sim} H^{k+1}. $$ 

Now since $Q_{k+1} \simeq \tau_{\leq 2(k+1)} R_k F_k$ we have 
an exact triangle
$$ Q_{k+1} \xrightarrow{\iota} R_k F_k
\rightarrow \tau_{\geq 2(k+1) + 1} R_k F_k \rightarrow Q_{k+1}[1]. $$
The cohomology sheaves of $\tau_{\geq 2(k+1) + 1} R_k F_k$
are concentrated in degrees $\geq 2(k+1) + 1$. By 
Lemma \ref{lemma-cohomology-sheaves-of-Ql-Qm}
the cohomology sheaves of $Q_1 Q_k$ are concentrated in degrees $\leq
2(k+1)$. Therefore 
$$ \homm_{D(Z \times Z)}(Q_1 Q_k, \tau_{\geq 2(k+1) + 1} R_k F_k) = 0, $$
and thus any map $Q_1 Q_k \rightarrow Q_n$ filters through $Q_{k+1}
\xrightarrow{\iota} Q_n$. 

On the other hand, the maps $Q_i \rightarrow H^i$ are simply
projections onto the rightmost (degree $2i$) cohomology. Therefore
by Lemma \ref{lemma-cohomology-sheaves-of-Ql-Qm} 
$Q_1 Q_k \rightarrow H^1 H^k$ is also the projection of $Q_1 Q_k$
onto its rightmost (degree $2(k+1)$) cohomology. Since the natural 
map $Q_{k+1} \rightarrow Q_n$ is an isomorphism on the degree $2(k+1)$
cohomology, it remains to show that the map 
\begin{equation}
\label{eqn-map-Q1Qk-to-Q_n-on-degree-2(k+1)-cohomology}
H^{2(k+1)}(Q_1 Q_k) 
\xrightarrow{H^{2(k+1)}(\iota\iota)}
H^{2(k+1)}(Q_n Q_n)
\xrightarrow{H^{2(k+1)}(m)} 
H^{2(k+1)}(Q_n)
\end{equation}
is an isomorphism. By Lemma \ref{lemma-cohomology-sheaves-of-Ql-Qm}
again, we have the natural direct sum decomposition 
$$ H^{2(k+1)}(Q_n Q_n) \simeq \bigoplus_{k+1 = p + q} H^{2p}(Q_n)
H^{2q}(Q_n) \simeq \Delta_* \mathcal{O}_Z^{\oplus(k+1)} $$
which identifies the first composant of 
\eqref{eqn-map-Q1Qk-to-Q_n-on-degree-2(k+1)-cohomology}
with the inclusion of a direct summand 
$$ H^{2}(Q_1) H^{2k}(Q_n) \hookrightarrow \bigoplus_{k+1 = p + q} 
H^{2p}(Q_n) H^{2q}(Q_n). $$
Moreover, this direct sum decomposition coincides with
the direct sum decomposition
\eqref{eqn-cohomology-sheaves-of-RkFkRkFk}.
In Proposition \ref{prps-enhancing-monad-RkFk-and-comonad-FkRk}
\eqref{item-RkFkRkFk-and-the-monad-multiplication}
we've computed the second composant of 
\eqref{eqn-map-Q1Qk-to-Q_n-on-degree-2(k+1)-cohomology}, the monad
multiplication, to be identified by the direct sum decomposition
\eqref{eqn-cohomology-sheaves-of-RkFkRkFk}
$$ H^{2(k+1)}(Q_n Q_n) \simeq \Delta_* \mathcal{O}_Z^{\oplus(k+1)} $$
and the isomorphism \eqref{eqn-cohomology-sheaves-of-RkFk}
$$ H^{2(k+1)}(Q_n) \simeq \Delta_* \mathcal{O}_Z $$
with the sum of the identity maps. It follows that 
$\eqref{eqn-map-Q1Qk-to-Q_n-on-degree-2(k+1)-cohomology}$ is a 
composition of a direct summand inclusion with a sum of isomorphisms
and thus is an isomorphism. 

\item \em Weak adjoints condition: \rm 

In \S\ref{section-family-pn-twist-enhanced-adjunction-monads-and-comonads}
we have established the isomorphisms \eqref{eqn-L_k-explicitly}
and \eqref{eqn-R_k-explicitly}:
\begin{align*}
L_k & \simeq (\iota,\pi)_* \mathcal{O}_{P}(-n-k-1)[n],
\\
R_k & \simeq (\iota,\pi)_* \mathcal{O}_{P}(-n-k-1)[-n].  
\end{align*}
Thus, as desired, we have 
$$ R_k \simeq L_k[-2n] \simeq H^n L_k.$$
\end{enumerate}

Finally, if the Mukai flop 
$X \xleftarrow{\beta} \tilde{X} \xrightarrow{\gamma} X'$ exists, 
the ``flop-flop = twist'' formula \eqref{eqn-flop-flop-equals-twist}
can be established using the beautiful method of
\cite[Prop.~4.6 and
4.8]{AddingtonDonovanMeachan-MukaiFlopsAndPTwists}. It goes through
completely unchanged, even though our $\mathbb{P}^n$-functor is not 
split and we do not assume the Hochschild cohomology vanishing
$HH^{\odd}(Z) = 0$. For the benefit of the reader we recall the main 
steps. 

First, the formula \eqref{eqn-flop-flop-equals-twist} is 
established in the ``local model'' case where $X = \Omega^1_{P/Z}$. 
In that case, our Mukai flop is a hyperplane section of an Atiyah flop
$\mathcal{X} \leftarrow \tilde{\mathcal{X}} \rightarrow
\tilde{\mathcal{X}}'.$
For Atiyah flops, the ``flop-flop = twist'' formula is established 
in \cite[\S2]{AddingtonDonovanMeachan-MukaiFlopsAndPTwists}
without assuming the Hochschild cohomology vanishing or that 
the corresponding spherical functor $\mathcal{F}_k$ is split. 
Thus it also goes through in our case. Then, we note as in 
\cite[Prop.~4.6]{AddingtonDonovanMeachan-MukaiFlopsAndPTwists}
that the spherical functor $\mathcal{F}_k$ is the composition 
$S_* F_k$ where $s$ is the inclusion $X \hookrightarrow \mathcal{X}$,
and that $S_*$ intertwines the $\mathbb{P}^n$-twist $P_{F_k}$ and
the spherical twist $T_{\mathcal{F}_k}$. This is deduced by exploiting
the fact that $F_k$ is a family of $\mathbb{P}^n$-objects over $Z$:
$F_k$ and $\mathcal{F}_k$ restrict in each fibre over each 
point of $Z$ to a $\mathbb{P}^n$-object and a spherical object
for which the claim was already established in  
\cite[Prop.~1.4]{HuybrechtsThomas-PnObjectsAndAutoequivalencesOfDerivedCategories}.
Since $S_*$ also intertwines Kawamata-Namikawa flopping equivalences
with the Bondal-Orlov flopping equivalences, the ``flop-flop = twist''
formula for $F_k$ follows from that for $\mathcal{F}_k$ by observing
that the only endofunctor of $D(X)$ which $S_*$ intertwines with
$\id_\mathcal{X}$ is $\id_X$. Finally, 
to obtain the ``flop-flop = twist'' formula for general $X$, we
use the deformation to the normal cone argument of 
\cite[Prop.~4.8]{AddingtonDonovanMeachan-MukaiFlopsAndPTwists}, 
which also never uses either the Hoschschild cohomology vanishing
or the fact that the $\mathbb{P}^n$-functor is split. 
\end{proof}

\vspace{1cm}


\appendix

\section{$\mathbb{P}^n$-functors and truncated twisted tensor algebras: 
a vector space example}
\begin{center}
by Rina Anno, Timothy Logvinenko, and Sophia Restad
\end{center}
~\newline 

\newcommand{\pair}[1]{\left(#1\right)} 
\newcommand{\set}[1]{\left\{#1\right\}} 
\newcommand{\chev}[1]{\langle#1\rangle} 
\newcommand{\abs}[1]{\left|#1\right|}
\newcommand{\norm}[1]{\left|\left|#1\right|\right|}
\renewcommand{\root}[1]{\sqrt{#1}}
\newcommand{\defbold}[1]{\textnormal{\textbf{#1}}} 
\newcommand{\cat}[1]{\textnormal{\textbf{#1}}} 
\newcommand{\der}[2]{\dfrac{\partial #1}{\partial #2}}

\subsection{Conjecture}
Since the $\mathbb{P}^n$-structure on a functor $F$ is a collection of additional data, the next logical step is investigating 
various $\mathbb{P}^n$-structures on a given functor. In particular, since the $\mathbb{P}$-twist only depends on
the autoequivalence $H$, the map $\sigma: H[-1] \to \A$, and the map 
$$
\gamma_1: \left\{H\xrightarrow{\sigma}\A\right\} \to RF,
$$
one can ask whether a $\mathbb{P}^n$-structure can be reconstructed from this data.
It is clear (see example below) that a $\mathbb{P}^n$-structure with this data may not be unique. We conjecture, however,
that there may be a distinguished $\mathbb{P}^n$-structure with fixed $H$, $\sigma$, and $\gamma_1$,
for which the cyclic co-extension structure on $RF$ is determined by $\sigma$.

\begin{conj}
Let $F: D(\A)\to D(B)$ be an enhanceable functor. Suppose that $(H, Q_n, \gamma)$ is a $\mathbb{P}^n$-structure on $F$.
Let $\sigma: H[-1] \to \A$ be the map that makes $Q_1$ a co-extension of $\id$ by $H$. Then there is a $\mathbb{P}^n$-structure
on $F$ with the same $H$, $\sigma$, and $\gamma_1: Q_1 \to RF$ (thus producing the same $\mathbb{P}$-twist) such that
$Q_n$ is isomorphic to the truncated twisted tensor algebra as an object of $D(\AbimA)$.
\end{conj}

The theorem below shows that this conjecture holds when $\A=k$ and $Q_n$ is split. Note that we cannot require $Q_n$ to be isomorphic to
the truncated twisted tensor algebra as an algebra, since for $\A=k$ that would mean that the algebra $RF$ is isomorphic to
$k[x]/(x^{n+1})$, while it is easy to show that for any polynomial $p(x)$ of degree $n+1$ the algebra $k[x]/(p(x))$ will work.
Another example to consider is extension of a $\mathbb{P}^n$-functor by zero 
(see \S \ref{section-extensions-by-zero})
where the strong monad condition does not hold, as it would were $Q_n$ the truncated twisted tensor algebra.

Let $k$ be a field and let $\A=k$ as a DG algebra concentrated in degree $0$. An autoequivalence $H$ then must be a shift.
Suppose $(H,Q_n,\gamma)$ is a $\mathbb{P}^n$-structure on a functor $F$.
If $H=[m]$ with $m\ne 0, 1$, the cyclic coextension $Q_n$ has to split and the algebra structure on it has to be that of 
$k[x]/x^{n+1}$ with $\mathrm{deg}\, x=m$. Consider the case where $H=k$ is the shift by $0$. Then
$Q_n$ is split as a cyclic coextension and as a $k$-algebra it is $n$-dimensional and concentrated in degree $0$.
Let $h\in Q_n$ be any non-zero element in $H\subset Q_n$.

Apply $R$ to the monad map $\nu: FHQ_{n-1} \to FJ_n$ and identify $RF$ and $Q_n$ via $\gamma$. We get the map
$$
Q_{n-1} \otimes H \otimes Q_n 
\xrightarrow{}
Q_n \otimes H \otimes Q_n
\simeq
Q_n \otimes Q_n
\xrightarrow{(-)h \otimes (-) - (-) \otimes h(-)}
Q_n \otimes Q_n
\xrightarrow{}
J_n \otimes Q_n.
$$
Since this is an isomorphism, the rank of the map in the middle has to be at least $n^2-n$. The theorem below then implies that
in this case, too, the algebra $Q_n$ has to be isomorphic to $k[x]/p(x)$, where $p(x)$ is a polynomial of degree $n+1$
where $x$ is identified with $h$. 
Then we can choose a basis $\{1,h,\ldots, h^n\}$ of $Q_n$ in which the monad multiplication satisfies the strong monad condition.

\subsection{Finite dimensional algebra over a field}

\begin{theorem*}
    Let $k$ be a field, and let $Q$ be a $k$-algebra of finite dimension $n$.
    Choose $h \in Q$, and let
        \[f_h: Q \otimes_kQ \to Q\otimes_K Q\]
    be given by
        \[f_h: q_1 \otimes q_2 \mapsto q_1h \otimes q_2 - q_1 \otimes hq_2.\]
    Then $\dim\ker f_h \leq n \iff Q \simeq k[x]/(p(x)), h \leftrightarrow x$, for some polynomial $p(x)$ of degree $n$.
\end{theorem*}
\begin{proof}
($\impliedby$)
    Let $p(x) = x^n + \sum_{i=1}^{n-1}a_ix^i$ be the minimal polynomial that represents multiplication by $x$,
    and let $A$ be the $n \times n$ companion matrix of $p(x)$.
    Explicitly,
        \[A = \begin{pmatrix}
            0 & & & -a_0 \\
            1 & \ddots & & -a_1 \\
              & \ddots & 0 & \vdots \\
              & & 1 & -a_{n-1}  
            \end{pmatrix}\]
Consider the basis $\set{1 \otimes 1, x \otimes 1, \ldots, x^{n-1}
\otimes 1, 1 \otimes x, \ldots, x^{n-1} \otimes x^{n-1}}$ for $Q \otimes Q$.  
With respect to it 
        \[f_h = A \otimes I - I \otimes A =
        \begin{pmatrix}
            A & & & a_0\cdot I \\
            -I & \ddots & & a_1\cdot I \\
               & \ddots & A & \vdots \\
               & & -I & A + a_{n-1}\cdot I
        \end{pmatrix}\]
where $\otimes$ denotes the Kronecker product. It remains to find the rank 
of matrix $A \otimes I - I \otimes A$.
The $n-1$ blocks of negative identity matrices on the subdiagonal
    ensure that the rank is at least $n^2 - n$,
    so $\dim\ker f_h \leq n$.
    To show equality, we use row reduction on the $n \times n$ blocks
    mapping the top ``row'' $R_1$ to $R_1 + AR_2 + A^2R_3 + \cdots  + A^{n-1}R_n$.
    The top ``row'' becomes all zero blocks until the very last block, which will contain the entry
        \[A^n + \sum_{i=0}^{n-1}a_i A^i = p(A).\]
    But $p$ is the characteristic polynomial of $A$; therefore $p(A)$ is the zero matrix,
    and $\dim\ker f_h = n$.
    
($\implies$)
    We first construct a basis of $Q$ that is conducive to
multiplication by $h$. As $Q$ is finite dimensional, there is
a linear dependence for the powers of $h$. Take the 
minimal polynomial representing it 
        \[p(x) = x^{m} + \sum_{i=0}^{m-1}a_{i}x^i\]
where $ m \leq n$. 
If $m = n$ then $\set{1, h, \ldots, h^n}$ is a basis for $Q$ and the proof is complete.
    
    Assume $m < n$, and, for now, that $Q$ is commutative.
    For the purposes of continuing,
    we will relabel $m := m_0, \: p(x) := p_0(x)$, and $A := A_0$.
    So far we have constructed a basis for the subalgebra $\chev{h}$ of $Q$, where
        \[\chev{h} \simeq k[x]/(p_0(x)).\]
    To extend this basis, choose $q_1 \in Q \setminus \chev{h}$, and consider the vector space $q_1\chev{h}$.
    We seek an element that behaves like a minimal polynomial representing the linear dependence of $\set{q_1, q_1h, q_1h^2, \ldots}$.
    Now, $I_1 = \set{ y \in\chev{h} \mid q_1y = 0}$ is an ideal in the principal ideal ring $\chev{h}$,
    and may be generated by a single element.
    Let $p_1(x)$ be the monic polynomial of minimal degree such that $(p_1(h)) = I_1$.
    The ideals of $K[x]/(p_0(x))$ are precisely those generated by polynomials that divide $p_0(x)$, so we may assume that $p_1 \mid p_0$.
    
    Proceeding in this manner, we inductively choose $q_j \in Q\setminus (\chev{h} \oplus q_1\chev{h} \oplus \cdots \oplus q_{j-1}\chev{h})$,
    with corresponding polynomials $p_j$ as constructed above, and companion matrices $A_j$.
    Similar to before, we have for each $i$ that $p_i \mid p_0$.
    This gives the basis
    $\set{1, h, h^2, \ldots, h^{m_0-1}, q_1, q_1h, \ldots, q_kh^{m_k-1}},$
    with
        \[A =
        \begin{pmatrix}
            A_0 & & \\
             & \ddots & \\
             & & A_k
        \end{pmatrix} \]
    representing multiplication by $h$ for all of $Q$,
    where $\sum_{i=0}^k m_i = n$.
    For $Q \otimes_K Q$, choose the basis
        \[\set{1 \otimes 1, h \otimes 1, \ldots, q_kh^{m_k-1} \otimes 1, 1 \otimes h, \ldots, q_kh^{m_k-1} \otimes q_kh^{m_k-1}},\]
    under which the map $f_h$ is represented by the matrix
        \[A \otimes I_n - I_n \otimes A = 
        \begin{pmatrix}
            A \otimes I_{m_0} - I_n \otimes A_0 & & \\
            & \ddots & \\
            & & A \otimes I_{m_k} - I_n \otimes A_k
        \end{pmatrix}.\]
    Clearly $\dim\ker f_h = \sum_{i=0}^k \operatorname{null} (A \otimes I_{m_i} - I_n \otimes A_i)$.
    For each $0 \leq i \leq k$,
        \[A \otimes I_{m_i} - I_n \otimes A_i = 
        \begin{pmatrix}
            A & & & a_{i,0}\cdot I \\
            -I & \ddots & & a_{i,1}\cdot I \\
               & \ddots & A & \vdots \\
               & & -I & A + a_{i,m_i-1}\cdot I
        \end{pmatrix} \mapsto \begin{pmatrix}
        0 & & & p_i(A) \\
        -I & \ddots & & a_{i,1}\cdot I \\
        & \ddots & 0 & \vdots \\
        & & -I & A + a_{i,m_i-1}\cdot I
        \end{pmatrix}\]
    by the same row reduction process outlined above.
    The null space of the resulting matrix is thus at least $m_i$-dimensional.
But as $p_i \mid p_0$ for each $i$, $p_0(A)$ is the zero matrix and $\dim\ker A\otimes I_{m_0} - I_n \otimes A_0 = n$.
    Hence 
        \[\dim\ker f_h = \sum_{i=0}^k \operatorname{null} (A\otimes I_{m_i} - I_n \otimes A_i) = \sum_{i=0}^k \operatorname{null}p_i(A) \geq n + \sum_{i=1}^k m_i > n.\]
    
    For the noncommutative case we construct bases
representing right and left multiplication by $h$ 
        \[\begin{aligned}
        \set{1, h, h^2, \ldots, h^{m_0-1}, q_1, q_1h, \ldots, q_kh^{m_k-1}}, \textnormal{ and }
        \set{1, h, h^2, \ldots, h^{m_0-1}, \tilde{q}_1, h\tilde{q}_1, \ldots, h^{\tilde{m}_1-1}\tilde{q}_1, \tilde{q}_2, \ldots, h^{\tilde{m}_l-1}\tilde{q}_l},
        \end{aligned}\]
    with similar corresponding minimal polynomials $\tilde{p}_j(x)$ of degree $\tilde{m}_j$, and companion matrices $\tilde{A}_j$,
    noting that $\tilde{p_0} = p_0$, and $\tilde{p}_j$ divides $\tilde{p}_0$ for each $1 \leq j \leq l$,
    and that $\sum_{i=0}^k m_i = \sum_{j=0}^l \tilde{m}_j = n$.
    Then 
        \[f_h = A \otimes I - I \otimes \tilde{A} =
            \begin{pmatrix}
                A \otimes I_{\tilde{m_0}} - I_n \otimes \tilde{A}_0 & & \\
                & \ddots & \\
                & & A \otimes I_{\tilde{m}_l} - I_n \otimes \tilde{A}_l
            \end{pmatrix}.\]
    Once again, $\dim\ker f_h = \sum_{j=0}^l \operatorname{null} (A \otimes I_{\tilde{m}_j} - I_n \otimes \tilde{A_j})$,
    and for each $j$ such that $0 \leq j \leq l$,
        \[A \otimes I_{\tilde{m}_j} - I_n \otimes \tilde{A}_j = 
            \begin{pmatrix}
                A & & & \tilde{a}_{j,0}\cdot I \\
                -I & \ddots & & \tilde{a}_{j,1}\cdot I \\
               & \ddots & A & \vdots \\
               & & -I & A + \tilde{a}_{j,\tilde{m}_j-1}\cdot I
            \end{pmatrix} \mapsto \begin{pmatrix}
                0 & & & \tilde{p}_j(A) \\
                -I & \ddots & & \tilde{a}_{j,1}\cdot I \\
                & \ddots & 0 & \vdots \\
                & & -I & A +  \tilde{a}_{j,\tilde{m}-1}\cdot I
            \end{pmatrix}.\]
    So then $\dim\ker f_h = \sum_{j=0}^l \operatorname{null}\tilde{p}_j(A)$.
    For the case $j = 0$ we have that $\operatorname{null}\tilde{p}_0(A) = n$, and for $j \geq 1$, $\operatorname{null}\tilde{p}_j(A) \geq \tilde{m}_j$ because $\tilde{p}_j | p_0$.
    Therefore
        \[\dim\ker f_h = \sum_{j=0}^l \operatorname{null}\tilde{p}_j(A) \geq n + \sum_{j=1}^l\tilde{m}_j > n.\]
\end{proof}

\bibliography{references}

\providecommand{\bysame}{\leavevmode\hbox to3em{\hrulefill}\thinspace}
\providecommand{\MR}{\relax\ifhmode\unskip\space\fi MR }
\providecommand{\MRhref}[2]{%
  \href{http://www.ams.org/mathscinet-getitem?mr=#1}{#2}
}
\providecommand{\href}[2]{#2}
\begin{thebibliography}{ADM19}

\bibitem[ACH14]{ArinkinCaldararuHablicsek-FormalityOfDerivedIntersectionsAndTheOrbifoldHKRIsomorphism}
Dima Arinkin, Andrei Caldararu, and Marton Hablicsek, \emph{Formality of
  derived intersections and the orbifold {HKR} isomorphism}, arXiv:1412.5233,
  (2014).

\bibitem[Add16]{Addington-NewDerivedSymmetriesOfSomeHyperkaehlerVarieties}
Nicolas Addington, \emph{New derived symmetries of some hyperkaehler
  varieties}, Alg. Geom. \textbf{3} (2016), no.~2, 223--260, arXiv:1112.0487.

\bibitem[ADM16]{AddingtonDonovanMeachan-ModuliSpacesOfTorsionSheavesOnK3SurfacesAndDerivedEquivalences}
Nicolas Addington, Will Donovan, and Ciaran Meachan, \emph{Moduli spaces of
  torsion sheaves on {K}3 surfaces and derived equivalences}, J. London Math.
  Soc. \textbf{93} (2016), no.~3, 846--865, arXiv:1507.02597.

\bibitem[ADM19]{AddingtonDonovanMeachan-MukaiFlopsAndPTwists}
\bysame, \emph{Mukai flops and {P}-twists}, J. Reine Angew. Math. \textbf{748}
  (2019), 227--240, arXiv:1507.02595.

\bibitem[AIL10]{AvramovIyengarLipman-ReflexivityAndRigidityForComplexesIISchemes}
Luchezar~L. Avramov, Srikanth~B. Iyengar, and Joseph Lipman, \emph{Reflexivity
  and rigidity for complexes, {I}{I}: Schemes}, arXiv:1001.3450, 2010.

\bibitem[AL12]{AnnoLogvinenko-OnTakingTwistsOfFourierMukaiFunctors}
Rina Anno and Timothy Logvinenko, \emph{On adjunctions for {F}ourier-{M}ukai
  transforms}, Adv. in Math. \textbf{231} (2012), no.~3--4, 2069--2115,
  arXiv:1004.3052.

\bibitem[AL16]{AnnoLogvinenko-OrthogonallySphericalObjectsAndSphericalFibrations}
\bysame, \emph{Orthogonally spherical objects and spherical fibrations}, Adv.
  in Math. \textbf{286} (2016), 338--386, arXiv:1011.0707.

\bibitem[AL17]{AnnoLogvinenko-SphericalDGFunctors}
\bysame, \emph{Spherical {DG}-functors}, J. Eur. Math. Soc. \textbf{19} (2017),
  2577--2656, arXiv:1309.5035.

\bibitem[AL21]{AnnoLogvinenko-BarCategoryOfModulesAndHomotopyAdjunctionForTensorFunctors}
\bysame, \emph{Bar category of modules and homotopy adjunction for tensor
  functors}, Int. Math. Res. Not. \textbf{2021} (2021), no.~2, 1353--1462,
  arXiv:1612.09530.

\bibitem[AL22]{AnnoLogvinenko-OnUniquenessOfPTwists}
\bysame, \emph{On uniqueness of {$\mathbb{P}$}-twists}, Int. Math. Res. Not.
  \textbf{2022} (2022), no.~14, 10533--10554, arXiv:1711.06649.

\bibitem[Bar22]{Barbacovi-OnTheCompositionOfTwoSphericalTwists}
Federico Barbacovi, \emph{On the composition of two spherical twists}, Int.
  Math. Res. Not. (2022).

\bibitem[BB22]{BodzentaBondal-FlopsAndSphericalFunctors}
Agnieszka Bodzenta and Alexey Bondal, \emph{Flops and spherical functors},
  Compos. Math. \textbf{158} (2022), no.~5, 1125--1187, arXiv:1511.00665.

\bibitem[B{\'e}n67]{Benabou-IntroductionToBicategories}
Jean B{\'e}nabou, \emph{Introduction to bicategories}, Reports of the Midwest
  Category Seminar (Berlin, Heidelberg), Springer Berlin Heidelberg, 1967,
  pp.~1--77.

\bibitem[Ber71]{Berthelot-ImmersionsRegulieresEtCalculDuKDUnSchemaEclate}
Pierre Berthelot, \emph{Immersions r{\'e}guli{\`e}res et calcul du
  ${K}^\bullet$ d'un sch{\'e}ma eclat{\'e}}, Th{\'e}orie des {I}ntersections et
  {T}h{\'e}or{\`e}me de {R}iemann-{R}och (SGA 6), Lecture Notes in Math., no.
  225, Springer-Verlag, 1971, pp.~416--465.

\bibitem[BK90]{BondalKapranov-EnhancedTriangulatedCategories}
Alexei Bondal and Mikhail Kapranov, \emph{Enhanced triangulated categories},
  Mat. Sb. \textbf{181} (1990), no.~5, 669--683.

\bibitem[BKR01]{BKR01}
Tom Bridgeland, Alastair King, and Miles Reid, \emph{The {M}c{K}ay
  correspondence as an equivalence of derived categories}, J. Amer. Math. Soc.
  \textbf{14} (2001), 535--554, math.AG/9908027.

\bibitem[BM14]{BayerMacri-MMPForModuliOfSheavesOnK3sViaWallcrossingNefAndMovableConesLagrangianFibrations}
Arend Bayer and Emanuele Macr{\`{\i}}, \emph{M{MP} for moduli of sheaves on
  {K}3s via wall-crossing: nef and movable cones, {L}agrangian fibrations},
  Invent. Math. \textbf{198} (2014), no.~3, 505--590.

\bibitem[Bon90]{Bondal-RepresentationOfAsssociativeAlgebrasAndCoherentSheaves}
Alexey Bondal, \emph{Representation of asssociative algebras and coherent
  sheaves}, Mathematics of the USSR-Izvestiya \textbf{34} (1990), no.~1, 23.

\bibitem[Bri08]{Bridgeland-StabilityConditionsOnK3Surfaces}
Tom Bridgeland, \emph{Stability conditions on {K}3 surfaces}, Duke Math J.
  \textbf{141} (2008), 241--291, arXiv:math/0307164.

\bibitem[Bri09]{Bridgeland-StabilityConditionsAndKleinianSingularities}
\bysame, \emph{Stability conditions and {K}leinian singularities}, Int. Math.
  Res. Not. \textbf{21} (2009), 4142--4157, arXiv:math/0508257.

\bibitem[Cau12]{Cautis-FlopsAndAboutAGuide}
Sabin Cautis, \emph{Flops and about: a guide}, Derived categories in algebraic
  geometry, EMS Ser. Congr. Rep., Eur. Math. Soc., Z\"urich, 2012,
  arXiv:1111.0688, pp.~61--101.

\bibitem[CK08]{CautisKamnitzer-KnotHomologyViaDerivedCategoriesOfCoherentSheavesISL2Case}
Sabin Cautis and Joel Kamnitzer, \emph{Knot homology via derived categories of
  coherent sheaves {I}, {S}{L}(2) case}, Duke Math J. \textbf{142} (2008),
  no.~3, 511--588, math.AG/0701194.

\bibitem[CW10]{CaldararuWillerton-TheMukaiPairingIACategoricalApproach}
Andrei Caldararu and Simon Willerton, \emph{The {M}ukai pairing, {I}: a
  categorical approach}, New York J. Math. \textbf{16} (2010), 61--98.

\bibitem[Dri04]{Drinfeld-DGQuotientsOfDGCategories}
Vladimir Drinfeld, \emph{{DG} quotients of {DG} categories}, J. Algebra
  \textbf{272} (2004), no.~2, 643--691, arXiv:math/0210114.

\bibitem[DS14]{DonovanSegal-WindowShiftsFlopEquivalencesAndGrassmannianTwists}
Will Donovan and Ed~Segal, \emph{Window shifts, flop equivalences and
  grassmannian twists}, Compositio Mathematica \textbf{150} (2014), no.~6,
  942--978, arXiv:1206.0219.

\bibitem[DW16]{DonovanWemyss-NoncommutativeDeformationsAndFlops}
Will Donovan and Michael Wemyss, \emph{Noncommutative deformations and flops},
  Duke Math. J. \textbf{165} (2016), no.~8, 1397--1474, arXiv:1309.0698.

\bibitem[Efi13]{Efimov-HomotopyFinitenessOfSomeDGCategoriesFromAlgebraicGeometry}
Alexander~I. Efimov, \emph{Homotopy finiteness of some {DG} categories from
  algebraic geometry}, pre-print arXiv:1308.0135, August 2013.

\bibitem[GD63]{Grothendieck-EGA-III-2}
Alexander Grothendieck and Jean Dieudonn{\'e}, \emph{{\'E}l{\'e}ments de
  g{\'e}om{\'e}trie alg{\'e}brique {III}: {\'E}tude cohomologique des faisceaux
  coh{\'e}rents. {S}econde partie.}, Publications math{\'e}matiques de
  l'I.H.{\'E}.S. \textbf{17} (1963), 5--91.

\bibitem[GM03]{GelfandManin-MethodsOfHomologicalAlgebra}
S.I. Gelfand and Yu.~I. Manin, \emph{Methods of homological algebra},
  Springer-Verlag, Berlin Heidelberg, 2003.

\bibitem[Gro95]{Grojnowski-InstantonsAndAffineAlgebrasITheHilbertSchemeAndVertexOperators}
Ian Grojnowski, \emph{Instantons and affine algebras {I}: {T}he {H}ilbert
  scheme and {V}ertex operators}, Math. Res. Lett. \textbf{3} (1995), no.~2,
  275--291, arXiv:alg-geom/9506020.

\bibitem[Hai01]{Haiman-HilbertSchemesPolygraphsAndTheMacdonaldPositivityConjecture}
Mark Haiman, \emph{Hilbert schemes, polygraphs, and the {M}acdonald positivity
  conjecture}, J. Amer. Math. Soc. \textbf{14} (2001), 941--1006.

\bibitem[Har66]{Hartshorne-Residues-and-Duality}
R.~Hartshorne, \emph{Residues and duality}, Springer-Verlag, 1966.

\bibitem[HLS16]{HalpernLeistnerShipman-AutoequivalencesOfDerivedCategoriesViaGeometricInvariantTheory}
Daniel Halpern-Leistner and Ian Shipman, \emph{Autoequivalences of derived
  categories via geometric invariant theory}, Adv. Math. \textbf{303} (2016),
  1264--1299, arXiv:1303.5531.

\bibitem[Hor05]{Horja-DerivedCategoryAutomorphismsFromMirrorSymmetry}
R.~Paul Horja, \emph{Derived category automorphisms from mirror symmetry}, Duke
  Math. J. \textbf{127} (2005), 1--34.

\bibitem[HT06]{HuybrechtsThomas-PnObjectsAndAutoequivalencesOfDerivedCategories}
Daniel Huybrechts and Richard Thomas, \emph{{$\mathbb P$}-objects and
  autoequivalences of derived categories}, Math. Res. Lett. \textbf{13} (2006),
  87--98, arXiv:math/0507040.

\bibitem[IU05]{IshiiUehara-AutoequivalencesOfDerivedCategoriesOnTheMinimalResolutionsOfA_nSingularitiesOnSurfaces}
Akira Ishii and Hokuto Uehara, \emph{Autoequivalences of derived categories on
  the minimal resolutions of ${A}_n$-singularities on surfaces}, J. Diff. Geom.
  \textbf{71} (2005), 385--435, arXiv:math/0409151.

\bibitem[Joh02]{Johnstone-SketchesOfAnElephantAToposTheoryCompendiumV1}
P.T. Johnstone, \emph{Sketches of an elephant: A topos theory compendium:
  volume 1}, Oxford Logic Guides, no.~43, Oxford University Press, 2002.

\bibitem[Kaw02]{Kawamata-DEquivalenceAndKEquivalence}
Yujiro Kawamata, \emph{{$D$}-equivalence and {$K$}-equivalence}, J.
  Differential Geom. \textbf{61} (2002), no.~1, 147--171.

\bibitem[Kel06]{Keller-AInfinityAlgebrasModulesAndFunctorCategories}
Bernhard Keller, \emph{${A}$-infinity algebras, modules and functor
  categories}, Trends in representation theory of algebras and related topics,
  Contemp. Math., vol. 406, Amer. Math. Soc., Providence, RI, 2006,
  arXiv:math/0510508, pp.~67--93.

\bibitem[KL15]{KuznetsovLunts-CategoricalResolutionsOfIrrationalSingularities}
Alexander Kuznetsov and Valery~A. Lunts, \emph{Categorical resolutions of
  irrational singularities}, Int. Math. Res. Notices \textbf{2015} (2015),
  no.~13, 4536--4625, pre-print arXiv:1212.6170.

\bibitem[KM17]{KrugMeachan-UniversalFunctorsOnSymmetricQuotientStacksOfAbelianVarieties}
Andreas Krug and Ciaran Meachan, \emph{Universal functors on symmetric quotient
  stacks of {A}belian varieties}, arXiv:1710.08618, 2017.

\bibitem[Kru15]{Krug-NewDerivedAutoequivalencesOfHilbertSchemesAndGeneralisedKummerVarieties}
Andreas Krug, \emph{On derived autoequivalences of {H}ilbert schemes and
  generalised {K}ummer varieties}, Int. Math. Res. Not. \textbf{2015} (2015),
  no.~20, 10680--10701, arXiv:1301.4970.

\bibitem[Kru18a]{Krug-RemarksOnTheDerivedMcKayCorrespondenceForHilbertSchemesOfPointsAndTautologicalBundles}
\bysame, \emph{Remarks on the derived mckay correspondence for hilbert schemes
  of points and tautological bundles}, Mathematische Annalen \textbf{371}
  (2018), no.~1, 461--486.

\bibitem[Kru18b]{Krug-VarietiesWithPUnits}
\bysame, \emph{Varieties with $\mathbb{P}$-units}, Trans. Amer. Math. Soc.
  \textbf{370} (2018), 7959--7983, arXiv:1604.03537.

\bibitem[Kru19]{Krug-P-functorVersionsOfTheNakajimaOperators}
\bysame, \emph{{$\mathbb P$}-functor versions of the {N}akajima operators},
  Alg. Geom. \textbf{6} (2019), no.~6, 678--715, arXiv:1405.1006.

\bibitem[KS14]{KapranovSchechtman-PerverseSchobers}
Mikhail Kapranov and Vadim Schechtman, \emph{Perverse schobers},
  arXiv:1411.2772, (2014).

\bibitem[KS15]{KrugSosna-OnTheDerivedCategoryOfTheHilbertSchemeOfPointsOnAnEnriquesSurface}
Andreas Krug and Pawel Sosna, \emph{On the derived category of the {H}ilbert
  scheme of points on an {E}nriques surface}, Sel. Math. \textbf{21} (2015),
  no.~4, 1339--1360.

\bibitem[KT07]{KhovanovThomas-BraidCobordismsTriangulatedCategoriesAndFlagVarieties}
Mikhail Khovanov and Richard Thomas, \emph{Braid cobordisms, triangulated
  categories, and flag varieties}, Homology, Homotopy Appl. \textbf{9} (2007),
  no.~2, 19--94, arXiv:math/0609335.

\bibitem[LH03]{Lefevre-SurLesAInftyCategories}
Kenji Lef{\`e}vre-Hasegawa, \emph{Sur les ${A}_\infty$-cat{\'e}gories}, Ph.D.
  thesis, Universit{\'e} Denis Diderot -- Paris 7, 2003, arXiv:math/0310337.

\bibitem[Lip09]{Lipman-NotesOnDerivedFunctorsAndGrothendieckDuality}
Joseph Lipman, \emph{Notes on derived functors and {G}rothendieck duality},
  Foundations of {G}rothendieck duality for diagrams of schemes, Lecture Notes
  in Math., vol. 1960, Springer, Berlin, 2009, pp.~1--259.

\bibitem[LO10]{LuntsOrlov-UniquenessOfEnhancementForTriangulatedCategories}
Valery~A. Lunts and Dmitri~O. Orlov, \emph{Uniqueness of enhancement for
  triangulated categories}, J. Amer. Math. Soc. \textbf{23} (2010), 853--908,
  arXiv:0908.4187.

\bibitem[LS16]{LuntsSchnurer-NewEnhancementsOfDerivedCategoriesOfCoherentSheavesAndApplications}
Valery~A. Lunts and Olaf~M. Schn{\"u}rer, \emph{New enhancements of derived
  categories of coherent sheaves and applications}, J. Algebra \textbf{446}
  (2016), 203--274.

\bibitem[Mea15]{Meachan-DerivedAutoequivalencesOfGeneralisedKummerVarieties}
Ciaran Meachan, \emph{Derived autoequivalences of generalised {K}ummer
  varieties}, Math. Res. Lett. \textbf{22} (2015), no.~4, 1193--1221.

\bibitem[Muk87]{Mukai-OnTheModuliSpaceOfBundlesOnK3SurfacesI}
Shigeru Mukai, \emph{On the moduli space of bundles on ${K}3$ surfaces {I}},
  Vector bundles on algebraic varieties, Oxford University Press, 1987,
  pp.~341--413.

\bibitem[Nak97]{Nakajima-HeisenbergAlgebraAndHilbertSchemesOfPointsOnProjectiveSurfaces}
Hiraku Nakajima, \emph{Heisenberg algebra and hilbert schemes of points on
  projective surfaces}, Annals of Mathematics \textbf{145} (1997), no.~2,
  379--388, arXiv:alg-geom/9507012.

\bibitem[Nam03]{Namikawa-MukaiFlopsAndDerivedCategories}
Yoshinori Namikawa, \emph{Mukai flops and derived categories}, J. Reine Angew.
  Math. \textbf{560} (2003), 65--76, arXiv:math/0203287.

\bibitem[Rou04]{Rouquier-CategorificationOfTheBraidGroups}
Raphael Rouquier, \emph{Categorification of the braid groups},
  arXiv:math/0409593, (2004).

\bibitem[Sca09]{Scala-CohomologyOfTheHilbertSchemeOfPointsOnASurfaceWithValuesInRepresentationsOfTautologicalBundles}
Luca Scala, \emph{Cohomology of the {H}ilbert scheme of points on a surface
  with values in representations of tautological bundles}, Duke Math. J.
  \textbf{150} (2009), no.~2, 211--267.

\bibitem[Sca20]{Scala-NotesOnDiagonalsOfTheProductAndSymmetricVarietyOfASurface}
Luca Scala, \emph{Notes on diagonals of the product and symmetric variety of a
  surface}, J. Pure Appl. Algebra \textbf{224} (2020), no.~8, arXiv:1510.04889.

\bibitem[Seg17]{Segal-AllAutoequivalencesAreSphericalTwists}
Ed~Segal, \emph{{All Autoequivalences are Spherical Twists}}, International
  Mathematics Research Notices \textbf{2018} (2017), no.~10, 3137--3154,
  arXiv:1603.06717.

\bibitem[Sei00]{Seidel-GradedLagrangianSubmanifolds}
Paul Seidel, \emph{Graded {L}agrangian submanifolds}, Bull. Soc. Math. France
  \textbf{128} (2000), 103--149, arXiv:math/9903049v2.

\bibitem[ST01]{SeidelThomas-BraidGroupActionsOnDerivedCategoriesOfCoherentSheaves}
Paul Seidel and Richard Thomas, \emph{Braid group actions on derived categories
  of coherent sheaves}, Duke Math. J. \textbf{108} (2001), no.~1, 37--108,
  arXiv:math/0001043.

\bibitem[Tab05]{Tabuada-UneStructureDeCategorieDeModelesDeQuillenSurLaCategorieDesDG-Categories}
Goncalo Tabuada, \emph{Une structure de cat{\'e}gorie de mod{\`e}les de
  {Q}uillen sur la cat{\'e}gorie des dg-cat{\'e}gories}, C. R. Acad. Sci.
  Paris, Ser. I \textbf{340} (2005), 15--19, arXiv:math/0407338.

\bibitem[Tod07]{Toda-OnACertainGeneralizationOfSphericalTwists}
Yukinobu Toda, \emph{On a certain generalization of spherical twists}, Bull.
  Soc. Math. France \textbf{135} (2007), no.~1, 119--134, arXiv:math/0603050.

\bibitem[To{\"e}07]{Toen-TheHomotopyTheoryOfDGCategoriesAndDerivedMoritaTheory}
Bertrand To{\"e}n, \emph{The homotopy theory of {\em dg}-categories and derived
  {M}orita theory}, Invent. Math. \textbf{167} (2007), no.~3, 615--667,
  arXiv:math/0408337.

\end{thebibliography}
\bibliographystyle{amsalpha}
\end{document}